\documentclass[10pt]{amsart}

\usepackage{times,bm,amsfonts,amsmath,amssymb,mathrsfs,tikz,hyperref,dsfont,
csquotes,enumerate} 
\hypersetup{pdfborder={0000}, colorlinks=true, linkcolor=blue,citecolor=citegreen}
\definecolor{citegreen}{rgb}{0.2,0.2,0.6}
\usepackage[babel=true,kerning=true]{microtype}

\ifx\figforTeXisloaded\relax \else\global\let\figforTeXisloaded=\relax\fi
\catcode`\@=11
\ifx\ctr@ln@m\undefined\else%
    \immediate\write16{*** Fig4TeX WARNING : \string\ctr@ln@m\space already defined.}\fi
\def\ctr@ln@m#1{\ifx#1\undefined\else%
    \immediate\write16{*** Fig4TeX WARNING : \string#1 already defined.}\fi}
\ctr@ln@m\ctr@ld@f
\def\ctr@ld@f#1#2{\ctr@ln@m#2#1#2}
\ctr@ld@f\def\ctr@ln@w#1#2{\ctr@ln@m#2\csname#1\endcsname#2}
{\catcode`\/=0 \catcode`/\=12 /ctr@ld@f/gdef/BS@{\}}
\ctr@ld@f\def\ctr@lcsn@m#1{\expandafter\ifx\csname#1\endcsname\relax\else%
    \immediate\write16{*** Fig4TeX WARNING : \BS@\expandafter\string#1\space already defined.}\fi}
\ctr@ld@f\edef\colonc@tcode{\the\catcode`\:}
\ctr@ld@f\edef\semicolonc@tcode{\the\catcode`\;}
\ctr@ld@f\def\t@stc@tcodech@nge{{\let\c@tcodech@nged=\z@%
    \ifnum\colonc@tcode=\the\catcode`\:\else\let\c@tcodech@nged=\@ne\fi%
    \ifnum\semicolonc@tcode=\the\catcode`\;\else\let\c@tcodech@nged=\@ne\fi%
    \ifx\c@tcodech@nged\@ne%
    \immediate\write16{}
    \immediate\write16{!!!=============================================================!!!}
    \immediate\write16{ Fig4TeX WARNING:}
    \immediate\write16{ The category code of some characters has been changed, which will}
    \immediate\write16{ result in an error (message "Runaway argument?").}
    \immediate\write16{ This probably comes from another package that changed the category}
    \immediate\write16{ code after Fig4TeX was loaded. If that proves to be exact, the}
    \immediate\write16{ solution is to exchange the loading commands on top of your file}
    \immediate\write16{ so that Fig4TeX is loaded last. For example, in LaTeX, we should}
    \immediate\write16{ say :}
    \immediate\write16{\BS@ usepackage[french]{babel}}
    \immediate\write16{\BS@ usepackage{fig4tex}}
    \immediate\write16{!!!=============================================================!!!}
    \immediate\write16{}
    \fi}}
\ctr@ld@f\def\FigforTeX{F\kern-.05em i\kern-.05em g\kern-.1em\raise-.14em\hbox{4}\kern-.19em\TeX}
\ctr@ld@f\def\W@rnmesoldA#1{\W@rnmesold}
\ctr@ld@f\def\W@rnmesoldAB#1(#2){\W@rnmesold}
\ctr@ld@f\def\W@rnmesold{%
    \immediate\write16{}
    \immediate\write16{!!!=============================================================!!!}
    \immediate\write16{ Fig4TeX WARNING:}
    \immediate\write16{ The file to be compiled is not compatible with the current version}
    \immediate\write16{ of Fig4TeX. To fix that, upgrade the source file (mainly change \BS@ ps*}
    \immediate\write16{ macros by \BS@ fig* macros), or use fig4tex184.tex instead (\BS@ input fig4tex184}
    \immediate\write16{ or \BS@ usepackage{fig4tex184}).}
    \immediate\write16{!!!=============================================================!!!}
    \immediate\write16{}}
\ctr@ln@m\psbeginfig\let\psbeginfig\W@rnmesoldA
\ctr@ln@m\psset\let\psset\W@rnmesoldAB
\ctr@ln@m\pssetdefault\let\pssetdefault\W@rnmesoldAB
\ctr@ln@m\pssetupdate\let\pssetupdate\W@rnmesoldA
\ctr@ln@w{newdimen}\epsil@n\epsil@n=0.00005pt
\ctr@ln@w{newdimen}\Cepsil@n\Cepsil@n=0.005pt
\ctr@ln@w{newdimen}\dcq@\dcq@=254pt
\ctr@ln@w{newdimen}\PI@\PI@=3.141592pt
\ctr@ln@w{newdimen}\DemiPI@deg\DemiPI@deg=90pt
\ctr@ln@w{newdimen}\PI@deg\PI@deg=180pt
\ctr@ln@w{newdimen}\DePI@deg\DePI@deg=360pt
\ctr@ld@f\chardef\t@n=10
\ctr@ld@f\chardef\c@nt=100
\ctr@ld@f\chardef\@lxxiv=74
\ctr@ld@f\chardef\@xci=91
\ctr@ld@f\mathchardef\@nMnCQn=9949
\ctr@ld@f\chardef\@vi=6
\ctr@ld@f\chardef\@xxx=30
\ctr@ld@f\chardef\@lvi=56
\ctr@ld@f\chardef\@@lxxi=71
\ctr@ld@f\chardef\@lxxxv=85
\ctr@ld@f\mathchardef\@@mmmmlxviii=4068
\ctr@ld@f\mathchardef\@ccclx=360
\ctr@ld@f\mathchardef\@dccxx=720
\ctr@ln@w{newcount}\p@rtent \ctr@ln@w{newcount}\f@ctech \ctr@ln@w{newcount}\result@tent
\ctr@ln@w{newdimen}\v@lmin \ctr@ln@w{newdimen}\v@lmax \ctr@ln@w{newdimen}\v@leur
\ctr@ln@w{newdimen}\result@t\ctr@ln@w{newdimen}\result@@t
\ctr@ln@w{newdimen}\mili@u \ctr@ln@w{newdimen}\c@rre \ctr@ln@w{newdimen}\delt@
\ctr@ld@f\def\degT@rd{0.017453 }  
\ctr@ld@f\def\rdT@deg{57.295779 } 
\ctr@ln@m\v@leurseule
{\catcode`p=12 \catcode`t=12 \gdef\v@leurseule#1pt{#1}}
\ctr@ld@f\def\repdecn@mb#1{\expandafter\v@leurseule\the#1\space}
\ctr@ld@f\def\arct@n#1(#2,#3){{\v@lmin=#2\v@lmax=#3%
    \maxim@m{\mili@u}{-\v@lmin}{\v@lmin}\maxim@m{\c@rre}{-\v@lmax}{\v@lmax}%
    \delt@=\mili@u\m@ech\mili@u%
    \ifdim\c@rre>\@nMnCQn\mili@u\divide\v@lmax\tw@\c@lATAN\v@leur(\z@,\v@lmax)
    \else%
    \maxim@m{\mili@u}{-\v@lmin}{\v@lmin}\maxim@m{\c@rre}{-\v@lmax}{\v@lmax}%
    \m@ech\c@rre%
    \ifdim\mili@u>\@nMnCQn\c@rre\divide\v@lmin\tw@
    \maxim@m{\mili@u}{-\v@lmin}{\v@lmin}\c@lATAN\v@leur(\mili@u,\z@)%
    \else\c@lATAN\v@leur(\delt@,\v@lmax)\fi\fi%
    \ifdim\v@lmin<\z@\v@leur=-\v@leur\ifdim\v@lmax<\z@\advance\v@leur-\PI@%
    \else\advance\v@leur\PI@\fi\fi%
    \global\result@t=\v@leur}#1=\result@t}
\ctr@ld@f\def\m@ech#1{\ifdim#1>1.646pt\divide\mili@u\t@n\divide\c@rre\t@n\m@ech#1\fi}
\ctr@ld@f\def\c@lATAN#1(#2,#3){{\v@lmin=#2\v@lmax=#3\v@leur=\z@\delt@=\tw@ pt%
    \un@iter{0.785398}{\v@lmax<}%
    \un@iter{0.463648}{\v@lmax<}%
    \un@iter{0.244979}{\v@lmax<}%
    \un@iter{0.124355}{\v@lmax<}%
    \un@iter{0.062419}{\v@lmax<}%
    \un@iter{0.031240}{\v@lmax<}%
    \un@iter{0.015624}{\v@lmax<}%
    \un@iter{0.007812}{\v@lmax<}%
    \un@iter{0.003906}{\v@lmax<}%
    \un@iter{0.001953}{\v@lmax<}%
    \un@iter{0.000976}{\v@lmax<}%
    \un@iter{0.000488}{\v@lmax<}%
    \un@iter{0.000244}{\v@lmax<}%
    \un@iter{0.000122}{\v@lmax<}%
    \un@iter{0.000061}{\v@lmax<}%
    \un@iter{0.000030}{\v@lmax<}%
    \un@iter{0.000015}{\v@lmax<}%
    \global\result@t=\v@leur}#1=\result@t}
\ctr@ld@f\def\un@iter#1#2{%
    \divide\delt@\tw@\edef\dpmn@{\repdecn@mb{\delt@}}%
    \mili@u=\v@lmin%
    \ifdim#2\z@%
      \advance\v@lmin-\dpmn@\v@lmax\advance\v@lmax\dpmn@\mili@u%
      \advance\v@leur-#1pt%
    \else%
      \advance\v@lmin\dpmn@\v@lmax\advance\v@lmax-\dpmn@\mili@u%
      \advance\v@leur#1pt%
    \fi}
\ctr@ld@f\def\c@ssin#1#2#3{\expandafter\ifx\csname COS@\number#3\endcsname\relax\c@lCS{#3pt}%
    \expandafter\xdef\csname COS@\number#3\endcsname{\repdecn@mb\result@t}%
    \expandafter\xdef\csname SIN@\number#3\endcsname{\repdecn@mb\result@@t}\fi%
    \edef#1{\csname COS@\number#3\endcsname}\edef#2{\csname SIN@\number#3\endcsname}}
\ctr@ld@f\def\c@lCS#1{{\mili@u=#1\p@rtent=\@ne%
    \relax\ifdim\mili@u<\z@\red@ng<-\else\red@ng>+\fi\f@ctech=\p@rtent%
    \relax\ifdim\mili@u<\z@\mili@u=-\mili@u\f@ctech=-\f@ctech\fi\c@@lCS}}
\ctr@ld@f\def\c@@lCS{\v@lmin=\mili@u\c@rre=-\mili@u\advance\c@rre\DemiPI@deg\v@lmax=\c@rre%
    \mili@u\@@lxxi\mili@u\divide\mili@u\@@mmmmlxviii%
    \edef\v@larg{\repdecn@mb{\mili@u}}\mili@u=-\v@larg\mili@u%
    \edef\v@lmxde{\repdecn@mb{\mili@u}}%
    \c@rre\@@lxxi\c@rre\divide\c@rre\@@mmmmlxviii%
    \edef\v@largC{\repdecn@mb{\c@rre}}\c@rre=-\v@largC\c@rre%
    \edef\v@lmxdeC{\repdecn@mb{\c@rre}}%
    \fctc@s\mili@u\v@lmin\global\result@t\p@rtent\v@leur%
    \let\t@mp=\v@larg\let\v@larg=\v@largC\let\v@largC=\t@mp%
    \let\t@mp=\v@lmxde\let\v@lmxde=\v@lmxdeC\let\v@lmxdeC=\t@mp%
    \fctc@s\c@rre\v@lmax\global\result@@t\f@ctech\v@leur}
\ctr@ld@f\def\fctc@s#1#2{\v@leur=#1\relax\ifdim#2<\@lxxxv\p@\cosser@h\else\sinser@t\fi}
\ctr@ld@f\def\cosser@h{\advance\v@leur\@lvi\p@\divide\v@leur\@lvi%
    \v@leur=\v@lmxde\v@leur\advance\v@leur\@xxx\p@%
    \v@leur=\v@lmxde\v@leur\advance\v@leur\@ccclx\p@%
    \v@leur=\v@lmxde\v@leur\advance\v@leur\@dccxx\p@\divide\v@leur\@dccxx}
\ctr@ld@f\def\sinser@t{\v@leur=\v@lmxdeC\p@\advance\v@leur\@vi\p@%
    \v@leur=\v@largC\v@leur\divide\v@leur\@vi}
\ctr@ld@f\def\red@ng#1#2{\relax\ifdim\mili@u#1#2\DemiPI@deg\advance\mili@u#2-\PI@deg%
    \p@rtent=-\p@rtent\red@ng#1#2\fi}
\ctr@ld@f\def\pr@c@lCS#1#2#3{\ctr@lcsn@m{COS@\number#3 }%
    \expandafter\xdef\csname COS@\number#3\endcsname{#1}%
    \expandafter\xdef\csname SIN@\number#3\endcsname{#2}}
\pr@c@lCS{1}{0}{0}
\pr@c@lCS{0.7071}{0.7071}{45}\pr@c@lCS{0.7071}{-0.7071}{-45}
\pr@c@lCS{0}{1}{90}          \pr@c@lCS{0}{-1}{-90}
\pr@c@lCS{-1}{0}{180}        \pr@c@lCS{-1}{0}{-180}
\pr@c@lCS{0}{-1}{270}        \pr@c@lCS{0}{1}{-270}
\ctr@ld@f\def\invers@#1#2{{\v@leur=#2\maxim@m{\v@lmax}{-\v@leur}{\v@leur}%
    \f@ctech=\@ne\m@inv@rs%
    \multiply\v@leur\f@ctech\edef\v@lv@leur{\repdecn@mb{\v@leur}}%
    \p@rtentiere{\p@rtent}{\v@leur}\v@lmin=\p@\divide\v@lmin\p@rtent%
    \inv@rs@\multiply\v@lmax\f@ctech\global\result@t=\v@lmax}#1=\result@t}
\ctr@ld@f\def\m@inv@rs{\ifdim\v@lmax<\p@\multiply\v@lmax\t@n\multiply\f@ctech\t@n\m@inv@rs\fi}
\ctr@ld@f\def\inv@rs@{\v@lmax=-\v@lmin\v@lmax=\v@lv@leur\v@lmax%
    \advance\v@lmax\tw@ pt\v@lmax=\repdecn@mb{\v@lmin}\v@lmax%
    \delt@=\v@lmax\advance\delt@-\v@lmin\ifdim\delt@<\z@\delt@=-\delt@\fi%
    \ifdim\delt@>\epsil@n\v@lmin=\v@lmax\inv@rs@\fi}
\ctr@ld@f\def\minim@m#1#2#3{\relax\ifdim#2<#3#1=#2\else#1=#3\fi}
\ctr@ld@f\def\maxim@m#1#2#3{\relax\ifdim#2>#3#1=#2\else#1=#3\fi}
\ctr@ld@f\def\p@rtentiere#1#2{#1=#2\divide#1by65536 }
\ctr@ld@f\def\r@undint#1#2{{\v@leur=#2\divide\v@leur\t@n\p@rtentiere{\p@rtent}{\v@leur}%
    \v@leur=\p@rtent pt\global\result@t=\t@n\v@leur}#1=\result@t}
\ctr@ld@f\def\sqrt@#1#2{{\v@leur=#2%
    \minim@m{\v@lmin}{\p@}{\v@leur}\maxim@m{\v@lmax}{\p@}{\v@leur}%
    \f@ctech=\@ne\m@sqrt@\sqrt@@%
    \mili@u=\v@lmin\advance\mili@u\v@lmax\divide\mili@u\tw@\multiply\mili@u\f@ctech%
    \global\result@t=\mili@u}#1=\result@t}
\ctr@ld@f\def\m@sqrt@{\ifdim\v@leur>\dcq@\divide\v@leur\c@nt\v@lmax=\v@leur%
    \multiply\f@ctech\t@n\m@sqrt@\fi}
\ctr@ld@f\def\sqrt@@{\mili@u=\v@lmin\advance\mili@u\v@lmax\divide\mili@u\tw@%
    \c@rre=\repdecn@mb{\mili@u}\mili@u%
    \ifdim\c@rre<\v@leur\v@lmin=\mili@u\else\v@lmax=\mili@u\fi%
    \delt@=\v@lmax\advance\delt@-\v@lmin\ifdim\delt@>\epsil@n\sqrt@@\fi}
\ctr@ld@f\def\extrairelepremi@r#1\de#2{\expandafter\lepremi@r#2@#1#2}
\ctr@ld@f\def\lepremi@r#1,#2@#3#4{\def#3{#1}\def#4{#2}\ignorespaces}
\ctr@ld@f\def\@cfor#1:=#2\do#3{%
  \edef\@fortemp{#2}%
  \ifx\@fortemp\empty\else\@cforloop#2,\@nil,\@nil\@@#1{#3}\fi}
\ctr@ln@m\@nextwhile
\ctr@ld@f\def\@cforloop#1,#2\@@#3#4{%
  \def#3{#1}%
  \ifx#3\Fig@nnil\let\@nextwhile=\Fig@fornoop\else#4\relax\let\@nextwhile=\@cforloop\fi%
  \@nextwhile#2\@@#3{#4}}

\ctr@ld@f\def\@ecfor#1:=#2\do#3{%
  \def\@@cfor{\@cfor#1:=}%
  \edef\@@@cfor{#2}%
  \expandafter\@@cfor\@@@cfor\do{#3}}
\ctr@ld@f\def\Fig@nnil{\@nil}
\ctr@ld@f\def\Fig@fornoop#1\@@#2#3{}
\ctr@ln@m\list@@rg
\ctr@ld@f\def\trtlis@rg#1#2{\def\list@@rg{#1}%
    \@ecfor\p@rv@l:=\list@@rg\do{\expandafter#2\p@rv@l|}}
\ctr@ld@f\def\trtlis@rgtok#1{\let@xte={}\let\n@xt\addt@t@xt\addt@t@xt #1}
\ctr@ln@m\M@cro
\ctr@ln@m\n@xt
\ctr@ld@f\def\addt@t@xt#1{\if#1|\let\n@xt\relax\else%
    \if#1,\expandafter\M@cro\the\let@xte|\let@xte={}%
    \else\let@xte=\expandafter{\the\let@xte #1}\fi\fi\n@xt}
\ctr@ln@w{newbox}\b@xvisu
\ctr@ln@w{newtoks}\let@xte
\ctr@ln@w{newif}\ifitis@K
\ctr@ln@w{newcount}\s@mme
\ctr@ln@w{newcount}\l@mbd@un \ctr@ln@w{newcount}\l@mbd@de
\ctr@ln@w{newcount}\superc@ntr@l\superc@ntr@l=\@ne        
\ctr@ln@w{newcount}\typec@ntr@l\typec@ntr@l=\superc@ntr@l 
\ctr@ln@w{newdimen}\v@lX  \ctr@ln@w{newdimen}\v@lY  \ctr@ln@w{newdimen}\v@lZ
\ctr@ln@w{newdimen}\v@lXa \ctr@ln@w{newdimen}\v@lYa \ctr@ln@w{newdimen}\v@lZa
\ctr@ln@w{newdimen}\unit@\unit@=\p@ 
\ctr@ld@f\def\unit@util{pt}
\ctr@ld@f\def\ptT@ptps{0.996264 }
\ctr@ld@f\def\ptpsT@pt{1.00375 }
\ctr@ld@f\def\ptT@unit@{1} 
\ctr@ld@f\def\setunit@#1{\def\unit@util{#1}\setunit@@#1:\invers@{\result@t}{\unit@}%
    \edef\ptT@unit@{\repdecn@mb\result@t}}
\ctr@ld@f\def\setunit@@#1#2:{\ifcat#1a\unit@=\@ne#1#2\else\unit@=#1#2\fi}
\ctr@ld@f\def\d@fm@cdim#1#2{{\v@leur=#2\v@leur=\ptT@unit@\v@leur\xdef#1{\repdecn@mb\v@leur}}}
\ctr@ln@w{newif}\ifBdingB@x\BdingB@xtrue
\ctr@ln@w{newdimen}\c@@rdXmin \ctr@ln@w{newdimen}\c@@rdYmin  
\ctr@ln@w{newdimen}\c@@rdXmax \ctr@ln@w{newdimen}\c@@rdYmax
\ctr@ld@f\def\b@undb@x#1#2{\ifBdingB@x%
    \relax\ifdim#1<\c@@rdXmin\global\c@@rdXmin=#1\fi%
    \relax\ifdim#2<\c@@rdYmin\global\c@@rdYmin=#2\fi%
    \relax\ifdim#1>\c@@rdXmax\global\c@@rdXmax=#1\fi%
    \relax\ifdim#2>\c@@rdYmax\global\c@@rdYmax=#2\fi\fi}
\ctr@ld@f\def\b@undb@xP#1{{\Figg@tXY{#1}\b@undb@x{\v@lX}{\v@lY}}}
\ctr@ld@f\def\ellBB@x#1;#2,#3(#4,#5,#6){{\s@uvc@ntr@l\et@tellBB@x%
    \setc@ntr@l{2}\figptell-2::#1;#2,#3(#4,#6)\b@undb@xP{-2}%
    \figptell-2::#1;#2,#3(#5,#6)\b@undb@xP{-2}%
    \c@ssin{\C@}{\S@}{#6}\v@lmin=\C@ pt\v@lmax=\S@ pt%
    \mili@u=#3\v@lmin\delt@=#2\v@lmax\arct@n\v@leur(\delt@,\mili@u)%
    \mili@u=-#3\v@lmax\delt@=#2\v@lmin\arct@n\c@rre(\delt@,\mili@u)%
    \v@leur=\rdT@deg\v@leur\advance\v@leur-\DePI@deg%
    \c@rre=\rdT@deg\c@rre\advance\c@rre-\DePI@deg%
    \v@lmin=#4pt\v@lmax=#5pt%
    \loop\ifdim\v@leur<\v@lmax\ifdim\v@leur>\v@lmin%
    \edef\@ngle{\repdecn@mb\v@leur}\figptell-2::#1;#2,#3(\@ngle,#6)%
    \b@undb@xP{-2}\fi\advance\v@leur\PI@deg\repeat%
    \loop\ifdim\c@rre<\v@lmax\ifdim\c@rre>\v@lmin%
    \edef\@ngle{\repdecn@mb\c@rre}\figptell-2::#1;#2,#3(\@ngle,#6)%
    \b@undb@xP{-2}\fi\advance\c@rre\PI@deg\repeat%
    \resetc@ntr@l\et@tellBB@x}\ignorespaces}
\ctr@ld@f\def\initb@undb@x{\c@@rdXmin=\maxdimen\c@@rdYmin=\maxdimen%
    \c@@rdXmax=-\maxdimen\c@@rdYmax=-\maxdimen}
\ctr@ld@f\def\c@ntr@lnum#1{%
    \relax\ifnum\typec@ntr@l=\@ne%
    \ifnum#1<\z@%
    \immediate\write16{*** Forbidden point number (#1). Abort.}\end\fi\fi%
    \set@bjc@de{#1}}
\ctr@ln@m\objc@de
\ctr@ld@f\def\set@bjc@de#1{\edef\objc@de{@BJ\ifnum#1<\z@ M\romannumeral-#1\else\romannumeral#1\fi}}
\s@mme=\m@ne\loop\ifnum\s@mme>-19
  \set@bjc@de{\s@mme}\ctr@lcsn@m\objc@de\ctr@lcsn@m{\objc@de T}
\advance\s@mme\m@ne\repeat
\s@mme=\@ne\loop\ifnum\s@mme<6
  \set@bjc@de{\s@mme}\ctr@lcsn@m\objc@de\ctr@lcsn@m{\objc@de T}
\advance\s@mme\@ne\repeat
\ctr@ld@f\def\setc@ntr@l#1{\ifnum\superc@ntr@l>#1\typec@ntr@l=\superc@ntr@l%
    \else\typec@ntr@l=#1\fi}
\ctr@ld@f\def\resetc@ntr@l#1{\global\superc@ntr@l=#1\setc@ntr@l{#1}}
\ctr@ld@f\def\s@uvc@ntr@l#1{\edef#1{\the\superc@ntr@l}}
\ctr@ln@m\c@lproscal
\ctr@ld@f\def\c@lproscalDD#1[#2,#3]{{\Figg@tXY{#2}%
    \edef\Xu@{\repdecn@mb{\v@lX}}\edef\Yu@{\repdecn@mb{\v@lY}}\Figg@tXY{#3}%
    \global\result@t=\Xu@\v@lX\global\advance\result@t\Yu@\v@lY}#1=\result@t}
\ctr@ld@f\def\c@lproscalTD#1[#2,#3]{{\Figg@tXY{#2}\edef\Xu@{\repdecn@mb{\v@lX}}%
    \edef\Yu@{\repdecn@mb{\v@lY}}\edef\Zu@{\repdecn@mb{\v@lZ}}%
    \Figg@tXY{#3}\global\result@t=\Xu@\v@lX\global\advance\result@t\Yu@\v@lY%
    \global\advance\result@t\Zu@\v@lZ}#1=\result@t}
\ctr@ld@f\def\c@lprovec#1{%
    \det@rmC\v@lZa(\v@lX,\v@lY,\v@lmin,\v@lmax)%
    \det@rmC\v@lXa(\v@lY,\v@lZ,\v@lmax,\v@leur)%
    \det@rmC\v@lYa(\v@lZ,\v@lX,\v@leur,\v@lmin)%
    \Figv@ctCreg#1(\v@lXa,\v@lYa,\v@lZa)}
\ctr@ld@f\def\det@rm#1[#2,#3]{{\Figg@tXY{#2}\Figg@tXYa{#3}%
    \delt@=\repdecn@mb{\v@lX}\v@lYa\advance\delt@-\repdecn@mb{\v@lY}\v@lXa%
    \global\result@t=\delt@}#1=\result@t}
\ctr@ld@f\def\det@rmC#1(#2,#3,#4,#5){{\global\result@t=\repdecn@mb{#2}#5%
    \global\advance\result@t-\repdecn@mb{#3}#4}#1=\result@t}
\ctr@ld@f\def\getredf@ctDD#1(#2,#3){{\maxim@m{\v@lXa}{-#2}{#2}\maxim@m{\v@lYa}{-#3}{#3}%
    \maxim@m{\v@lXa}{\v@lXa}{\v@lYa}
    \ifdim\v@lXa>\@xci pt\divide\v@lXa\@xci%
    \p@rtentiere{\p@rtent}{\v@lXa}\advance\p@rtent\@ne\else\p@rtent=\@ne\fi%
    \global\result@tent=\p@rtent}#1=\result@tent\ignorespaces}
\ctr@ld@f\def\getredf@ctTD#1(#2,#3,#4){{\maxim@m{\v@lXa}{-#2}{#2}\maxim@m{\v@lYa}{-#3}{#3}%
    \maxim@m{\v@lZa}{-#4}{#4}\maxim@m{\v@lXa}{\v@lXa}{\v@lYa}%
    \maxim@m{\v@lXa}{\v@lXa}{\v@lZa}
    \ifdim\v@lXa>\@lxxiv pt\divide\v@lXa\@lxxiv%
    \p@rtentiere{\p@rtent}{\v@lXa}\advance\p@rtent\@ne\else\p@rtent=\@ne\fi%
    \global\result@tent=\p@rtent}#1=\result@tent\ignorespaces}
\ctr@ln@m\getredf@ctB
\ctr@ld@f\def\getredf@ctBDD#1{\getredf@ctDD#1(\v@lX,\v@lY)}
\ctr@ld@f\def\getredf@ctBTD#1{\getredf@ctTD#1(\v@lX,\v@lY,\v@lZ)}
\ctr@ld@f\def\FigptintercircB@zDD#1:#2:#3,#4[#5,#6,#7,#8]{{\s@uvc@ntr@l\et@tfigptintercircB@zDD%
    \setc@ntr@l{2}\figvectPDD-1[#5,#8]\Figg@tXY{-1}\getredf@ctDD\f@ctech(\v@lX,\v@lY)%
    \mili@u=#4\unit@\divide\mili@u\f@ctech\c@rre=\repdecn@mb{\mili@u}\mili@u%
    \figptBezierDD-5::#3[#5,#6,#7,#8]%
    \v@lmin=#3\p@\v@lmax=\v@lmin\advance\v@lmax0.1\p@%
    \loop\edef\T@{\repdecn@mb{\v@lmax}}\figptBezierDD-2::\T@[#5,#6,#7,#8]%
    \figvectPDD-1[-5,-2]\n@rmeucCDD{\delt@}{-1}\ifdim\delt@<\c@rre\v@lmin=\v@lmax%
    \advance\v@lmax0.1\p@\repeat%
    \loop\mili@u=\v@lmin\advance\mili@u\v@lmax%
    \divide\mili@u\tw@\edef\T@{\repdecn@mb{\mili@u}}\figptBezierDD-2::\T@[#5,#6,#7,#8]%
    \figvectPDD-1[-5,-2]\n@rmeucCDD{\delt@}{-1}\ifdim\delt@>\c@rre\v@lmax=\mili@u%
    \else\v@lmin=\mili@u\fi\v@leur=\v@lmax\advance\v@leur-\v@lmin%
    \ifdim\v@leur>\epsil@n\repeat\figptcopyDD#1:#2/-2/%
    \resetc@ntr@l\et@tfigptintercircB@zDD}\ignorespaces}
\ctr@ln@m\figptinterlines
\ctr@ld@f\def\inters@cDD#1:#2[#3,#4;#5,#6]{{\s@uvc@ntr@l\et@tinters@cDD%
    \setc@ntr@l{2}\vecunit@{-1}{#4}\vecunit@{-2}{#6}%
    \Figg@tXY{-1}\setc@ntr@l{1}\Figg@tXYa{#3}%
    \edef\A@{\repdecn@mb{\v@lX}}\edef\B@{\repdecn@mb{\v@lY}}%
    \v@lmin=\B@\v@lXa\advance\v@lmin-\A@\v@lYa%
    \Figg@tXYa{#5}\setc@ntr@l{2}\Figg@tXY{-2}%
    \edef\C@{\repdecn@mb{\v@lX}}\edef\D@{\repdecn@mb{\v@lY}}%
    \v@lmax=\D@\v@lXa\advance\v@lmax-\C@\v@lYa%
    \delt@=\A@\v@lY\advance\delt@-\B@\v@lX%
    \invers@{\v@leur}{\delt@}\edef\v@ldelta{\repdecn@mb{\v@leur}}%
    \v@lXa=\A@\v@lmax\advance\v@lXa-\C@\v@lmin%
    \v@lYa=\B@\v@lmax\advance\v@lYa-\D@\v@lmin%
    \v@lXa=\v@ldelta\v@lXa\v@lYa=\v@ldelta\v@lYa%
    \setc@ntr@l{1}\Figp@intregDD#1:{#2}(\v@lXa,\v@lYa)%
    \resetc@ntr@l\et@tinters@cDD}\ignorespaces}
\ctr@ld@f\def\inters@cTD#1:#2[#3,#4;#5,#6]{{\s@uvc@ntr@l\et@tinters@cTD%
    \setc@ntr@l{2}\figvectNVTD-1[#4,#6]\figvectNVTD-2[#6,-1]\figvectPTD-1[#3,#5]%
    \r@pPSTD\v@leur[-2,-1,#4]\edef\v@lcoef{\repdecn@mb{\v@leur}}%
    \figpttraTD#1:{#2}=#3/\v@lcoef,#4/\resetc@ntr@l\et@tinters@cTD}\ignorespaces}
\ctr@ld@f\def\r@pPSTD#1[#2,#3,#4]{{\Figg@tXY{#2}\edef\Xu@{\repdecn@mb{\v@lX}}%
    \edef\Yu@{\repdecn@mb{\v@lY}}\edef\Zu@{\repdecn@mb{\v@lZ}}%
    \Figg@tXY{#3}\v@lmin=\Xu@\v@lX\advance\v@lmin\Yu@\v@lY\advance\v@lmin\Zu@\v@lZ%
    \Figg@tXY{#4}\v@lmax=\Xu@\v@lX\advance\v@lmax\Yu@\v@lY\advance\v@lmax\Zu@\v@lZ%
    \invers@{\v@leur}{\v@lmax}\global\result@t=\repdecn@mb{\v@leur}\v@lmin}%
    #1=\result@t}
\ctr@ln@m\n@rminf
\ctr@ld@f\def\n@rminfDD#1#2{{\Figg@tXY{#2}\maxim@m{\v@lX}{\v@lX}{-\v@lX}%
    \maxim@m{\v@lY}{\v@lY}{-\v@lY}\maxim@m{\global\result@t}{\v@lX}{\v@lY}}%
    #1=\result@t}
\ctr@ld@f\def\n@rminfTD#1#2{{\Figg@tXY{#2}\maxim@m{\v@lX}{\v@lX}{-\v@lX}%
    \maxim@m{\v@lY}{\v@lY}{-\v@lY}\maxim@m{\v@lZ}{\v@lZ}{-\v@lZ}%
    \maxim@m{\v@lX}{\v@lX}{\v@lY}\maxim@m{\global\result@t}{\v@lX}{\v@lZ}}%
    #1=\result@t}
\ctr@ln@m\n@rmeucC
\ctr@ld@f\def\n@rmeucCDD#1#2{\Figg@tXY{#2}\divide\v@lX\f@ctech\divide\v@lY\f@ctech%
    #1=\repdecn@mb{\v@lX}\v@lX\v@lX=\repdecn@mb{\v@lY}\v@lY\advance#1\v@lX}
\ctr@ld@f\def\n@rmeucCTD#1#2{\Figg@tXY{#2}%
    \divide\v@lX\f@ctech\divide\v@lY\f@ctech\divide\v@lZ\f@ctech%
    #1=\repdecn@mb{\v@lX}\v@lX\v@lX=\repdecn@mb{\v@lY}\v@lY\advance#1\v@lX%
    \v@lX=\repdecn@mb{\v@lZ}\v@lZ\advance#1\v@lX}
\ctr@ln@m\n@rmeucSV
\ctr@ld@f\def\n@rmeucSVDD#1#2{{\Figg@tXY{#2}%
    \v@lXa=\repdecn@mb{\v@lX}\v@lX\v@lYa=\repdecn@mb{\v@lY}\v@lY%
    \advance\v@lXa\v@lYa\sqrt@{\global\result@t}{\v@lXa}}#1=\result@t}
\ctr@ld@f\def\n@rmeucSVTD#1#2{{\Figg@tXY{#2}\v@lXa=\repdecn@mb{\v@lX}\v@lX%
    \v@lYa=\repdecn@mb{\v@lY}\v@lY\v@lZa=\repdecn@mb{\v@lZ}\v@lZ%
    \advance\v@lXa\v@lYa\advance\v@lXa\v@lZa\sqrt@{\global\result@t}{\v@lXa}}#1=\result@t}
\ctr@ln@m\n@rmeuc
\ctr@ld@f\def\n@rmeucDD#1#2{{\Figg@tXY{#2}\getredf@ctDD\f@ctech(\v@lX,\v@lY)%
    \divide\v@lX\f@ctech\divide\v@lY\f@ctech%
    \v@lXa=\repdecn@mb{\v@lX}\v@lX\v@lYa=\repdecn@mb{\v@lY}\v@lY%
    \advance\v@lXa\v@lYa\sqrt@{\global\result@t}{\v@lXa}%
    \global\multiply\result@t\f@ctech}#1=\result@t}
\ctr@ld@f\def\n@rmeucTD#1#2{{\Figg@tXY{#2}\getredf@ctTD\f@ctech(\v@lX,\v@lY,\v@lZ)%
    \divide\v@lX\f@ctech\divide\v@lY\f@ctech\divide\v@lZ\f@ctech%
    \v@lXa=\repdecn@mb{\v@lX}\v@lX%
    \v@lYa=\repdecn@mb{\v@lY}\v@lY\v@lZa=\repdecn@mb{\v@lZ}\v@lZ%
    \advance\v@lXa\v@lYa\advance\v@lXa\v@lZa\sqrt@{\global\result@t}{\v@lXa}%
    \global\multiply\result@t\f@ctech}#1=\result@t}
\ctr@ln@m\vecunit@
\ctr@ld@f\def\vecunit@DD#1#2{{\Figg@tXY{#2}\getredf@ctDD\f@ctech(\v@lX,\v@lY)%
    \divide\v@lX\f@ctech\divide\v@lY\f@ctech%
    \Figv@ctCreg#1(\v@lX,\v@lY)\n@rmeucSV{\v@lYa}{#1}%
    \invers@{\v@lXa}{\v@lYa}\edef\v@lv@lXa{\repdecn@mb{\v@lXa}}%
    \v@lX=\v@lv@lXa\v@lX\v@lY=\v@lv@lXa\v@lY%
    \Figv@ctCreg#1(\v@lX,\v@lY)\multiply\v@lYa\f@ctech\global\result@t=\v@lYa}}
\ctr@ld@f\def\vecunit@TD#1#2{{\Figg@tXY{#2}\getredf@ctTD\f@ctech(\v@lX,\v@lY,\v@lZ)%
    \divide\v@lX\f@ctech\divide\v@lY\f@ctech\divide\v@lZ\f@ctech%
    \Figv@ctCreg#1(\v@lX,\v@lY,\v@lZ)\n@rmeucSV{\v@lYa}{#1}%
    \invers@{\v@lXa}{\v@lYa}\edef\v@lv@lXa{\repdecn@mb{\v@lXa}}%
    \v@lX=\v@lv@lXa\v@lX\v@lY=\v@lv@lXa\v@lY\v@lZ=\v@lv@lXa\v@lZ%
    \Figv@ctCreg#1(\v@lX,\v@lY,\v@lZ)\multiply\v@lYa\f@ctech\global\result@t=\v@lYa}}
\ctr@ld@f\def\vecunitC@TD[#1,#2]{\Figg@tXYa{#1}\Figg@tXY{#2}%
    \advance\v@lX-\v@lXa\advance\v@lY-\v@lYa\advance\v@lZ-\v@lZa\c@lvecunitTD}
\ctr@ld@f\def\vecunitCV@TD#1{\Figg@tXY{#1}\c@lvecunitTD}
\ctr@ld@f\def\c@lvecunitTD{\getredf@ctTD\f@ctech(\v@lX,\v@lY,\v@lZ)%
    \divide\v@lX\f@ctech\divide\v@lY\f@ctech\divide\v@lZ\f@ctech%
    \v@lXa=\repdecn@mb{\v@lX}\v@lX%
    \v@lYa=\repdecn@mb{\v@lY}\v@lY\v@lZa=\repdecn@mb{\v@lZ}\v@lZ%
    \advance\v@lXa\v@lYa\advance\v@lXa\v@lZa\sqrt@{\v@lYa}{\v@lXa}%
    \invers@{\v@lXa}{\v@lYa}\edef\v@lv@lXa{\repdecn@mb{\v@lXa}}%
    \v@lX=\v@lv@lXa\v@lX\v@lY=\v@lv@lXa\v@lY\v@lZ=\v@lv@lXa\v@lZ}
\ctr@ln@m\figgetangle
\ctr@ld@f\def\figgetangleDD#1[#2,#3,#4]{\ifGR@cri{\s@uvc@ntr@l\et@tfiggetangleDD\setc@ntr@l{2}%
    \figvectPDD-1[#2,#3]\figvectPDD-2[#2,#4]\vecunit@{-1}{-1}%
    \c@lproscalDD\delt@[-2,-1]\figvectNVDD-1[-1]\c@lproscalDD\v@leur[-2,-1]%
    \arct@n\v@lmax(\delt@,\v@leur)\v@lmax=\rdT@deg\v@lmax%
    \ifdim\v@lmax<\z@\advance\v@lmax\DePI@deg\fi\xdef#1{\repdecn@mb{\v@lmax}}%
    \resetc@ntr@l\et@tfiggetangleDD}\ignorespaces\fi}
\ctr@ld@f\def\figgetangleTD#1[#2,#3,#4,#5]{\ifGR@cri{\s@uvc@ntr@l\et@tfiggetangleTD\setc@ntr@l{2}%
    \figvectPTD-1[#2,#3]\figvectPTD-2[#2,#5]\figvectNVTD-3[-1,-2]%
    \figvectPTD-2[#2,#4]\figvectNVTD-4[-3,-1]%
    \vecunit@{-1}{-1}\c@lproscalTD\delt@[-2,-1]\c@lproscalTD\v@leur[-2,-4]%
    \arct@n\v@lmax(\delt@,\v@leur)\v@lmax=\rdT@deg\v@lmax%
    \ifdim\v@lmax<\z@\advance\v@lmax\DePI@deg\fi\xdef#1{\repdecn@mb{\v@lmax}}%
    \resetc@ntr@l\et@tfiggetangleTD}\ignorespaces\fi}    
\ctr@ld@f\def\figgetdist#1[#2,#3]{\ifGR@cri{\s@uvc@ntr@l\et@tfiggetdist\setc@ntr@l{2}%
    \figvectP-1[#2,#3]\n@rmeuc{\v@lX}{-1}\v@lX=\ptT@unit@\v@lX\xdef#1{\repdecn@mb{\v@lX}}%
    \resetc@ntr@l\et@tfiggetdist}\ignorespaces\fi}
\ctr@ld@f\def\figget#1=#2[#3]{\keln@mun#1|%
    \def\n@mref{a}\ifx\l@debut\n@mref\figgetangle#2[#3]\else
    \def\n@mref{d}\ifx\l@debut\n@mref\figgetdist#2[#3]\else
    \W@rnmeskwd{figget}{#1}\fi\fi\ignorespaces}
\ctr@ld@f\def\Figg@tT#1{\c@ntr@lnum{#1}%
    {\expandafter\expandafter\expandafter\extr@ctT\csname\objc@de\endcsname:%
     \ifnum\B@@ltxt=\z@\ptn@me{#1}\else\csname\objc@de T\endcsname\fi}}
\ctr@ld@f\def\extr@ctT#1,#2,#3/#4:{\def\B@@ltxt{#3}}
\ctr@ld@f\def\Figg@tXY#1{\c@ntr@lnum{#1}%
    \expandafter\expandafter\expandafter\extr@ctC\csname\objc@de\endcsname:}
\ctr@ln@m\extr@ctC
\ctr@ld@f\def\extr@ctCDD#1/#2,#3,#4:{\v@lX=#2\v@lY=#3}
\ctr@ld@f\def\extr@ctCTD#1/#2,#3,#4:{\v@lX=#2\v@lY=#3\v@lZ=#4}
\ctr@ld@f\def\Figg@tXYa#1{\c@ntr@lnum{#1}%
    \expandafter\expandafter\expandafter\extr@ctCa\csname\objc@de\endcsname:}
\ctr@ln@m\extr@ctCa
\ctr@ld@f\def\extr@ctCaDD#1/#2,#3,#4:{\v@lXa=#2\v@lYa=#3}
\ctr@ld@f\def\extr@ctCaTD#1/#2,#3,#4:{\v@lXa=#2\v@lYa=#3\v@lZa=#4}
\ctr@ln@m\t@xt@
\ctr@ld@f\def\figinit#1{\t@stc@tcodech@nge\initpr@lim\Figinit@#1,:\initpss@ttings\ignorespaces}
\ctr@ld@f\def\Figinit@#1,#2:{\setunit@{#1}\def\t@xt@{#2}\ifx\t@xt@\empty\else\Figinit@@#2:\fi}
\ctr@ld@f\def\Figinit@@#1#2:{\if#12 \else\Figs@tproj{#1}\initTD@\fi}
\ctr@ln@w{newif}\ifTr@isDim
\ctr@ld@f\def\UnD@fined{UNDEFINED}
\ctr@ln@m\@utoFN
\ctr@ln@m\@utoFInDone
\ctr@ln@m\disob@unit
\ctr@ld@f\def\initpr@lim{\initb@undb@x\figsetmark{}\figsetptname{$A_{##1}$}\def\Sc@leFact{1}%
    \initDD@\figsetroundcoord{yes}\GR@critrue\expandafter\setupd@te\D@FTupdate:%
    \edef\disob@unit{\UnD@fined}\edef\t@rgetpt{\UnD@fined}\gdef\@utoFInDone{1}\gdef\@utoFN{0}}
\ctr@ld@f\def\initDD@{\Tr@isDimfalse%
    \ifPDFm@ke%
     \let\Ps@rcerc=\Ps@rcercBz%
     \let\Ps@rell=\Ps@rellBz%
    \fi
    \let\c@lDCUn=\c@lDCUnDD%
    \let\c@lDCDeux=\c@lDCDeuxDD%
    \let\c@ldefproj=\relax%
    \let\c@lproscal=\c@lproscalDD%
    \let\c@lprojSP=\relax%
    \let\extr@ctC=\extr@ctCDD%
    \let\extr@ctCa=\extr@ctCaDD%
    \let\extr@ctCF=\extr@ctCFDD%
    \let\Figp@intreg=\Figp@intregDD%
    \let\Figpts@xes=\Figpts@xesDD%
    \let\getredf@ctB=\getredf@ctBDD%
    \let\n@rmeucSV=\n@rmeucSVDD\let\n@rmeuc=\n@rmeucDD\let\n@rmeucC\n@rmeucCDD\let\n@rminf=\n@rminfDD%
    \let\pr@dMatV=\pr@dMatVDD%
    \let\Q@@xes=\Q@@xesDD%
    \let\vecunit@=\vecunit@DD%
    \let\figcoord=\figcoordDD%
    \let\figgetangle=\figgetangleDD%
    \let\figpt=\figptDD%
    \let\figptBezier=\figptBezierDD%
    \let\figptbary=\figptbaryDD%
    \let\figptcirc=\figptcircDD%
    \let\figptcircumcenter=\figptcircumcenterDD%
    \let\figptcopy=\figptcopyDD%
    \let\figptcurvcenter=\figptcurvcenterDD%
    \let\figptell=\figptellDD%
    \let\figptendnormal=\figptendnormalDD%
    \let\figptinterlineplane=\figptinterlineplaneDD%
    \let\figptinterlines=\inters@cDD%
    \let\figptorthocenter=\figptorthocenterDD%
    \let\figptorthoprojline=\figptorthoprojlineDD%
    \let\figptorthoprojplane=\figptorthoprojplaneDD%
    \let\figptrot=\figptrotDD%
    \let\figptscontrol=\figptscontrolDD%
    \let\figptsintercirc=\figptsintercircDD%
    \let\figptsinterlinell=\figptsinterlinellDD%
    \let\figptsorthoprojline=\figptsorthoprojlineDD%
    \let\figptorthoprojplane=\figptorthoprojplaneDD%
    \let\figptsrot=\figptsrotDD%
    \let\figptssym=\figptssymDD%
    \let\figptstra=\figptstraDD%
    \let\figptsym=\figptsymDD%
    \let\figpttraC=\figpttraCDD%
    \let\figpttra=\figpttraDD%
    \let\figptvisilimSL=\figptvisilimSLDD%
    \let\figsetobdist=\figsetobdistDD%
    \let\figsettarget=\figsettargetDD%
    \let\figsetview=\figsetviewDD%
    \let\figvectDBezier=\figvectDBezierDD%
    \let\figvectN=\figvectNDD%
    \let\figvectNV=\figvectNVDD%
    \let\figvectP=\figvectPDD%
    \let\figvectU=\figvectUDD%
    \let\figdrawarccircP=\Q@arccircPDD%
    \let\figdrawarccirc=\Q@arccircDD%
    \let\figdrawarcell=\Q@arcellDD%
    \let\figdrawarcellPA=\Q@arcellPADD%
    \let\figdrawarrowBezier=\Q@arrowBezierDD%
    \let\figdrawarrowcircP=\Q@arrowcircPDD%
    \let\figdrawarrowcirc=\Q@arrowcircDD%
    \let\figdrawarrowhead=\Q@arrowheadDD%
    \let\figdrawarrow=\Q@arrowDD%
    \let\figdrawBezier=\Q@BezierDD%
    \let\figdrawcirc=\Q@circDD%
    \let\figdrawcurve=\Q@curveDD%
    \let\figdrawnormal=\Q@normalDD%
    }
\ctr@ld@f\def\initTD@{\Tr@isDimtrue\initb@undb@xTD\newt@rgetptfalse\newdis@bfalse%
    \let\c@lDCUn=\c@lDCUnTD%
    \let\c@lDCDeux=\c@lDCDeuxTD%
    \let\c@ldefproj=\c@ldefprojTD%
    \let\c@lproscal=\c@lproscalTD%
    \let\extr@ctC=\extr@ctCTD%
    \let\extr@ctCa=\extr@ctCaTD%
    \let\extr@ctCF=\extr@ctCFTD%
    \let\Figp@intreg=\Figp@intregTD%
    \let\Figpts@xes=\Figpts@xesTD%
    \let\getredf@ctB=\getredf@ctBTD%
    \let\n@rmeucSV=\n@rmeucSVTD\let\n@rmeuc=\n@rmeucTD\let\n@rmeucC\n@rmeucCTD\let\n@rminf=\n@rminfTD%
    \let\pr@dMatV=\pr@dMatVTD%
    \let\Q@@xes=\Q@@xesTD%
    \let\vecunit@=\vecunit@TD%
    \let\figcoord=\figcoordTD%
    \let\figgetangle=\figgetangleTD%
    \let\figpt=\figptTD%
    \let\figptBezier=\figptBezierTD%
    \let\figptbary=\figptbaryTD%
    \let\figptcirc=\figptcircTD%
    \let\figptcircumcenter=\figptcircumcenterTD%
    \let\figptcopy=\figptcopyTD%
    \let\figptcurvcenter=\figptcurvcenterTD%
    \let\figptinterlineplane=\figptinterlineplaneTD%
    \let\figptinterlines=\inters@cTD%
    \let\figptorthocenter=\figptorthocenterTD%
    \let\figptorthoprojline=\figptorthoprojlineTD%
    \let\figptorthoprojplane=\figptorthoprojplaneTD%
    \let\figptrot=\figptrotTD%
    \let\figptscontrol=\figptscontrolTD%
    \let\figptsintercirc=\figptsintercircTD%
    \let\figptsorthoprojline=\figptsorthoprojlineTD%
    \let\figptsorthoprojplane=\figptsorthoprojplaneTD%
    \let\figptsrot=\figptsrotTD%
    \let\figptssym=\figptssymTD%
    \let\figptstra=\figptstraTD%
    \let\figptsym=\figptsymTD%
    \let\figpttraC=\figpttraCTD%
    \let\figpttra=\figpttraTD%
    \let\figptvisilimSL=\figptvisilimSLTD%
    \let\figsetobdist=\figsetobdistTD%
    \let\figsettarget=\figsettargetTD%
    \let\figsetview=\figsetviewTD%
    \let\figvectDBezier=\figvectDBezierTD%
    \let\figvectN=\figvectNTD%
    \let\figvectNV=\figvectNVTD%
    \let\figvectP=\figvectPTD%
    \let\figvectU=\figvectUTD%
    \let\figdrawarccircP=\Q@arccircPTD%
    \let\figdrawarccirc=\Q@arccircTD%
    \let\figdrawarcell=\Q@arcellTD%
    \let\figdrawarcellPA=\Q@arcellPATD%
    \let\figdrawarrowBezier=\Q@arrowBezierTD%
    \let\figdrawarrowcircP=\Q@arrowcircPTD%
    \let\figdrawarrowcirc=\Q@arrowcircTD%
    \let\figdrawarrowhead=\Q@arrowheadTD%
    \let\figdrawarrow=\Q@arrowTD%
    \let\figdrawBezier=\Q@BezierTD%
    \let\figdrawcirc=\Q@circTD%
    \let\figdrawcurve=\Q@curveTD%
    }
\ctr@ld@f\def\un@v@ilable#1{\immediate\write16{*** The macro #1 is not available in the current context.}}
\ctr@ld@f\def\figinsert#1{{\def\t@xt@{#1}\relax%
    \ifx\t@xt@\empty\ifnum\@utoFInDone>\z@\Figinsert@\DefGIfilen@me,:\fi%
    \else\expandafter\FiginsertNu@#1 :\fi}\ignorespaces}
\ctr@ld@f\def\FiginsertNu@#1 #2:{\def\t@xt@{#1}\relax\ifx\t@xt@\empty\def\t@xt@{#2}%
    \ifx\t@xt@\empty\ifnum\@utoFInDone>\z@\Figinsert@\DefGIfilen@me,:\fi%
    \else\FiginsertNu@#2:\fi\else\expandafter\FiginsertNd@#1 #2:\fi}
\ctr@ld@f\def\FiginsertNd@#1#2:{\ifcat#1a\Figinsert@#1#2,:\else%
    \ifnum\@utoFInDone>\z@\Figinsert@\DefGIfilen@me,#1#2,:\fi\fi}
\ctr@ln@m\Sc@leFact
\ctr@ld@f\def\Figinsert@#1,#2:{\def\t@xt@{#2}\ifx\t@xt@\empty\xdef\Sc@leFact{1}\else%
    \X@rgdeux@#2\xdef\Sc@leFact{\@rgdeux}\fi%
    \Figdisc@rdLTS{#1}{\t@xt@}\@psfgetbb{\t@xt@}%
    \v@lX=\@psfllx\p@\v@lX=\ptpsT@pt\v@lX\v@lX=\Sc@leFact\v@lX%
    \v@lY=\@psflly\p@\v@lY=\ptpsT@pt\v@lY\v@lY=\Sc@leFact\v@lY%
    \b@undb@x{\v@lX}{\v@lY}%
    \v@lX=\@psfurx\p@\v@lX=\ptpsT@pt\v@lX\v@lX=\Sc@leFact\v@lX%
    \v@lY=\@psfury\p@\v@lY=\ptpsT@pt\v@lY\v@lY=\Sc@leFact\v@lY%
    \b@undb@x{\v@lX}{\v@lY}%
    \ifPDFm@ke\Figinclud@PDF{\t@xt@}{\Sc@leFact}\else%
    \v@lX=\c@nt pt\v@lX=\Sc@leFact\v@lX\edef\F@ct{\repdecn@mb{\v@lX}}%
    \ifx\TeXturesonMacOSltX\special{postscriptfile #1 vscale=\F@ct\space hscale=\F@ct}%
    \else\includegraphics{#1}\fi\fi%
    \message{[\t@xt@]}\ignorespaces}
\ctr@ld@f\def\Figdisc@rdLTS#1#2{\expandafter\Figdisc@rdLTS@#1 :#2}
\ctr@ld@f\def\Figdisc@rdLTS@#1 #2:#3{\def#3{#1}\relax\ifx#3\empty\expandafter\Figdisc@rdLTS@#2:#3\fi}
\ctr@ld@f\def\figinsertE#1{\FiginsertE@#1,:\ignorespaces}
\ctr@ld@f\def\FiginsertE@#1,#2:{{\def\t@xt@{#2}\ifx\t@xt@\empty\xdef\Sc@leFact{1}\else%
    \X@rgdeux@#2\xdef\Sc@leFact{\@rgdeux}\fi%
    \Figdisc@rdLTS{#1}{\t@xt@}\pdfximage{\t@xt@}%
    \setbox\Gb@x=\hbox{\pdfrefximage\pdflastximage}%
    \v@lX=\z@\v@lY=-\Sc@leFact\dp\Gb@x\b@undb@x{\v@lX}{\v@lY}%
    \advance\v@lX\Sc@leFact\wd\Gb@x\advance\v@lY\Sc@leFact\dp\Gb@x%
    \advance\v@lY\Sc@leFact\ht\Gb@x\b@undb@x{\v@lX}{\v@lY}%
    \v@lX=\Sc@leFact\wd\Gb@x\pdfximage width \v@lX {\t@xt@}%
    \rlap{\pdfrefximage\pdflastximage}\message{[\t@xt@]}}\ignorespaces}
\ctr@ld@f\def\X@rgdeux@#1,{\edef\@rgdeux{#1}}
\ctr@ln@m\figpt
\ctr@ld@f\def\figptDD#1:#2(#3,#4){\ifGR@cri\c@ntr@lnum{#1}%
    {\v@lX=#3\unit@\v@lY=#4\unit@\Fig@dmpt{#2}{\z@}}\ignorespaces\fi}
\ctr@ld@f\def\Fig@dmpt#1#2{\def\t@xt@{#1}\ifx\t@xt@\empty\def\B@@ltxt{\z@}%
    \else\expandafter\gdef\csname\objc@de T\endcsname{#1}\def\B@@ltxt{\@ne}\fi%
    \expandafter\xdef\csname\objc@de\endcsname{\ifitis@vect@r\C@dCl@svect%
    \else\C@dCl@spt\fi,\z@,\B@@ltxt/\the\v@lX,\the\v@lY,#2}}
\ctr@ld@f\def\C@dCl@spt{P}
\ctr@ld@f\def\C@dCl@svect{V}
\ctr@ln@m\c@@rdYZ
\ctr@ln@m\c@@rdY
\ctr@ld@f\def\figptTD#1:#2(#3,#4){\ifGR@cri\c@ntr@lnum{#1}%
    \def\c@@rdYZ{#4,0,0}\extrairelepremi@r\c@@rdY\de\c@@rdYZ%
    \extrairelepremi@r\c@@rdZ\de\c@@rdYZ%
    {\v@lX=#3\unit@\v@lY=\c@@rdY\unit@\v@lZ=\c@@rdZ\unit@\Fig@dmpt{#2}{\the\v@lZ}%
    \b@undb@xTD{\v@lX}{\v@lY}{\v@lZ}}\ignorespaces\fi}
\ctr@ln@m\Figp@intreg
\ctr@ld@f\def\Figp@intregDD#1:#2(#3,#4){\c@ntr@lnum{#1}%
    {\result@t=#4\v@lX=#3\v@lY=\result@t\Fig@dmpt{#2}{\z@}}\ignorespaces}
\ctr@ld@f\def\Figp@intregTD#1:#2(#3,#4){\c@ntr@lnum{#1}%
    \def\c@@rdYZ{#4,\z@,\z@}\extrairelepremi@r\c@@rdY\de\c@@rdYZ%
    \extrairelepremi@r\c@@rdZ\de\c@@rdYZ%
    {\v@lX=#3\v@lY=\c@@rdY\v@lZ=\c@@rdZ\Fig@dmpt{#2}{\the\v@lZ}%
    \b@undb@xTD{\v@lX}{\v@lY}{\v@lZ}}\ignorespaces}
\ctr@ln@m\figptBezier
\ctr@ld@f\def\figptBezierDD#1:#2:#3[#4,#5,#6,#7]{\ifGR@cri{\s@uvc@ntr@l\et@tfigptBezierDD%
    \FigptBezier@#3[#4,#5,#6,#7]\Figp@intregDD#1:{#2}(\v@lX,\v@lY)%
    \resetc@ntr@l\et@tfigptBezierDD}\ignorespaces\fi}
\ctr@ld@f\def\figptBezierTD#1:#2:#3[#4,#5,#6,#7]{\ifGR@cri{\s@uvc@ntr@l\et@tfigptBezierTD%
    \FigptBezier@#3[#4,#5,#6,#7]\Figp@intregTD#1:{#2}(\v@lX,\v@lY,\v@lZ)%
    \resetc@ntr@l\et@tfigptBezierTD}\ignorespaces\fi}
\ctr@ld@f\def\FigptBezier@#1[#2,#3,#4,#5]{\setc@ntr@l{2}%
    \edef\T@{#1}\v@leur=\p@\advance\v@leur-#1pt\edef\UNmT@{\repdecn@mb{\v@leur}}%
    \figptcopy-4:/#2/\figptcopy-3:/#3/\figptcopy-2:/#4/\figptcopy-1:/#5/%
    \l@mbd@un=-4 \l@mbd@de=-\thr@@\p@rtent=\m@ne\c@lDecast%
    \l@mbd@un=-4 \l@mbd@de=-\thr@@\p@rtent=-\tw@\c@lDecast%
    \l@mbd@un=-4 \l@mbd@de=-\thr@@\p@rtent=-\thr@@\c@lDecast\Figg@tXY{-4}}
\ctr@ln@m\c@lDCUn
\ctr@ld@f\def\c@lDCUnDD#1#2{\Figg@tXY{#1}\v@lX=\UNmT@\v@lX\v@lY=\UNmT@\v@lY%
    \Figg@tXYa{#2}\advance\v@lX\T@\v@lXa\advance\v@lY\T@\v@lYa%
    \Figp@intregDD#1:(\v@lX,\v@lY)}
\ctr@ld@f\def\c@lDCUnTD#1#2{\Figg@tXY{#1}\v@lX=\UNmT@\v@lX\v@lY=\UNmT@\v@lY\v@lZ=\UNmT@\v@lZ%
    \Figg@tXYa{#2}\advance\v@lX\T@\v@lXa\advance\v@lY\T@\v@lYa\advance\v@lZ\T@\v@lZa%
    \Figp@intregTD#1:(\v@lX,\v@lY,\v@lZ)}
\ctr@ld@f\def\c@lDecast{\relax\ifnum\l@mbd@un<\p@rtent\c@lDCUn{\l@mbd@un}{\l@mbd@de}%
    \advance\l@mbd@un\@ne\advance\l@mbd@de\@ne\c@lDecast\fi}
\ctr@ld@f\def\figptmap#1:#2=#3/#4/#5/{\ifGR@cri{\s@uvc@ntr@l\et@tfigptmap%
    \setc@ntr@l{2}\figvectP-1[#4,#3]\Figg@tXY{-1}%
    \pr@dMatV/#5/\figpttra#1:{#2}=#4/1,-1/%
    \resetc@ntr@l\et@tfigptmap}\ignorespaces\fi}
\ctr@ln@m\pr@dMatV
\ctr@ld@f\def\pr@dMatVDD/#1,#2;#3,#4/{\v@lXa=#1\v@lX\advance\v@lXa#2\v@lY%
    \v@lYa=#3\v@lX\advance\v@lYa#4\v@lY\Figv@ctCreg-1(\v@lXa,\v@lYa)}
\ctr@ld@f\def\pr@dMatVTD/#1,#2,#3;#4,#5,#6;#7,#8,#9/{%
    \v@lXa=#1\v@lX\advance\v@lXa#2\v@lY\advance\v@lXa#3\v@lZ%
    \v@lYa=#4\v@lX\advance\v@lYa#5\v@lY\advance\v@lYa#6\v@lZ%
    \v@lZa=#7\v@lX\advance\v@lZa#8\v@lY\advance\v@lZa#9\v@lZ%
    \Figv@ctCreg-1(\v@lXa,\v@lYa,\v@lZa)}
\ctr@ln@m\figptbary
\ctr@ld@f\def\figptbaryDD#1:#2[#3;#4]{\ifGR@cri{\edef\list@num{#3}\extrairelepremi@r\p@int\de\list@num%
    \s@mme=\z@\@ecfor\c@ef:=#4\do{\advance\s@mme\c@ef}%
    \edef\listec@ef{#4,0}\extrairelepremi@r\c@ef\de\listec@ef%
    \Figg@tXY{\p@int}\divide\v@lX\s@mme\divide\v@lY\s@mme%
    \multiply\v@lX\c@ef\multiply\v@lY\c@ef%
    \@ecfor\p@int:=\list@num\do{\extrairelepremi@r\c@ef\de\listec@ef%
           \Figg@tXYa{\p@int}\divide\v@lXa\s@mme\divide\v@lYa\s@mme%
           \multiply\v@lXa\c@ef\multiply\v@lYa\c@ef%
           \advance\v@lX\v@lXa\advance\v@lY\v@lYa}%
    \Figp@intregDD#1:{#2}(\v@lX,\v@lY)}\ignorespaces\fi}
\ctr@ld@f\def\figptbaryTD#1:#2[#3;#4]{\ifGR@cri{\edef\list@num{#3}\extrairelepremi@r\p@int\de\list@num%
    \s@mme=\z@\@ecfor\c@ef:=#4\do{\advance\s@mme\c@ef}%
    \edef\listec@ef{#4,0}\extrairelepremi@r\c@ef\de\listec@ef%
    \Figg@tXY{\p@int}\divide\v@lX\s@mme\divide\v@lY\s@mme\divide\v@lZ\s@mme%
    \multiply\v@lX\c@ef\multiply\v@lY\c@ef\multiply\v@lZ\c@ef%
    \@ecfor\p@int:=\list@num\do{\extrairelepremi@r\c@ef\de\listec@ef%
           \Figg@tXYa{\p@int}\divide\v@lXa\s@mme\divide\v@lYa\s@mme\divide\v@lZa\s@mme%
           \multiply\v@lXa\c@ef\multiply\v@lYa\c@ef\multiply\v@lZa\c@ef%
           \advance\v@lX\v@lXa\advance\v@lY\v@lYa\advance\v@lZ\v@lZa}%
    \Figp@intregTD#1:{#2}(\v@lX,\v@lY,\v@lZ)}\ignorespaces\fi}
\ctr@ld@f\def\figptbaryR#1:#2[#3;#4]{\ifGR@cri{%
    \v@leur=\z@\@ecfor\c@ef:=#4\do{\maxim@m{\v@lmax}{\c@ef pt}{-\c@ef pt}%
    \ifdim\v@lmax>\v@leur\v@leur=\v@lmax\fi}%
    \ifdim\v@leur<\p@\f@ctech=\@M\else\ifdim\v@leur<\t@n\p@\f@ctech=\@m\else%
    \ifdim\v@leur<\c@nt\p@\f@ctech=\c@nt\else\ifdim\v@leur<\@m\p@\f@ctech=\t@n\else%
    \f@ctech=\@ne\fi\fi\fi\fi%
    \def\listec@ef{0}%
    \@ecfor\c@ef:=#4\do{\sc@lec@nvRI{\c@ef pt}\edef\listec@ef{\listec@ef,\the\s@mme}}%
    \extrairelepremi@r\c@ef\de\listec@ef\figptbary#1:#2[#3;\listec@ef]}\ignorespaces\fi}
\ctr@ld@f\def\sc@lec@nvRI#1{\v@leur=#1\p@rtentiere{\s@mme}{\v@leur}\advance\v@leur-\s@mme\p@%
    \multiply\v@leur\f@ctech\p@rtentiere{\p@rtent}{\v@leur}%
    \multiply\s@mme\f@ctech\advance\s@mme\p@rtent}
\ctr@ln@m\figptcirc
\ctr@ld@f\def\figptcircDD#1:#2:#3;#4(#5){\ifGR@cri{\s@uvc@ntr@l\et@tfigptcircDD%
    \c@lptellDD#1:{#2}:#3;#4,#4(#5)\resetc@ntr@l\et@tfigptcircDD}\ignorespaces\fi}
\ctr@ld@f\def\figptcircTD#1:#2:#3,#4,#5;#6(#7){\ifGR@cri{\s@uvc@ntr@l\et@tfigptcircTD%
    \setc@ntr@l{2}\c@lExtAxes#3,#4,#5(#6)\figptellP#1:{#2}:#3,-4,-5(#7)%
    \resetc@ntr@l\et@tfigptcircTD}\ignorespaces\fi}
\ctr@ln@m\figptcircumcenter
\ctr@ld@f\def\figptcircumcenterDD#1:#2[#3,#4,#5]{\ifGR@cri{\s@uvc@ntr@l\et@tfigptcircumcenterDD%
    \setc@ntr@l{2}\figvectNDD-5[#3,#4]\figptbaryDD-3:[#3,#4;1,1]%
                  \figvectNDD-6[#4,#5]\figptbaryDD-4:[#4,#5;1,1]%
    \resetc@ntr@l{2}\inters@cDD#1:{#2}[-3,-5;-4,-6]%
    \resetc@ntr@l\et@tfigptcircumcenterDD}\ignorespaces\fi}
\ctr@ld@f\def\figptcircumcenterTD#1:#2[#3,#4,#5]{\ifGR@cri{\s@uvc@ntr@l\et@tfigptcircumcenterTD%
    \setc@ntr@l{2}\figvectNTD-1[#3,#4,#5]%
    \figvectPTD-3[#3,#4]\figvectNVTD-5[-1,-3]\figptbaryTD-3:[#3,#4;1,1]%
    \figvectPTD-4[#4,#5]\figvectNVTD-6[-1,-4]\figptbaryTD-4:[#4,#5;1,1]%
    \resetc@ntr@l{2}\inters@cTD#1:{#2}[-3,-5;-4,-6]%
    \resetc@ntr@l\et@tfigptcircumcenterTD}\ignorespaces\fi}
\ctr@ln@m\figptcopy
\ctr@ld@f\def\figptcopyDD#1:#2/#3/{\ifGR@cri{\Figg@tXY{#3}%
    \Figp@intregDD#1:{#2}(\v@lX,\v@lY)}\ignorespaces\fi}
\ctr@ld@f\def\figptcopyTD#1:#2/#3/{\ifGR@cri{\Figg@tXY{#3}%
    \Figp@intregTD#1:{#2}(\v@lX,\v@lY,\v@lZ)}\ignorespaces\fi}
\ctr@ln@m\figptcurvcenter
\ctr@ld@f\def\figptcurvcenterDD#1:#2:#3[#4,#5,#6,#7]{\ifGR@cri{\s@uvc@ntr@l\et@tfigptcurvcenterDD%
    \setc@ntr@l{2}\c@lcurvradDD#3[#4,#5,#6,#7]\edef\Sprim@{\repdecn@mb{\result@t}}%
    \figptBezierDD-1::#3[#4,#5,#6,#7]\figpttraDD#1:{#2}=-1/\Sprim@,-5/%
    \resetc@ntr@l\et@tfigptcurvcenterDD}\ignorespaces\fi}
\ctr@ld@f\def\figptcurvcenterTD#1:#2:#3[#4,#5,#6,#7]{\ifGR@cri{\s@uvc@ntr@l\et@tfigptcurvcenterTD%
    \setc@ntr@l{2}\figvectDBezierTD -5:1,#3[#4,#5,#6,#7]%
    \figvectDBezierTD -6:2,#3[#4,#5,#6,#7]\vecunit@TD{-5}{-5}%
    \edef\Sprim@{\repdecn@mb{\result@t}}\figvectNVTD-1[-6,-5]%
    \figvectNVTD-5[-5,-1]\c@lproscalTD\v@leur[-6,-5]%
    \invers@{\v@leur}{\v@leur}\v@leur=\Sprim@\v@leur\v@leur=\Sprim@\v@leur%
    \figptBezierTD-1::#3[#4,#5,#6,#7]\edef\Sprim@{\repdecn@mb{\v@leur}}%
    \figpttraTD#1:{#2}=-1/\Sprim@,-5/\resetc@ntr@l\et@tfigptcurvcenterTD}\ignorespaces\fi}
\ctr@ld@f\def\c@lcurvradDD#1[#2,#3,#4,#5]{{\figvectDBezierDD -5:1,#1[#2,#3,#4,#5]%
    \figvectDBezierDD -6:2,#1[#2,#3,#4,#5]\vecunit@DD{-5}{-5}%
    \edef\Sprim@{\repdecn@mb{\result@t}}\figvectNVDD-5[-5]\c@lproscalDD\v@leur[-6,-5]%
    \invers@{\v@leur}{\v@leur}\v@leur=\Sprim@\v@leur\v@leur=\Sprim@\v@leur%
    \global\result@t=\v@leur}}
\ctr@ln@m\figptell
\ctr@ld@f\def\figptellDD#1:#2:#3;#4,#5(#6,#7){\ifGR@cri{\s@uvc@ntr@l\et@tfigptell%
    \c@lptellDD#1::#3;#4,#5(#6)\figptrotDD#1:{#2}=#1/#3,#7/%
    \resetc@ntr@l\et@tfigptell}\ignorespaces\fi}
\ctr@ld@f\def\c@lptellDD#1:#2:#3;#4,#5(#6){\c@ssin{\C@}{\S@}{#6}\v@lmin=\C@ pt\v@lmax=\S@ pt%
    \v@lmin=#4\v@lmin\v@lmax=#5\v@lmax%
    \edef\Xc@mp{\repdecn@mb{\v@lmin}}\edef\Yc@mp{\repdecn@mb{\v@lmax}}%
    \setc@ntr@l{2}\figvectC-1(\Xc@mp,\Yc@mp)\figpttraDD#1:{#2}=#3/1,-1/}
\ctr@ld@f\def\figptellP#1:#2:#3,#4,#5(#6){\ifGR@cri{\s@uvc@ntr@l\et@tfigptellP%
    \setc@ntr@l{2}\figvectP-1[#3,#4]\figvectP-2[#3,#5]%
    \v@leur=#6pt\c@lptellP{#3}{-1}{-2}\figptcopy#1:{#2}/-3/%
    \resetc@ntr@l\et@tfigptellP}\ignorespaces\fi}
\ctr@ln@m\@ngle
\ctr@ld@f\def\c@lptellP#1#2#3{\edef\@ngle{\repdecn@mb\v@leur}\c@ssin{\C@}{\S@}{\@ngle}%
    \figpttra-3:=#1/\C@,#2/\figpttra-3:=-3/\S@,#3/}
\ctr@ln@m\figptendnormal
\ctr@ld@f\def\figptendnormalDD#1:#2:#3,#4[#5,#6]{\ifGR@cri{\s@uvc@ntr@l\et@tfigptendnormal%
    \Figg@tXYa{#5}\Figg@tXY{#6}%
    \advance\v@lX-\v@lXa\advance\v@lY-\v@lYa%
    \setc@ntr@l{2}\Figv@ctCreg-1(\v@lX,\v@lY)\vecunit@{-1}{-1}\Figg@tXY{-1}%
    \delt@=#3\unit@\maxim@m{\delt@}{\delt@}{-\delt@}\edef\l@ngueur{\repdecn@mb{\delt@}}%
    \v@lX=\l@ngueur\v@lX\v@lY=\l@ngueur\v@lY%
    \delt@=\p@\advance\delt@-#4pt\edef\l@ngueur{\repdecn@mb{\delt@}}%
    \figptbaryR-1:[#5,#6;#4,\l@ngueur]\Figg@tXYa{-1}%
    \advance\v@lXa\v@lY\advance\v@lYa-\v@lX%
    \setc@ntr@l{1}\Figp@intregDD#1:{#2}(\v@lXa,\v@lYa)\resetc@ntr@l\et@tfigptendnormal}%
    \ignorespaces\fi}
\ctr@ld@f\def\figptexcenter#1:#2[#3,#4,#5]{\ifGR@cri{\let@xte={-}
    \Figptexinsc@nter#1:#2[#3,#4,#5]}\ignorespaces\fi}
\ctr@ld@f\def\figptincenter#1:#2[#3,#4,#5]{\ifGR@cri{\let@xte={}
    \Figptexinsc@nter#1:#2[#3,#4,#5]}\ignorespaces\fi}
\ctr@ld@f
\ctr@ld@f\def\Figptexinsc@nter#1:#2[#3,#4,#5]{%
    \figgetdist\LA@[#4,#5]\figgetdist\LB@[#3,#5]\figgetdist\LC@[#3,#4]%
    \figptbaryR#1:{#2}[#3,#4,#5;\the\let@xte\LA@,\LB@,\LC@]}
\ctr@ln@m\figptinterlineplane
\ctr@ld@f\def\figptinterlineplaneDD{\un@v@ilable{figptinterlineplane}}
\ctr@ld@f\def\figptinterlineplaneTD#1:#2[#3,#4;#5,#6]{\ifGR@cri{\s@uvc@ntr@l\et@tfigptinterlineplane%
    \setc@ntr@l{2}\figvectPTD-1[#3,#5]\vecunit@TD{-2}{#6}%
    \r@pPSTD\v@leur[-2,-1,#4]\edef\v@lcoef{\repdecn@mb{\v@leur}}%
    \figpttraTD#1:{#2}=#3/\v@lcoef,#4/\resetc@ntr@l\et@tfigptinterlineplane}\ignorespaces\fi}
\ctr@ln@m\figptorthocenter
\ctr@ld@f\def\figptorthocenterDD#1:#2[#3,#4,#5]{\ifGR@cri{\s@uvc@ntr@l\et@tfigptorthocenterDD%
    \setc@ntr@l{2}\figvectNDD-3[#3,#4]\figvectNDD-4[#4,#5]%
    \resetc@ntr@l{2}\inters@cDD#1:{#2}[#5,-3;#3,-4]%
    \resetc@ntr@l\et@tfigptorthocenterDD}\ignorespaces\fi}
\ctr@ld@f\def\figptorthocenterTD#1:#2[#3,#4,#5]{\ifGR@cri{\s@uvc@ntr@l\et@tfigptorthocenterTD%
    \setc@ntr@l{2}\figvectNTD-1[#3,#4,#5]%
    \figvectPTD-2[#3,#4]\figvectNVTD-3[-1,-2]%
    \figvectPTD-2[#4,#5]\figvectNVTD-4[-1,-2]%
    \resetc@ntr@l{2}\inters@cTD#1:{#2}[#5,-3;#3,-4]%
    \resetc@ntr@l\et@tfigptorthocenterTD}\ignorespaces\fi}
\ctr@ln@m\figptorthoprojline
\ctr@ld@f\def\figptorthoprojlineDD#1:#2=#3/#4,#5/{\ifGR@cri{\s@uvc@ntr@l\et@tfigptorthoprojlineDD%
    \setc@ntr@l{2}\figvectPDD-3[#4,#5]\figvectNVDD-4[-3]\resetc@ntr@l{2}%
    \inters@cDD#1:{#2}[#3,-4;#4,-3]\resetc@ntr@l\et@tfigptorthoprojlineDD}\ignorespaces\fi}
\ctr@ld@f\def\figptorthoprojlineTD#1:#2=#3/#4,#5/{\ifGR@cri{\s@uvc@ntr@l\et@tfigptorthoprojlineTD%
    \setc@ntr@l{2}\figvectPTD-1[#4,#3]\figvectPTD-2[#4,#5]\vecunit@TD{-2}{-2}%
    \c@lproscalTD\v@leur[-1,-2]\edef\v@lcoef{\repdecn@mb{\v@leur}}%
    \figpttraTD#1:{#2}=#4/\v@lcoef,-2/\resetc@ntr@l\et@tfigptorthoprojlineTD}\ignorespaces\fi}
\ctr@ln@m\figptorthoprojplane
\ctr@ld@f\def\figptorthoprojplaneDD{\un@v@ilable{figptorthoprojplane}}
\ctr@ld@f\def\figptorthoprojplaneTD#1:#2=#3/#4,#5/{\ifGR@cri{\s@uvc@ntr@l\et@tfigptorthoprojplane%
    \setc@ntr@l{2}\figvectPTD-1[#3,#4]\vecunit@TD{-2}{#5}%
    \c@lproscalTD\v@leur[-1,-2]\edef\v@lcoef{\repdecn@mb{\v@leur}}%
    \figpttraTD#1:{#2}=#3/\v@lcoef,-2/\resetc@ntr@l\et@tfigptorthoprojplane}\ignorespaces\fi}
\ctr@ld@f\def\figpthom#1:#2=#3/#4,#5/{\ifGR@cri{\s@uvc@ntr@l\et@tfigpthom%
    \setc@ntr@l{2}\figvectP-1[#4,#3]\figpttra#1:{#2}=#4/#5,-1/%
    \resetc@ntr@l\et@tfigpthom}\ignorespaces\fi}
\ctr@ld@f\def\figptinv#1:#2=#3/#4,#5/{\ifGR@cri{\s@uvc@ntr@l\et@tfigptinv%
    \setc@ntr@l{2}\figvectP-1[#4,#3]\Figg@tXY{-1}%
    \getredf@ctB\f@ctech\n@rmeucC{\delt@}{-1}%
    \delt@=\ptT@unit@\delt@\delt@=\ptT@unit@\delt@%
    \invers@{\delt@}{\delt@}\multiply\f@ctech\f@ctech\divide\delt@\f@ctech%
    \delt@=#5\delt@\edef\v@lcoef{\repdecn@mb{\delt@}}\figpttra#1:{#2}=#4/\v@lcoef,-1/%
    \resetc@ntr@l\et@tfigptinv}\ignorespaces\fi}
\ctr@ln@m\figptrot
\ctr@ld@f\def\figptrotDD#1:#2=#3/#4,#5/{\ifGR@cri{\s@uvc@ntr@l\et@tfigptrotDD%
    \c@ssin{\C@}{\S@}{#5}\setc@ntr@l{2}\figvectPDD-1[#4,#3]\Figg@tXY{-1}%
    \v@lXa=\C@\v@lX\advance\v@lXa-\S@\v@lY%
    \v@lYa=\S@\v@lX\advance\v@lYa\C@\v@lY%
    \Figv@ctCreg-1(\v@lXa,\v@lYa)\figpttraDD#1:{#2}=#4/1,-1/%
    \resetc@ntr@l\et@tfigptrotDD}\ignorespaces\fi}
\ctr@ld@f\def\figptrotTD#1:#2=#3/#4,#5,#6/{\ifGR@cri{\s@uvc@ntr@l\et@tfigptrotTD%
    \c@ssin{\C@}{\S@}{#5}%
    \setc@ntr@l{2}\figptorthoprojplaneTD-3:=#4/#3,#6/\figvectPTD-2[-3,#3]%
    \n@rmeucTD\v@leur{-2}\ifdim\v@leur<\Cepsil@n\Figg@tXYa{#3}\else%
    \edef\v@lcoef{\repdecn@mb{\v@leur}}\figvectNVTD-1[#6,-2]%
    \Figg@tXYa{-1}\v@lXa=\v@lcoef\v@lXa\v@lYa=\v@lcoef\v@lYa\v@lZa=\v@lcoef\v@lZa%
    \v@lXa=\S@\v@lXa\v@lYa=\S@\v@lYa\v@lZa=\S@\v@lZa\Figg@tXY{-2}%
    \advance\v@lXa\C@\v@lX\advance\v@lYa\C@\v@lY\advance\v@lZa\C@\v@lZ%
    \Figg@tXY{-3}\advance\v@lXa\v@lX\advance\v@lYa\v@lY\advance\v@lZa\v@lZ\fi%
    \Figp@intregTD#1:{#2}(\v@lXa,\v@lYa,\v@lZa)\resetc@ntr@l\et@tfigptrotTD}\ignorespaces\fi}
\ctr@ln@m\figptsym
\ctr@ld@f\def\figptsymDD#1:#2=#3/#4,#5/{\ifGR@cri{\s@uvc@ntr@l\et@tfigptsymDD%
    \resetc@ntr@l{2}\figptorthoprojlineDD-5:=#3/#4,#5/\figvectPDD-2[#3,-5]%
    \figpttraDD#1:{#2}=#3/2,-2/\resetc@ntr@l\et@tfigptsymDD}\ignorespaces\fi}
\ctr@ld@f\def\figptsymTD#1:#2=#3/#4,#5/{\ifGR@cri{\s@uvc@ntr@l\et@tfigptsymTD%
    \resetc@ntr@l{2}\figptorthoprojplaneTD-3:=#3/#4,#5/\figvectPTD-2[#3,-3]%
    \figpttraTD#1:{#2}=#3/2,-2/\resetc@ntr@l\et@tfigptsymTD}\ignorespaces\fi}
\ctr@ln@m\figpttra
\ctr@ld@f\def\figpttraDD#1:#2=#3/#4,#5/{\ifGR@cri{\Figg@tXYa{#5}\v@lXa=#4\v@lXa\v@lYa=#4\v@lYa%
    \Figg@tXY{#3}\advance\v@lX\v@lXa\advance\v@lY\v@lYa%
    \Figp@intregDD#1:{#2}(\v@lX,\v@lY)}\ignorespaces\fi}
\ctr@ld@f\def\figpttraTD#1:#2=#3/#4,#5/{\ifGR@cri{\Figg@tXYa{#5}\v@lXa=#4\v@lXa\v@lYa=#4\v@lYa%
    \v@lZa=#4\v@lZa\Figg@tXY{#3}\advance\v@lX\v@lXa\advance\v@lY\v@lYa%
    \advance\v@lZ\v@lZa\Figp@intregTD#1:{#2}(\v@lX,\v@lY,\v@lZ)}\ignorespaces\fi}
\ctr@ln@m\figpttraC
\ctr@ld@f\def\figpttraCDD#1:#2=#3/#4,#5/{\ifGR@cri{\v@lXa=#4\unit@\v@lYa=#5\unit@%
    \Figg@tXY{#3}\advance\v@lX\v@lXa\advance\v@lY\v@lYa%
    \Figp@intregDD#1:{#2}(\v@lX,\v@lY)}\ignorespaces\fi}
\ctr@ld@f\def\figpttraCTD#1:#2=#3/#4,#5,#6/{\ifGR@cri{\v@lXa=#4\unit@\v@lYa=#5\unit@\v@lZa=#6\unit@%
    \Figg@tXY{#3}\advance\v@lX\v@lXa\advance\v@lY\v@lYa\advance\v@lZ\v@lZa%
    \Figp@intregTD#1:{#2}(\v@lX,\v@lY,\v@lZ)}\ignorespaces\fi}
\ctr@ld@f\def\figptsaxes#1:#2(#3){\ifGR@cri{\an@lys@xes#3,:\ifx\t@xt@\empty%
    \ifTr@isDim\Figpts@xes#1:#2(0,#3,0,#3,0,#3)\else\Figpts@xes#1:#2(0,#3,0,#3)\fi%
    \else\Figpts@xes#1:#2(#3)\fi}\ignorespaces\fi}
\ctr@ln@m\Figpts@xes
\ctr@ld@f\def\Figpts@xesDD#1:#2(#3,#4,#5,#6){%
    \s@mme=#1\figpttraC\the\s@mme:$x$=#2/#4,0/%
    \advance\s@mme\@ne\figpttraC\the\s@mme:$y$=#2/0,#6/}
\ctr@ld@f\def\Figpts@xesTD#1:#2(#3,#4,#5,#6,#7,#8){%
    \s@mme=#1\figpttraC\the\s@mme:$x$=#2/#4,0,0/%
    \advance\s@mme\@ne\figpttraC\the\s@mme:$y$=#2/0,#6,0/%
    \advance\s@mme\@ne\figpttraC\the\s@mme:$z$=#2/0,0,#8/}
\ctr@ld@f\def\figptsmap#1=#2/#3/#4/{\ifGR@cri{\s@uvc@ntr@l\et@tfigptsmap%
    \setc@ntr@l{2}\def\list@num{#2}\s@mme=#1%
    \@ecfor\p@int:=\list@num\do{\figvectP-1[#3,\p@int]\Figg@tXY{-1}%
    \pr@dMatV/#4/\figpttra\the\s@mme:=#3/1,-1/\advance\s@mme\@ne}%
    \resetc@ntr@l\et@tfigptsmap}\ignorespaces\fi}
\ctr@ln@m\figptscontrol
\ctr@ld@f\def\figptscontrolDD#1[#2,#3,#4,#5]{\ifGR@cri{\s@uvc@ntr@l\et@tfigptscontrolDD\setc@ntr@l{2}%
    \v@lX=\z@\v@lY=\z@\Figtr@nptDD{-5}{#2}\Figtr@nptDD{2}{#5}%
    \divide\v@lX\@vi\divide\v@lY\@vi%
    \Figtr@nptDD{3}{#3}\Figtr@nptDD{-1.5}{#4}\Figp@intregDD-1:(\v@lX,\v@lY)%
    \v@lX=\z@\v@lY=\z@\Figtr@nptDD{2}{#2}\Figtr@nptDD{-5}{#5}%
    \divide\v@lX\@vi\divide\v@lY\@vi\Figtr@nptDD{-1.5}{#3}\Figtr@nptDD{3}{#4}%
    \s@mme=#1\advance\s@mme\@ne\Figp@intregDD\the\s@mme:(\v@lX,\v@lY)%
    \figptcopyDD#1:/-1/\resetc@ntr@l\et@tfigptscontrolDD}\ignorespaces\fi}
\ctr@ld@f\def\figptscontrolTD#1[#2,#3,#4,#5]{\ifGR@cri{\s@uvc@ntr@l\et@tfigptscontrolTD\setc@ntr@l{2}%
    \v@lX=\z@\v@lY=\z@\v@lZ=\z@\Figtr@nptTD{-5}{#2}\Figtr@nptTD{2}{#5}%
    \divide\v@lX\@vi\divide\v@lY\@vi\divide\v@lZ\@vi%
    \Figtr@nptTD{3}{#3}\Figtr@nptTD{-1.5}{#4}\Figp@intregTD-1:(\v@lX,\v@lY,\v@lZ)%
    \v@lX=\z@\v@lY=\z@\v@lZ=\z@\Figtr@nptTD{2}{#2}\Figtr@nptTD{-5}{#5}%
    \divide\v@lX\@vi\divide\v@lY\@vi\divide\v@lZ\@vi\Figtr@nptTD{-1.5}{#3}\Figtr@nptTD{3}{#4}%
    \s@mme=#1\advance\s@mme\@ne\Figp@intregTD\the\s@mme:(\v@lX,\v@lY,\v@lZ)%
    \figptcopyTD#1:/-1/\resetc@ntr@l\et@tfigptscontrolTD}\ignorespaces\fi}
\ctr@ld@f\def\Figtr@nptDD#1#2{\Figg@tXYa{#2}\v@lXa=#1\v@lXa\v@lYa=#1\v@lYa%
    \advance\v@lX\v@lXa\advance\v@lY\v@lYa}
\ctr@ld@f\def\Figtr@nptTD#1#2{\Figg@tXYa{#2}\v@lXa=#1\v@lXa\v@lYa=#1\v@lYa\v@lZa=#1\v@lZa%
    \advance\v@lX\v@lXa\advance\v@lY\v@lYa\advance\v@lZ\v@lZa}
\ctr@ld@f\def\figptscontrolcurve#1,#2[#3]{\ifGR@cri{\s@uvc@ntr@l\et@tfigptscontrolcurve%
    \def\list@num{#3}\extrairelepremi@r\Ak@\de\list@num%
    \extrairelepremi@r\Ai@\de\list@num\extrairelepremi@r\Aj@\de\list@num%
    \s@mme=#1\figptcopy\the\s@mme:/\Ai@/%
    \setc@ntr@l{2}\figvectP -1[\Ak@,\Aj@]%
    \@ecfor\Ak@:=\list@num\do{\advance\s@mme\@ne\figpttra\the\s@mme:=\Ai@/\curv@roundness,-1/%
       \figvectP -1[\Ai@,\Ak@]\advance\s@mme\@ne\figpttra\the\s@mme:=\Aj@/-\curv@roundness,-1/%
       \advance\s@mme\@ne\figptcopy\the\s@mme:/\Aj@/%
       \edef\Ai@{\Aj@}\edef\Aj@{\Ak@}}\advance\s@mme-#1\divide\s@mme\thr@@%
       \xdef#2{\the\s@mme}%
    \resetc@ntr@l\et@tfigptscontrolcurve}\ignorespaces\fi}
\ctr@ln@m\figptsintercirc
\ctr@ld@f\def\figptsintercircDD#1[#2,#3;#4,#5]{\ifGR@cri{\s@uvc@ntr@l\et@tfigptsintercircDD%
    \setc@ntr@l{2}\let\c@lNVintc=\c@lNVintcDD\Figptsintercirc@#1[#2,#3;#4,#5]%
    \resetc@ntr@l\et@tfigptsintercircDD}\ignorespaces\fi}
\ctr@ld@f\def\figptsintercircTD#1[#2,#3;#4,#5;#6]{\ifGR@cri{\s@uvc@ntr@l\et@tfigptsintercircTD%
    \setc@ntr@l{2}\let\c@lNVintc=\c@lNVintcTD\vecunitC@TD[#2,#6]%
    \Figv@ctCreg-3(\v@lX,\v@lY,\v@lZ)\Figptsintercirc@#1[#2,#3;#4,#5]%
    \resetc@ntr@l\et@tfigptsintercircTD}\ignorespaces\fi}
\ctr@ld@f\def\Figptsintercirc@#1[#2,#3;#4,#5]{\figvectP-1[#2,#4]%
    \vecunit@{-1}{-1}\delt@=\result@t\f@ctech=\result@tent%
    \s@mme=#1\advance\s@mme\@ne\figptcopy#1:/#2/\figptcopy\the\s@mme:/#4/%
    \ifdim\delt@=\z@\else%
    \v@lmin=#3\unit@\v@lmax=#5\unit@\v@leur=\v@lmin\advance\v@leur\v@lmax%
    \ifdim\v@leur>\delt@%
    \v@leur=\v@lmin\advance\v@leur-\v@lmax\maxim@m{\v@leur}{\v@leur}{-\v@leur}%
    \ifdim\v@leur<\delt@%
    \divide\v@lmin\f@ctech\divide\v@lmax\f@ctech\divide\delt@\f@ctech%
    \v@lmin=\repdecn@mb{\v@lmin}\v@lmin\v@lmax=\repdecn@mb{\v@lmax}\v@lmax%
    \invers@{\v@leur}{\delt@}\advance\v@lmax-\v@lmin%
    \v@lmax=-\repdecn@mb{\v@leur}\v@lmax\advance\delt@\v@lmax\delt@=.5\delt@%
    \v@lmax=\delt@\multiply\v@lmax\f@ctech%
    \edef\t@ille{\repdecn@mb{\v@lmax}}\figpttra-2:=#2/\t@ille,-1/%
    \delt@=\repdecn@mb{\delt@}\delt@\advance\v@lmin-\delt@%
    \sqrt@{\v@leur}{\v@lmin}\multiply\v@leur\f@ctech\edef\t@ille{\repdecn@mb{\v@leur}}%
    \c@lNVintc\figpttra#1:=-2/-\t@ille,-1/\figpttra\the\s@mme:=-2/\t@ille,-1/\fi\fi\fi}
\ctr@ld@f\def\c@lNVintcDD{\Figg@tXY{-1}\Figv@ctCreg-1(-\v@lY,\v@lX)} 
\ctr@ld@f\def\c@lNVintcTD{{\Figg@tXY{-3}\v@lmin=\v@lX\v@lmax=\v@lY\v@leur=\v@lZ%
    \Figg@tXY{-1}\c@lprovec{-3}\vecunit@{-3}{-3}
    \Figg@tXY{-1}\v@lmin=\v@lX\v@lmax=\v@lY%
    \v@leur=\v@lZ\Figg@tXY{-3}\c@lprovec{-1}}} 
\ctr@ln@m\figptsinterlinell
\ctr@ld@f\def\figptsinterlinellDD#1[#2,#3,#4,#5;#6,#7]{\ifGR@cri{\s@uvc@ntr@l\et@tfigptsinterlinellDD%
    \figptcopy#1:/#6/\s@mme=#1\advance\s@mme\@ne\figptcopy\the\s@mme:/#7/%
    \v@lmin=#3\unit@\v@lmax=#4\unit@
    \setc@ntr@l{2}\figptbaryDD-4:[#6,#7;1,1]\figptsrotDD-3=-4,#7/#2,-#5/
    \Figg@tXY{-3}\Figg@tXYa{#2}\advance\v@lX-\v@lXa\advance\v@lY-\v@lYa
    \figvectP-1[-3,-2]\Figg@tXYa{-1}\figvectP-3[-4,#7]\Figptsint@rLE{#1}
    \resetc@ntr@l\et@tfigptsinterlinellDD}\ignorespaces\fi}
\ctr@ld@f\def\figptsinterlinellP#1[#2,#3,#4;#5,#6]{\ifGR@cri{\s@uvc@ntr@l\et@tfigptsinterlinellP%
    \figptcopy#1:/#5/\s@mme=#1\advance\s@mme\@ne\figptcopy\the\s@mme:/#6/\setc@ntr@l{2}%
    \figvectP-1[#2,#3]\vecunit@{-1}{-1}\v@lmin=\result@t
    \figvectP-2[#2,#4]\vecunit@{-2}{-2}\v@lmax=\result@t
    \figptbary-4:[#5,#6;1,1]
    \figvectP-3[#2,-4]\c@lproscal\v@lX[-3,-1]\c@lproscal\v@lY[-3,-2]
    \figvectP-3[-4,#6]\c@lproscal\v@lXa[-3,-1]\c@lproscal\v@lYa[-3,-2]
    \Figptsint@rLE{#1}\resetc@ntr@l\et@tfigptsinterlinellP}\ignorespaces\fi}
\ctr@ld@f\def\Figptsint@rLE#1{%
    \getredf@ctDD\f@ctech(\v@lmin,\v@lmax)%
    \getredf@ctDD\p@rtent(\v@lX,\v@lY)\ifnum\p@rtent>\f@ctech\f@ctech=\p@rtent\fi%
    \getredf@ctDD\p@rtent(\v@lXa,\v@lYa)\ifnum\p@rtent>\f@ctech\f@ctech=\p@rtent\fi%
    \divide\v@lmin\f@ctech\divide\v@lmax\f@ctech\divide\v@lX\f@ctech\divide\v@lY\f@ctech%
    \divide\v@lXa\f@ctech\divide\v@lYa\f@ctech%
    \c@rre=\repdecn@mb\v@lXa\v@lmax\mili@u=\repdecn@mb\v@lYa\v@lmin%
    \getredf@ctDD\f@ctech(\c@rre,\mili@u)%
    \c@rre=\repdecn@mb\v@lX\v@lmax\mili@u=\repdecn@mb\v@lY\v@lmin%
    \getredf@ctDD\p@rtent(\c@rre,\mili@u)\ifnum\p@rtent>\f@ctech\f@ctech=\p@rtent\fi%
    \divide\v@lmin\f@ctech\divide\v@lmax\f@ctech\divide\v@lX\f@ctech\divide\v@lY\f@ctech%
    \divide\v@lXa\f@ctech\divide\v@lYa\f@ctech%
    \v@lmin=\repdecn@mb{\v@lmin}\v@lmin\v@lmax=\repdecn@mb{\v@lmax}\v@lmax%
    \edef\G@xde{\repdecn@mb\v@lmin}\edef\P@xde{\repdecn@mb\v@lmax}%
    \c@rre=-\v@lmax\v@leur=\repdecn@mb\v@lY\v@lY\advance\c@rre\v@leur\c@rre=\G@xde\c@rre%
    \v@leur=\repdecn@mb\v@lX\v@lX\v@leur=\P@xde\v@leur\advance\c@rre\v@leur
    \v@lmin=\repdecn@mb\v@lYa\v@lmin\v@lmax=\repdecn@mb\v@lXa\v@lmax%
    \mili@u=\repdecn@mb\v@lX\v@lmax\advance\mili@u\repdecn@mb\v@lY\v@lmin
    \v@lmax=\repdecn@mb\v@lXa\v@lmax\advance\v@lmax\repdecn@mb\v@lYa\v@lmin
    \ifdim\v@lmax>\epsil@n%
    \maxim@m{\v@leur}{\c@rre}{-\c@rre}\maxim@m{\v@lmin}{\mili@u}{-\mili@u}%
    \maxim@m{\v@leur}{\v@leur}{\v@lmin}\maxim@m{\v@lmin}{\v@lmax}{-\v@lmax}%
    \maxim@m{\v@leur}{\v@leur}{\v@lmin}\p@rtentiere{\p@rtent}{\v@leur}\advance\p@rtent\@ne%
    \divide\c@rre\p@rtent\divide\mili@u\p@rtent\divide\v@lmax\p@rtent%
    \delt@=\repdecn@mb{\mili@u}\mili@u\v@leur=\repdecn@mb{\v@lmax}\c@rre%
    \advance\delt@-\v@leur\ifdim\delt@<\z@\else\sqrt@\delt@\delt@%
    \invers@\v@lmax\v@lmax\edef\Uns@rAp{\repdecn@mb\v@lmax}%
    \v@leur=-\mili@u\advance\v@leur-\delt@\v@leur=\Uns@rAp\v@leur%
    \edef\t@ille{\repdecn@mb\v@leur}\figpttra#1:=-4/\t@ille,-3/\s@mme=#1\advance\s@mme\@ne%
    \v@leur=-\mili@u\advance\v@leur\delt@\v@leur=\Uns@rAp\v@leur%
    \edef\t@ille{\repdecn@mb\v@leur}\figpttra\the\s@mme:=-4/\t@ille,-3/\fi\fi}
\ctr@ln@m\figptsorthoprojline
\ctr@ld@f\def\figptsorthoprojlineDD#1=#2/#3,#4/{\ifGR@cri{\s@uvc@ntr@l\et@tfigptsorthoprojlineDD%
    \setc@ntr@l{2}\figvectPDD-3[#3,#4]\figvectNVDD-4[-3]\resetc@ntr@l{2}%
    \def\list@num{#2}\s@mme=#1\@ecfor\p@int:=\list@num\do{%
    \inters@cDD\the\s@mme:[\p@int,-4;#3,-3]\advance\s@mme\@ne}%
    \resetc@ntr@l\et@tfigptsorthoprojlineDD}\ignorespaces\fi}
\ctr@ld@f\def\figptsorthoprojlineTD#1=#2/#3,#4/{\ifGR@cri{\s@uvc@ntr@l\et@tfigptsorthoprojlineTD%
    \setc@ntr@l{2}\figvectPTD-2[#3,#4]\vecunit@TD{-2}{-2}%
    \def\list@num{#2}\s@mme=#1\@ecfor\p@int:=\list@num\do{%
    \figvectPTD-1[#3,\p@int]\c@lproscalTD\v@leur[-1,-2]%
    \edef\v@lcoef{\repdecn@mb{\v@leur}}\figpttraTD\the\s@mme:=#3/\v@lcoef,-2/%
    \advance\s@mme\@ne}\resetc@ntr@l\et@tfigptsorthoprojlineTD}\ignorespaces\fi}
\ctr@ln@m\figptsorthoprojplane
\ctr@ld@f\def\figptsorthoprojplaneDD{\un@v@ilable{figptsorthoprojplane}}
\ctr@ld@f\def\figptsorthoprojplaneTD#1=#2/#3,#4/{\ifGR@cri{\s@uvc@ntr@l\et@tfigptsorthoprojplane%
    \setc@ntr@l{2}\vecunit@TD{-2}{#4}%
    \def\list@num{#2}\s@mme=#1\@ecfor\p@int:=\list@num\do{\figvectPTD-1[\p@int,#3]%
    \c@lproscalTD\v@leur[-1,-2]\edef\v@lcoef{\repdecn@mb{\v@leur}}%
    \figpttraTD\the\s@mme:=\p@int/\v@lcoef,-2/\advance\s@mme\@ne}%
    \resetc@ntr@l\et@tfigptsorthoprojplane}\ignorespaces\fi}
\ctr@ld@f\def\figptshom#1=#2/#3,#4/{\ifGR@cri{\s@uvc@ntr@l\et@tfigptshom%
    \setc@ntr@l{2}\def\list@num{#2}\s@mme=#1%
    \@ecfor\p@int:=\list@num\do{\figvectP-1[#3,\p@int]%
    \figpttra\the\s@mme:=#3/#4,-1/\advance\s@mme\@ne}%
    \resetc@ntr@l\et@tfigptshom}\ignorespaces\fi}
\ctr@ld@f\def\figptsinv#1=#2/#3,#4/{\ifGR@cri{\s@uvc@ntr@l\et@tfigptsinv%
    \setc@ntr@l{2}\def\list@num{#2}\s@mme=#1%
    \@ecfor\p@int:=\list@num\do{\figvectP-1[#3,\p@int]\Figg@tXY{-1}%
    \getredf@ctB\f@ctech\n@rmeucC{\delt@}{-1}%
    \delt@=\ptT@unit@\delt@\delt@=\ptT@unit@\delt@%
    \invers@{\delt@}{\delt@}\multiply\f@ctech\f@ctech\divide\delt@\f@ctech%
    \delt@=#4\delt@\edef\v@lcoef{\repdecn@mb{\delt@}}\figpttra\the\s@mme:=#3/\v@lcoef,-1/%
    \advance\s@mme\@ne}\resetc@ntr@l\et@tfigptsinv}\ignorespaces\fi}
\ctr@ln@m\figptsrot
\ctr@ld@f\def\figptsrotDD#1=#2/#3,#4/{\ifGR@cri{\s@uvc@ntr@l\et@tfigptsrotDD%
    \c@ssin{\C@}{\S@}{#4}\setc@ntr@l{2}\def\list@num{#2}\s@mme=#1%
    \@ecfor\p@int:=\list@num\do{\figvectPDD-1[#3,\p@int]\Figg@tXY{-1}%
    \v@lXa=\C@\v@lX\advance\v@lXa-\S@\v@lY%
    \v@lYa=\S@\v@lX\advance\v@lYa\C@\v@lY%
    \Figv@ctCreg-1(\v@lXa,\v@lYa)\figpttraDD\the\s@mme:=#3/1,-1/\advance\s@mme\@ne}%
    \resetc@ntr@l\et@tfigptsrotDD}\ignorespaces\fi}
\ctr@ld@f\def\figptsrotTD#1=#2/#3,#4,#5/{\ifGR@cri{\s@uvc@ntr@l\et@tfigptsrotTD%
    \c@ssin{\C@}{\S@}{#4}%
    \setc@ntr@l{2}\def\list@num{#2}\s@mme=#1%
    \@ecfor\p@int:=\list@num\do{\figptorthoprojplaneTD-3:=#3/\p@int,#5/%
    \figvectPTD-2[-3,\p@int]%
    \figvectNVTD-1[#5,-2]\n@rmeucTD\v@leur{-2}\edef\v@lcoef{\repdecn@mb{\v@leur}}%
    \Figg@tXYa{-1}\v@lXa=\v@lcoef\v@lXa\v@lYa=\v@lcoef\v@lYa\v@lZa=\v@lcoef\v@lZa%
    \v@lXa=\S@\v@lXa\v@lYa=\S@\v@lYa\v@lZa=\S@\v@lZa\Figg@tXY{-2}%
    \advance\v@lXa\C@\v@lX\advance\v@lYa\C@\v@lY\advance\v@lZa\C@\v@lZ%
    \Figg@tXY{-3}\advance\v@lXa\v@lX\advance\v@lYa\v@lY\advance\v@lZa\v@lZ%
    \Figp@intregTD\the\s@mme:(\v@lXa,\v@lYa,\v@lZa)\advance\s@mme\@ne}%
    \resetc@ntr@l\et@tfigptsrotTD}\ignorespaces\fi}
\ctr@ln@m\figptssym
\ctr@ld@f\def\figptssymDD#1=#2/#3,#4/{\ifGR@cri{\s@uvc@ntr@l\et@tfigptssymDD%
    \setc@ntr@l{2}\figvectPDD-3[#3,#4]\Figg@tXY{-3}\Figv@ctCreg-4(-\v@lY,\v@lX)%
    \resetc@ntr@l{2}\def\list@num{#2}\s@mme=#1%
    \@ecfor\p@int:=\list@num\do{\inters@cDD-5:[#3,-3;\p@int,-4]\figvectPDD-2[\p@int,-5]%
    \figpttraDD\the\s@mme:=\p@int/2,-2/\advance\s@mme\@ne}%
    \resetc@ntr@l\et@tfigptssymDD}\ignorespaces\fi}
\ctr@ld@f\def\figptssymTD#1=#2/#3,#4/{\ifGR@cri{\s@uvc@ntr@l\et@tfigptssymTD%
    \setc@ntr@l{2}\vecunit@TD{-2}{#4}\def\list@num{#2}\s@mme=#1%
    \@ecfor\p@int:=\list@num\do{\figvectPTD-1[\p@int,#3]%
    \c@lproscalTD\v@leur[-1,-2]\v@leur=2\v@leur\edef\v@lcoef{\repdecn@mb{\v@leur}}%
    \figpttraTD\the\s@mme:=\p@int/\v@lcoef,-2/\advance\s@mme\@ne}%
    \resetc@ntr@l\et@tfigptssymTD}\ignorespaces\fi}
\ctr@ln@m\figptstra
\ctr@ld@f\def\figptstraDD#1=#2/#3,#4/{\ifGR@cri{\Figg@tXYa{#4}\v@lXa=#3\v@lXa\v@lYa=#3\v@lYa%
    \def\list@num{#2}\s@mme=#1\@ecfor\p@int:=\list@num\do{\Figg@tXY{\p@int}%
    \advance\v@lX\v@lXa\advance\v@lY\v@lYa%
    \Figp@intregDD\the\s@mme:(\v@lX,\v@lY)\advance\s@mme\@ne}}\ignorespaces\fi}
\ctr@ld@f\def\figptstraTD#1=#2/#3,#4/{\ifGR@cri{\Figg@tXYa{#4}\v@lXa=#3\v@lXa\v@lYa=#3\v@lYa%
    \v@lZa=#3\v@lZa\def\list@num{#2}\s@mme=#1\@ecfor\p@int:=\list@num\do{\Figg@tXY{\p@int}%
    \advance\v@lX\v@lXa\advance\v@lY\v@lYa\advance\v@lZ\v@lZa%
    \Figp@intregTD\the\s@mme:(\v@lX,\v@lY,\v@lZ)\advance\s@mme\@ne}}\ignorespaces\fi}
\ctr@ln@m\figptvisilimSL
\ctr@ld@f\def\figptvisilimSLDD{\un@v@ilable{figptvisilimSL}}
\ctr@ld@f\def\figptvisilimSLTD#1:#2[#3,#4;#5,#6]{\ifGR@cri{\s@uvc@ntr@l\et@tfigptvisilimSLTD%
    \setc@ntr@l{2}\figvectP-1[#3,#4]\n@rminf{\delt@}{-1}%
    \ifcase\CUR@proj\v@lX=\cxa@\p@\v@lY=-\p@\v@lZ=\cxb@\p@
    \Figv@ctCreg-2(\v@lX,\v@lY,\v@lZ)\figvectP-3[#5,#6]\figvectNV-1[-2,-3]%
    \or\figvectP-1[#5,#6]\vecunitCV@TD{-1}\v@lmin=\v@lX\v@lmax=\v@lY
    \v@leur=\v@lZ\v@lX=\cza@\p@\v@lY=\czb@\p@\v@lZ=\czc@\p@\c@lprovec{-1}%
    \or\c@ley@pt{-2}\figvectN-1[#5,#6,-2]\fi
    \edef\Ai@{#3}\edef\Aj@{#4}\figvectP-2[#5,\Ai@]\c@lproscal\v@leur[-1,-2]%
    \ifdim\v@leur>\z@\p@rtent=\@ne\else\p@rtent=\m@ne\fi%
    \figvectP-2[#5,\Aj@]\c@lproscal\v@leur[-1,-2]%
    \ifdim\p@rtent\v@leur>\z@\figptcopy#1:#2/#3/%
    \message{*** \BS@ figptvisilimSL: points are on the same side.}\else%
    \figptcopy-3:/#3/\figptcopy-4:/#4/%
    \loop\figptbary-5:[-3,-4;1,1]\figvectP-2[#5,-5]\c@lproscal\v@leur[-1,-2]%
    \ifdim\p@rtent\v@leur>\z@\figptcopy-3:/-5/\else\figptcopy-4:/-5/\fi%
    \divide\delt@\tw@\ifdim\delt@>\epsil@n\repeat%
    \figptbary#1:#2[-3,-4;1,1]\fi\resetc@ntr@l\et@tfigptvisilimSLTD}\ignorespaces\fi}
\ctr@ld@f\def\c@ley@pt#1{\t@stp@r\ifitis@K\v@lX=\cza@\p@\v@lY=\czb@\p@\v@lZ=\czc@\p@%
    \Figv@ctCreg-1(\v@lX,\v@lY,\v@lZ)\Figp@intreg-2:(\wd\Bt@rget,\ht\Bt@rget,\dp\Bt@rget)%
    \figpttra#1:=-2/-\disob@intern,-1/\else\end\fi}
\ctr@ld@f\def\t@stp@r{\itis@Ktrue\ifnewt@rgetpt\else\itis@Kfalse%
    \message{*** \BS@ figptvisilimXX: target point undefined.}\fi\ifnewdis@b\else%
    \itis@Kfalse\message{*** \BS@ figptvisilimXX: observation distance undefined.}\fi%
    \ifitis@K\else\message{*** This macro must be called after \BS@ figdrawbegin or after
    having set the missing parameter(s) with \BS@ figset proj()}\fi}
\ctr@ld@f\def\figscan#1(#2,#3){{\s@uvc@ntr@l\et@tfigscan\@psfgetbb{#1}\if@psfbbfound\else%
    \def\@psfllx{0}\def\@psflly{20}\def\@psfurx{540}\def\@psfury{640}\fi\figscan@{#2}{#3}%
    \resetc@ntr@l\et@tfigscan}\ignorespaces}
\ctr@ld@f\def\figscan@#1#2{%
    \unit@=\@ne bp\setc@ntr@l{2}\figsetmark{}%
    \def\minst@p{20pt}%
    \v@lX=\@psfllx\p@\v@lX=\Sc@leFact\v@lX\r@undint\v@lX\v@lX%
    \v@lY=\@psflly\p@\v@lY=\Sc@leFact\v@lY\ifdim\v@lY>\z@\r@undint\v@lY\v@lY\fi%
    \delt@=\@psfury\p@\delt@=\Sc@leFact\delt@%
    \advance\delt@-\v@lY\v@lXa=\@psfurx\p@\v@lXa=\Sc@leFact\v@lXa\v@leur=\minst@p%
    \edef\valv@lY{\repdecn@mb{\v@lY}}\edef\LgTr@it{\the\delt@}%
    \loop\ifdim\v@lX<\v@lXa\edef\valv@lX{\repdecn@mb{\v@lX}}%
    \figptDD -1:(\valv@lX,\valv@lY)\figwriten -1:\hbox{\vrule height\LgTr@it}(0)%
    \ifdim\v@leur<\minst@p\else\figsetmark{\raise-8bp\hbox{$\scriptscriptstyle\triangle$}}%
    \figwrites -1:\@ffichnb{0}{\valv@lX}(6)\v@leur=\z@\figsetmark{}\fi%
    \advance\v@leur#1pt\advance\v@lX#1pt\repeat%
    \def\minst@p{10pt}%
    \v@lX=\@psfllx\p@\v@lX=\Sc@leFact\v@lX\ifdim\v@lX>\z@\r@undint\v@lX\v@lX\fi%
    \v@lY=\@psflly\p@\v@lY=\Sc@leFact\v@lY\r@undint\v@lY\v@lY%
    \delt@=\@psfurx\p@\delt@=\Sc@leFact\delt@%
    \advance\delt@-\v@lX\v@lYa=\@psfury\p@\v@lYa=\Sc@leFact\v@lYa\v@leur=\minst@p%
    \edef\valv@lX{\repdecn@mb{\v@lX}}\edef\LgTr@it{\the\delt@}%
    \loop\ifdim\v@lY<\v@lYa\edef\valv@lY{\repdecn@mb{\v@lY}}%
    \figptDD -1:(\valv@lX,\valv@lY)\figwritee -1:\vbox{\hrule width\LgTr@it}(0)%
    \ifdim\v@leur<\minst@p\else\figsetmark{$\triangleright$\kern4bp}%
    \figwritew -1:\@ffichnb{0}{\valv@lY}(6)\v@leur=\z@\figsetmark{}\fi%
    \advance\v@leur#2pt\advance\v@lY#2pt\repeat}
\ctr@ld@f
\ctr@ld@f\def\figscan@E#1(#2,#3){{\s@uvc@ntr@l\et@tfigscan@E%
    \Figdisc@rdLTS{#1}{\t@xt@}\pdfximage{\t@xt@}%
    \setbox\Gb@x=\hbox{\pdfrefximage\pdflastximage}%
    \edef\@psfllx{0}\v@lY=-\dp\Gb@x\edef\@psflly{\repdecn@mb{\v@lY}}%
    \edef\@psfurx{\repdecn@mb{\wd\Gb@x}}%
    \v@lY=\dp\Gb@x\advance\v@lY\ht\Gb@x\edef\@psfury{\repdecn@mb{\v@lY}}%
    \figscan@{#2}{#3}\resetc@ntr@l\et@tfigscan@E}\ignorespaces}
\ctr@ld@f\def\figshowpts[#1,#2]{{\figsetmark{$\bullet$}\figsetptname{\bf ##1}%
    \p@rtent=#2\relax\ifnum\p@rtent<\z@\p@rtent=\z@\fi%
    \s@mme=#1\relax\ifnum\s@mme<\z@\s@mme=\z@\fi%
    \loop\ifnum\s@mme<\p@rtent\pt@rvect{\s@mme}%
    \ifitis@K\figwriten{\the\s@mme}:(4pt)\fi\advance\s@mme\@ne\repeat%
    \pt@rvect{\s@mme}\ifitis@K\figwriten{\the\s@mme}:(4pt)\fi}\ignorespaces}
\ctr@ld@f\def\pt@rvect#1{\set@bjc@de{#1}%
    \expandafter\expandafter\expandafter\inqpt@rvec\csname\objc@de\endcsname:}
\ctr@ld@f\def\inqpt@rvec#1#2:{\if#1\C@dCl@spt\itis@Ktrue\else\itis@Kfalse\fi}
\ctr@ld@f\def\figshowsettings{{%
    \immediate\write16{====================================================================}%
    \immediate\write16{ Current settings are (DDV means "with dynamic default value"):}%
    \immediate\write16{ --- GENERAL ---}%
    \immediate\write16{Scale factor and Unit = \unit@util\space (\the\unit@)
     \space -> \BS@ figinit{ScaleFactorUnit}}%
    \immediate\write16{Update mode = \ifGRupdatem@de yes\else no\fi
     \space-> \BS@ figset(update=yes/no) or \BS@ figsetdefault(update=yes/no)}%
    \immediate\write16{ --- WRITING ---}%
    \immediate\write16{Implicit point name = \ptn@me{i} \space-> \BS@ figset write(ptname={Name})}%
    \immediate\write16{Point marker = \the\c@nsymb \space -> \BS@ figset write(mark=Mark)}%
    \immediate\write16{Print rounded coordinates = \ifr@undcoord yes\else no\fi
     \space-> \BS@ figset write(roundcoord=yes/no)}%
    \immediate\write16{ --- GRAPHICAL (general) ---}%
    \immediate\write16{Color = \CUR@color \space-> \BS@ figset(color=ColorDefinition)}%
    \immediate\write16{Filling mode = \iffillm@de yes\else no\fi
     \space-> \BS@ figset(fillmode=yes/no)}%
    \immediate\write16{Line join = \CUR@join \space-> \BS@ figset(join=miter/round/bevel)}%
    \immediate\write16{Line style = \CUR@dash \space-> \BS@ figset(dash=Index/Pattern)}%
    \immediate\write16{Line width = \CUR@width
     \space-> \BS@ figset(width=real in PostScript units)}%
    \immediate\write16{ --- GRAPHICAL (specific) ---}%
    \immediate\write16{Altitude (all the following attributes are DDV):}%
    \immediate\write16{ Base line color =
     \ifx\DDV@blcolor\D@FTref general color\else\DDV@blcolor\fi
     \space-> \BS@ figset altitude(blcolor=ColorDefinition)}%
    \immediate\write16{ Base line style =
     \ifx\DDV@bldash\D@FTref general style\else\DDV@bldash\fi
     \space-> \BS@ figset altitude(bldash=Index/Pattern)}%
    \immediate\write16{ Base line width =
     \ifx\DDV@blwidth\D@FTref general width\else\DDV@blwidth\fi
     \space-> \BS@ figset altitude(blwidth=real in PostScript units)}%
    \immediate\write16{ Square line color =
     \ifx\DDV@sqcolor\D@FTref general color\else\DDV@sqcolor\fi
     \space-> \BS@ figset altitude(sqcolor=ColorDefinition)}%
    \immediate\write16{ Square line style =
     \ifx\DDV@sqdash\D@FTref general style\else\DDV@sqdash\fi
     \space-> \BS@ figset altitude(sqdash=Index/Pattern)}%
    \immediate\write16{ Square line width =
     \ifx\DDV@sqwidth\D@FTref general width\else\DDV@sqwidth\fi
     \space-> \BS@ figset altitude(sqwidth=real in PostScript units)}%
    \immediate\write16{Arrowhead:}%
    \immediate\write16{ (half-)Angle = \@rrowheadangle
     \space-> \BS@ figset arrowhead(angle=real in degrees)}%
    \immediate\write16{ Filling mode = \if@rrowhfill yes\else no\fi
     \space-> \BS@ figset arrowhead(fillmode=yes/no)}%
    \immediate\write16{ "Outside" = \if@rrowhout yes\else no\fi
     \space-> \BS@ figset arrowhead(out=yes/no)}%
    \immediate\write16{ Length = \@rrowheadlength
     \if@rrowratio\space(not active)\else\space(active)\fi
     \space-> \BS@ figset arrowhead(length=real in user coord.)}%
    \immediate\write16{ Ratio = \@rrowheadratio
     \if@rrowratio\space(active)\else\space(not active)\fi
     \space-> \BS@ figset arrowhead(ratio=real in [0,1])}%
    \immediate\write16{Curve:}%
    \immediate\write16{ Roundness = \curv@roundness
     \space-> \BS@ figset curve(roundness=real in [0,0.5])}%
    \immediate\write16{Flow chart:}%
    \immediate\write16{ Arrow position = \@rrowp@s
     \space-> \BS@ figset flowchart(arrowposition=real in [0,1])}%
    \immediate\write16{ Arrow reference point = \ifcase\@rrowr@fpt start\else end\fi
     \space-> \BS@ figset flowchart(arrowrefpt = start/end)}%
    \immediate\write16{ Background color = \fcbgc@lor
     \space-> \BS@ figset flowchart(bgcolor=ColorDefinition)}%
    \immediate\write16{ Line type = \ifcase\fclin@typ@ curve\else polygon\fi
     \space-> \BS@ figset flowchart(line=polygon/curve)}%
    \immediate\write16{ Padding = (\Xp@dd, \Yp@dd)
     \space-> \BS@ figset flowchart(padding = real in user coord.)}%
    \immediate\write16{\space\space\space\space(or
     \BS@ figset flowchart(xpadding=real, ypadding=real) )}%
    \immediate\write16{ Radius = \fclin@r@d
     \space-> \BS@ figset flowchart(radius=positive real in user coord.)}%
    \immediate\write16{ Shape = \fcsh@pe
     \space-> \BS@ figset flowchart(shape = rectangle, ellipse or lozenge)}%
    \immediate\write16{ Thickness color (DDV) = 
     \ifx\DDV@thickcolor\D@FTref general color\else\DDV@thickcolor\fi
     \space-> \BS@ figset flowchart(thickcolor=ColorDefinition)}%
    \immediate\write16{ Thickness = \thickn@ss
     \space-> \BS@ figset flowchart(thickness = real in user coord.)}%
    \immediate\write16{Mesh:}%
    \immediate\write16{ Diagonal = \c@ntrolmesh
     \space-> \BS@ figset mesh(diag=integer in {-1,0,1})}%
    \immediate\write16{ Lines color (DDV) =
     \ifx\DDV@meshcolor\D@FTref general color\else\DDV@meshcolor\fi
     \space-> \BS@ figset mesh(color=ColorDefinition)}%
    \immediate\write16{ Lines style (DDV) =
     \ifx\DDV@meshdash\D@FTref general style\else\DDV@meshdash\fi
     \space-> \BS@ figset mesh(dash=Index/Pattern)}%
    \immediate\write16{ Lines width (DDV) =
     \ifx\DDV@meshwidth\D@FTref general width\else\DDV@meshwidth\fi
     \space-> \BS@ figset mesh(width=real in PostScript units)}%
    \immediate\write16{Trimesh:}%
    \immediate\write16{ Lines color (DDV) =
     \ifx\DDV@tmeshcolor\D@FTref general color\else\DDV@tmeshcolor\fi
     \space-> \BS@ figset trimesh(color=ColorDefinition)}%
    \immediate\write16{ Lines style (DDV) =
     \ifx\DDV@tmeshdash\D@FTref general style\else\DDV@tmeshdash\fi
     \space-> \BS@ figset trimesh(dash=Index/Pattern)}%
    \immediate\write16{ Lines width (DDV) =
     \ifx\DDV@tmeshwidth\D@FTref general width\else\DDV@tmeshwidth\fi
     \space-> \BS@ figset trimesh(width=real in PostScript units)}%
    \ifTr@isDim%
    \immediate\write16{ --- 3D to 2D PROJECTION ---}%
    \immediate\write16{Projection : \typ@proj \space-> \BS@ figinit{ScaleFactorUnit, ProjType}}%
    \immediate\write16{Longitude (psi) = \v@lPsi \space-> \BS@ figset proj(psi=real in degrees)}%
    \ifcase\CUR@proj\immediate\write16{Depth coeff. (Lambda)
     \space = \v@lTheta \space-> \BS@ figset proj(lambda=real in [0,1])}%
    \else\immediate\write16{Latitude (theta)
     \space = \v@lTheta \space-> \BS@ figset proj(theta=real in degrees)}%
    \fi%
    \ifnum\CUR@proj=\tw@%
    \immediate\write16{Observation distance = \disob@unit
     \space-> \BS@ figset proj(dist=real in user coord.)}%
    \immediate\write16{Target point = \t@rgetpt \space-> \BS@ figset proj(targetpt=pt number)}%
     \v@lX=\ptT@unit@\wd\Bt@rget\v@lY=\ptT@unit@\ht\Bt@rget\v@lZ=\ptT@unit@\dp\Bt@rget%
    \immediate\write16{ Its coordinates are
     (\repdecn@mb{\v@lX}, \repdecn@mb{\v@lY}, \repdecn@mb{\v@lZ})}%
    \fi%
    \fi%
    \immediate\write16{====================================================================}%
    \ignorespaces}}
\ctr@ln@w{newif}\ifitis@vect@r
\ctr@ld@f\def\figvectC#1(#2,#3){{\itis@vect@rtrue\figpt#1:(#2,#3)}\ignorespaces}
\ctr@ld@f\def\Figv@ctCreg#1(#2,#3){{\itis@vect@rtrue\Figp@intreg#1:(#2,#3)}\ignorespaces}
\ctr@ln@m\figvectDBezier
\ctr@ld@f\def\figvectDBezierDD#1:#2,#3[#4,#5,#6,#7]{\ifGR@cri{\s@uvc@ntr@l\et@tfigvectDBezierDD%
    \FigvectDBezier@#2,#3[#4,#5,#6,#7]\v@lX=\c@ef\v@lX\v@lY=\c@ef\v@lY%
    \Figv@ctCreg#1(\v@lX,\v@lY)\resetc@ntr@l\et@tfigvectDBezierDD}\ignorespaces\fi}
\ctr@ld@f\def\figvectDBezierTD#1:#2,#3[#4,#5,#6,#7]{\ifGR@cri{\s@uvc@ntr@l\et@tfigvectDBezierTD%
    \FigvectDBezier@#2,#3[#4,#5,#6,#7]\v@lX=\c@ef\v@lX\v@lY=\c@ef\v@lY\v@lZ=\c@ef\v@lZ%
    \Figv@ctCreg#1(\v@lX,\v@lY,\v@lZ)\resetc@ntr@l\et@tfigvectDBezierTD}\ignorespaces\fi}
\ctr@ld@f\def\FigvectDBezier@#1,#2[#3,#4,#5,#6]{\setc@ntr@l{2}%
    \edef\T@{#2}\v@leur=\p@\advance\v@leur-#2pt\edef\UNmT@{\repdecn@mb{\v@leur}}%
    \ifnum#1=\tw@\def\c@ef{6}\else\def\c@ef{3}\fi%
    \figptcopy-4:/#3/\figptcopy-3:/#4/\figptcopy-2:/#5/\figptcopy-1:/#6/%
    \l@mbd@un=-4 \l@mbd@de=-\thr@@\p@rtent=\m@ne\c@lDecast%
    \ifnum#1=\tw@\c@lDCDeux{-4}{-3}\c@lDCDeux{-3}{-2}\c@lDCDeux{-4}{-3}\else%
    \l@mbd@un=-4 \l@mbd@de=-\thr@@\p@rtent=-\tw@\c@lDecast%
    \c@lDCDeux{-4}{-3}\fi\Figg@tXY{-4}}
\ctr@ln@m\c@lDCDeux
\ctr@ld@f\def\c@lDCDeuxDD#1#2{\Figg@tXY{#2}\Figg@tXYa{#1}%
    \advance\v@lX-\v@lXa\advance\v@lY-\v@lYa\Figp@intregDD#1:(\v@lX,\v@lY)}
\ctr@ld@f\def\c@lDCDeuxTD#1#2{\Figg@tXY{#2}\Figg@tXYa{#1}\advance\v@lX-\v@lXa%
    \advance\v@lY-\v@lYa\advance\v@lZ-\v@lZa\Figp@intregTD#1:(\v@lX,\v@lY,\v@lZ)}
\ctr@ln@m\figvectN
\ctr@ld@f\def\figvectNDD#1[#2,#3]{\ifGR@cri{\Figg@tXYa{#2}\Figg@tXY{#3}%
    \advance\v@lX-\v@lXa\advance\v@lY-\v@lYa%
    \Figv@ctCreg#1(-\v@lY,\v@lX)}\ignorespaces\fi}
\ctr@ld@f\def\figvectNTD#1[#2,#3,#4]{\ifGR@cri{\vecunitC@TD[#2,#4]\v@lmin=\v@lX\v@lmax=\v@lY%
    \v@leur=\v@lZ\vecunitC@TD[#2,#3]\c@lprovec{#1}}\ignorespaces\fi}
\ctr@ln@m\figvectNV
\ctr@ld@f\def\figvectNVDD#1[#2]{\ifGR@cri{\Figg@tXY{#2}\Figv@ctCreg#1(-\v@lY,\v@lX)}\ignorespaces\fi}
\ctr@ld@f\def\figvectNVTD#1[#2,#3]{\ifGR@cri{\vecunitCV@TD{#3}\v@lmin=\v@lX\v@lmax=\v@lY%
    \v@leur=\v@lZ\vecunitCV@TD{#2}\c@lprovec{#1}}\ignorespaces\fi}
\ctr@ln@m\figvectP
\ctr@ld@f\def\figvectPDD#1[#2,#3]{\ifGR@cri{\Figg@tXYa{#2}\Figg@tXY{#3}%
    \advance\v@lX-\v@lXa\advance\v@lY-\v@lYa%
    \Figv@ctCreg#1(\v@lX,\v@lY)}\ignorespaces\fi}
\ctr@ld@f\def\figvectPTD#1[#2,#3]{\ifGR@cri{\Figg@tXYa{#2}\Figg@tXY{#3}%
    \advance\v@lX-\v@lXa\advance\v@lY-\v@lYa\advance\v@lZ-\v@lZa%
    \Figv@ctCreg#1(\v@lX,\v@lY,\v@lZ)}\ignorespaces\fi}
\ctr@ln@m\figvectU
\ctr@ld@f\def\figvectUDD#1[#2]{\ifGR@cri{\n@rmeuc\v@leur{#2}\invers@\v@leur\v@leur%
    \delt@=\repdecn@mb{\v@leur}\unit@\edef\v@ldelt@{\repdecn@mb{\delt@}}%
    \Figg@tXY{#2}\v@lX=\v@ldelt@\v@lX\v@lY=\v@ldelt@\v@lY%
    \Figv@ctCreg#1(\v@lX,\v@lY)}\ignorespaces\fi}
\ctr@ld@f\def\figvectUTD#1[#2]{\ifGR@cri{\n@rmeuc\v@leur{#2}\invers@\v@leur\v@leur%
    \delt@=\repdecn@mb{\v@leur}\unit@\edef\v@ldelt@{\repdecn@mb{\delt@}}%
    \Figg@tXY{#2}\v@lX=\v@ldelt@\v@lX\v@lY=\v@ldelt@\v@lY\v@lZ=\v@ldelt@\v@lZ%
    \Figv@ctCreg#1(\v@lX,\v@lY,\v@lZ)}\ignorespaces\fi}
\ctr@ld@f\def\figvisu#1#2#3{\c@ldefproj\initb@undb@x\xdef\figforTeXFigno{\figforTeXnextFigno}%
    \s@mme=\figforTeXnextFigno\advance\s@mme\@ne\xdef\figforTeXnextFigno{\number\s@mme}%
    \setbox\b@xvisu=\hbox{\ifnum\@utoFN>\z@\figinsert{}\gdef\@utoFInDone{0}\fi\ignorespaces#3}%
    \gdef\@utoFInDone{1}\gdef\@utoFN{0}%
    \v@lXa=-\c@@rdYmin\v@lYa=\c@@rdYmax\advance\v@lYa-\c@@rdYmin%
    \v@lX=\c@@rdXmax\advance\v@lX-\c@@rdXmin%
    \setbox#1=\hbox{#2}\v@lY=-\v@lX\maxim@m{\v@lX}{\v@lX}{\wd#1}%
    \advance\v@lY\v@lX\divide\v@lY\tw@\advance\v@lY-\c@@rdXmin%
    \setbox#1=\vbox{\parindent\z@\hsize=\v@lX\vskip\v@lYa%
    \rlap{\hskip\v@lY\smash{\raise\v@lXa\box\b@xvisu}}%
    \def\t@xt@{#2}\ifx\t@xt@\empty\else\medskip\centerline{#2}\fi}\wd#1=\v@lX}
\ctr@ld@f\def\figDecrementFigno{{\xdef\figforTeXnextFigno{\figforTeXFigno}%
    \s@mme=\figforTeXFigno\advance\s@mme\m@ne\xdef\figforTeXFigno{\number\s@mme}}}
\ctr@ln@w{newbox}\Bt@rget\setbox\Bt@rget=\null
\ctr@ln@w{newbox}\BminTD@\setbox\BminTD@=\null
\ctr@ln@w{newbox}\BmaxTD@\setbox\BmaxTD@=\null
\ctr@ln@w{newif}\ifnewt@rgetpt\ctr@ln@w{newif}\ifnewdis@b
\ctr@ld@f\def\b@undb@xTD#1#2#3{%
    \relax\ifdim#1<\wd\BminTD@\global\wd\BminTD@=#1\fi%
    \relax\ifdim#2<\ht\BminTD@\global\ht\BminTD@=#2\fi%
    \relax\ifdim#3<\dp\BminTD@\global\dp\BminTD@=#3\fi%
    \relax\ifdim#1>\wd\BmaxTD@\global\wd\BmaxTD@=#1\fi%
    \relax\ifdim#2>\ht\BmaxTD@\global\ht\BmaxTD@=#2\fi%
    \relax\ifdim#3>\dp\BmaxTD@\global\dp\BmaxTD@=#3\fi}
\ctr@ld@f\def\c@ldefdisob{{\ifdim\wd\BminTD@<\maxdimen\v@leur=\wd\BmaxTD@\advance\v@leur-\wd\BminTD@%
    \delt@=\ht\BmaxTD@\advance\delt@-\ht\BminTD@\maxim@m{\v@leur}{\v@leur}{\delt@}%
    \delt@=\dp\BmaxTD@\advance\delt@-\dp\BminTD@\maxim@m{\v@leur}{\v@leur}{\delt@}%
    \v@leur=5\v@leur\else\v@leur=800pt\fi\c@ldefdisob@{\v@leur}}}
\ctr@ln@m\disob@intern
\ctr@ln@m\disob@
\ctr@ln@m\divf@ctproj
\ctr@ld@f\def\c@ldefdisob@#1{{\v@leur=#1\ifdim\v@leur<\p@\v@leur=800pt\fi%
    \xdef\disob@intern{\repdecn@mb{\v@leur}}%
    \delt@=\ptT@unit@\v@leur\xdef\disob@unit{\repdecn@mb{\delt@}}%
    \f@ctech=\@ne\loop\ifdim\v@leur>\t@n pt\divide\v@leur\t@n\multiply\f@ctech\t@n\repeat%
    \xdef\disob@{\repdecn@mb{\v@leur}}\xdef\divf@ctproj{\the\f@ctech}}%
    \global\newdis@btrue}
\ctr@ln@m\t@rgetpt
\ctr@ld@f\def\c@ldeft@rgetpt{\newt@rgetpttrue\def\t@rgetpt{CenterBoundBox}{%
    \delt@=\wd\BmaxTD@\advance\delt@-\wd\BminTD@\divide\delt@\tw@%
    \v@leur=\wd\BminTD@\advance\v@leur\delt@\global\wd\Bt@rget=\v@leur%
    \delt@=\ht\BmaxTD@\advance\delt@-\ht\BminTD@\divide\delt@\tw@%
    \v@leur=\ht\BminTD@\advance\v@leur\delt@\global\ht\Bt@rget=\v@leur%
    \delt@=\dp\BmaxTD@\advance\delt@-\dp\BminTD@\divide\delt@\tw@%
    \v@leur=\dp\BminTD@\advance\v@leur\delt@\global\dp\Bt@rget=\v@leur}}
\ctr@ln@m\c@ldefproj
\ctr@ld@f\def\c@ldefprojTD{\ifnewt@rgetpt\else\c@ldeft@rgetpt\fi\ifnewdis@b\else\c@ldefdisob\fi}
\ctr@ld@f\def\c@lprojcav{
    \v@lZa=\cxa@\v@lY\advance\v@lX\v@lZa%
    \v@lZa=\cxb@\v@lY\v@lY=\v@lZ\advance\v@lY\v@lZa\ignorespaces}
\ctr@ln@m\v@lcoef
\ctr@ld@f\def\c@lprojrea{
    \advance\v@lX-\wd\Bt@rget\advance\v@lY-\ht\Bt@rget\advance\v@lZ-\dp\Bt@rget%
    \v@lZa=\cza@\v@lX\advance\v@lZa\czb@\v@lY\advance\v@lZa\czc@\v@lZ%
    \divide\v@lZa\divf@ctproj\advance\v@lZa\disob@ pt\invers@{\v@lZa}{\v@lZa}%
    \v@lZa=\disob@\v@lZa\edef\v@lcoef{\repdecn@mb{\v@lZa}}%
    \v@lXa=\cxa@\v@lX\advance\v@lXa\cxb@\v@lY\v@lXa=\v@lcoef\v@lXa%
    \v@lY=\cyb@\v@lY\advance\v@lY\cya@\v@lX\advance\v@lY\cyc@\v@lZ%
    \v@lY=\v@lcoef\v@lY\v@lX=\v@lXa\ignorespaces}
\ctr@ld@f\def\c@lprojort{
    \v@lXa=\cxa@\v@lX\advance\v@lXa\cxb@\v@lY%
    \v@lY=\cyb@\v@lY\advance\v@lY\cya@\v@lX\advance\v@lY\cyc@\v@lZ%
    \v@lX=\v@lXa\ignorespaces}
\ctr@ld@f\def\Figptpr@j#1:#2/#3/{{\Figg@tXY{#3}\superc@lprojSP%
    \Figp@intregDD#1:{#2}(\v@lX,\v@lY)}\ignorespaces}
\ctr@ln@m\figsetobdist
\ctr@ld@f\def\figsetobdistDD{\un@v@ilable{figsetobdist}}
\ctr@ld@f\def\figsetobdistTD(#1){{\ifCUR@PS\W@rnmesIgn{figset proj(dist=...)}%
    \else\v@leur=#1\unit@\c@ldefdisob@{\v@leur}\fi}\ignorespaces}
\ctr@ln@m\c@lprojSP
\ctr@ln@m\CUR@proj
\ctr@ln@m\typ@proj
\ctr@ln@m\superc@lprojSP
\ctr@ld@f\def\Figs@tproj#1{%
    \if#13 \def@ultproj\else\if#1c\def@ultproj%
    \else\if#1o\xdef\CUR@proj{1}\xdef\typ@proj{orthogonal}%
         \figsetviewTD(\def@ultpsi,\def@ulttheta)%
         \global\let\c@lprojSP=\c@lprojort\global\let\superc@lprojSP=\c@lprojort%
    \else\if#1r\xdef\CUR@proj{2}\xdef\typ@proj{realistic}%
         \figsetviewTD(\def@ultpsi,\def@ulttheta)%
         \global\let\c@lprojSP=\c@lprojrea\global\let\superc@lprojSP=\c@lprojrea%
    \else\def@ultproj\message{*** Unknown projection. Cavalier projection assumed.}%
    \fi\fi\fi\fi}
\ctr@ld@f\def\def@ultproj{\xdef\CUR@proj{0}\xdef\typ@proj{cavalier}\figsetviewTD(\def@ultpsi,0.5)%
         \global\let\c@lprojSP=\c@lprojcav\global\let\superc@lprojSP=\c@lprojcav}
\ctr@ln@m\figsettarget
\ctr@ld@f\def\figsettargetDD{\un@v@ilable{figsettarget}}
\ctr@ld@f\def\figsettargetTD[#1]{{\ifCUR@PS\W@rnmesIgn{figset proj(targetpt=...)}%
    \else\global\newt@rgetpttrue\xdef\t@rgetpt{#1}\Figg@tXY{#1}\global\wd\Bt@rget=\v@lX%
    \global\ht\Bt@rget=\v@lY\global\dp\Bt@rget=\v@lZ\fi}\ignorespaces}
\ctr@ln@m\figsetview
\ctr@ld@f\def\figsetviewDD{\un@v@ilable{figsetview}}
\ctr@ld@f\def\figsetviewTD(#1){\ifCUR@PS\W@rnmesIgn{figset proj(Psi|Theta|Lambda=...)}%
     \else\Figsetview@#1,:\fi\ignorespaces}
\ctr@ld@f\def\Figsetview@#1,#2:{{\xdef\v@lPsi{#1}\def\t@xt@{#2}%
    \ifx\t@xt@\empty\def\@rgdeux{\v@lTheta}\else\X@rgdeux@#2\fi%
    \c@ssin{\costhet@}{\sinthet@}{#1}\v@lmin=\costhet@ pt\v@lmax=\sinthet@ pt%
    \ifcase\CUR@proj%
    \v@leur=\@rgdeux\v@lmin\xdef\cxa@{\repdecn@mb{\v@leur}}%
    \v@leur=\@rgdeux\v@lmax\xdef\cxb@{\repdecn@mb{\v@leur}}\v@leur=\@rgdeux pt%
    \relax\ifdim\v@leur>\p@\message{*** Lambda too large ! See \BS@ figset proj() !}\fi%
    \else%
    \v@lmax=-\v@lmax\xdef\cxa@{\repdecn@mb{\v@lmax}}\xdef\cxb@{\costhet@}%
    \ifx\t@xt@\empty\edef\@rgdeux{\def@ulttheta}\fi\c@ssin{\C@}{\S@}{\@rgdeux}%
    \v@lmax=-\S@ pt%
    \v@leur=\v@lmax\v@leur=\costhet@\v@leur\xdef\cya@{\repdecn@mb{\v@leur}}%
    \v@leur=\v@lmax\v@leur=\sinthet@\v@leur\xdef\cyb@{\repdecn@mb{\v@leur}}%
    \xdef\cyc@{\C@}\v@lmin=-\C@ pt%
    \v@leur=\v@lmin\v@leur=\costhet@\v@leur\xdef\cza@{\repdecn@mb{\v@leur}}%
    \v@leur=\v@lmin\v@leur=\sinthet@\v@leur\xdef\czb@{\repdecn@mb{\v@leur}}%
    \xdef\czc@{\repdecn@mb{\v@lmax}}\fi%
    \xdef\v@lTheta{\@rgdeux}}}
\ctr@ld@f\def\def@ultpsi{40}
\ctr@ld@f\def\def@ulttheta{25}
\ctr@ln@m\l@debut
\ctr@ln@m\n@mref
\ctr@ld@f\def\Figsetpr@j#1=#2|{\keln@mtr#1|%
    \def\n@mref{dep}\ifx\l@debut\n@mref\Figsetd@p{#2}\else
    \def\n@mref{dis}\ifx\l@debut\n@mref%
     \ifnum\CUR@proj=\tw@\figsetobdist(#2)\else\Figset@rr\fi\else
    \def\n@mref{lam}\ifx\l@debut\n@mref\Figsetd@p{#2}\else
    \def\n@mref{lat}\ifx\l@debut\n@mref\Figsetth@{#2}\else
    \def\n@mref{lon}\ifx\l@debut\n@mref\figsetview(#2)\else
    \def\n@mref{psi}\ifx\l@debut\n@mref\figsetview(#2)\else
    \def\n@mref{tar}\ifx\l@debut\n@mref%
     \ifnum\CUR@proj=\tw@\figsettarget[#2]\else\Figset@rr\fi\else
    \def\n@mref{the}\ifx\l@debut\n@mref\Figsetth@{#2}\else
    \W@rnmesAttr{figset proj}{#1}\fi\fi\fi\fi\fi\fi\fi\fi}
\ctr@ld@f\def\Figsetd@p#1{\ifnum\CUR@proj=\z@\figsetview(\v@lPsi,#1)\else\Figset@rr\fi}
\ctr@ld@f\def\Figsetth@#1{\ifnum\CUR@proj=\z@\Figset@rr\else\figsetview(\v@lPsi,#1)\fi}
\ctr@ld@f\def\Figset@rr{\message{*** \BS@ figset proj(): Attribute "\n@mref" ignored, incompatible
    with current projection}}
\ctr@ld@f\def\initb@undb@xTD{\wd\BminTD@=\maxdimen\ht\BminTD@=\maxdimen\dp\BminTD@=\maxdimen%
    \wd\BmaxTD@=-\maxdimen\ht\BmaxTD@=-\maxdimen\dp\BmaxTD@=-\maxdimen}
\ctr@ln@w{newbox}\Gb@x      
\ctr@ln@w{newbox}\Gb@xSC    
\ctr@ln@w{newtoks}\c@nsymb  
\ctr@ln@w{newif}\ifr@undcoord\ctr@ln@w{newif}\ifunitpr@sent
\ctr@ld@f\def\unssqrttw@{0.707106 }
\ctr@ld@f\def\figAst{\raise-1.15ex\hbox{$\ast$}}
\ctr@ld@f\def\figBullet{\raise-1.15ex\hbox{$\bullet$}}
\ctr@ld@f\def\figCirc{\raise-1.15ex\hbox{$\circ$}}
\ctr@ld@f\def\figDiamond{\raise-1.15ex\hbox{$\diamond$}}%
\ctr@ld@f\def\boxit#1#2{\leavevmode\hbox{\vrule\vbox{\hrule\vglue#1%
    \vtop{\hbox{\kern#1{#2}\kern#1}\vglue#1\hrule}}\vrule}}
\ctr@ld@f
\ctr@ld@f
\ctr@ld@f\def\c@nterpt{\ignorespaces%
    \kern-.5\wd\Gb@xSC%
    \raise-.5\ht\Gb@xSC\rlap{\hbox{\raise.5\dp\Gb@xSC\hbox{\copy\Gb@xSC}}}%
    \kern .5\wd\Gb@xSC\ignorespaces}
\ctr@ld@f\def\b@undb@xSC#1#2{{\v@lXa=#1\v@lYa=#2%
    \v@leur=\ht\Gb@xSC\advance\v@leur\dp\Gb@xSC%
    \advance\v@lXa-.5\wd\Gb@xSC\advance\v@lYa-.5\v@leur\b@undb@x{\v@lXa}{\v@lYa}%
    \advance\v@lXa\wd\Gb@xSC\advance\v@lYa\v@leur\b@undb@x{\v@lXa}{\v@lYa}}}
\ctr@ln@m\Dist@n
\ctr@ln@m\l@suite
\ctr@ld@f\def\@keldist#1#2{\edef\Dist@n{#2}\y@tiunit{\Dist@n}%
    \ifunitpr@sent#1=\Dist@n\else#1=\Dist@n\unit@\fi}
\ctr@ld@f\def\y@tiunit#1{\unitpr@sentfalse\expandafter\y@tiunit@#1:}
\ctr@ld@f\def\y@tiunit@#1#2:{\ifcat#1a\unitpr@senttrue\else\def\l@suite{#2}%
    \ifx\l@suite\empty\else\y@tiunit@#2:\fi\fi}
\ctr@ln@m\figcoord
\ctr@ld@f\def\figcoordDD#1{{\v@lX=\ptT@unit@\v@lX\v@lY=\ptT@unit@\v@lY%
    \ifr@undcoord\ifcase#1\v@leur=0.5pt\or\v@leur=0.05pt\or\v@leur=0.005pt%
    \or\v@leur=0.0005pt\else\v@leur=\z@\fi%
    \ifdim\v@lX<\z@\advance\v@lX-\v@leur\else\advance\v@lX\v@leur\fi%
    \ifdim\v@lY<\z@\advance\v@lY-\v@leur\else\advance\v@lY\v@leur\fi\fi%
    (\@ffichnb{#1}{\repdecn@mb{\v@lX}},\ifmmode\else\thinspace\fi%
    \@ffichnb{#1}{\repdecn@mb{\v@lY}})}}
\ctr@ld@f\def\@ffichnb#1#2{{\def\@@ffich{\@ffich#1(}\edef\n@mbre{#2}%
    \expandafter\@@ffich\n@mbre)}}
\ctr@ld@f\def\@ffich#1(#2.#3){{#2\ifnum#1>\z@.\fi\def\dig@ts{#3}\s@mme=\z@%
    \loop\ifnum\s@mme<#1\expandafter\@ffichdec\dig@ts:\advance\s@mme\@ne\repeat}}
\ctr@ld@f\def\@ffichdec#1#2:{\relax#1\def\dig@ts{#20}}
\ctr@ld@f\def\figcoordTD#1{{\v@lX=\ptT@unit@\v@lX\v@lY=\ptT@unit@\v@lY\v@lZ=\ptT@unit@\v@lZ%
    \ifr@undcoord\ifcase#1\v@leur=0.5pt\or\v@leur=0.05pt\or\v@leur=0.005pt%
    \or\v@leur=0.0005pt\else\v@leur=\z@\fi%
    \ifdim\v@lX<\z@\advance\v@lX-\v@leur\else\advance\v@lX\v@leur\fi%
    \ifdim\v@lY<\z@\advance\v@lY-\v@leur\else\advance\v@lY\v@leur\fi%
    \ifdim\v@lZ<\z@\advance\v@lZ-\v@leur\else\advance\v@lZ\v@leur\fi\fi%
    (\@ffichnb{#1}{\repdecn@mb{\v@lX}},\ifmmode\else\thinspace\fi%
     \@ffichnb{#1}{\repdecn@mb{\v@lY}},\ifmmode\else\thinspace\fi%
     \@ffichnb{#1}{\repdecn@mb{\v@lZ}})}}
\ctr@ld@f\def\figsetroundcoord#1{\expandafter\Figsetr@undcoord#1:\ignorespaces}
\ctr@ld@f\def\Figsetr@undcoord#1#2:{\if#1n\r@undcoordfalse\else\r@undcoordtrue\fi}
\ctr@ld@f\def\Figsetwr@te#1=#2|{\keln@mun#1|%
    \def\n@mref{m}\ifx\l@debut\n@mref\figsetmark{#2}\else
    \def\n@mref{p}\ifx\l@debut\n@mref\figsetptname{#2}\else
    \def\n@mref{r}\ifx\l@debut\n@mref\figsetroundcoord{#2}\else
    \W@rnmesAttr{figset write}{#1}\fi\fi\fi}
\ctr@ld@f\def\figsetmark#1{\c@nsymb={#1}\setbox\Gb@xSC=\hbox{\the\c@nsymb}\ignorespaces}
\ctr@ln@m\ptn@me
\ctr@ld@f\def\figsetptname#1{\def\ptn@me##1{#1}\ignorespaces}
\ctr@ld@f\def\FigWrit@L#1:#2(#3,#4){\ignorespaces\@keldist\v@leur{#3}\@keldist\delt@{#4}%
    \C@rp@r@m\def\list@num{#1}\@ecfor\p@int:=\list@num\do{\FigWrit@pt{\p@int}{#2}}}
\ctr@ld@f\def\FigWrit@pt#1#2{\FigWp@r@m{#1}{#2}\Vc@rrect\figWp@si%
    \ifdim\wd\Gb@xSC>\z@\b@undb@xSC{\v@lX}{\v@lY}\fi\figWBB@x}
\ctr@ld@f\def\FigWp@r@m#1#2{\Figg@tXY{#1}%
    \setbox\Gb@x=\hbox{\def\t@xt@{#2}\ifx\t@xt@\empty\Figg@tT{#1}\else#2\fi}\c@lprojSP}
\ctr@ld@f\let\Vc@rrect=\relax
\ctr@ld@f\let\C@rp@r@m=\relax
\ctr@ld@f\def\figwrite[#1]#2{{\ignorespaces\def\list@num{#1}\@ecfor\p@int:=\list@num\do{%
    \setbox\Gb@x=\hbox{\def\t@xt@{#2}\ifx\t@xt@\empty\Figg@tT{\p@int}\else#2\fi}%
    \Figwrit@{\p@int}}}\ignorespaces}
\ctr@ld@f\def\Figwrit@#1{\Figg@tXY{#1}\c@lprojSP%
    \rlap{\kern\v@lX\raise\v@lY\hbox{\unhcopy\Gb@x}}\v@leur=\v@lY%
    \advance\v@lY\ht\Gb@x\b@undb@x{\v@lX}{\v@lY}\advance\v@lX\wd\Gb@x%
    \v@lY=\v@leur\advance\v@lY-\dp\Gb@x\b@undb@x{\v@lX}{\v@lY}}
\ctr@ld@f\def\figwritec[#1]#2{{\ignorespaces\def\list@num{#1}%
    \@ecfor\p@int:=\list@num\do{\Figwrit@c{\p@int}{#2}}}\ignorespaces}
\ctr@ld@f\def\Figwrit@c#1#2{\FigWp@r@m{#1}{#2}%
    \rlap{\kern\v@lX\raise\v@lY\hbox{\rlap{\kern-.5\wd\Gb@x%
    \raise-.5\ht\Gb@x\hbox{\raise.5\dp\Gb@x\hbox{\unhcopy\Gb@x}}}}}%
    \v@leur=\ht\Gb@x\advance\v@leur\dp\Gb@x%
    \advance\v@lX-.5\wd\Gb@x\advance\v@lY-.5\v@leur\b@undb@x{\v@lX}{\v@lY}%
    \advance\v@lX\wd\Gb@x\advance\v@lY\v@leur\b@undb@x{\v@lX}{\v@lY}}
\ctr@ld@f\def\figwritep[#1]{{\ignorespaces\def\list@num{#1}\setbox\Gb@x=\hbox{\c@nterpt}%
    \@ecfor\p@int:=\list@num\do{\Figwrit@{\p@int}}}\ignorespaces}
\ctr@ld@f\def\figwritew#1:#2(#3){\figwritegcw#1:{#2}(#3,0pt)}
\ctr@ld@f\def\figwritee#1:#2(#3){\figwritegce#1:{#2}(#3,0pt)}
\ctr@ld@f\def\figwriten#1:#2(#3){{\def\Vc@rrect{\v@lZ=\v@leur\advance\v@lZ\dp\Gb@x}%
    \Figwrit@NS#1:{#2}(#3)}\ignorespaces}
\ctr@ld@f\def\figwrites#1:#2(#3){{\def\Vc@rrect{\v@lZ=-\v@leur\advance\v@lZ-\ht\Gb@x}%
    \Figwrit@NS#1:{#2}(#3)}\ignorespaces}
\ctr@ld@f\def\Figwrit@NS#1:#2(#3){\let\figWp@si=\FigWp@siNS\let\figWBB@x=\FigWBB@xNS%
    \FigWrit@L#1:{#2}(#3,0pt)}
\ctr@ld@f\def\FigWp@siNS{\rlap{\kern\v@lX\raise\v@lY\hbox{\rlap{\kern-.5\wd\Gb@x%
    \raise\v@lZ\hbox{\unhcopy\Gb@x}}\c@nterpt}}}
\ctr@ld@f\def\FigWBB@xNS{\advance\v@lY\v@lZ%
    \advance\v@lY-\dp\Gb@x\advance\v@lX-.5\wd\Gb@x\b@undb@x{\v@lX}{\v@lY}%
    \advance\v@lY\ht\Gb@x\advance\v@lY\dp\Gb@x%
    \advance\v@lX\wd\Gb@x\b@undb@x{\v@lX}{\v@lY}}
\ctr@ld@f\def\figwritenw#1:#2(#3){{\let\figWp@si=\FigWp@sigW\let\figWBB@x=\FigWBB@xgWE%
    \def\C@rp@r@m{\v@leur=\unssqrttw@\v@leur\delt@=\v@leur%
    \ifdim\delt@=\z@\delt@=\epsil@n\fi}\let@xte={-}\FigWrit@L#1:{#2}(#3,0pt)}\ignorespaces}
\ctr@ld@f\def\figwritesw#1:#2(#3){{\let\figWp@si=\FigWp@sigW\let\figWBB@x=\FigWBB@xgWE%
    \def\C@rp@r@m{\v@leur=\unssqrttw@\v@leur\delt@=-\v@leur%
    \ifdim\delt@=\z@\delt@=-\epsil@n\fi}\let@xte={-}\FigWrit@L#1:{#2}(#3,0pt)}\ignorespaces}
\ctr@ld@f\def\figwritene#1:#2(#3){{\let\figWp@si=\FigWp@sigE\let\figWBB@x=\FigWBB@xgWE%
    \def\C@rp@r@m{\v@leur=\unssqrttw@\v@leur\delt@=\v@leur%
    \ifdim\delt@=\z@\delt@=\epsil@n\fi}\let@xte={}\FigWrit@L#1:{#2}(#3,0pt)}\ignorespaces}
\ctr@ld@f\def\figwritese#1:#2(#3){{\let\figWp@si=\FigWp@sigE\let\figWBB@x=\FigWBB@xgWE%
    \def\C@rp@r@m{\v@leur=\unssqrttw@\v@leur\delt@=-\v@leur%
    \ifdim\delt@=\z@\delt@=-\epsil@n\fi}\let@xte={}\FigWrit@L#1:{#2}(#3,0pt)}\ignorespaces}
\ctr@ld@f\def\figwritegw#1:#2(#3,#4){{\let\figWp@si=\FigWp@sigW\let\figWBB@x=\FigWBB@xgWE%
    \let@xte={-}\FigWrit@L#1:{#2}(#3,#4)}\ignorespaces}
\ctr@ld@f\def\figwritege#1:#2(#3,#4){{\let\figWp@si=\FigWp@sigE\let\figWBB@x=\FigWBB@xgWE%
    \let@xte={}\FigWrit@L#1:{#2}(#3,#4)}\ignorespaces}
\ctr@ld@f\def\FigWp@sigW{\v@lXa=\z@\v@lYa=\ht\Gb@x\advance\v@lYa\dp\Gb@x%
    \ifdim\delt@>\z@\relax%
    \rlap{\kern\v@lX\raise\v@lY\hbox{\rlap{\kern-\wd\Gb@x\kern-\v@leur%
          \raise\delt@\hbox{\raise\dp\Gb@x\hbox{\unhcopy\Gb@x}}}\c@nterpt}}%
    \else\ifdim\delt@<\z@\relax\v@lYa=-\v@lYa%
    \rlap{\kern\v@lX\raise\v@lY\hbox{\rlap{\kern-\wd\Gb@x\kern-\v@leur%
          \raise\delt@\hbox{\raise-\ht\Gb@x\hbox{\unhcopy\Gb@x}}}\c@nterpt}}%
    \else\v@lXa=-.5\v@lYa%
    \rlap{\kern\v@lX\raise\v@lY\hbox{\rlap{\kern-\wd\Gb@x\kern-\v@leur%
          \raise-.5\ht\Gb@x\hbox{\raise.5\dp\Gb@x\hbox{\unhcopy\Gb@x}}}\c@nterpt}}%
    \fi\fi}
\ctr@ld@f\def\FigWp@sigE{\v@lXa=\z@\v@lYa=\ht\Gb@x\advance\v@lYa\dp\Gb@x%
    \ifdim\delt@>\z@\relax%
    \rlap{\kern\v@lX\raise\v@lY\hbox{\c@nterpt\kern\v@leur%
          \raise\delt@\hbox{\raise\dp\Gb@x\hbox{\unhcopy\Gb@x}}}}%
    \else\ifdim\delt@<\z@\relax\v@lYa=-\v@lYa%
    \rlap{\kern\v@lX\raise\v@lY\hbox{\c@nterpt\kern\v@leur%
          \raise\delt@\hbox{\raise-\ht\Gb@x\hbox{\unhcopy\Gb@x}}}}%
    \else\v@lXa=-.5\v@lYa%
    \rlap{\kern\v@lX\raise\v@lY\hbox{\c@nterpt\kern\v@leur%
          \raise-.5\ht\Gb@x\hbox{\raise.5\dp\Gb@x\hbox{\unhcopy\Gb@x}}}}%
    \fi\fi}
\ctr@ld@f\def\FigWBB@xgWE{\advance\v@lY\delt@%
    \advance\v@lX\the\let@xte\v@leur\advance\v@lY\v@lXa\b@undb@x{\v@lX}{\v@lY}%
    \advance\v@lX\the\let@xte\wd\Gb@x\advance\v@lY\v@lYa\b@undb@x{\v@lX}{\v@lY}}
\ctr@ld@f\def\figwritegcw#1:#2(#3,#4){{\let\figWp@si=\FigWp@sigcW\let\figWBB@x=\FigWBB@xgcWE%
    \let@xte={-}\FigWrit@L#1:{#2}(#3,#4)}\ignorespaces}
\ctr@ld@f\def\figwritegce#1:#2(#3,#4){{\let\figWp@si=\FigWp@sigcE\let\figWBB@x=\FigWBB@xgcWE%
    \let@xte={}\FigWrit@L#1:{#2}(#3,#4)}\ignorespaces}
\ctr@ld@f\def\FigWp@sigcW{\rlap{\kern\v@lX\raise\v@lY\hbox{\rlap{\kern-\wd\Gb@x\kern-\v@leur%
     \raise-.5\ht\Gb@x\hbox{\raise\delt@\hbox{\raise.5\dp\Gb@x\hbox{\unhcopy\Gb@x}}}}%
     \c@nterpt}}}
\ctr@ld@f\def\FigWp@sigcE{\rlap{\kern\v@lX\raise\v@lY\hbox{\c@nterpt\kern\v@leur%
    \raise-.5\ht\Gb@x\hbox{\raise\delt@\hbox{\raise.5\dp\Gb@x\hbox{\unhcopy\Gb@x}}}}}}
\ctr@ld@f\def\FigWBB@xgcWE{\v@lZ=\ht\Gb@x\advance\v@lZ\dp\Gb@x%
    \advance\v@lX\the\let@xte\v@leur\advance\v@lY\delt@\advance\v@lY.5\v@lZ%
    \b@undb@x{\v@lX}{\v@lY}%
    \advance\v@lX\the\let@xte\wd\Gb@x\advance\v@lY-\v@lZ\b@undb@x{\v@lX}{\v@lY}}
\ctr@ld@f\def\figwritebn#1:#2(#3){{\def\Vc@rrect{\v@lZ=\v@leur}\Figwrit@NS#1:{#2}(#3)}\ignorespaces}
\ctr@ld@f\def\figwritebs#1:#2(#3){{\def\Vc@rrect{\v@lZ=-\v@leur}\Figwrit@NS#1:{#2}(#3)}\ignorespaces}
\ctr@ld@f\def\figwritebw#1:#2(#3){{\let\figWp@si=\FigWp@sibW\let\figWBB@x=\FigWBB@xbWE%
    \let@xte={-}\FigWrit@L#1:{#2}(#3,0pt)}\ignorespaces}
\ctr@ld@f\def\figwritebe#1:#2(#3){{\let\figWp@si=\FigWp@sibE\let\figWBB@x=\FigWBB@xbWE%
    \let@xte={}\FigWrit@L#1:{#2}(#3,0pt)}\ignorespaces}
\ctr@ld@f\def\FigWp@sibW{\rlap{\kern\v@lX\raise\v@lY\hbox{\rlap{\kern-\wd\Gb@x\kern-\v@leur%
          \hbox{\unhcopy\Gb@x}}\c@nterpt}}}
\ctr@ld@f\def\FigWp@sibE{\rlap{\kern\v@lX\raise\v@lY\hbox{\c@nterpt\kern\v@leur%
          \hbox{\unhcopy\Gb@x}}}}
\ctr@ld@f\def\FigWBB@xbWE{\v@lZ=\ht\Gb@x\advance\v@lZ\dp\Gb@x%
    \advance\v@lX\the\let@xte\v@leur\advance\v@lY\ht\Gb@x\b@undb@x{\v@lX}{\v@lY}%
    \advance\v@lX\the\let@xte\wd\Gb@x\advance\v@lY-\v@lZ\b@undb@x{\v@lX}{\v@lY}}
\ctr@ln@w{newread}\frf@g  \ctr@ln@w{newwrite}\fwf@g
\ctr@ln@w{newif}\ifCUR@PS
\ctr@ln@w{newif}\ifGR@cri
\ctr@ln@w{newif}\ifUse@llipse
\ctr@ln@w{newif}\ifGRdebugm@de \GRdebugm@defalse 
\ctr@ln@w{newif}\ifPDFm@ke
\ifx\pdfliteral\undefined\else\ifnum\pdfoutput>\z@\PDFm@ketrue\fi\fi
\ctr@ld@f\def\initPDF@rDVI{%
\ifPDFm@ke
 \let\figscan=\figscan@E
 \let\newGr@FN=\newGr@FNPDF
 \ctr@ld@f\def\c@mcurveto{c}
 \ctr@ld@f\def\c@mfill{f}
 \ctr@ld@f\def\c@mgsave{q}
 \ctr@ld@f\def\c@mgrestore{Q}
 \ctr@ld@f\def\c@mlineto{l}
 \ctr@ld@f\def\c@mmoveto{m}
 \ctr@ld@f\def\c@msetgray{g}     \ctr@ld@f\def\c@msetgrayStroke{G}
 \ctr@ld@f\def\c@msetcmykcolor{k}\ctr@ld@f\def\c@msetcmykcolorStroke{K}
 \ctr@ld@f\def\c@msetrgbcolor{rg}\ctr@ld@f\def\c@msetrgbcolorStroke{RG}
 \ctr@ld@f\def\d@fprimarC@lor{\CUR@color\space\CUR@colorc@md%
               \space\CUR@color\space\CUR@colorc@mdStroke}
 \ctr@ld@f\def\c@msetdash{d}
 \ctr@ld@f\def\c@msetlinejoin{j}
 \ctr@ld@f\def\c@msetlinewidth{w}
 \ctr@ld@f\def\f@gclosestroke{\immediate\write\fwf@g{s}}
 \ctr@ld@f\def\f@gfill{\immediate\write\fwf@g{\fillc@md}}
 \ctr@ld@f\def\f@gnewpath{}
 \ctr@ld@f\def\f@gstroke{\immediate\write\fwf@g{S}}
\else
 \let\figinsertE=\figinsert
 \let\newGr@FN=\newGr@FNDVI
 \ctr@ld@f\def\c@mcurveto{curveto}
 \ctr@ld@f\def\c@mfill{fill}
 \ctr@ld@f\def\c@mgsave{gsave}
 \ctr@ld@f\def\c@mgrestore{grestore}
 \ctr@ld@f\def\c@mlineto{lineto}
 \ctr@ld@f\def\c@mmoveto{moveto}
 \ctr@ld@f\def\c@msetgray{setgray}          \ctr@ld@f\def\c@msetgrayStroke{}
 \ctr@ld@f\def\c@msetcmykcolor{setcmykcolor}\ctr@ld@f\def\c@msetcmykcolorStroke{}
 \ctr@ld@f\def\c@msetrgbcolor{setrgbcolor}  \ctr@ld@f\def\c@msetrgbcolorStroke{}
 \ctr@ld@f\def\d@fprimarC@lor{\CUR@color\space\CUR@colorc@md}
 \ctr@ld@f\def\c@msetdash{setdash}
 \ctr@ld@f\def\c@msetlinejoin{setlinejoin}
 \ctr@ld@f\def\c@msetlinewidth{setlinewidth}
 \ctr@ld@f\def\f@gclosestroke{\immediate\write\fwf@g{closepath\space stroke}}
 \ctr@ld@f\def\f@gfill{\immediate\write\fwf@g{\fillc@md}}
 \ctr@ld@f\def\f@gnewpath{\immediate\write\fwf@g{newpath}}
 \ctr@ld@f\def\f@gstroke{\immediate\write\fwf@g{stroke}}
\fi}
\ctr@ld@f\def\c@pypsfile#1#2{\c@pyfil@{\immediate\write#1}{#2}}
\ctr@ld@f\def\Figinclud@PDF#1#2{\openin\frf@g=#1\pdfliteral{q #2 0 0 #2 0 0 cm}%
    \c@pyfil@{\pdfliteral}{\frf@g}\pdfliteral{Q}\closein\frf@g}
\ctr@ln@w{newif}\ifmored@ta
\ctr@ln@m\bl@nkline
\ctr@ld@f\def\c@pyfil@#1#2{\def\bl@nkline{\par}{\catcode`\%=12
    \loop\ifeof#2\mored@tafalse\else\mored@tatrue\immediate\read#2 to\tr@c
    \ifx\tr@c\bl@nkline\else#1{\tr@c}\fi\fi\ifmored@ta\repeat}}
\ctr@ld@f\def\keln@mun#1#2|{\def\l@debut{#1}\def\l@suite{#2}}
\ctr@ld@f\def\keln@mde#1#2#3|{\def\l@debut{#1#2}\def\l@suite{#3}}
\ctr@ld@f\def\keln@mtr#1#2#3#4|{\def\l@debut{#1#2#3}\def\l@suite{#4}}
\ctr@ld@f\def\keln@mqu#1#2#3#4#5|{\def\l@debut{#1#2#3#4}\def\l@suite{#5}}
\ctr@ld@f\let\@psffilein=\frf@g 
\ctr@ln@w{newif}\if@psffileok    
\ctr@ln@w{newif}\if@psfbbfound   
\ctr@ln@w{newif}\if@psfverbose   
\@psfverbosetrue
\ctr@ln@m\@psfllx \ctr@ln@m\@psflly
\ctr@ln@m\@psfurx \ctr@ln@m\@psfury
\ctr@ln@m\resetcolonc@tcode
\ctr@ld@f\def\@psfgetbb#1{\global\@psfbbfoundfalse%
\global\def\@psfllx{0}\global\def\@psflly{0}%
\global\def\@psfurx{30}\global\def\@psfury{30}%
\openin\@psffilein=#1\relax
\ifeof\@psffilein\errmessage{I couldn't open #1, will ignore it}\else
   \edef\resetcolonc@tcode{\catcode`\noexpand\:\the\catcode`\:\relax}%
   {\@psffileoktrue \chardef\other=12
    \def\do##1{\catcode`##1=\other}\dospecials \catcode`\ =10 \resetcolonc@tcode
    \loop
       \read\@psffilein to \@psffileline
       \ifeof\@psffilein\@psffileokfalse\else
          \expandafter\@psfaux\@psffileline:. \\%
       \fi
   \if@psffileok\repeat
   \if@psfbbfound\else
    \if@psfverbose\message{No bounding box comment in #1; using defaults}\fi\fi
   }\closein\@psffilein\fi}%
\ctr@ln@m\@psfbblit
\ctr@ln@m\@psfpercent
{\catcode`\%=12 \global\let\@psfpercent=
\ctr@ln@m\@psfaux
\long\def\@psfaux#1#2:#3\\{\ifx#1\@psfpercent
   \def\testit{#2}\ifx\testit\@psfbblit
      \@psfgrab #3 . . . \\%
      \@psffileokfalse
      \global\@psfbbfoundtrue
   \fi\else\ifx#1\par\else\@psffileokfalse\fi\fi}%
\ctr@ld@f\def\@psfempty{}%
\ctr@ld@f\def\@psfgrab #1 #2 #3 #4 #5\\{%
\global\def\@psfllx{#1}\ifx\@psfllx\@psfempty
      \@psfgrab #2 #3 #4 #5 .\\\else
   \global\def\@psflly{#2}%
   \global\def\@psfurx{#3}\global\def\@psfury{#4}\fi}%
\ctr@ld@f\def\PSwrit@cmd#1#2#3{{\Figg@tXY{#1}\c@lprojSP\b@undb@x{\v@lX}{\v@lY}%
    \v@lX=\ptT@ptps\v@lX\v@lY=\ptT@ptps\v@lY%
    \immediate\write#3{\repdecn@mb{\v@lX}\space\repdecn@mb{\v@lY}\space#2}}}
\ctr@ld@f\def\PSwrit@cmdS#1#2#3#4#5{{\Figg@tXY{#1}\c@lprojSP\b@undb@x{\v@lX}{\v@lY}%
    \global\result@t=\v@lX\global\result@@t=\v@lY%
    \v@lX=\ptT@ptps\v@lX\v@lY=\ptT@ptps\v@lY%
    \immediate\write#3{\repdecn@mb{\v@lX}\space\repdecn@mb{\v@lY}\space#2}}%
    \edef#4{\the\result@t}\edef#5{\the\result@@t}}
\ctr@ld@f\def\update@ttr#1#2#3{\Figdisc@rdLTS{#3}{\n@mref}%
    \ifx\n@mref\D@FTref#2{#1}\else#2{#3}\fi}
\ctr@ld@f\def\D@FTref{default}
\ctr@ld@f\def\W@rnmesAttr#1#2{%
    \immediate\write16{*** Unknown attribute: \BS@ #1(..., #2=...)}}
\ctr@ld@f\def\W@rnmeskwd#1#2{%
    \immediate\write16{*** Unknown keyword #2 in \BS@ #1}}
\ctr@ld@f\def\W@rnmesIgn#1{\immediate\write16{*** \BS@ #1 is ignored inside a
     \BS@ figdrawbegin-\BS@ figdrawend block.}}
\ctr@ld@f\def\Psset@lti#1=#2|{\keln@mtr#1|%
    \def\n@mref{blc}\ifx\l@debut\n@mref\update@ttr\D@FTref\P@setblcolor{#2}\else
    \def\n@mref{bld}\ifx\l@debut\n@mref\update@ttr\D@FTref\P@setbldash{#2}\else
    \def\n@mref{blw}\ifx\l@debut\n@mref\update@ttr\D@FTref\P@setblwidth{#2}\else
    \def\n@mref{sqc}\ifx\l@debut\n@mref\update@ttr\D@FTref\P@setsqcolor{#2}\else
    \def\n@mref{sqd}\ifx\l@debut\n@mref\update@ttr\D@FTref\P@setsqdash{#2}\else
    \def\n@mref{sqw}\ifx\l@debut\n@mref\update@ttr\D@FTref\P@setsqwidth{#2}\else
    \W@rnmesAttr{figset altitude}{#1}\fi\fi\fi\fi\fi\fi}
\ctr@ln@m\DDV@blcolor
\ctr@ld@f\def\P@setblcolor#1{\edef\DDV@blcolor{#1}}
\ctr@ln@m\DDV@bldash
\ctr@ld@f\def\P@setbldash#1{\edef\DDV@bldash{#1}}
\ctr@ln@m\DDV@blwidth
\ctr@ld@f\def\P@setblwidth#1{\edef\DDV@blwidth{#1}}
\ctr@ln@m\DDV@sqcolor
\ctr@ld@f\def\P@setsqcolor#1{\edef\DDV@sqcolor{#1}}
\ctr@ln@m\DDV@sqdash
\ctr@ld@f\def\P@setsqdash#1{\edef\DDV@sqdash{#1}}
\ctr@ln@m\DDV@sqwidth
\ctr@ld@f\def\P@setsqwidth#1{\edef\DDV@sqwidth{#1}}
\ctr@ld@f\def\figdrawaltitude#1[#2,#3,#4]{{\ifCUR@PS\ifGR@cri%
    \PSc@mment{altitude Square Dim=#1, Triangle=[#2 / #3,#4]}%
    \s@uvc@ntr@l\et@tpsaltitude\resetc@ntr@l{2}\figptorthoprojline-5:=#2/#3,#4/%
    \figvectP -1[#3,#4]\n@rminf{\v@leur}{-1}\vecunit@{-3}{-1}%
    \figvectP -1[-5,#3]\n@rminf{\v@lmin}{-1}\figvectP -2[-5,#4]\n@rminf{\v@lmax}{-2}%
    \ifdim\v@lmin<\v@lmax\s@mme=#3\else\v@lmax=\v@lmin\s@mme=#4\fi%
    \figvectP -4[-5,#2]\vecunit@{-4}{-4}\delt@=#1\unit@%
    \edef\t@ille{\repdecn@mb{\delt@}}\figpttra-1:=-5/\t@ille,-3/%
    \figptstra-3=-5,-1/\t@ille,-4/\figdrawline[#2,-5]%
    \Pss@tspecifSt{color=\DDV@sqcolor,dash=\DDV@sqdash,width=\DDV@sqwidth}%
    \figdrawline[-1,-2,-3]%
    \Psrest@reSt{color=\DDV@sqcolor,dash=\DDV@sqdash,width=\DDV@sqwidth}%
    \ifdim\v@leur<\v@lmax%
    \Pss@tspecifSt{color=\DDV@blcolor,dash=\DDV@bldash,width=\DDV@blwidth}%
    \figdrawline[-5,\the\s@mme]%
    \Psrest@reSt{color=\DDV@blcolor,dash=\DDV@bldash,width=\DDV@blwidth}%
    \fi\PSc@mment{End altitude}\resetc@ntr@l\et@tpsaltitude\fi\fi}}
\ctr@ld@f\def\Ps@rcerc#1;#2(#3,#4){\ellBB@x#1;#2,#2(#3,#4,0)%
    \f@gnewpath{\delt@=#2\unit@\delt@=\ptT@ptps\delt@%
    \BdingB@xfalse%
    \PSwrit@cmd{#1}{\repdecn@mb{\delt@}\space #3\space #4\space arc}{\fwf@g}}}
\ctr@ln@m\figdrawarccirc
\ctr@ld@f\def\Q@arccircDD#1;#2(#3,#4){\ifCUR@PS\ifGR@cri%
    \PSc@mment{arccircDD Center=#1 ; Radius=#2 (Ang1=#3, Ang2=#4)}%
    \iffillm@de\Ps@rcerc#1;#2(#3,#4)%
    \f@gfill%
    \else\Ps@rcerc#1;#2(#3,#4)\f@gstroke\fi%
    \PSc@mment{End arccircDD}\fi\fi}
\ctr@ld@f\def\Q@arccircTD#1,#2,#3;#4(#5,#6){{\ifCUR@PS\ifGR@cri\s@uvc@ntr@l\et@tpsarccircTD%
    \PSc@mment{arccircTD Center=#1,P1=#2,P2=#3 ; Radius=#4 (Ang1=#5, Ang2=#6)}%
    \setc@ntr@l{2}\c@lExtAxes#1,#2,#3(#4)\Q@arcellPATD#1,-4,-5(#5,#6)%
    \PSc@mment{End arccircTD}\resetc@ntr@l\et@tpsarccircTD\fi\fi}}
\ctr@ld@f\def\c@lExtAxes#1,#2,#3(#4){%
    \figvectPTD-5[#1,#2]\vecunit@{-5}{-5}\figvectNTD-4[#1,#2,#3]\vecunit@{-4}{-4}%
    \figvectNVTD-3[-4,-5]\delt@=#4\unit@\edef\r@yon{\repdecn@mb{\delt@}}%
    \figpttra-4:=#1/\r@yon,-5/\figpttra-5:=#1/\r@yon,-3/}
\ctr@ln@m\figdrawarccircP
\ctr@ld@f\def\Q@arccircPDD#1;#2[#3,#4]{{\ifCUR@PS\ifGR@cri\s@uvc@ntr@l\et@tpsarccircPDD%
    \PSc@mment{arccircPDD Center=#1; Radius=#2, [P1=#3, P2=#4]}%
    \Ps@ngleparam#1;#2[#3,#4]\ifdim\v@lmin>\v@lmax\advance\v@lmax\DePI@deg\fi%
    \edef\@ngdeb{\repdecn@mb{\v@lmin}}\edef\@ngfin{\repdecn@mb{\v@lmax}}%
    \figdrawarccirc#1;\r@dius(\@ngdeb,\@ngfin)%
    \PSc@mment{End arccircPDD}\resetc@ntr@l\et@tpsarccircPDD\fi\fi}}
\ctr@ld@f\def\Q@arccircPTD#1;#2[#3,#4,#5]{{\ifCUR@PS\ifGR@cri\s@uvc@ntr@l\et@tpsarccircPTD%
    \PSc@mment{arccircPTD Center=#1; Radius=#2, [P1=#3, P2=#4, P3=#5]}%
    \setc@ntr@l{2}\c@lExtAxes#1,#3,#5(#2)\figdrawarcellPP#1,-4,-5[#3,#4]%
    \PSc@mment{End arccircPTD}\resetc@ntr@l\et@tpsarccircPTD\fi\fi}}
\ctr@ld@f\def\Ps@ngleparam#1;#2[#3,#4]{\setc@ntr@l{2}%
    \figvectPDD-1[#1,#3]\vecunit@{-1}{-1}\Figg@tXY{-1}\arct@n\v@lmin(\v@lX,\v@lY)%
    \figvectPDD-2[#1,#4]\vecunit@{-2}{-2}\Figg@tXY{-2}\arct@n\v@lmax(\v@lX,\v@lY)%
    \v@lmin=\rdT@deg\v@lmin\v@lmax=\rdT@deg\v@lmax%
    \v@leur=#2pt\maxim@m{\mili@u}{-\v@leur}{\v@leur}%
    \edef\r@dius{\repdecn@mb{\mili@u}}}
\ctr@ld@f\def\Ps@rcercBz#1;#2(#3,#4){\Ps@rellBz#1;#2,#2(#3,#4,0)}
\ctr@ld@f\def\Ps@rellBz#1;#2,#3(#4,#5,#6){%
    \ellBB@x#1;#2,#3(#4,#5,#6)\BdingB@xfalse%
    \c@lNbarcs{#4}{#5}\v@leur=#4pt\setc@ntr@l{2}\figptell-13::#1;#2,#3(#4,#6)%
    \f@gnewpath\PSwrit@cmd{-13}{\c@mmoveto}{\fwf@g}%
    \s@mme=\z@\bcl@rellBz#1;#2,#3(#6)\BdingB@xtrue}
\ctr@ld@f\def\bcl@rellBz#1;#2,#3(#4){\relax%
    \ifnum\s@mme<\p@rtent\advance\s@mme\@ne%
    \advance\v@leur\delt@\edef\@ngle{\repdecn@mb\v@leur}\figptell-14::#1;#2,#3(\@ngle,#4)%
    \advance\v@leur\delt@\edef\@ngle{\repdecn@mb\v@leur}\figptell-15::#1;#2,#3(\@ngle,#4)%
    \advance\v@leur\delt@\edef\@ngle{\repdecn@mb\v@leur}\figptell-16::#1;#2,#3(\@ngle,#4)%
    \figptscontrolDD-18[-13,-14,-15,-16]%
    \PSwrit@cmd{-18}{}{\fwf@g}\PSwrit@cmd{-17}{}{\fwf@g}%
    \PSwrit@cmd{-16}{\c@mcurveto}{\fwf@g}%
    \figptcopyDD-13:/-16/\bcl@rellBz#1;#2,#3(#4)\fi}
\ctr@ld@f\def\Ps@rell#1;#2,#3(#4,#5,#6){\ellBB@x#1;#2,#3(#4,#5,#6)%
    \f@gnewpath{\v@lmin=#2\unit@\v@lmin=\ptT@ptps\v@lmin%
    \v@lmax=#3\unit@\v@lmax=\ptT@ptps\v@lmax\BdingB@xfalse%
    \PSwrit@cmd{#1}%
    {#6\space\repdecn@mb{\v@lmin}\space\repdecn@mb{\v@lmax}\space #4\space #5\space ellipse}{\fwf@g}}%
    \global\Use@llipsetrue}
\ctr@ln@m\figdrawarcell
\ctr@ld@f\def\Q@arcellDD#1;#2,#3(#4,#5,#6){{\ifCUR@PS\ifGR@cri%
    \PSc@mment{arcellDD Center=#1 ; XRad=#2, YRad=#3 (Ang1=#4, Ang2=#5, Inclination=#6)}%
    \iffillm@de\Ps@rell#1;#2,#3(#4,#5,#6)%
    \f@gfill%
    \else\Ps@rell#1;#2,#3(#4,#5,#6)\f@gstroke\fi%
    \PSc@mment{End arcellDD}\fi\fi}}
\ctr@ld@f\def\Q@arcellTD#1;#2,#3(#4,#5,#6){{\ifCUR@PS\ifGR@cri\s@uvc@ntr@l\et@tpsarcellTD%
    \PSc@mment{arcellTD Center=#1 ; XRad=#2, YRad=#3 (Ang1=#4, Ang2=#5, Inclination=#6)}%
    \setc@ntr@l{2}\figpttraC -8:=#1/#2,0,0/\figpttraC -7:=#1/0,#3,0/%
    \figvectC -4(0,0,1)\figptsrot -8=-8,-7/#1,#6,-4/\Q@arcellPATD#1,-8,-7(#4,#5)%
    \PSc@mment{End arcellTD}\resetc@ntr@l\et@tpsarcellTD\fi\fi}}
\ctr@ln@m\figdrawarcellPA
\ctr@ld@f\def\Q@arcellPADD#1,#2,#3(#4,#5){{\ifCUR@PS\ifGR@cri\s@uvc@ntr@l\et@tpsarcellPADD%
    \PSc@mment{arcellPADD Center=#1,PtAxis1=#2,PtAxis2=#3 (Ang1=#4, Ang2=#5)}%
    \setc@ntr@l{2}\figvectPDD-1[#1,#2]\vecunit@DD{-1}{-1}\v@lX=\ptT@unit@\result@t%
    \edef\XR@d{\repdecn@mb{\v@lX}}\Figg@tXY{-1}\arct@n\v@lmin(\v@lX,\v@lY)%
    \v@lmin=\rdT@deg\v@lmin\edef\Inclin@{\repdecn@mb{\v@lmin}}%
    \figgetdist\YR@d[#1,#3]\Q@arcellDD#1;\XR@d,\YR@d(#4,#5,\Inclin@)%
    \PSc@mment{End arcellPADD}\resetc@ntr@l\et@tpsarcellPADD\fi\fi}}
\ctr@ld@f\def\Q@arcellPATD#1,#2,#3(#4,#5){{\ifCUR@PS\ifGR@cri\s@uvc@ntr@l\et@tpsarcellPATD%
    \PSc@mment{arcellPATD Center=#1,PtAxis1=#2,PtAxis2=#3 (Ang1=#4, Ang2=#5)}%
    \iffillm@de\Ps@rellPATD#1,#2,#3(#4,#5)%
    \f@gfill%
    \else\Ps@rellPATD#1,#2,#3(#4,#5)\f@gstroke\fi%
    \PSc@mment{End arcellPATD}\resetc@ntr@l\et@tpsarcellPATD\fi\fi}}
\ctr@ld@f\def\Ps@rellPATD#1,#2,#3(#4,#5){\let\c@lprojSP=\relax%
    \setc@ntr@l{2}\figvectPTD-1[#1,#2]\figvectPTD-2[#1,#3]\c@lNbarcs{#4}{#5}%
    \v@leur=#4pt\c@lptellP{#1}{-1}{-2}\Figptpr@j-5:/-3/%
    \f@gnewpath\PSwrit@cmdS{-5}{\c@mmoveto}{\fwf@g}{\X@un}{\Y@un}%
    \edef\C@nt@r{#1}\s@mme=\z@\bcl@rellPATD}
\ctr@ld@f\def\bcl@rellPATD{\relax%
    \ifnum\s@mme<\p@rtent\advance\s@mme\@ne%
    \advance\v@leur\delt@\c@lptellP{\C@nt@r}{-1}{-2}\Figptpr@j-4:/-3/%
    \advance\v@leur\delt@\c@lptellP{\C@nt@r}{-1}{-2}\Figptpr@j-6:/-3/%
    \advance\v@leur\delt@\c@lptellP{\C@nt@r}{-1}{-2}\Figptpr@j-3:/-3/%
    \v@lX=\z@\v@lY=\z@\Figtr@nptDD{-5}{-5}\Figtr@nptDD{2}{-3}%
    \divide\v@lX\@vi\divide\v@lY\@vi%
    \Figtr@nptDD{3}{-4}\Figtr@nptDD{-1.5}{-6}\v@lmin=\v@lX\v@lmax=\v@lY%
    \v@lX=\z@\v@lY=\z@\Figtr@nptDD{2}{-5}\Figtr@nptDD{-5}{-3}%
    \divide\v@lX\@vi\divide\v@lY\@vi\Figtr@nptDD{-1.5}{-4}\Figtr@nptDD{3}{-6}%
    \BdingB@xfalse%
    \Figp@intregDD-4:(\v@lmin,\v@lmax)\PSwrit@cmdS{-4}{}{\fwf@g}{\X@de}{\Y@de}%
    \Figp@intregDD-4:(\v@lX,\v@lY)\PSwrit@cmdS{-4}{}{\fwf@g}{\X@tr}{\Y@tr}%
    \BdingB@xtrue\PSwrit@cmdS{-3}{\c@mcurveto}{\fwf@g}{\X@qu}{\Y@qu}%
    \B@zierBB@x{1}{\Y@un}(\X@un,\X@de,\X@tr,\X@qu)%
    \B@zierBB@x{2}{\X@un}(\Y@un,\Y@de,\Y@tr,\Y@qu)%
    \edef\X@un{\X@qu}\edef\Y@un{\Y@qu}\figptcopyDD-5:/-3/\bcl@rellPATD\fi}
\ctr@ld@f\def\c@lNbarcs#1#2{%
    \delt@=#2pt\advance\delt@-#1pt\maxim@m{\v@lmax}{\delt@}{-\delt@}%
    \v@leur=\v@lmax\divide\v@leur45 \p@rtentiere{\p@rtent}{\v@leur}\advance\p@rtent\@ne%
    \s@mme=\p@rtent\multiply\s@mme\thr@@\divide\delt@\s@mme}
\ctr@ld@f\def\figdrawarcellPP#1,#2,#3[#4,#5]{{\ifCUR@PS\ifGR@cri\s@uvc@ntr@l\et@tpsarcellPP%
    \PSc@mment{arcellPP Center=#1,PtAxis1=#2,PtAxis2=#3 [Point1=#4, Point2=#5]}%
    \setc@ntr@l{2}\figvectP-2[#1,#3]\vecunit@{-2}{-2}\v@lmin=\result@t%
    \invers@{\v@lmax}{\v@lmin}%
    \figvectP-1[#1,#2]\vecunit@{-1}{-1}\v@leur=\result@t%
    \v@leur=\repdecn@mb{\v@lmax}\v@leur\edef\AsB@{\repdecn@mb{\v@leur}}
    \c@lAngle{#1}{#4}{\v@lmin}\edef\@ngdeb{\repdecn@mb{\v@lmin}}%
    \c@lAngle{#1}{#5}{\v@lmax}\ifdim\v@lmin>\v@lmax\advance\v@lmax\DePI@deg\fi%
    \edef\@ngfin{\repdecn@mb{\v@lmax}}\figdrawarcellPA#1,#2,#3(\@ngdeb,\@ngfin)%
    \PSc@mment{End arcellPP}\resetc@ntr@l\et@tpsarcellPP\fi\fi}}
\ctr@ld@f\def\c@lAngle#1#2#3{\figvectP-3[#1,#2]%
    \c@lproscal\delt@[-3,-1]\c@lproscal\v@leur[-3,-2]%
    \v@leur=\AsB@\v@leur\arct@n#3(\delt@,\v@leur)#3=\rdT@deg#3}
\ctr@ln@w{newif}\if@rrowratio\@rrowratiotrue
\ctr@ln@w{newif}\if@rrowhfill
\ctr@ln@w{newif}\if@rrowhout
\ctr@ld@f\def\Psset@rrowhe@d#1=#2|{\keln@mun#1|%
    \def\n@mref{a}\ifx\l@debut\n@mref\update@ttr\D@FTarrowheadangle\Q@s@tarrowheadangle{#2}\else
    \def\n@mref{f}\ifx\l@debut\n@mref\update@ttr\D@FTarrowheadfill\Q@s@tarrowheadfill{#2}\else
    \def\n@mref{l}\ifx\l@debut\n@mref\update@ttr\D@FTarrowheadlength\Q@s@tarrowheadlength{#2}\else
    \def\n@mref{o}\ifx\l@debut\n@mref\update@ttr\D@FTarrowheadout\Q@s@tarrowheadout{#2}\else
    \def\n@mref{r}\ifx\l@debut\n@mref\update@ttr\D@FTarrowheadratio\Q@s@tarrowheadratio{#2}\else
    \W@rnmesAttr{figset arrowhead}{#1}\fi\fi\fi\fi\fi}
\ctr@ln@m\@rrowheadangle
\ctr@ln@m\C@AHANG \ctr@ln@m\S@AHANG \ctr@ln@m\UNSS@N
\ctr@ld@f\def\Q@s@tarrowheadangle#1{\edef\@rrowheadangle{#1}{\c@ssin{\C@}{\S@}{#1}%
    \xdef\C@AHANG{\C@}\xdef\S@AHANG{\S@}\v@lmax=\S@ pt%
    \invers@{\v@leur}{\v@lmax}\maxim@m{\v@leur}{\v@leur}{-\v@leur}%
    \xdef\UNSS@N{\the\v@leur}}}
\ctr@ld@f\def\Q@s@tarrowheadfill#1{\expandafter\set@rrowhfill#1:}
\ctr@ld@f\def\set@rrowhfill#1#2:{\if#1n\@rrowhfillfalse\else\@rrowhfilltrue\fi}
\ctr@ld@f\def\Q@s@tarrowheadout#1{\expandafter\set@rrowhout#1:}
\ctr@ld@f\def\set@rrowhout#1#2:{\if#1n\@rrowhoutfalse\else\@rrowhouttrue\fi}
\ctr@ln@m\@rrowheadlength
\ctr@ld@f\def\Q@s@tarrowheadlength#1{\edef\@rrowheadlength{#1}\@rrowratiofalse}
\ctr@ln@m\@rrowheadratio
\ctr@ld@f\def\Q@s@tarrowheadratio#1{\edef\@rrowheadratio{#1}\@rrowratiotrue}
\ctr@ln@m\D@FTarrowheadlength
\ctr@ld@f\def\figresetarrowhead{%
    \Q@s@tarrowheadangle{\D@FTarrowheadangle}%
    \Q@s@tarrowheadfill{\D@FTarrowheadfill}%
    \Q@s@tarrowheadout{\D@FTarrowheadout}%
    \Q@s@tarrowheadratio{\D@FTarrowheadratio}%
    \d@fm@cdim\D@FTarrowheadlength{\D@FTh@rdahlength}
    \Q@s@tarrowheadlength{\D@FTarrowheadlength}}
\ctr@ld@f\def\D@FTarrowheadratio{0.1}
\ctr@ld@f\def\D@FTarrowheadangle{20}
\ctr@ld@f\def\D@FTarrowheadfill{no}
\ctr@ld@f\def\D@FTarrowheadout{no}
\ctr@ld@f\def\D@FTh@rdahlength{8pt}
\ctr@ln@m\figdrawarrow
\ctr@ld@f\def\Q@arrowDD[#1,#2]{{\ifCUR@PS\ifGR@cri\s@uvc@ntr@l\et@tpsarrow%
    \PSc@mment{arrowDD [Pt1,Pt2]=[#1,#2]}\Q@s@tfillmode{no}%
    \Q@arrowheadDD[#1,#2]\setc@ntr@l{2}\figdrawline[#1,-3]%
    \PSc@mment{End arrowDD}\resetc@ntr@l\et@tpsarrow\fi\fi}}
\ctr@ld@f\def\Q@arrowTD[#1,#2]{{\ifCUR@PS\ifGR@cri\s@uvc@ntr@l\et@tpsarrowTD%
    \PSc@mment{arrowTD [Pt1,Pt2]=[#1,#2]}\resetc@ntr@l{2}%
    \Figptpr@j-5:/#1/\Figptpr@j-6:/#2/\let\c@lprojSP=\relax\Q@arrowDD[-5,-6]%
    \PSc@mment{End arrowTD}\resetc@ntr@l\et@tpsarrowTD\fi\fi}}
\ctr@ln@m\figdrawarrowhead
\ctr@ld@f\def\Q@arrowheadDD[#1,#2]{{\ifCUR@PS\ifGR@cri\s@uvc@ntr@l\et@tpsarrowheadDD%
    \if@rrowhfill\def\@hangle{-\@rrowheadangle}\else\def\@hangle{\@rrowheadangle}\fi%
    \if@rrowratio%
    \if@rrowhout\def\@hratio{-\@rrowheadratio}\else\def\@hratio{\@rrowheadratio}\fi%
    \PSc@mment{arrowheadDD Ratio=\@hratio, Angle=\@hangle, [Pt1,Pt2]=[#1,#2]}%
    \Ps@rrowhead\@hratio,\@hangle[#1,#2]%
    \else%
    \if@rrowhout\def\@hlength{-\@rrowheadlength}\else\def\@hlength{\@rrowheadlength}\fi%
    \PSc@mment{arrowheadDD Length=\@hlength, Angle=\@hangle, [Pt1,Pt2]=[#1,#2]}%
    \Ps@rrowheadfd\@hlength,\@hangle[#1,#2]%
    \fi%
    \PSc@mment{End arrowheadDD}\resetc@ntr@l\et@tpsarrowheadDD\fi\fi}}
\ctr@ld@f\def\Q@arrowheadTD[#1,#2]{{\ifCUR@PS\ifGR@cri\s@uvc@ntr@l\et@tpsarrowheadTD%
    \PSc@mment{arrowheadTD [Pt1,Pt2]=[#1,#2]}\resetc@ntr@l{2}%
    \Figptpr@j-5:/#1/\Figptpr@j-6:/#2/\let\c@lprojSP=\relax\Q@arrowheadDD[-5,-6]%
    \PSc@mment{End arrowheadTD}\resetc@ntr@l\et@tpsarrowheadTD\fi\fi}}
\ctr@ld@f\def\Ps@rrowhead#1,#2[#3,#4]{\v@leur=#1\p@\maxim@m{\v@leur}{\v@leur}{-\v@leur}%
    \ifdim\v@leur>\Cepsil@n{
    \PSc@mment{@rrowhead Ratio=#1, Angle=#2, [Pt1,Pt2]=[#3,#4]}\v@leur=\UNSS@N%
    \v@leur=\CUR@width\v@leur\v@leur=\ptpsT@pt\v@leur\delt@=.5\v@leur
    \setc@ntr@l{2}\figvectPDD-3[#4,#3]%
    \Figg@tXY{-3}\v@lX=#1\v@lX\v@lY=#1\v@lY\Figv@ctCreg-3(\v@lX,\v@lY)%
    \vecunit@{-4}{-3}\mili@u=\result@t%
    \ifdim#2pt>\z@\v@lXa=-\C@AHANG\delt@%
     \edef\c@ef{\repdecn@mb{\v@lXa}}\figpttraDD-3:=-3/\c@ef,-4/\fi%
    \edef\c@ef{\repdecn@mb{\delt@}}%
    \v@lXa=\mili@u\v@lXa=\C@AHANG\v@lXa%
    \v@lYa=\ptpsT@pt\p@\v@lYa=\CUR@width\v@lYa\v@lYa=\sDcc@ngle\v@lYa%
    \advance\v@lXa-\v@lYa\gdef\sDcc@ngle{0}%
    \ifdim\v@lXa>\v@leur\edef\c@efendpt{\repdecn@mb{\v@leur}}%
    \else\edef\c@efendpt{\repdecn@mb{\v@lXa}}\fi%
    \Figg@tXY{-3}\v@lmin=\v@lX\v@lmax=\v@lY%
    \v@lXa=\C@AHANG\v@lmin\v@lYa=\S@AHANG\v@lmax\advance\v@lXa\v@lYa%
    \v@lYa=-\S@AHANG\v@lmin\v@lX=\C@AHANG\v@lmax\advance\v@lYa\v@lX%
    \setc@ntr@l{1}\Figg@tXY{#4}\advance\v@lX\v@lXa\advance\v@lY\v@lYa%
    \setc@ntr@l{2}\Figp@intregDD-2:(\v@lX,\v@lY)%
    \v@lXa=\C@AHANG\v@lmin\v@lYa=-\S@AHANG\v@lmax\advance\v@lXa\v@lYa%
    \v@lYa=\S@AHANG\v@lmin\v@lX=\C@AHANG\v@lmax\advance\v@lYa\v@lX%
    \setc@ntr@l{1}\Figg@tXY{#4}\advance\v@lX\v@lXa\advance\v@lY\v@lYa%
    \setc@ntr@l{2}\Figp@intregDD-1:(\v@lX,\v@lY)%
    \ifdim#2pt<\z@\fillm@detrue\figdrawline[-2,#4,-1]
    \else\figptstraDD-3=#4,-2,-1/\c@ef,-4/\s@uvdash{\typ@dash}\Q@s@tdash{\D@FTdash}%
    \figdrawline[-2,-3,-1]\Q@s@tdash{\typ@dash}\fi
    \ifdim#1pt>\z@\figpttraDD-3:=#4/\c@efendpt,-4/\else\figptcopyDD-3:/#4/\fi%
    \PSc@mment{End @rrowhead}}\fi}
\ctr@ld@f\def\sDcc@ngle{0}
\ctr@ld@f\def\Ps@rrowheadfd#1,#2[#3,#4]{{%
    \PSc@mment{@rrowheadfd Length=#1, Angle=#2, [Pt1,Pt2]=[#3,#4]}%
    \setc@ntr@l{2}\figvectPDD-1[#3,#4]\n@rmeucDD{\v@leur}{-1}\v@leur=\ptT@unit@\v@leur%
    \invers@{\v@leur}{\v@leur}\v@leur=#1\v@leur\edef\R@tio{\repdecn@mb{\v@leur}}%
    \Ps@rrowhead\R@tio,#2[#3,#4]\PSc@mment{End @rrowheadfd}}}
\ctr@ln@m\figdrawarrowBezier
\ctr@ld@f\def\Q@arrowBezierDD[#1,#2,#3,#4]{{\ifCUR@PS\ifGR@cri\s@uvc@ntr@l\et@tpsarrowBezierDD%
    \PSc@mment{arrowBezierDD Control points=#1,#2,#3,#4}\setc@ntr@l{2}%
    \if@rrowratio\c@larclengthDD\v@leur,10[#1,#2,#3,#4]\else\v@leur=\z@\fi%
    \Ps@rrowB@zDD\v@leur[#1,#2,#3,#4]%
    \PSc@mment{End arrowBezierDD}\resetc@ntr@l\et@tpsarrowBezierDD\fi\fi}}
\ctr@ld@f\def\Q@arrowBezierTD[#1,#2,#3,#4]{{\ifCUR@PS\ifGR@cri\s@uvc@ntr@l\et@tpsarrowBezierTD%
    \PSc@mment{arrowBezierTD Control points=#1,#2,#3,#4}\resetc@ntr@l{2}%
    \Figptpr@j-7:/#1/\Figptpr@j-8:/#2/\Figptpr@j-9:/#3/\Figptpr@j-10:/#4/%
    \let\c@lprojSP=\relax\ifnum\CUR@proj<\tw@\Q@arrowBezierDD[-7,-8,-9,-10]%
    \else\f@gnewpath\PSwrit@cmd{-7}{\c@mmoveto}{\fwf@g}%
    \if@rrowratio\c@larclengthDD\mili@u,10[-7,-8,-9,-10]\else\mili@u=\z@\fi%
    \p@rtent=\NBz@rcs\advance\p@rtent\m@ne\subB@zierTD\p@rtent[#1,#2,#3,#4]%
    \f@gstroke%
    \advance\v@lmin\p@rtent\delt@
    \v@leur=\v@lmin\advance\v@leur0.33333 \delt@\edef\unti@rs{\repdecn@mb{\v@leur}}%
    \v@leur=\v@lmin\advance\v@leur0.66666 \delt@\edef\deti@rs{\repdecn@mb{\v@leur}}%
    \figptcopyDD-8:/-10/\c@lsubBzarc\unti@rs,\deti@rs[#1,#2,#3,#4]%
    \figptcopyDD-8:/-4/\figptcopyDD-9:/-3/\Ps@rrowB@zDD\mili@u[-7,-8,-9,-10]\fi%
    \PSc@mment{End arrowBezierTD}\resetc@ntr@l\et@tpsarrowBezierTD\fi\fi}}
\ctr@ld@f\def\c@larclengthDD#1,#2[#3,#4,#5,#6]{{\p@rtent=#2\figptcopyDD-5:/#3/%
    \delt@=\p@\divide\delt@\p@rtent\c@rre=\z@\v@leur=\z@\s@mme=\z@%
    \loop\ifnum\s@mme<\p@rtent\advance\s@mme\@ne\advance\v@leur\delt@%
    \edef\T@{\repdecn@mb{\v@leur}}\figptBezierDD-6::\T@[#3,#4,#5,#6]%
    \figvectPDD-1[-5,-6]\n@rmeucDD{\mili@u}{-1}\advance\c@rre\mili@u%
    \figptcopyDD-5:/-6/\repeat\global\result@t=\ptT@unit@\c@rre}#1=\result@t}
\ctr@ld@f\def\Ps@rrowB@zDD#1[#2,#3,#4,#5]{{\Q@s@tfillmode{no}%
    \if@rrowratio\delt@=\@rrowheadratio#1\else\delt@=\@rrowheadlength pt\fi%
    \v@leur=\C@AHANG\delt@\edef\R@dius{\repdecn@mb{\v@leur}}%
    \FigptintercircB@zDD-5::0,\R@dius[#5,#4,#3,#2]%
    \Q@s@tarrowheadlength{\repdecn@mb{\delt@}}\Q@arrowheadDD[-5,#5]%
    \let\n@rmeuc=\n@rmeucDD\figgetdist\R@dius[#5,-3]%
    \FigptintercircB@zDD-6::0,\R@dius[#5,#4,#3,#2]%
    \figptBezierDD-5::0.33333[#5,#4,#3,#2]\figptBezierDD-3::0.66666[#5,#4,#3,#2]%
    \figptscontrolDD-5[-6,-5,-3,#2]\Q@BezierDD1[-6,-5,-4,#2]}}
\ctr@ln@m\figdrawarrowcirc
\ctr@ld@f\def\Q@arrowcircDD#1;#2(#3,#4){{\ifCUR@PS\ifGR@cri\s@uvc@ntr@l\et@tpsarrowcircDD%
    \PSc@mment{arrowcircDD Center=#1 ; Radius=#2 (Ang1=#3,Ang2=#4)}%
    \Q@s@tfillmode{no}\Pscirc@rrowhead#1;#2(#3,#4)%
    \setc@ntr@l{2}\figvectPDD -4[#1,-3]\vecunit@{-4}{-4}%
    \Figg@tXY{-4}\arct@n\v@lmin(\v@lX,\v@lY)%
    \v@lmin=\rdT@deg\v@lmin\v@leur=#4pt\advance\v@leur-\v@lmin%
    \maxim@m{\v@leur}{\v@leur}{-\v@leur}%
    \ifdim\v@leur>\DemiPI@deg\relax\ifdim\v@lmin<#4pt\advance\v@lmin\DePI@deg%
    \else\advance\v@lmin-\DePI@deg\fi\fi\edef\ar@ngle{\repdecn@mb{\v@lmin}}%
    \ifdim#3pt<#4pt\figdrawarccirc#1;#2(#3,\ar@ngle)\else\figdrawarccirc#1;#2(\ar@ngle,#3)\fi%
    \PSc@mment{End arrowcircDD}\resetc@ntr@l\et@tpsarrowcircDD\fi\fi}}
\ctr@ld@f\def\Q@arrowcircTD#1,#2,#3;#4(#5,#6){{\ifCUR@PS\ifGR@cri\s@uvc@ntr@l\et@tpsarrowcircTD%
    \PSc@mment{arrowcircTD Center=#1,P1=#2,P2=#3 ; Radius=#4 (Ang1=#5, Ang2=#6)}%
    \resetc@ntr@l{2}\c@lExtAxes#1,#2,#3(#4)\let\c@lprojSP=\relax%
    \figvectPTD-11[#1,-4]\figvectPTD-12[#1,-5]\c@lNbarcs{#5}{#6}%
    \if@rrowratio\v@lmax=\degT@rd\v@lmax\edef\D@lpha{\repdecn@mb{\v@lmax}}\fi%
    \advance\p@rtent\m@ne\mili@u=\z@%
    \v@leur=#5pt\c@lptellP{#1}{-11}{-12}\Figptpr@j-9:/-3/%
    \f@gnewpath\PSwrit@cmdS{-9}{\c@mmoveto}{\fwf@g}{\X@un}{\Y@un}%
    \edef\C@nt@r{#1}\s@mme=\z@\bcl@rcircTD\f@gstroke%
    \advance\v@leur\delt@\c@lptellP{#1}{-11}{-12}\Figptpr@j-5:/-3/%
    \advance\v@leur\delt@\c@lptellP{#1}{-11}{-12}\Figptpr@j-6:/-3/%
    \advance\v@leur\delt@\c@lptellP{#1}{-11}{-12}\Figptpr@j-10:/-3/%
    \figptscontrolDD-8[-9,-5,-6,-10]%
    \if@rrowratio\c@lcurvradDD0.5[-9,-8,-7,-10]\advance\mili@u\result@t%
    \maxim@m{\mili@u}{\mili@u}{-\mili@u}\mili@u=\ptT@unit@\mili@u%
    \mili@u=\D@lpha\mili@u\advance\p@rtent\@ne\divide\mili@u\p@rtent\fi%
    \Ps@rrowB@zDD\mili@u[-9,-8,-7,-10]%
    \PSc@mment{End arrowcircTD}\resetc@ntr@l\et@tpsarrowcircTD\fi\fi}}
\ctr@ld@f\def\bcl@rcircTD{\relax%
    \ifnum\s@mme<\p@rtent\advance\s@mme\@ne%
    \advance\v@leur\delt@\c@lptellP{\C@nt@r}{-11}{-12}\Figptpr@j-5:/-3/%
    \advance\v@leur\delt@\c@lptellP{\C@nt@r}{-11}{-12}\Figptpr@j-6:/-3/%
    \advance\v@leur\delt@\c@lptellP{\C@nt@r}{-11}{-12}\Figptpr@j-10:/-3/%
    \figptscontrolDD-8[-9,-5,-6,-10]\BdingB@xfalse%
    \PSwrit@cmdS{-8}{}{\fwf@g}{\X@de}{\Y@de}\PSwrit@cmdS{-7}{}{\fwf@g}{\X@tr}{\Y@tr}%
    \BdingB@xtrue\PSwrit@cmdS{-10}{\c@mcurveto}{\fwf@g}{\X@qu}{\Y@qu}%
    \if@rrowratio\c@lcurvradDD0.5[-9,-8,-7,-10]\advance\mili@u\result@t\fi%
    \B@zierBB@x{1}{\Y@un}(\X@un,\X@de,\X@tr,\X@qu)%
    \B@zierBB@x{2}{\X@un}(\Y@un,\Y@de,\Y@tr,\Y@qu)%
    \edef\X@un{\X@qu}\edef\Y@un{\Y@qu}\figptcopyDD-9:/-10/\bcl@rcircTD\fi}
\ctr@ld@f\def\Pscirc@rrowhead#1;#2(#3,#4){{%
    \PSc@mment{circ@rrowhead Center=#1 ; Radius=#2 (Ang1=#3,Ang2=#4)}%
    \v@leur=#2\unit@\edef\s@glen{\repdecn@mb{\v@leur}}\v@lY=\z@\v@lX=\v@leur%
    \resetc@ntr@l{2}\Figv@ctCreg-3(\v@lX,\v@lY)\figpttraDD-5:=#1/1,-3/%
    \figptrotDD-5:=-5/#1,#4/%
    \figvectPDD-3[#1,-5]\Figg@tXY{-3}\v@leur=\v@lX%
    \ifdim#3pt<#4pt\v@lX=\v@lY\v@lY=-\v@leur\else\v@lX=-\v@lY\v@lY=\v@leur\fi%
    \Figv@ctCreg-3(\v@lX,\v@lY)\vecunit@{-3}{-3}%
    \if@rrowratio\v@leur=#4pt\advance\v@leur-#3pt\maxim@m{\mili@u}{-\v@leur}{\v@leur}%
    \mili@u=\degT@rd\mili@u\v@leur=\s@glen\mili@u\edef\s@glen{\repdecn@mb{\v@leur}}%
    \mili@u=#2\mili@u\mili@u=\@rrowheadratio\mili@u\else\mili@u=\@rrowheadlength pt\fi%
    \figpttraDD-6:=-5/\s@glen,-3/\v@leur=#2pt\v@leur=2\v@leur%
    \invers@{\v@leur}{\v@leur}\c@rre=\repdecn@mb{\v@leur}\mili@u
    \mili@u=\c@rre\mili@u=\repdecn@mb{\c@rre}\mili@u%
    \v@leur=\p@\advance\v@leur-\mili@u
    \invers@{\mili@u}{2\v@leur}\delt@=\c@rre\delt@=\repdecn@mb{\mili@u}\delt@%
    \xdef\sDcc@ngle{\repdecn@mb{\delt@}}
    \sqrt@{\mili@u}{\v@leur}\arct@n\v@leur(\mili@u,\c@rre)%
    \v@leur=\rdT@deg\v@leur
    \ifdim#3pt<#4pt\v@leur=-\v@leur\fi%
    \if@rrowhout\v@leur=-\v@leur\fi\edef\cor@ngle{\repdecn@mb{\v@leur}}%
    \figptrotDD-6:=-6/-5,\cor@ngle/\Q@arrowheadDD[-6,-5]%
    \PSc@mment{End circ@rrowhead}}}
\ctr@ln@m\figdrawarrowcircP
\ctr@ld@f\def\Q@arrowcircPDD#1;#2[#3,#4]{{\ifCUR@PS\ifGR@cri%
    \PSc@mment{arrowcircPDD Center=#1; Radius=#2, [P1=#3,P2=#4]}%
    \s@uvc@ntr@l\et@tpsarrowcircPDD\Ps@ngleparam#1;#2[#3,#4]%
    \ifdim\v@leur>\z@\ifdim\v@lmin>\v@lmax\advance\v@lmax\DePI@deg\fi%
    \else\ifdim\v@lmin<\v@lmax\advance\v@lmin\DePI@deg\fi\fi%
    \edef\@ngdeb{\repdecn@mb{\v@lmin}}\edef\@ngfin{\repdecn@mb{\v@lmax}}%
    \figdrawarrowcirc#1;\r@dius(\@ngdeb,\@ngfin)%
    \PSc@mment{End arrowcircPDD}\resetc@ntr@l\et@tpsarrowcircPDD\fi\fi}}
\ctr@ld@f\def\Q@arrowcircPTD#1;#2[#3,#4,#5]{{\ifCUR@PS\ifGR@cri\s@uvc@ntr@l\et@tpsarrowcircPTD%
    \PSc@mment{arrowcircPTD Center=#1; Radius=#2, [P1=#3,P2=#4,P3=#5]}%
    \figgetangleTD\@ngfin[#1,#3,#4,#5]\v@leur=#2pt%
    \maxim@m{\mili@u}{-\v@leur}{\v@leur}\edef\r@dius{\repdecn@mb{\mili@u}}%
    \ifdim\v@leur<\z@\v@lmax=\@ngfin pt\advance\v@lmax-\DePI@deg%
    \edef\@ngfin{\repdecn@mb{\v@lmax}}\fi\Q@arrowcircTD#1,#3,#5;\r@dius(0,\@ngfin)%
    \PSc@mment{End arrowcircPTD}\resetc@ntr@l\et@tpsarrowcircPTD\fi\fi}}
\ctr@ld@f\def\figdrawaxes#1(#2){{\ifCUR@PS\ifGR@cri\s@uvc@ntr@l\et@tpsaxes%
    \PSc@mment{axes Origin=#1 Range=(#2)}\an@lys@xes#2,:\resetc@ntr@l{2}%
    \ifx\t@xt@\empty\ifTr@isDim\Q@@xes#1(0,#2,0,#2,0,#2)\else\Q@@xes#1(0,#2,0,#2)\fi%
    \else\Q@@xes#1(#2)\fi\PSc@mment{End axes}\resetc@ntr@l\et@tpsaxes\fi\fi}}
\ctr@ld@f\def\an@lys@xes#1,#2:{\def\t@xt@{#2}}
\ctr@ln@m\Q@@xes
\ctr@ld@f\def\Q@@xesDD#1(#2,#3,#4,#5){%
    \figpttraC-5:=#1/#2,0/\figpttraC-6:=#1/#3,0/\Q@arrowDD[-5,-6]%
    \figpttraC-5:=#1/0,#4/\figpttraC-6:=#1/0,#5/\Q@arrowDD[-5,-6]}
\ctr@ld@f\def\Q@@xesTD#1(#2,#3,#4,#5,#6,#7){%
    \figpttraC-7:=#1/#2,0,0/\figpttraC-8:=#1/#3,0,0/\Q@arrowTD[-7,-8]%
    \figpttraC-7:=#1/0,#4,0/\figpttraC-8:=#1/0,#5,0/\Q@arrowTD[-7,-8]%
    \figpttraC-7:=#1/0,0,#6/\figpttraC-8:=#1/0,0,#7/\Q@arrowTD[-7,-8]}
\ctr@ln@m\newGr@FN
\ctr@ld@f\def\newGr@FNPDF#1{\s@mme=\Gr@FNb\advance\s@mme\@ne\xdef\Gr@FNb{\number\s@mme}}
\ctr@ld@f\def\newGr@FNDVI#1{\newGr@FNPDF{}\xdef#1{\jobname GI\Gr@FNb.anx}}
\ctr@ld@f\def\figdrawbegin#1{\newGr@FN\DefGIfilen@me\gdef\@utoFN{0}%
    \def\t@xt@{#1}\relax\ifx\t@xt@\empty\GRupdatem@detrue%
    \gdef\@utoFN{1}\Psb@ginfig\DefGIfilen@me\else\expandafter\Psb@ginfigNu@#1 :\fi}
\ctr@ld@f\def\Psb@ginfigNu@#1 #2:{\def\t@xt@{#1}\relax\ifx\t@xt@\empty\def\t@xt@{#2}%
    \ifx\t@xt@\empty\GRupdatem@detrue\gdef\@utoFN{1}\Psb@ginfig\DefGIfilen@me%
    \else\Psb@ginfigNu@#2:\fi\else\Psb@ginfig{#1}\fi}
\ctr@ln@m\PSfilen@me \ctr@ln@m\auxfilen@me
\ctr@ld@f\def\Psb@ginfig#1{\ifCUR@PS\else%
    \edef\PSfilen@me{#1}\edef\auxfilen@me{\jobname.anx}%
    \ifGRupdatem@de\GR@critrue\else\openin\frf@g=\PSfilen@me\relax%
    \ifeof\frf@g\GR@critrue\else\GR@crifalse\fi\closein\frf@g\fi%
    \CUR@PStrue\c@ldefproj\expandafter\setupd@te\D@FTupdate:%
    \ifGR@cri\initb@undb@x%
    \immediate\openout\fwf@g=\auxfilen@me\initpss@ttings\fi%
    \fi}
\ctr@ld@f\def\Gr@FNb{0}
\ctr@ld@f\def\figforTeXFileno{\Gr@FNb}
\ctr@ld@f\def\figforTeXFigno{0 }
\ctr@ld@f\def\figforTeXnextFigno{1 }
\ctr@ld@f\edef\DefGIfilen@me{\jobname GI.anx}
\ctr@ld@f\def\initpss@ttings{\figreset{altitude,arrowhead,curve,general,flowchart,mesh,trimesh}%
    \Use@llipsefalse}
\ctr@ld@f\def\B@zierBB@x#1#2(#3,#4,#5,#6){{\c@rre=\t@n\epsil@n
    \v@lmax=#4\advance\v@lmax-#5\v@lmax=\thr@@\v@lmax\advance\v@lmax#6\advance\v@lmax-#3%
    \mili@u=#4\mili@u=-\tw@\mili@u\advance\mili@u#3\advance\mili@u#5%
    \v@lmin=#4\advance\v@lmin-#3\maxim@m{\v@leur}{-\v@lmax}{\v@lmax}%
    \maxim@m{\delt@}{-\mili@u}{\mili@u}\maxim@m{\v@leur}{\v@leur}{\delt@}%
    \maxim@m{\delt@}{-\v@lmin}{\v@lmin}\maxim@m{\v@leur}{\v@leur}{\delt@}%
    \ifdim\v@leur>\c@rre\invers@{\v@leur}{\v@leur}\edef\Uns@rM@x{\repdecn@mb{\v@leur}}%
    \v@lmax=\Uns@rM@x\v@lmax\mili@u=\Uns@rM@x\mili@u\v@lmin=\Uns@rM@x\v@lmin%
    \maxim@m{\v@leur}{-\v@lmax}{\v@lmax}\ifdim\v@leur<\c@rre%
    \maxim@m{\v@leur}{-\mili@u}{\mili@u}\ifdim\v@leur<\c@rre\else%
    \invers@{\mili@u}{\mili@u}\v@leur=-0.5\v@lmin%
    \v@leur=\repdecn@mb{\mili@u}\v@leur\m@jBBB@x{\v@leur}{#1}{#2}(#3,#4,#5,#6)\fi%
    \else\delt@=\repdecn@mb{\mili@u}\mili@u\v@leur=\repdecn@mb{\v@lmax}\v@lmin%
    \advance\delt@-\v@leur\ifdim\delt@<\z@\else\invers@{\v@lmax}{\v@lmax}%
    \edef\Uns@rAp{\repdecn@mb{\v@lmax}}\sqrt@{\delt@}{\delt@}%
    \v@leur=-\mili@u\advance\v@leur\delt@\v@leur=\Uns@rAp\v@leur%
    \m@jBBB@x{\v@leur}{#1}{#2}(#3,#4,#5,#6)%
    \v@leur=-\mili@u\advance\v@leur-\delt@\v@leur=\Uns@rAp\v@leur%
    \m@jBBB@x{\v@leur}{#1}{#2}(#3,#4,#5,#6)\fi\fi\fi}}
\ctr@ld@f\def\m@jBBB@x#1#2#3(#4,#5,#6,#7){{\relax\ifdim#1>\z@\ifdim#1<\p@%
    \edef\T@{\repdecn@mb{#1}}\v@lX=\p@\advance\v@lX-#1\edef\UNmT@{\repdecn@mb{\v@lX}}%
    \v@lX=#4\v@lY=#5\v@lZ=#6\v@lXa=#7\v@lX=\UNmT@\v@lX\advance\v@lX\T@\v@lY%
    \v@lY=\UNmT@\v@lY\advance\v@lY\T@\v@lZ\v@lZ=\UNmT@\v@lZ\advance\v@lZ\T@\v@lXa%
    \v@lX=\UNmT@\v@lX\advance\v@lX\T@\v@lY\v@lY=\UNmT@\v@lY\advance\v@lY\T@\v@lZ%
    \v@lX=\UNmT@\v@lX\advance\v@lX\T@\v@lY%
    \ifcase#2\or\v@lY=#3\or\v@lY=\v@lX\v@lX=#3\fi\b@undb@x{\v@lX}{\v@lY}\fi\fi}}
\ctr@ld@f\def\PsB@zier#1[#2]{{\f@gnewpath%
    \s@mme=\z@\def\list@num{#2,0}\extrairelepremi@r\p@int\de\list@num%
    \PSwrit@cmdS{\p@int}{\c@mmoveto}{\fwf@g}{\X@un}{\Y@un}\p@rtent=#1\bclB@zier}}
\ctr@ld@f\def\bclB@zier{\relax%
    \ifnum\s@mme<\p@rtent\advance\s@mme\@ne\BdingB@xfalse%
    \extrairelepremi@r\p@int\de\list@num\PSwrit@cmdS{\p@int}{}{\fwf@g}{\X@de}{\Y@de}%
    \extrairelepremi@r\p@int\de\list@num\PSwrit@cmdS{\p@int}{}{\fwf@g}{\X@tr}{\Y@tr}%
    \BdingB@xtrue%
    \extrairelepremi@r\p@int\de\list@num\PSwrit@cmdS{\p@int}{\c@mcurveto}{\fwf@g}{\X@qu}{\Y@qu}%
    \B@zierBB@x{1}{\Y@un}(\X@un,\X@de,\X@tr,\X@qu)%
    \B@zierBB@x{2}{\X@un}(\Y@un,\Y@de,\Y@tr,\Y@qu)%
    \edef\X@un{\X@qu}\edef\Y@un{\Y@qu}\bclB@zier\fi}
\ctr@ln@m\figdrawBezier
\ctr@ld@f\def\Q@BezierDD#1[#2]{\ifCUR@PS\ifGR@cri%
    \PSc@mment{BezierDD N arcs=#1, Control points=#2}%
    \iffillm@de\PsB@zier#1[#2]%
    \f@gfill%
    \else\PsB@zier#1[#2]\f@gstroke\fi%
    \PSc@mment{End BezierDD}\fi\fi}
\ctr@ln@m\et@tpsBezierTD
\ctr@ld@f\def\Q@BezierTD#1[#2]{\ifCUR@PS\ifGR@cri\s@uvc@ntr@l\et@tpsBezierTD%
    \PSc@mment{BezierTD N arcs=#1, Control points=#2}%
    \iffillm@de\PsB@zierTD#1[#2]%
    \f@gfill%
    \else\PsB@zierTD#1[#2]\f@gstroke\fi%
    \PSc@mment{End BezierTD}\resetc@ntr@l\et@tpsBezierTD\fi\fi}
\ctr@ld@f\def\PsB@zierTD#1[#2]{\ifnum\CUR@proj<\tw@\PsB@zier#1[#2]\else\PsB@zier@TD#1[#2]\fi}
\ctr@ld@f\def\PsB@zier@TD#1[#2]{{\f@gnewpath%
    \s@mme=\z@\def\list@num{#2,0}\extrairelepremi@r\p@int\de\list@num%
    \let\c@lprojSP=\relax\setc@ntr@l{2}\Figptpr@j-7:/\p@int/%
    \PSwrit@cmd{-7}{\c@mmoveto}{\fwf@g}%
    \loop\ifnum\s@mme<#1\advance\s@mme\@ne\extrairelepremi@r\p@intun\de\list@num%
    \extrairelepremi@r\p@intde\de\list@num\extrairelepremi@r\p@inttr\de\list@num%
    \subB@zierTD\NBz@rcs[\p@int,\p@intun,\p@intde,\p@inttr]\edef\p@int{\p@inttr}\repeat}}
\ctr@ld@f\def\subB@zierTD#1[#2,#3,#4,#5]{\delt@=\p@\divide\delt@\NBz@rcs\v@lmin=\z@%
    {\Figg@tXY{-7}\edef\X@un{\the\v@lX}\edef\Y@un{\the\v@lY}%
    \s@mme=\z@\loop\ifnum\s@mme<#1\advance\s@mme\@ne%
    \v@leur=\v@lmin\advance\v@leur0.33333 \delt@\edef\unti@rs{\repdecn@mb{\v@leur}}%
    \v@leur=\v@lmin\advance\v@leur0.66666 \delt@\edef\deti@rs{\repdecn@mb{\v@leur}}%
    \advance\v@lmin\delt@\edef\trti@rs{\repdecn@mb{\v@lmin}}%
    \figptBezierTD-8::\trti@rs[#2,#3,#4,#5]\Figptpr@j-8:/-8/%
    \c@lsubBzarc\unti@rs,\deti@rs[#2,#3,#4,#5]\BdingB@xfalse%
    \PSwrit@cmdS{-4}{}{\fwf@g}{\X@de}{\Y@de}\PSwrit@cmdS{-3}{}{\fwf@g}{\X@tr}{\Y@tr}%
    \BdingB@xtrue\PSwrit@cmdS{-8}{\c@mcurveto}{\fwf@g}{\X@qu}{\Y@qu}%
    \B@zierBB@x{1}{\Y@un}(\X@un,\X@de,\X@tr,\X@qu)%
    \B@zierBB@x{2}{\X@un}(\Y@un,\Y@de,\Y@tr,\Y@qu)%
    \edef\X@un{\X@qu}\edef\Y@un{\Y@qu}\figptcopyDD-7:/-8/\repeat}}
\ctr@ld@f\def\NBz@rcs{2}
\ctr@ld@f\def\c@lsubBzarc#1,#2[#3,#4,#5,#6]{\figptBezierTD-5::#1[#3,#4,#5,#6]%
    \figptBezierTD-6::#2[#3,#4,#5,#6]\Figptpr@j-4:/-5/\Figptpr@j-5:/-6/%
    \figptscontrolDD-4[-7,-4,-5,-8]}
\ctr@ln@m\figdrawcirc
\ctr@ld@f\def\Q@circDD#1(#2){\ifCUR@PS\ifGR@cri\PSc@mment{circDD Center=#1 (Radius=#2)}%
    \Q@arccircDD#1;#2(0,360)\PSc@mment{End circDD}\fi\fi}
\ctr@ld@f\def\Q@circTD#1,#2,#3(#4){\ifCUR@PS\ifGR@cri%
    \PSc@mment{circTD Center=#1,P1=#2,P2=#3 (Radius=#4)}%
    \Q@arccircTD#1,#2,#3;#4(0,360)\PSc@mment{End circTD}\fi\fi}
\ctr@ln@m\p@urcent
{\catcode`\%=12\gdef\p@urcent{
\ctr@ld@f\def\PSc@mment#1{\ifGRdebugm@de\immediate\write\fwf@g{\p@urcent\space#1}\fi}
\ctr@ln@m\acc@louv \ctr@ln@m\acc@lfer
{\catcode`\[=1\catcode`\{=12\gdef\acc@louv[{}}
{\catcode`\]=2\catcode`\}=12\gdef\acc@lfer{}]]
\ctr@ld@f\def\PSdict@{\ifUse@llipse%
    \immediate\write\fwf@g{/ellipsedict 9 dict def ellipsedict /mtrx matrix put}%
    \immediate\write\fwf@g{/ellipse \acc@louv ellipsedict begin}%
    \immediate\write\fwf@g{ /endangle exch def /startangle exch def}%
    \immediate\write\fwf@g{ /yrad exch def /xrad exch def}%
    \immediate\write\fwf@g{ /rotangle exch def /y exch def /x exch def}%
    \immediate\write\fwf@g{ /savematrix mtrx currentmatrix def}%
    \immediate\write\fwf@g{ x y translate rotangle rotate xrad yrad scale}%
    \immediate\write\fwf@g{ 0 0 1 startangle endangle arc}%
    \immediate\write\fwf@g{ savematrix setmatrix end\acc@lfer def}%
    \fi\PShe@der{EndProlog}}
\ctr@ld@f\def\Pssetc@rve#1=#2|{\keln@mun#1|%
    \def\n@mref{r}\ifx\l@debut\n@mref\update@ttr\D@FTroundness\Q@s@troundness{#2}\else
    \W@rnmesAttr{figset curve}{#1}\fi}
\ctr@ln@m\curv@roundness
\ctr@ld@f\def\Q@s@troundness#1{\edef\curv@roundness{#1}}
\ctr@ld@f\def\D@FTroundness{0.2} 
\ctr@ln@m\figdrawcurve
\ctr@ld@f\def\Q@curveDD[#1]{{\ifCUR@PS\ifGR@cri\PSc@mment{curveDD Points=#1}%
    \s@uvc@ntr@l\et@tpscurveDD%
    \iffillm@de\Psc@rveDD\curv@roundness[#1]%
    \f@gfill%
    \else\Psc@rveDD\curv@roundness[#1]\f@gstroke\fi%
    \PSc@mment{End curveDD}\resetc@ntr@l\et@tpscurveDD\fi\fi}}
\ctr@ld@f\def\Q@curveTD[#1]{{\ifCUR@PS\ifGR@cri%
    \PSc@mment{curveTD Points=#1}\s@uvc@ntr@l\et@tpscurveTD\let\c@lprojSP=\relax%
    \iffillm@de\Psc@rveTD\curv@roundness[#1]%
    \f@gfill%
    \else\Psc@rveTD\curv@roundness[#1]\f@gstroke\fi%
    \PSc@mment{End curveTD}\resetc@ntr@l\et@tpscurveTD\fi\fi}}
\ctr@ld@f\def\Psc@rveDD#1[#2]{%
    \def\list@num{#2}\extrairelepremi@r\Ak@\de\list@num%
    \extrairelepremi@r\Ai@\de\list@num\extrairelepremi@r\Aj@\de\list@num%
    \f@gnewpath\PSwrit@cmdS{\Ai@}{\c@mmoveto}{\fwf@g}{\X@un}{\Y@un}%
    \setc@ntr@l{2}\figvectPDD -1[\Ak@,\Aj@]%
    \@ecfor\Ak@:=\list@num\do{\figpttraDD-2:=\Ai@/#1,-1/\BdingB@xfalse%
       \PSwrit@cmdS{-2}{}{\fwf@g}{\X@de}{\Y@de}%
       \figvectPDD -1[\Ai@,\Ak@]\figpttraDD-2:=\Aj@/-#1,-1/%
       \PSwrit@cmdS{-2}{}{\fwf@g}{\X@tr}{\Y@tr}\BdingB@xtrue%
       \PSwrit@cmdS{\Aj@}{\c@mcurveto}{\fwf@g}{\X@qu}{\Y@qu}%
       \B@zierBB@x{1}{\Y@un}(\X@un,\X@de,\X@tr,\X@qu)%
       \B@zierBB@x{2}{\X@un}(\Y@un,\Y@de,\Y@tr,\Y@qu)%
       \edef\X@un{\X@qu}\edef\Y@un{\Y@qu}\edef\Ai@{\Aj@}\edef\Aj@{\Ak@}}}
\ctr@ld@f\def\Psc@rveTD#1[#2]{\ifnum\CUR@proj<\tw@\Psc@rvePPTD#1[#2]\else\Psc@rveCPTD#1[#2]\fi}
\ctr@ld@f\def\Psc@rvePPTD#1[#2]{\setc@ntr@l{2}%
    \def\list@num{#2}\extrairelepremi@r\Ak@\de\list@num\Figptpr@j-5:/\Ak@/%
    \extrairelepremi@r\Ai@\de\list@num\Figptpr@j-3:/\Ai@/%
    \extrairelepremi@r\Aj@\de\list@num\Figptpr@j-4:/\Aj@/%
    \f@gnewpath\PSwrit@cmdS{-3}{\c@mmoveto}{\fwf@g}{\X@un}{\Y@un}%
    \figvectPDD -1[-5,-4]%
    \@ecfor\Ak@:=\list@num\do{\Figptpr@j-5:/\Ak@/\figpttraDD-2:=-3/#1,-1/%
       \BdingB@xfalse\PSwrit@cmdS{-2}{}{\fwf@g}{\X@de}{\Y@de}%
       \figvectPDD -1[-3,-5]\figpttraDD-2:=-4/-#1,-1/%
       \PSwrit@cmdS{-2}{}{\fwf@g}{\X@tr}{\Y@tr}\BdingB@xtrue%
       \PSwrit@cmdS{-4}{\c@mcurveto}{\fwf@g}{\X@qu}{\Y@qu}%
       \B@zierBB@x{1}{\Y@un}(\X@un,\X@de,\X@tr,\X@qu)%
       \B@zierBB@x{2}{\X@un}(\Y@un,\Y@de,\Y@tr,\Y@qu)%
       \edef\X@un{\X@qu}\edef\Y@un{\Y@qu}\figptcopyDD-3:/-4/\figptcopyDD-4:/-5/}}
\ctr@ld@f\def\Psc@rveCPTD#1[#2]{\setc@ntr@l{2}%
    \def\list@num{#2}\extrairelepremi@r\Ak@\de\list@num%
    \extrairelepremi@r\Ai@\de\list@num\extrairelepremi@r\Aj@\de\list@num%
    \Figptpr@j-7:/\Ai@/%
    \f@gnewpath\PSwrit@cmd{-7}{\c@mmoveto}{\fwf@g}%
    \figvectPTD -9[\Ak@,\Aj@]%
    \@ecfor\Ak@:=\list@num\do{\figpttraTD-10:=\Ai@/#1,-9/%
       \figvectPTD -9[\Ai@,\Ak@]\figpttraTD-11:=\Aj@/-#1,-9/%
       \subB@zierTD\NBz@rcs[\Ai@,-10,-11,\Aj@]\edef\Ai@{\Aj@}\edef\Aj@{\Ak@}}}
\ctr@ld@f\def\figdrawend{\ifCUR@PS\ifGR@cri\immediate\closeout\fwf@g%
    \immediate\openout\fwf@g=\PSfilen@me\relax%
    \ifPDFm@ke\PSBdingB@x\else%
    \immediate\write\fwf@g{\p@urcent\string!PS-Adobe-2.0 EPSF-2.0}%
    \PShe@der{Creator\string: TeX (fig4tex.tex)}%
    \PShe@der{Title\string: \PSfilen@me}%
    \PShe@der{CreationDate\string: \the\day/\the\month/\the\year}%
    \PSBdingB@x%
    \PShe@der{EndComments}\PSdict@\fi%
    \immediate\write\fwf@g{\c@mgsave}%
    \openin\frf@g=\auxfilen@me\c@pypsfile\fwf@g\frf@g\closein\frf@g%
    \immediate\write\fwf@g{\c@mgrestore}%
    \PSc@mment{End of file.}\immediate\closeout\fwf@g%
    \immediate\openout\fwf@g=\auxfilen@me\immediate\closeout\fwf@g%
    \immediate\write16{File \PSfilen@me\space created.}\fi\fi\CUR@PSfalse\GR@critrue}
\ctr@ld@f\def\PShe@der#1{\immediate\write\fwf@g{\p@urcent\p@urcent#1}}
\ctr@ld@f\def\PSBdingB@x{{\v@lX=\ptT@ptps\c@@rdXmin\v@lY=\ptT@ptps\c@@rdYmin%
     \v@lXa=\ptT@ptps\c@@rdXmax\v@lYa=\ptT@ptps\c@@rdYmax%
     \PShe@der{BoundingBox\string: \repdecn@mb{\v@lX}\space\repdecn@mb{\v@lY}%
     \space\repdecn@mb{\v@lXa}\space\repdecn@mb{\v@lYa}}}}
\ctr@ld@f\def\figdrawfcconnect[#1]{{\ifCUR@PS\ifGR@cri\PSc@mment{fcconnect Points=#1}%
    \Q@s@tfillmode{no}\s@uvc@ntr@l\et@tpsfcconnect\resetc@ntr@l{2}%
    \fcc@nnect@[#1]\resetc@ntr@l\et@tpsfcconnect\PSc@mment{End fcconnect}\fi\fi}}
\ctr@ld@f\def\fcc@nnect@[#1]{\let\N@rm=\n@rmeucDD\def\list@num{#1}%
    \extrairelepremi@r\Ai@\de\list@num\edef\pr@m{\Ai@}\v@leur=\z@\p@rtent=\@ne\c@llgtot%
    \ifcase\fclin@typ@\edef\list@num{[\pr@m,#1,\Ai@}\expandafter\figdrawcurve\list@num]%
    \else\ifdim\fclin@r@d\p@>\z@\Pslin@conge[#1]\else\figdrawline[#1]\fi\fi%
    \v@leur=\@rrowp@s\v@leur\edef\list@num{#1,\Ai@,0}%
    \extrairelepremi@r\Ai@\de\list@num\mili@u=\epsil@n\c@llgpart%
    \advance\mili@u-\epsil@n\advance\mili@u-\delt@\advance\v@leur-\mili@u%
    \ifcase\fclin@typ@\invers@\mili@u\delt@%
    \ifnum\@rrowr@fpt>\z@\advance\delt@-\v@leur\v@leur=\delt@\fi%
    \v@leur=\repdecn@mb\v@leur\mili@u\edef\v@lt{\repdecn@mb\v@leur}%
    \extrairelepremi@r\Ak@\de\list@num%
    \figvectPDD-1[\pr@m,\Aj@]\figpttraDD-6:=\Ai@/\curv@roundness,-1/%
    \figvectPDD-1[\Ak@,\Ai@]\figpttraDD-7:=\Aj@/\curv@roundness,-1/%
    \delt@=\@rrowheadlength\p@\delt@=\C@AHANG\delt@\edef\R@dius{\repdecn@mb{\delt@}}%
    \ifcase\@rrowr@fpt%
    \FigptintercircB@zDD-8::\v@lt,\R@dius[\Ai@,-6,-7,\Aj@]\Q@arrowheadDD[-5,-8]\else%
    \FigptintercircB@zDD-8::\v@lt,\R@dius[\Aj@,-7,-6,\Ai@]\Q@arrowheadDD[-8,-5]\fi%
    \else\advance\delt@-\v@leur%
    \p@rtentiere{\p@rtent}{\delt@}\edef\C@efun{\the\p@rtent}%
    \p@rtentiere{\p@rtent}{\v@leur}\edef\C@efde{\the\p@rtent}%
    \figptbaryDD-5:[\Ai@,\Aj@;\C@efun,\C@efde]\ifcase\@rrowr@fpt%
    \delt@=\@rrowheadlength\unit@\delt@=\C@AHANG\delt@\edef\t@ille{\repdecn@mb{\delt@}}%
    \figvectPDD-2[\Ai@,\Aj@]\vecunit@{-2}{-2}\figpttraDD-5:=-5/\t@ille,-2/\fi%
    \Q@arrowheadDD[\Ai@,-5]\fi}
\ctr@ld@f\def\c@llgtot{\@ecfor\Aj@:=\list@num\do{\figvectP-1[\Ai@,\Aj@]\N@rm\delt@{-1}%
    \advance\v@leur\delt@\advance\p@rtent\@ne\edef\Ai@{\Aj@}}}
\ctr@ld@f\def\c@llgpart{\extrairelepremi@r\Aj@\de\list@num\figvectP-1[\Ai@,\Aj@]\N@rm\delt@{-1}%
    \advance\mili@u\delt@\ifdim\mili@u<\v@leur\edef\pr@m{\Ai@}\edef\Ai@{\Aj@}\c@llgpart\fi}
\ctr@ld@f\def\Pslin@conge[#1]{\ifnum\p@rtent>\tw@{\def\list@num{#1}%
    \extrairelepremi@r\Ai@\de\list@num\extrairelepremi@r\Aj@\de\list@num%
    \figptcopy-6:/\Ai@/\figvectP-3[\Ai@,\Aj@]\vecunit@{-3}{-3}\v@lmax=\result@t%
    \@ecfor\Ak@:=\list@num\do{\figvectP-4[\Aj@,\Ak@]\vecunit@{-4}{-4}%
    \minim@m\v@lmin\v@lmax\result@t\v@lmax=\result@t%
    \det@rm\delt@[-3,-4]\maxim@m\mili@u{\delt@}{-\delt@}\ifdim\mili@u>\Cepsil@n%
    \ifdim\delt@>\z@\figgetangleDD\Angl@[\Aj@,\Ak@,\Ai@]\else%
    \figgetangleDD\Angl@[\Aj@,\Ai@,\Ak@]\fi%
    \v@leur=\PI@deg\advance\v@leur-\Angl@\p@\divide\v@leur\tw@%
    \edef\Angl@{\repdecn@mb\v@leur}\c@ssin{\C@}{\S@}{\Angl@}\v@leur=\fclin@r@d\unit@%
    \v@leur=\S@\v@leur\mili@u=\C@\p@\invers@\mili@u\mili@u%
    \v@leur=\repdecn@mb{\mili@u}\v@leur%
    \minim@m\v@leur\v@leur\v@lmin\edef\t@ille{\repdecn@mb{\v@leur}}%
    \figpttra-5:=\Aj@/-\t@ille,-3/\figdrawline[-6,-5]\figpttra-6:=\Aj@/\t@ille,-4/%
    \figvectNVDD-3[-3]\figvectNVDD-8[-4]\inters@cDD-7:[-5,-3;-6,-8]%
    \ifdim\delt@>\z@\figdrawarccircP-7;\fclin@r@d[-5,-6]\else\figdrawarccircP-7;\fclin@r@d[-6,-5]\fi%
    \else\figdrawline[-6,\Aj@]\figptcopy-6:/\Aj@/\fi
    \edef\Ai@{\Aj@}\edef\Aj@{\Ak@}\figptcopy-3:/-4/}\figdrawline[-6,\Aj@]}\else\figdrawline[#1]\fi}
\ctr@ld@f\def\figdrawfcnode[#1]#2{{\ifCUR@PS\ifGR@cri\PSc@mment{fcnode Points=#1}%
    \s@uvc@ntr@l\et@tpsfcnode\resetc@ntr@l{2}%
    \def\t@xt@{#2}\ifx\t@xt@\empty\def\g@tt@xt{\setbox\Gb@x=\hbox{\Figg@tT{\p@int}}}%
    \else\def\g@tt@xt{\setbox\Gb@x=\hbox{#2}}\fi%
    \v@lmin=\h@rdfcXp@dd\advance\v@lmin\Xp@dd\unit@\multiply\v@lmin\tw@%
    \v@lmax=\h@rdfcYp@dd\advance\v@lmax\Yp@dd\unit@\multiply\v@lmax\tw@%
    \Figv@ctCreg-8(\unit@,-\unit@)\def\list@num{#1}%
    \delt@=\CUR@width bp\divide\delt@\tw@%
    \fcn@de\PSc@mment{End fcnode}\resetc@ntr@l\et@tpsfcnode\fi\fi}}
\ctr@ld@f\def\d@butn@de{\g@tt@xt\v@lX=\wd\Gb@x%
    \v@lY=\ht\Gb@x\advance\v@lY\dp\Gb@x\advance\v@lX\v@lmin\advance\v@lY\v@lmax}
\ctr@ld@f\def\fcn@deE{%
    \@ecfor\p@int:=\list@num\do{\d@butn@de\v@lX=\unssqrttw@\v@lX\v@lY=\unssqrttw@\v@lY%
    \ifdim\thickn@ss\p@>\z@
    \v@lXa=\v@lX\advance\v@lXa\delt@\v@lXa=\ptT@unit@\v@lXa\edef\XR@d{\repdecn@mb\v@lXa}%
    \v@lYa=\v@lY\advance\v@lYa\delt@\v@lYa=\ptT@unit@\v@lYa\edef\YR@d{\repdecn@mb\v@lYa}%
    \arct@n\v@leur(\v@lXa,\v@lYa)\v@leur=\rdT@deg\v@leur\edef\@nglde{\repdecn@mb\v@leur}%
    {\c@lptellDD-2::\p@int;\XR@d,\YR@d(\@nglde)}
    \advance\v@leur-\PI@deg\edef\@nglun{\repdecn@mb\v@leur}%
    {\c@lptellDD-3::\p@int;\XR@d,\YR@d(\@nglun)}%
    \figptstra-6=-3,-2,\p@int/\thickn@ss,-8/\Q@s@tfillmode{yes}%
    \Pss@tspecifSt{color=\DDV@thickcolor}%
    \figdrawline[-2,-3,-6,-5]\figdrawarcell-4;\XR@d,\YR@d(\@nglun,\@nglde,0)%
    \Psrest@reSt{color=\DDV@thickcolor}\fi
    \v@lX=\ptT@unit@\v@lX\v@lY=\ptT@unit@\v@lY%
    \edef\XR@d{\repdecn@mb\v@lX}\edef\YR@d{\repdecn@mb\v@lY}%
    \Q@s@tfillmode{yes}\Pss@tspecifSt{color=\fcbgc@lor}%
    \figdrawarcell\p@int;\XR@d,\YR@d(0,360,0)%
    \Q@s@tfillmode{no}\Psrest@reSt{color=\fcbgc@lor}\figdrawarcell\p@int;\XR@d,\YR@d(0,360,0)}}
\ctr@ld@f\def\fcn@deL{\delt@=\ptT@unit@\delt@\edef\t@ille{\repdecn@mb\delt@}%
    \@ecfor\p@int:=\list@num\do{\Figg@tXYa{\p@int}\d@butn@de%
    \ifdim\v@lX>\v@lY\itis@Ktrue\else\itis@Kfalse\fi%
    \advance\v@lXa-\v@lX\Figp@intreg-1:(\v@lXa,\v@lYa)%
    \advance\v@lXa\v@lX\advance\v@lYa-\v@lY\Figp@intreg-2:(\v@lXa,\v@lYa)%
    \advance\v@lXa\v@lX\advance\v@lYa\v@lY\Figp@intreg-3:(\v@lXa,\v@lYa)%
    \advance\v@lXa-\v@lX\advance\v@lYa\v@lY\Figp@intreg-4:(\v@lXa,\v@lYa)%
    \ifdim\thickn@ss\p@>\z@
    \Figg@tXYa{\p@int}\Q@s@tfillmode{yes}\Pss@tspecifSt{color=\DDV@thickcolor}%
    \c@lpt@xt{-1}{-4}\c@lpt@xt@\v@lXa\v@lYa\v@lX\v@lY\c@rre\delt@%
    \Figp@intregDD-9:(\v@lZ,\v@lYa)\Figp@intregDD-11:(\v@lZa,\v@lYa)%
    \c@lpt@xt{-4}{-3}\c@lpt@xt@\v@lYa\v@lXa\v@lY\v@lX\delt@\c@rre%
    \Figp@intregDD-12:(\v@lXa,\v@lZ)\Figp@intregDD-10:(\v@lXa,\v@lZa)%
    \ifitis@K\figptstra-7=-9,-10,-11/\thickn@ss,-8/\figdrawline[-9,-11,-5,-6,-7]\else%
    \figptstra-7=-10,-11,-12/\thickn@ss,-8/\figdrawline[-10,-12,-5,-6,-7]\fi%
    \Psrest@reSt{color=\DDV@thickcolor}\fi
    \Q@s@tfillmode{yes}\Pss@tspecifSt{color=\fcbgc@lor}\figdrawline[-1,-2,-3,-4]%
    \Q@s@tfillmode{no}\Psrest@reSt{color=\fcbgc@lor}\figdrawline[-1,-2,-3,-4,-1]}}
\ctr@ld@f\def\c@lpt@xt#1#2{\figvectN-7[#1,#2]\vecunit@{-7}{-7}\figpttra-5:=#1/\t@ille,-7/%
    \figvectP-7[#1,#2]\Figg@tXY{-7}\c@rre=\v@lX\delt@=\v@lY\Figg@tXY{-5}}
\ctr@ld@f\def\c@lpt@xt@#1#2#3#4#5#6{\v@lZ=#6\invers@{\v@lZ}{\v@lZ}\v@leur=\repdecn@mb{#5}\v@lZ%
    \v@lZ=#2\advance\v@lZ-#4\mili@u=\repdecn@mb{\v@leur}\v@lZ%
    \v@lZ=#3\advance\v@lZ\mili@u\v@lZa=-\v@lZ\advance\v@lZa\tw@#1}
\ctr@ld@f\def\fcn@deR{\@ecfor\p@int:=\list@num\do{\Figg@tXYa{\p@int}\d@butn@de%
    \advance\v@lXa-0.5\v@lX\advance\v@lYa-0.5\v@lY\Figp@intreg-1:(\v@lXa,\v@lYa)%
    \advance\v@lXa\v@lX\Figp@intreg-2:(\v@lXa,\v@lYa)%
    \advance\v@lYa\v@lY\Figp@intreg-3:(\v@lXa,\v@lYa)%
    \advance\v@lXa-\v@lX\Figp@intreg-4:(\v@lXa,\v@lYa)%
    \ifdim\thickn@ss\p@>\z@
    \Q@s@tfillmode{yes}\Pss@tspecifSt{color=\DDV@thickcolor}%
    \Figv@ctCreg-5(-\delt@,-\delt@)\figpttra-9:=-1/1,-5/%
    \Figv@ctCreg-5(\delt@,-\delt@)\figpttra-10:=-2/1,-5/%
    \Figv@ctCreg-5(\delt@,\delt@)\figpttra-11:=-3/1,-5/%
    \figptstra-7=-9,-10,-11/\thickn@ss,-8/\figdrawline[-9,-11,-5,-6,-7]%
    \Psrest@reSt{color=\DDV@thickcolor}\fi
    \Q@s@tfillmode{yes}\Pss@tspecifSt{color=\fcbgc@lor}\figdrawline[-1,-2,-3,-4]%
    \Q@s@tfillmode{no}\Psrest@reSt{color=\fcbgc@lor}\figdrawline[-1,-2,-3,-4,-1]}}
\ctr@ld@f\def\Pssetfl@wchart#1=#2|{\keln@mtr#1|%
    \def\n@mref{arr}\ifx\l@debut\n@mref\expandafter\keln@mtr\l@suite|%
     \def\n@mref{owp}\ifx\l@debut\n@mref\update@ttr\D@FTfcarrowposition\P@setfcarrowposition{#2}\else
     \def\n@mref{owr}\ifx\l@debut\n@mref\update@ttr\D@FTfcarrowrefpt\P@setfcarrowrefpt{#2}\else
     \W@rnmesAttr{figset flowchart}{#1}\fi\fi\else%
    \def\n@mref{bgc}\ifx\l@debut\n@mref\update@ttr\D@FTfcbgcolor\P@setfcbgcolor{#2}\else
    \def\n@mref{lin}\ifx\l@debut\n@mref\update@ttr\D@FTfcline\P@setfcline{#2}\else
    \def\n@mref{pad}\ifx\l@debut\n@mref\update@ttr\D@FTfcxpadding\P@setfcxpadding{#2}%
                                       \update@ttr\D@FTfcypadding\P@setfcypadding{#2}\else
    \def\n@mref{rad}\ifx\l@debut\n@mref\update@ttr\D@FTfcradius\P@setfcradius{#2}\else
    \def\n@mref{sha}\ifx\l@debut\n@mref\update@ttr\D@FTfcshape\P@setfcshape{#2}\else
    \def\n@mref{thi}\ifx\l@debut\n@mref\expandafter\keln@mtr\l@suite|%
     \def\n@mref{ckc}\ifx\l@debut\n@mref\update@ttr\D@FTref\P@setfcthickcolor{#2}\else
     \def\n@mref{ckn}\ifx\l@debut\n@mref\update@ttr\D@FTfcthickness\P@setfcthickness{#2}\else
     \W@rnmesAttr{figset flowchart}{#1}\fi\fi\else%
    \def\n@mref{xpa}\ifx\l@debut\n@mref\update@ttr\D@FTfcxpadding\P@setfcxpadding{#2}\else
    \def\n@mref{ypa}\ifx\l@debut\n@mref\update@ttr\D@FTfcypadding\P@setfcypadding{#2}\else
    \W@rnmesAttr{figset flowchart}{#1}\fi\fi\fi\fi\fi\fi\fi\fi\fi}
\ctr@ln@m\@rrowp@s
\ctr@ld@f\def\P@setfcarrowposition#1{\edef\@rrowp@s{#1}}
\ctr@ln@m\@rrowr@fpt
\ctr@ld@f\def\P@setfcarrowrefpt#1{\setfcr@fpt#1|}
\ctr@ld@f\def\setfcr@fpt#1#2|{\if#1e\def\@rrowr@fpt{1}\else\def\@rrowr@fpt{0}\fi}
\ctr@ln@m\fcbgc@lor
\ctr@ld@f\def\P@setfcbgcolor#1{\edef\fcbgc@lor{#1}}
\ctr@ln@m\fclin@typ@
\ctr@ld@f\def\P@setfcline#1{\setfccurv@#1|}
\ctr@ld@f\def\setfccurv@#1#2|{\if#1c\def\fclin@typ@{0}\else\def\fclin@typ@{1}\fi}
\ctr@ln@m\fclin@r@d
\ctr@ld@f\def\P@setfcradius#1{\edef\fclin@r@d{#1}}
\ctr@ln@m\fcn@de \ctr@ln@m\fcsh@pe
\ctr@ln@m\h@rdfcXp@dd \ctr@ln@m\h@rdfcYp@dd
\ctr@ld@f\def\P@setfcshape#1{\setfcshap@#1|}
\ctr@ld@f\def\setfcshap@#1#2|{%
    \if#1e\let\fcn@de=\fcn@deE\def\h@rdfcXp@dd{4pt}\def\h@rdfcYp@dd{4pt}%
     \edef\fcsh@pe{ellipse}\else%
    \if#1l\let\fcn@de=\fcn@deL\def\h@rdfcXp@dd{4pt}\def\h@rdfcYp@dd{4pt}%
     \edef\fcsh@pe{lozenge}\else%
          \let\fcn@de=\fcn@deR\def\h@rdfcXp@dd{6pt}\def\h@rdfcYp@dd{6pt}%
     \edef\fcsh@pe{rectangle}\fi\fi}
\ctr@ln@m\DDV@thickcolor
\ctr@ld@f\def\P@setfcthickcolor#1{\edef\DDV@thickcolor{#1}}
\ctr@ln@m\thickn@ss
\ctr@ld@f\def\P@setfcthickness#1{\edef\thickn@ss{#1}}
\ctr@ln@m\Xp@dd
\ctr@ld@f\def\P@setfcxpadding#1{\edef\Xp@dd{#1}}
\ctr@ln@m\Yp@dd
\ctr@ld@f\def\P@setfcypadding#1{\edef\Yp@dd{#1}}
\ctr@ld@f\def\figdrawline[#1]{{\ifCUR@PS\ifGR@cri\PSc@mment{line Points=#1}%
    \let\figdrawlign@=\Pslign@P\Pslin@{#1}\PSc@mment{End line}\fi\fi}}
\ctr@ld@f\def\figdrawlineF#1{{\ifCUR@PS\ifGR@cri\PSc@mment{lineF Filename=#1}%
    \let\figdrawlign@=\Pslign@F\Pslin@{#1}\PSc@mment{End lineF}\fi\fi}}
\ctr@ld@f\def\figdrawlineC(#1){{\ifCUR@PS\ifGR@cri\PSc@mment{lineC}%
    \let\figdrawlign@=\Pslign@C\Pslin@{#1}\PSc@mment{End lineC}\fi\fi}}
\ctr@ld@f\def\Pslin@#1{\iffillm@de\figdrawlign@{#1}%
    \f@gfill%
    \else\figdrawlign@{#1}\ifx\derp@int\premp@int%
    \f@gclosestroke%
    \else\f@gstroke\fi\fi}
\ctr@ld@f\def\Pslign@P#1{\def\list@num{#1}\extrairelepremi@r\p@int\de\list@num%
    \edef\premp@int{\p@int}\f@gnewpath%
    \PSwrit@cmd{\p@int}{\c@mmoveto}{\fwf@g}%
    \@ecfor\p@int:=\list@num\do{\PSwrit@cmd{\p@int}{\c@mlineto}{\fwf@g}%
    \edef\derp@int{\p@int}}}
\ctr@ld@f\def\Pslign@F#1{\s@uvc@ntr@l\et@tPslign@F\setc@ntr@l{2}\openin\frf@g=#1\relax%
    \ifeof\frf@g\message{*** File #1 not found !}\end\else%
    \read\frf@g to\tr@c\edef\premp@int{\tr@c}\expandafter\extr@ctCF\tr@c:%
    \f@gnewpath\PSwrit@cmd{-1}{\c@mmoveto}{\fwf@g}%
    \loop\read\frf@g to\tr@c\ifeof\frf@g\mored@tafalse\else\mored@tatrue\fi%
    \ifmored@ta\expandafter\extr@ctCF\tr@c:\PSwrit@cmd{-1}{\c@mlineto}{\fwf@g}%
    \edef\derp@int{\tr@c}\repeat\fi\closein\frf@g\resetc@ntr@l\et@tPslign@F}
\ctr@ln@m\extr@ctCF
\ctr@ld@f\def\extr@ctCFDD#1 #2:{\v@lX=#1\unit@\v@lY=#2\unit@\Figp@intregDD-1:(\v@lX,\v@lY)}
\ctr@ld@f\def\extr@ctCFTD#1 #2 #3:{\v@lX=#1\unit@\v@lY=#2\unit@\v@lZ=#3\unit@%
    \Figp@intregTD-1:(\v@lX,\v@lY,\v@lZ)}
\ctr@ld@f\def\Pslign@C#1{\s@uvc@ntr@l\et@tPslign@C\setc@ntr@l{2}%
    \def\list@num{#1}\extrairelepremi@r\p@int\de\list@num%
    \edef\premp@int{\p@int}\f@gnewpath%
    \expandafter\Pslign@C@\p@int:\PSwrit@cmd{-1}{\c@mmoveto}{\fwf@g}%
    \@ecfor\p@int:=\list@num\do{\expandafter\Pslign@C@\p@int:%
    \PSwrit@cmd{-1}{\c@mlineto}{\fwf@g}\edef\derp@int{\p@int}}%
    \resetc@ntr@l\et@tPslign@C}
\ctr@ld@f\def\Pslign@C@#1 #2:{{\def\t@xt@{#1}\ifx\t@xt@\empty\Pslign@C@#2:
    \else\extr@ctCF#1 #2:\fi}}
\ctr@ld@f\def\Pssetm@sh#1=#2|{\keln@mde#1|%
    \def\n@mref{co}\ifx\l@debut\n@mref\update@ttr\D@FTref\P@setmeshcolor{#2}\else
    \def\n@mref{da}\ifx\l@debut\n@mref\update@ttr\D@FTref\P@setmeshdash{#2}\else
    \def\n@mref{di}\ifx\l@debut\n@mref\update@ttr\D@FTmeshdiag\Q@s@tmeshdiag{#2}\else
    \def\n@mref{wi}\ifx\l@debut\n@mref\update@ttr\D@FTref\P@setmeshwidth{#2}\else
    \W@rnmesAttr{figset mesh}{#1}\fi\fi\fi\fi}
\ctr@ln@m\c@ntrolmesh
\ctr@ld@f\def\Q@s@tmeshdiag#1{\edef\c@ntrolmesh{#1}}
\ctr@ld@f\def\D@FTmeshdiag{0}    
\ctr@ln@m\DDV@meshcolor
\ctr@ld@f\def\P@setmeshcolor#1{\edef\DDV@meshcolor{#1}}
\ctr@ln@m\DDV@meshdash
\ctr@ld@f\def\P@setmeshdash#1{\edef\DDV@meshdash{#1}}
\ctr@ln@m\DDV@meshwidth
\ctr@ld@f\def\P@setmeshwidth#1{\edef\DDV@meshwidth{#1}}
\ctr@ld@f\def\figdrawmesh#1,#2[#3,#4,#5,#6]{{\ifCUR@PS\ifGR@cri%
    \PSc@mment{mesh N1=#1, N2=#2, Quadrangle=[#3,#4,#5,#6]}\s@uvc@ntr@l\et@tpsmesh%
    \Pss@tspecifSt{color=\DDV@meshcolor,dash=\DDV@meshdash,width=\DDV@meshwidth}%
    \setc@ntr@l{2}%
    \ifnum#1>\@ne\Psmeshp@rt#1[#3,#4,#5,#6]\fi%
    \ifnum#2>\@ne\Psmeshp@rt#2[#4,#5,#6,#3]\fi%
    \ifnum\c@ntrolmesh>\z@\Psmeshdi@g#1,#2[#3,#4,#5,#6]\fi%
    \ifnum\c@ntrolmesh<\z@\Psmeshdi@g#2,#1[#4,#5,#6,#3]\fi%
    \Psrest@reSt{color=\DDV@meshcolor,dash=\DDV@meshdash,width=\DDV@meshwidth}%
    \figdrawline[#3,#4,#5,#6,#3]\PSc@mment{End mesh}\resetc@ntr@l\et@tpsmesh\fi\fi}}
\ctr@ld@f\def\Psmeshp@rt#1[#2,#3,#4,#5]{{\l@mbd@un=\@ne\l@mbd@de=#1\loop%
    \ifnum\l@mbd@un<#1\advance\l@mbd@de\m@ne\figptbary-1:[#2,#3;\l@mbd@de,\l@mbd@un]%
    \figptbary-2:[#5,#4;\l@mbd@de,\l@mbd@un]\figdrawline[-1,-2]\advance\l@mbd@un\@ne\repeat}}
\ctr@ld@f\def\Psmeshdi@g#1,#2[#3,#4,#5,#6]{\figptcopy-2:/#3/\figptcopy-3:/#6/%
    \l@mbd@un=\z@\l@mbd@de=#1\loop\ifnum\l@mbd@un<#1%
    \advance\l@mbd@un\@ne\advance\l@mbd@de\m@ne\figptcopy-1:/-2/\figptcopy-4:/-3/%
    \figptbary-2:[#3,#4;\l@mbd@de,\l@mbd@un]%
    \figptbary-3:[#6,#5;\l@mbd@de,\l@mbd@un]\Psmeshdi@gp@rt#2[-1,-2,-3,-4]\repeat}
\ctr@ld@f\def\Psmeshdi@gp@rt#1[#2,#3,#4,#5]{{\l@mbd@un=\z@\l@mbd@de=#1\loop%
    \ifnum\l@mbd@un<#1\figptbary-5:[#2,#5;\l@mbd@de,\l@mbd@un]%
    \advance\l@mbd@de\m@ne\advance\l@mbd@un\@ne%
    \figptbary-6:[#3,#4;\l@mbd@de,\l@mbd@un]\figdrawline[-5,-6]\repeat}}
\ctr@ln@m\figdrawnormal
\ctr@ld@f\def\Q@normalDD#1,#2[#3,#4]{{\ifCUR@PS\ifGR@cri%
    \PSc@mment{normal Length=#1, Lambda=#2 [Pt1,Pt2]=[#3,#4]}%
    \s@uvc@ntr@l\et@tpsnormal\resetc@ntr@l{2}\figptendnormal-6::#1,#2[#3,#4]%
    \figptcopyDD-5:/-1/\figdrawarrow[-5,-6]%
    \PSc@mment{End normal}\resetc@ntr@l\et@tpsnormal\fi\fi}}
\ctr@ld@f\def\figreset#1{\trtlis@rg{#1}{\Psreset@}}
\ctr@ld@f\def\Psreset@#1|{\def\t@xt@{#1}\ifx\t@xt@\empty\P@resetg@n
    \else\keln@mde#1|%
    \def\n@mref{al}\ifx\l@debut\n@mref%
        \figset altitude(blcolor=default,bldash=default,blwidth=default,%
        sqcolor=default,sqdash=default,sqwidth=default)\else
    \def\n@mref{ar}\ifx\l@debut\n@mref\figresetarrowhead\else
    \def\n@mref{cu}\ifx\l@debut\n@mref\figset curve(roundness=\D@FTroundness)\else
    \def\n@mref{ge}\ifx\l@debut\n@mref\P@resetg@n\else
    \def\n@mref{fl}\ifx\l@debut\n@mref%
        \figset flowchart(arrowp=\D@FTfcarrowposition,arrowr=\D@FTfcarrowrefpt,%
	bgcolor=\D@FTfcbgcolor,line=\D@FTfcline,radius=\D@FTfcradius,%
	shape=\D@FTfcshape,thickcolor=default,thickness=\D@FTfcthickness,%
	xpadd=\D@FTfcxpadding,ypadd=\D@FTfcypadding)\else
    \def\n@mref{me}\ifx\l@debut\n@mref\figset mesh(diag=\D@FTmeshdiag,%
        color=default,dash=default,width=default)\else
    \def\n@mref{tr}\ifx\l@debut\n@mref%
        \figset trimesh(color=default,dash=default,width=default)\else
    \W@rnmeskwd{figreset}{#1}\fi\fi\fi\fi\fi\fi\fi\fi}
\ctr@ld@f\def\P@resetg@n{\figset (color=\D@FTcolor,dash=\D@FTdash,fill=\D@FTfill,%
    join=\D@FTjoin,width=\D@FTwidth)}
\ctr@ld@f\def\figset#1(#2){\def\t@xt@{#1}\ifx\t@xt@\empty\trtlis@rg{#2}{\Pssetg@n}
    \else\keln@mde#1|%
    \def\n@mref{al}\ifx\l@debut\n@mref\trtlis@rg{#2}{\Psset@lti}\else
    \def\n@mref{ar}\ifx\l@debut\n@mref\trtlis@rg{#2}{\Psset@rrowhe@d}\else
    \def\n@mref{cu}\ifx\l@debut\n@mref\trtlis@rg{#2}{\Pssetc@rve}\else
    \def\n@mref{fl}\ifx\l@debut\n@mref\trtlis@rg{#2}{\Pssetfl@wchart}\else
    \def\n@mref{ge}\ifx\l@debut\n@mref\trtlis@rg{#2}{\Pssetg@n}\else
    \def\n@mref{me}\ifx\l@debut\n@mref\trtlis@rg{#2}{\Pssetm@sh}\else
    \def\n@mref{pr}\ifx\l@debut\n@mref\ifCUR@PS\W@rnmesIgn{figset proj(...)}%
     \else\trtlis@rg{#2}{\Figsetpr@j}\fi\else
    \def\n@mref{tr}\ifx\l@debut\n@mref\trtlis@rg{#2}{\Pssettrim@sh}\else
    \def\n@mref{wr}\ifx\l@debut\n@mref\let\M@cro=\Figsetwr@te\trtlis@rgtok{#2,|}\else
    \W@rnmeskwd{figset}{#1}\fi\fi\fi\fi\fi\fi\fi\fi\fi\fi\ignorespaces}
\ctr@ld@f\def\figsetdefault#1(#2){\ifCUR@PS\W@rnmesIgn{figsetdefault}\else%
    \def\t@xt@{#1}\ifx\t@xt@\empty\trtlis@rg{#2}{\Pssd@g@n}\else\keln@mun#1|
    \def\n@mref{a}\ifx\l@debut\n@mref\trtlis@rg{#2}{\Pssd@@rrowhe@d}\else
    \def\n@mref{c}\ifx\l@debut\n@mref\trtlis@rg{#2}{\Pssd@c@rve}\else
    \def\n@mref{g}\ifx\l@debut\n@mref\trtlis@rg{#2}{\Pssd@g@n}\else
    \def\n@mref{f}\ifx\l@debut\n@mref\trtlis@rg{#2}{\Pssd@fl@wchart}\else
    \def\n@mref{m}\ifx\l@debut\n@mref\trtlis@rg{#2}{\Pssd@m@sh}\else
    \W@rnmeskwd{figsetdefault}{#1}\fi\fi\fi\fi\fi\fi\initpss@ttings\fi}
\ctr@ld@f\def\Pssd@g@n#1=#2|{\keln@mun#1|%
    \def\n@mref{c}\ifx\l@debut\n@mref\edef\D@FTcolor{#2}\else
    \def\n@mref{d}\ifx\l@debut\n@mref\edef\D@FTdash{#2}\else
    \def\n@mref{f}\ifx\l@debut\n@mref\edef\D@FTfill{#2}\else
    \def\n@mref{j}\ifx\l@debut\n@mref\edef\D@FTjoin{#2}\else
    \def\n@mref{u}\ifx\l@debut\n@mref\edef\D@FTupdate{#2}\Q@s@tupdate{#2}\else
    \def\n@mref{w}\ifx\l@debut\n@mref\edef\D@FTwidth{#2}\else
    \W@rnmesAttr{figsetdefault}{#1}\fi\fi\fi\fi\fi\fi}
\ctr@ld@f\def\Pssd@@rrowhe@d#1=#2|{\keln@mun#1|%
    \def\n@mref{a}\ifx\l@debut\n@mref\edef\D@FTarrowheadangle{#2}\else
    \def\n@mref{f}\ifx\l@debut\n@mref\edef\D@FTarrowheadfill{#2}\else
    \def\n@mref{l}\ifx\l@debut\n@mref\y@tiunit{#2}\ifunitpr@sent%
     \edef\D@FTh@rdahlength{#2}\else\edef\D@FTh@rdahlength{#2pt}%
     \message{*** \BS@ figsetdefault (..., #1=#2, ...) : unit is missing, pt is assumed.}%
     \fi\else
    \def\n@mref{o}\ifx\l@debut\n@mref\edef\D@FTarrowheadout{#2}\else
    \def\n@mref{r}\ifx\l@debut\n@mref\edef\D@FTarrowheadratio{#2}\else
    \W@rnmesAttr{figsetdefault arrowhead}{#1}\fi\fi\fi\fi\fi}
\ctr@ld@f\def\Pssd@c@rve#1=#2|{\keln@mun#1|%
    \def\n@mref{r}\ifx\l@debut\n@mref\edef\D@FTroundness{#2}\else%
    \W@rnmesAttr{figsetdefault curve}{#1}\fi}
\ctr@ld@f\def\Pssd@fl@wchart#1=#2|{\keln@mtr#1|%
    \def\n@mref{arr}\ifx\l@debut\n@mref\expandafter\keln@mtr\l@suite|%
     \def\n@mref{owp}\ifx\l@debut\n@mref\edef\D@FTfcarrowposition{#2}\else
     \def\n@mref{owr}\ifx\l@debut\n@mref\edef\D@FTfcarrowrefpt{#2}\else
                     \W@rnmesAttr{figsetdefault flowchart}{#1}\fi\fi\else%
    \def\n@mref{bgc}\ifx\l@debut\n@mref\edef\D@FTfcbgcolor{#2}\else
    \def\n@mref{lin}\ifx\l@debut\n@mref\edef\D@FTfcline{#2}\else
    \def\n@mref{pad}\ifx\l@debut\n@mref\edef\D@FTfcxpadding{#2}%
                    \edef\D@FTfcypadding{#2}\else
    \def\n@mref{rad}\ifx\l@debut\n@mref\edef\D@FTfcradius{#2}\else
    \def\n@mref{sha}\ifx\l@debut\n@mref\edef\D@FTfcshape{#2}\else
    \def\n@mref{thi}\ifx\l@debut\n@mref\expandafter\keln@mtr\l@suite|%
     \def\n@mref{ckn}\ifx\l@debut\n@mref\edef\D@FTfcthickness{#2}\else
                     \W@rnmesAttr{figsetdefault flowchart}{#1}\fi\else%
    \def\n@mref{xpa}\ifx\l@debut\n@mref\edef\D@FTfcxpadding{#2}\else
    \def\n@mref{ypa}\ifx\l@debut\n@mref\edef\D@FTfcypadding{#2}\else
    \W@rnmesAttr{figsetdefault flowchart}{#1}\fi\fi\fi\fi\fi\fi\fi\fi\fi}
\ctr@ld@f\def\D@FTfcarrowposition{0.5}
\ctr@ld@f\def\D@FTfcarrowrefpt{start}
\ctr@ld@f\def\D@FTfcbgcolor{1}
\ctr@ld@f\def\D@FTfcline{polygon}
\ctr@ld@f\def\D@FTfcradius{0}
\ctr@ld@f\def\D@FTfcshape{rectangle}
\ctr@ld@f\def\D@FTfcthickness{0}
\ctr@ld@f\def\D@FTfcxpadding{0}
\ctr@ld@f\def\D@FTfcypadding{0}
\ctr@ld@f\def\Pssd@m@sh#1=#2|{\keln@mun#1|%
    \def\n@mref{d}\ifx\l@debut\n@mref\edef\D@FTmeshdiag{#2}\else%
    \W@rnmesAttr{figsetdefault mesh}{#1}\fi}
\ctr@ln@w{newif}\iffillm@de
\ctr@ld@f\def\Q@s@tfillmode#1{\expandafter\setfillm@de#1:}
\ctr@ld@f\def\setfillm@de#1#2:{\if#1n\fillm@defalse\else\fillm@detrue\fi}
\ctr@ld@f\def\D@FTfill{no}     
\ctr@ln@w{newif}\ifGRupdatem@de
\ctr@ld@f\def\Q@s@tupdate#1{\ifCUR@PS\W@rnmesIgn{figset (update=...)}%
    \else\expandafter\setupd@te#1:\fi}
\ctr@ld@f\def\setupd@te#1#2:{\if#1n\GRupdatem@defalse\else\GRupdatem@detrue\fi}
\ctr@ld@f\def\D@FTupdate{no}     
\ctr@ln@m\CUR@color \ctr@ln@m\CUR@colorc@md
\ctr@ld@f\def\s@uvcolor#1{\edef#1{\CUR@color}}
\ctr@ld@f\def\D@FTcolor{0}       
\ctr@ld@f\def\Pssetc@lor#1{\ifGR@cri\result@tent=\@ne\expandafter\c@lnbV@l#1 :%
    \def\CUR@color{}\def\CUR@colorc@md{}%
    \ifcase\result@tent\or\Q@s@tgray{#1}\or\or\Q@s@trgb{#1}\or\Q@s@tcmyk{#1}\fi\fi}
\ctr@ln@m\CUR@colorc@mdStroke
\ctr@ld@f\def\Q@s@tcmyk#1{\ifGR@cri\def\CUR@color{#1}\def\CUR@colorc@md{\c@msetcmykcolor}%
    \def\CUR@colorc@mdStroke{\c@msetcmykcolorStroke}%
    \ifCUR@PS\PSc@mment{setcmyk Color=#1}\us@primarC@lor\fi\fi}
\ctr@ld@f\def\Q@s@trgb#1{\ifGR@cri\def\CUR@color{#1}\def\CUR@colorc@md{\c@msetrgbcolor}%
    \def\CUR@colorc@mdStroke{\c@msetrgbcolorStroke}%
    \ifCUR@PS\PSc@mment{setrgb Color=#1}\us@primarC@lor\fi\fi}
\ctr@ld@f\def\Q@s@tgray#1{\ifGR@cri\def\CUR@color{#1}\def\CUR@colorc@md{\c@msetgray}%
    \def\CUR@colorc@mdStroke{\c@msetgrayStroke}%
    \ifCUR@PS\PSc@mment{setgray Gray level=#1}\us@primarC@lor\fi\fi}
\ctr@ln@m\fillc@md
\ctr@ld@f\def\us@primarC@lor{\immediate\write\fwf@g{\d@fprimarC@lor}%
    \let\fillc@md=\prfillc@md}
\ctr@ld@f\def\prfillc@md{\d@fprimarC@lor\space\c@mfill}
\ctr@ld@f\def\c@lnbV@l#1 #2:{\def\t@xt@{#1}\relax\ifx\t@xt@\empty\c@lnbV@l#2:
    \else\c@lnbV@l@#1 #2:\fi}
\ctr@ld@f\def\c@lnbV@l@#1 #2:{\def\t@xt@{#2}\ifx\t@xt@\empty%
    \def\t@xt@{#1}\ifx\t@xt@\empty\advance\result@tent\m@ne\fi
    \else\advance\result@tent\@ne\c@lnbV@l@#2:\fi}
\ctr@ld@f\def\Blackcmyk{0 0 0 1}
\ctr@ld@f\def\Whitecmyk{0 0 0 0}
\ctr@ld@f\def\Cyancmyk{1 0 0 0}
\ctr@ld@f\def\Magentacmyk{0 1 0 0}
\ctr@ld@f\def\Yellowcmyk{0 0 1 0}
\ctr@ld@f\def\Redcmyk{0 1 1 0}
\ctr@ld@f\def\Greencmyk{1 0 1 0}
\ctr@ld@f\def\Bluecmyk{1 1 0 0}
\ctr@ld@f\def\Graycmyk{0 0 0 0.50}
\ctr@ld@f\def\BrickRedcmyk{0 0.89 0.94 0.28} 
\ctr@ld@f\def\Browncmyk{0 0.81 1 0.60} 
\ctr@ld@f\def\ForestGreencmyk{0.91 0 0.88 0.12} 
\ctr@ld@f\def\Goldenrodcmyk{ 0 0.10 0.84 0} 
\ctr@ld@f\def\Marooncmyk{0 0.87 0.68 0.32} 
\ctr@ld@f\def\Orangecmyk{0 0.61 0.87 0} 
\ctr@ld@f\def\Purplecmyk{0.45 0.86 0 0} 
\ctr@ld@f\def\RoyalBluecmyk{1. 0.50 0 0} 
\ctr@ld@f\def\Violetcmyk{0.79 0.88 0 0} 
\ctr@ld@f\def\Blackrgb{0 0 0}
\ctr@ld@f\def\Whitergb{1 1 1}
\ctr@ld@f\def\Redrgb{1 0 0}
\ctr@ld@f\def\Greenrgb{0 1 0}
\ctr@ld@f\def\Bluergb{0 0 1}
\ctr@ld@f\def\Cyanrgb{0 1 1}
\ctr@ld@f\def\Magentargb{1 0 1}
\ctr@ld@f\def\Yellowrgb{1 1 0}
\ctr@ld@f\def\Grayrgb{0.5 0.5 0.5}
\ctr@ld@f\def\Chocolatergb{0.824 0.412 0.118}
\ctr@ld@f\def\DarkGoldenrodrgb{0.722 0.525 0.043}
\ctr@ld@f\def\DarkOrangergb{1 0.549 0}
\ctr@ld@f\def\Firebrickrgb{0.698 0.133 0.133}
\ctr@ld@f\def\ForestGreenrgb{0.133 0.545 0.133}
\ctr@ld@f\def\Goldrgb{1 0.843 0}
\ctr@ld@f\def\HotPinkrgb{1 0.412 0.706}
\ctr@ld@f\def\Maroonrgb{0.690 0.188 0.376}
\ctr@ld@f\def\Pinkrgb{1 0.753 0.796}
\ctr@ld@f\def\RoyalBluergb{0.255 0.412 0.882}
\ctr@ld@f\def\Pssetg@n#1=#2|{\keln@mun#1|%
    \def\n@mref{c}\ifx\l@debut\n@mref\update@ttr\D@FTcolor\Pssetc@lor{#2}\else
    \def\n@mref{d}\ifx\l@debut\n@mref\update@ttr\D@FTdash\Q@s@tdash{#2}\else
    \def\n@mref{f}\ifx\l@debut\n@mref\update@ttr\D@FTfill\Q@s@tfillmode{#2}\else
    \def\n@mref{j}\ifx\l@debut\n@mref\update@ttr\D@FTjoin\Q@s@tjoin{#2}\else
    \def\n@mref{u}\ifx\l@debut\n@mref\update@ttr\D@FTupdate\Q@s@tupdate{#2}\else
    \def\n@mref{w}\ifx\l@debut\n@mref\update@ttr\D@FTwidth\Q@s@twidth{#2}\else
    \W@rnmesAttr{figset}{#1}\fi\fi\fi\fi\fi\fi}
\ctr@ln@m\CUR@dash
\ctr@ld@f\def\s@uvdash#1{\edef#1{\CUR@dash}}
\ctr@ld@f\def\D@FTdash{1}        
\ctr@ld@f\def\Q@s@tdash#1{\ifGR@cri\edef\CUR@dash{#1}\ifCUR@PS\expandafter\Pssetd@sh#1 :\fi\fi}
\ctr@ld@f\def\Pssetd@shI#1{\PSc@mment{setdash Index=#1}\ifcase#1%
    \or\immediate\write\fwf@g{[] 0 \c@msetdash}
    \or\immediate\write\fwf@g{[6 2] 0 \c@msetdash}
    \or\immediate\write\fwf@g{[4 2] 0 \c@msetdash}
    \or\immediate\write\fwf@g{[2 2] 0 \c@msetdash}
    \or\immediate\write\fwf@g{[1 2] 0 \c@msetdash}
    \or\immediate\write\fwf@g{[2 4] 0 \c@msetdash}
    \or\immediate\write\fwf@g{[3 5] 0 \c@msetdash}
    \or\immediate\write\fwf@g{[3 3] 0 \c@msetdash}
    \or\immediate\write\fwf@g{[3 5 1 5] 0 \c@msetdash}
    \or\immediate\write\fwf@g{[6 4 2 4] 0 \c@msetdash}
    \fi}
\ctr@ld@f\def\Pssetd@sh#1 #2:{{\def\t@xt@{#1}\ifx\t@xt@\empty\Pssetd@sh#2:
    \else\def\t@xt@{#2}\ifx\t@xt@\empty\Pssetd@shI{#1}\else\s@mme=\@ne\def\debutp@t{#1}%
    \an@lysd@sh#2:\ifodd\s@mme\edef\debutp@t{\debutp@t\space\finp@t}\def\finp@t{0}\fi%
    \PSc@mment{setdash Pattern=#1 #2}%
    \immediate\write\fwf@g{[\debutp@t] \finp@t\space\c@msetdash}\fi\fi}}
\ctr@ld@f\def\an@lysd@sh#1 #2:{\def\t@xt@{#2}\ifx\t@xt@\empty\def\finp@t{#1}\else%
    \edef\debutp@t{\debutp@t\space#1}\advance\s@mme\@ne\an@lysd@sh#2:\fi}
\ctr@ln@m\CUR@width
\ctr@ld@f\def\s@uvwidth#1{\edef#1{\CUR@width}}
\ctr@ld@f\def\D@FTwidth{0.4}     
\ctr@ld@f\def\Q@s@twidth#1{\ifGR@cri\edef\CUR@width{#1}\ifCUR@PS%
    \PSc@mment{setwidth Width=#1}\immediate\write\fwf@g{#1 \c@msetlinewidth}\fi\fi}
\ctr@ln@m\CUR@join
\ctr@ld@f\def\s@uvjoin#1{\edef#1{\CUR@join}}
\ctr@ld@f\def\D@FTjoin{miter}   
\ctr@ld@f\def\Q@s@tjoin#1{\ifGR@cri\edef\CUR@join{#1}\ifCUR@PS\expandafter\Pssetj@in#1:\fi\fi}
\ctr@ld@f\def\Pssetj@in#1#2:{\PSc@mment{setjoin join=#1}%
    \if#1r\def\t@xt@{1}\else\if#1b\def\t@xt@{2}\else\def\t@xt@{0}\fi\fi%
    \immediate\write\fwf@g{\t@xt@\space\c@msetlinejoin}}
\ctr@ld@f\def\Pss@tspecifSt#1{\trtlis@rg{#1}{\Pss@tspecifSt@}}
\ctr@ld@f\def\Pss@tspecifSt@#1=#2|{\keln@mun#1|%
    \def\n@mref{c}\ifx\l@debut\n@mref\def\n@mref{#2}\ifx\n@mref\D@FTref\else%
     \s@uvcolor{\typ@color}\Pssetc@lor{#2}\fi\else
    \def\n@mref{d}\ifx\l@debut\n@mref\def\n@mref{#2}\ifx\n@mref\D@FTref\else%
     \s@uvdash{\typ@dash}\Q@s@tdash{#2}\fi\else
    \def\n@mref{j}\ifx\l@debut\n@mref\def\n@mref{#2}\ifx\n@mref\D@FTref\else%
     \s@uvjoin{\typ@join}\Q@s@tjoin{#2}\fi\else
    \def\n@mref{w}\ifx\l@debut\n@mref\def\n@mref{#2}\ifx\n@mref\D@FTref\else%
     \s@uvwidth{\typ@width}\Q@s@twidth{#2}\fi\else
    \W@rnmeskwd{Pss@tspecifSt}{#1}\fi\fi\fi\fi}
\ctr@ld@f\def\Psrest@reSt#1{\trtlis@rg{#1}{\Psrest@reSt@}}
\ctr@ld@f\def\Psrest@reSt@#1=#2|{\keln@mun#1|%
    \def\n@mref{c}\ifx\l@debut\n@mref\def\n@mref{#2}\ifx\n@mref\D@FTref\else%
     \Pssetc@lor{\typ@color}\fi\else
    \def\n@mref{d}\ifx\l@debut\n@mref\def\n@mref{#2}\ifx\n@mref\D@FTref\else%
     \Q@s@tdash{\typ@dash}\fi\else
    \def\n@mref{j}\ifx\l@debut\n@mref\def\n@mref{#2}\ifx\n@mref\D@FTref\else%
     \Q@s@tjoin{\typ@join}\fi\else
    \def\n@mref{w}\ifx\l@debut\n@mref\def\n@mref{#2}\ifx\n@mref\D@FTref\else%
     \Q@s@twidth{\typ@width}\fi\else
    \W@rnmeskwd{Psrest@reSt}{#1}\fi\fi\fi\fi}
\ctr@ld@f\def\Pssettrim@sh#1=#2|{\keln@mde#1|%
    \def\n@mref{co}\ifx\l@debut\n@mref\update@ttr\D@FTref\P@settmeshcolor{#2}\else
    \def\n@mref{da}\ifx\l@debut\n@mref\update@ttr\D@FTref\P@settmeshdash{#2}\else
    \def\n@mref{wi}\ifx\l@debut\n@mref\update@ttr\D@FTref\P@settmeshwidth{#2}\else
    \W@rnmesAttr{figset trimesh}{#1}\fi\fi\fi}
\ctr@ln@m\DDV@tmeshcolor
\ctr@ld@f\def\P@settmeshcolor#1{\edef\DDV@tmeshcolor{#1}}
\ctr@ln@m\DDV@tmeshdash
\ctr@ld@f\def\P@settmeshdash#1{\edef\DDV@tmeshdash{#1}}
\ctr@ln@m\DDV@tmeshwidth
\ctr@ld@f\def\P@settmeshwidth#1{\edef\DDV@tmeshwidth{#1}}
\ctr@ld@f\def\figdrawtrimesh#1[#2,#3,#4]{{\ifCUR@PS\ifGR@cri%
    \PSc@mment{trimesh Type=#1, Triangle=[#2,#3,#4]}%
    \s@uvc@ntr@l\et@tpstrimesh\ifnum#1>\@ne%
    \Pss@tspecifSt{color=\DDV@tmeshcolor,dash=\DDV@tmeshdash,width=\DDV@tmeshwidth}%
    \setc@ntr@l{2}%
    \Pstrimeshp@rt#1[#2,#3,#4]\Pstrimeshp@rt#1[#3,#4,#2]\Pstrimeshp@rt#1[#4,#2,#3]%
    \Psrest@reSt{color=\DDV@tmeshcolor,dash=\DDV@tmeshdash,width=\DDV@tmeshwidth}%
    \fi\figdrawline[#2,#3,#4,#2]%
    \PSc@mment{End trimesh}\resetc@ntr@l\et@tpstrimesh\fi\fi}}
\ctr@ld@f\def\Pstrimeshp@rt#1[#2,#3,#4]{{\l@mbd@un=\@ne\l@mbd@de=#1\loop\ifnum\l@mbd@de>\@ne%
    \advance\l@mbd@de\m@ne\figptbary-1:[#2,#3;\l@mbd@de,\l@mbd@un]%
    \figptbary-2:[#2,#4;\l@mbd@de,\l@mbd@un]\figdrawline[-1,-2]%
    \advance\l@mbd@un\@ne\repeat}}
\initpr@lim\initpss@ttings\initPDF@rDVI
\ctr@ln@w{newbox}\figBoxA
\ctr@ln@w{newbox}\figBoxB
\ctr@ln@w{newbox}\figBoxC
\catcode`\@=12

\def\gap{[0.8ex]}

\def\myotimes{\otimes}

\newcommand\sfT{\mathsf{T}}
\newcommand\frH{\mathfrak{H}}
\newcommand\frc{\mathfrak{c}}
\newcommand\frd{\mathfrak{d}}
\newcommand\frT{\mathfrak{T}}
\newcommand\frl{\mathfrak{l}}

\def\FrmV{\frH_{\alpha,\tt, V}}
\def\FrmV{\frH_{\alpha,\tt, V}}

\def\OpV{\sfH_{\alpha, \tt, V}}

\def\Frm{Q_{\alpha, \cC}}
\def\Op{\sfH_{\alpha, \cC}}

\def\OpL{\mathsf{L}_{\alpha,\cC}}
\def\OpLf{\mathsf{L}_{\alpha,\Gamma_\tt}}
\def\frmLf{Q_{\alpha,\Gamma_\tt}}
\def\frmLft{\wt{Q}_{\alpha,\Gamma_\tt}}
\def\frmLfr{{Q}_{\alpha,\Gamma}}
\def\frmLfrt{\wt{Q}_{\alpha,\Gamma}}

\def\frmLWz{Q_0}
\def\frmLWo{Q_1}
\def\frmLWj{Q_j}

\def\frmLz{Q_0'}

\def\frmLzz{Q_{00}'}
\def\frmLzo{Q_{01}'}
\def\frmLzj{Q_{0j}'}

\def\frmLzpr{Q_0'}

\def\frmLWot{Q_{\alpha, \Omega^{K}_\tt}}
\def\frmLWone{\wt{Q}_{\alpha,\Omega^{K}_\tt}}

\def\frmVg{\frT_{\alpha, \tt,\gamma, V}}
\def\opVg{\sfT_{\alpha,\tt, \gamma, V}}

\def\frmg{\frT_{\alpha, \tt, \gamma}}
\def\opg{\sfT_{\alpha, \tt,\gamma}}

\newcommand\frm{\frT_{\alpha, \tt}}
\def\op{\sfT_{\alpha,\tt}}

\def\opID{\sfh_{\alpha,L}^{\rm D}}
\def\opIN{\sfh_{\alpha,L}^{\rm N}}
\def\frmID{\frq_{\alpha,L}^{\rm D}}
\def\frmIN{\frq_{\alpha,L}^{\rm N}}

\def\frmKSN{\frq_c^{\rm N}}
\def\frmKSD{\frq_c^{\rm D}}

\newcommand\N{\cN}

\def\frmR{\frT_{\tt,R}}
\def\opR{\sfT_{\tt,R}}

\def\frmRN{\frS_{\tt, R}}

\def\frmRI{\fra_{\tt, R}}
\def\frmRII{\frb_{\tt, R}}
\def\frmRIII{\frc_{\tt, R}}
\def\frmRIV{\frd_{\tt, R}}

\def\frmRk{\frS_{\aa,\tt,R}^{(k)}}

\def\frmtrans{\frt_{\aa,\tt, R}}
\def\frmlong{\frl_{\aa,\tt,R}}

\newcommand\medsum{{\textstyle\sum}}

\newcommand\Cd{{\cC_{d,\theta}}}

\newcommand\m{\mathfrak{m}}
\def\tt{\theta}
\def\aa{\alpha}
\def\lm{\lambda}
\def\sfh{\mathsf{h}}
\def\Gt{{\Gamma_\tt}}
\def\St{{\Sigma_\tt}}
\newcommand\G{\Gamma}
\def\arr{\rightarrow}
\def\frg{\mathfrak{g}}
\def\fri{\mathfrak{i}}

\def\restric#1#2{{#1} |_{#2}}

\newcommand\myemph[1]{\textbf{\emph{#1}}} 

\definecolor{darkred}{rgb}{0.5,0.1,0.1}
\newcommand{\vl}[1]{{\color{red} #1}}
\newcommand{\vlc}[1]{{\color{red} VL: #1}}

\newcommand\s{\sigma}
\newcommand\sd{\sigma_{\rm dis}}
\newcommand\sess{\sigma_{\rm ess}}

\newcommand\ii{{\mathsf{i}}}
\newcommand\p{\partial} 

\newcommand\pp{\prime}
\newcommand\dpp{{\prime\prime}}
\newcommand\tpp{{\prime\prime\prime}}

\newcommand\sfH{\mathsf{H}}
\renewcommand\sfh{\mathsf{h}}
\newcommand\one{\mathbbm{1}}
\newcommand\Z{\mathsf{z}}
\newcommand\sfP{\mathsf{P}}
\newcommand\dd{{\mathsf{d}}}
\newcommand\uhr{\upharpoonright}

\newcommand\frf{{\mathfrak f}}
\newcommand\frj{{\mathfrak j}}
\newcommand\omg{\omega}

\newcommand\sfw{\mathsf{w}}
\newcommand\sfz{\mathsf{z}}

\newcommand\sfS{\mathsf{S}}
\newcommand\sfC{\mathsf{C}}
\newcommand\sfW{\mathsf{W}}
\newcommand\sfZ{\mathsf{Z}}
\newcommand\sfU{\mathsf{U}}
\newcommand\sfV{\mathsf{V}}

\newcommand*{\medcap}{\mathbin{\scalebox{1.5}{\ensuremath{\cap}}}}%
\newcommand*{\medcup}{\mathbin{\scalebox{1.5}{\ensuremath{\cup}}}}%
\newcommand*{\medoplus}{\mathbin{\scalebox{1.5}{\ensuremath{\oplus}}}}%

\def\radius{3cm}
\def\softness{0.4}

\definecolor{softred}{rgb}{1,\softness,\softness}
\definecolor{softgreen}{rgb}{\softness,1,\softness}
\definecolor{softblue}{rgb}{\softness,\softness,1}
\definecolor{softrg}{rgb}{1,1,\softness}
\definecolor{softrb}{rgb}{1,\softness,1}
\definecolor{softgb}{rgb}{\softness,1,1}

\newcounter{counter_a}
\newenvironment{myenum}{\begin{list}{{\rm(\roman{counter_a})}}%
{\usecounter{counter_a}
\setlength{\itemsep}{1.ex}\setlength{\topsep}{1.5ex}
\setlength{\leftmargin}{5ex}\setlength{\labelwidth}{5ex}}}{\end{list}}
\newcommand\ds{\displaystyle}
\usepackage[latin1]{inputenc}
\usepackage[T1]{fontenc}
\newcommand{\red}{\color{red}}

\newcommand{\eg}{{\it e.g.}\,}
\newcommand{\ie}{{\it i.e.}\,}
\newcommand{\cf}{{\it cf.}\,}

\numberwithin{figure}{section}
\numberwithin{equation}{section}
\theoremstyle{plain}
\newtheorem*{thm*}{Theorem}
\newtheorem{thm}{Theorem}[section]

\newtheorem{lem}[thm]{Lemma}
\newtheorem{prop}[thm]{Proposition}

\newtheorem{cor}[thm]{Corollary}

\newtheorem{notn}[thm]{Notation}

\newtheorem{dfn}[thm]{Definition}
\theoremstyle{remark}

\theoremstyle{plain}

\newcommand{\dsp}{\displaystyle}
\newcommand{\spec}{{\mathrm{spec}}}
\newcommand{\rmd}{\mathrm{d}}
\newcommand{\rmi}{\mathrm{i}}
\newcommand{\supp}{\mathrm{supp}\,}
\newcommand{\beu}{\begin{equation*}}
\newcommand{\eeu}{\end{equation*}}
\newcommand{\besu}{\begin{equation*}
\begin{aligned}}
\newcommand{\eesu}{\end{aligned}
\end{equation*}}
\newcommand{\bes}{\begin{equation}
\begin{aligned}}
\newcommand{\ees}{\end{aligned}
\end{equation}}

\newcommand\cA{\mathcal A}
\newcommand\cB{\mathcal B}
\newcommand\cD{\mathcal D}
\newcommand\cF{\mathcal F}
\newcommand\cG{\mathcal G}
\newcommand\cH{\mathcal H}
\newcommand\cK{\mathcal K}
\newcommand\cJ{\mathcal J}
\newcommand\cL{\mathcal L}
\newcommand\cM{\mathcal M}
\newcommand\cN{\mathcal N}
\newcommand\cP{\mathcal P}
\newcommand\cR{\mathcal R}
\newcommand\cS{\mathcal S}
\newcommand\CC{\mathbb C}
\newcommand\NN{\mathbb N}
\newcommand\RR{\mathbb R}
\newcommand\ZZ{\mathbb Z}
\newcommand\frA{\mathfrak A}
\newcommand\frB{\mathfrak B}
\newcommand\frS{\mathfrak S}
\newcommand\fra{\mathfrak a}
\newcommand\frq{\mathfrak q}
\newcommand\frs{\mathfrak s}
\newcommand\frp{\mathfrak p}
\newcommand\frh{\mathfrak h}
\newcommand\fraD{\mathfrak{a}_{\rm D}}
\newcommand\fraN{\mathfrak{a}}
\newcommand\dis{\displaystyle}
\newcommand\ov{\overline}
\newcommand\wt{\widetilde}
\newcommand\wh{\widehat}
\newcommand{\defeq}{\mathrel{\mathop:}=}
\newcommand{\eqdef}{=\mathrel{\mathop:}}
\newcommand{\defequ}{\mathrel{\mathop:}\hspace*{-0.72ex}&=}
\newcommand\vectn[2]{\begin{pmatrix} #1 \\ #2 \end{pmatrix}}
\newcommand\vect[2]{\begin{pmatrix} #1 \\[1ex] #2 \end{pmatrix}}
\newcommand\restr[1]{\!\bigm|_{#1}}
\newcommand\RE{\text{\rm Re}}

\newcommand\sign{{\rm sign\,}}

\newcommand\void[1]{}

\newcommand\eps{\varepsilon}
\newcommand\ran{{\rm ran\,}}
\newcommand\sigp{\sigma_{\rm p}}
\DeclareMathOperator\tr{tr}

\def\sA{{\mathfrak A}}   \def\sB{{\mathfrak B}}   \def\sC{{\mathfrak C}}
\def\sD{{\mathfrak D}}   \def\sE{{\mathfrak E}}   \def\sF{{\mathfrak F}}
\def\sG{{\mathfrak G}}   \def\sH{{\mathfrak H}}   \def\sI{{\mathfrak I}}
\def\sJ{{\mathfrak J}}   \def\sK{{\mathfrak K}}   \def\sL{{\mathfrak L}}
\def\sM{{\mathfrak M}}   \def\sN{{\mathfrak N}}   \def\sO{{\mathfrak O}}
\def\sP{{\mathfrak P}}   \def\sQ{{\mathfrak Q}}   \def\sR{{\mathfrak R}}
\def\sS{{\mathfrak S}}   \def\sT{{\mathfrak T}}   \def\sU{{\mathfrak U}}
\def\sV{{\mathfrak V}}   \def\sW{{\mathfrak W}}   \def\sX{{\mathfrak X}}
\def\sY{{\mathfrak Y}}   \def\sZ{{\mathfrak Z}}

\def\frb{{\mathfrak b}}
\def\frl{{\mathfrak l}}
\def\frs{{\mathfrak s}}
\def\frt{{\mathfrak t}}

\def\dA{{\mathbb A}}   \def\dB{{\mathbb B}}   \def\dC{{\mathbb C}}
\def\dD{{\mathbb D}}   \def\dE{{\mathbb E}}   \def\dF{{\mathbb F}}
\def\dG{{\mathbb G}}   \def\dH{{\mathbb H}}   \def\dI{{\mathbb I}}
\def\dJ{{\mathbb J}}   \def\dK{{\mathbb K}}   \def\dL{{\mathbb L}}
\def\dM{{\mathbb M}}   \def\dN{{\mathbb N}}   \def\dO{{\mathbb O}}
\def\dP{{\mathbb P}}   \def\dQ{{\mathbb Q}}   \def\dR{{\mathbb R}}
\def\dS{{\mathbb S}}   \def\dT{{\mathbb T}}   \def\dU{{\mathbb U}}
\def\dV{{\mathbb V}}   \def\dW{{\mathbb W}}   \def\dX{{\mathbb X}}
\def\dY{{\mathbb Y}}   \def\dZ{{\mathbb Z}}

\def\cA{{\mathcal A}}   \def\cB{{\mathcal B}}   \def\cC{{\mathcal C}}
\def\cD{{\mathcal D}}   \def\cE{{\mathcal E}}   \def\cF{{\mathcal F}}
\def\cG{{\mathcal G}}   \def\cH{{\mathcal H}}   \def\cI{{\mathcal I}}
\def\cJ{{\mathcal J}}   \def\cK{{\mathcal K}}   \def\cL{{\mathcal L}}
\def\cM{{\mathcal M}}   \def\cN{{\mathcal N}}   \def\cO{{\mathcal O}}
\def\cP{{\mathcal P}}   \def\cQ{{\mathcal Q}}   \def\cR{{\mathcal R}}
\def\cS{{\mathcal S}}   \def\cT{{\mathcal T}}   \def\cU{{\mathcal U}}
\def\cV{{\mathcal V}}   \def\cW{{\mathcal W}}   \def\cX{{\mathcal X}}
\def\cY{{\mathcal Y}}   \def\cZ{{\mathcal Z}}

\renewcommand{\div}{\mathrm{div}\,}
\newcommand{\grad}{\mathrm{grad}\,}
\newcommand{\Tr}{\mathrm{Tr}\,}

\newcommand{\dom}{\mathrm{dom}\,}
\newcommand{\mes}{\mathrm{mes}\,}

\makeatletter
\renewcommand\theequation{\thesection.\arabic{equation}}
\@addtoreset{equation}{section}

\def\@listI{
    \leftmargin\leftmargini
    \parsep 1.5pt plus 1pt minus 1pt
    \topsep -1.5pt plus 1pt minus 1pt
    \itemsep \parsep}
\let\@listi\@listI
\makeatother

\usepackage{color}
\newcommand{\Bk}{\color{black}}
\newcommand{\Rd}{\color{red}}
\newcommand{\Gn}{\color{green}}
\newcommand{\Bl}{\color{blue}}
\def\to#1{\Bl {#1} \Bk}

\newcommand{\R}{\mathbb{R}}

\newcommand{\spn}{\mathrm{span}\,}

\newcommand{\fq}{\mathfrak{q}}
\newcommand{\frQ}{\mathfrak{Q}}

\newcommand{\Gui}{\mathsf{Gui}}
\newcommand{\Lay}{\mathsf{Lay}}
\newcommand{\Dom}{\mathsf{Dom}}
\newcommand{\Tri}{\mathsf{Tri}}

\newcommand{\far}{\mathsf{far}}
\newcommand{\near}{\mathsf{near}}
\newcommand{\trans}{\mathsf{trans}}

\newcommand{\ps}[2]{\left\langle#1,#2\right\rangle}
\newcommand{\vabs}[1]{\left|#1\right|}

\title[Bound states of $\delta$-interactions on conical surfaces]
{On the bound states of Schr\"odinger operators with \boldmath{$\delta$}-interactions on conical surfaces}

\author{Vladimir Lotoreichik} 
\address{Department of Theoretical Physics,
Nuclear Physics Institute, Czech Academy of Sciences, 250 68, \v{R}e\v{z} near Prague, Czech Republic}
\email{lotoreichik@ujf.cas.cz}

\author{Thomas Ourmi\`eres-Bonafos}
\address{BCAM - Basque Center for Applied Mathematics, Alameda de Mazarredo, 14 E48009 Bilbao, Basque Country -  Spain}
\email{tourmieres@bcamath.org}

\begin{document}



\maketitle

\begin{abstract}
In dimension greater than or equal to three, we investigate the spectrum 
of a Schr\"odinger operator with a $\delta$-interaction supported on a cone whose cross section 
is the sphere of co-dimension two.
After decomposing into fibers, we prove that there is discrete spectrum only in dimension three and that 
it is generated by the axisymmetric fiber. We get that these eigenvalues are non-decreasing functions 
of the aperture of the cone and we exhibit the precise logarithmic accumulation of the discrete spectrum below the threshold of the essential spectrum.
\end{abstract}

\section{Introduction}

\subsection{Motivation}

Some physical systems are efficiently described by 
Schr\"{o}dinger operators with singular $\delta$-type interactions 
supported on various sets of zero Lebesgue measure
(points, curves, surfaces or hypersurfaces). 
For instance, these operators are used to approximate atomic Hamiltonians in strong homogeneous magnetic fields~\cite{BD06} or 
photonic crystals with high-contrast~\cite{FK96}. 
The spectra of such Schr\"odinger operators are related to 
admissible values of the energy in quantum mechanics or, to admissible propagation frequencies of electromagnetic waves in optics. 

A natural issue is to understand how the geometry of the support of the
$\delta$-interaction influences the spectrum of these Schr\"odinger operators. This question is important, not only 
because of prospective applications in physics, but because it is also mathematically relevant in spectral geometry. 
For references on this topic, we refer to the review paper~\cite{E08}, 
the monograph~\cite{EK15} and the references therein.

In dimension two, for $\delta$-interactions supported on curves,  
the question of finding a connection between the spectrum and the
geometry was first addressed in~\cite{EI01}. In that paper, 
the authors consider two-dimensional Schr\"odinger operators
with attractive $\delta$-interactions supported 
on asymptotically straight curves. 
Provided that the curve is not a straight line, 
they prove that there exists at least one 
bound state below the threshold of the essential spectrum. 
In the same spirit, we mention the special case of $\delta$-interactions supported on broken lines, 
investigated in~\cite{BEW08, DR14, EKo15, EN03}.

In dimension three the state of the art is not as complete as in dimension two. 
Instead of dealing with asymptotically straight curves, one is interested in attractive $\delta$-interactions 
supported on asymptotically flat surfaces and such Schr\"odinger operators were first studied in~\cite{EK03}. 
Provided the cross section is smooth, infinite conical surfaces give rise to a family of asymptotically flat surfaces. 
The special case of a circular cross section is investigated in~\cite{BEL14} where the main result is the existence 
of infinitely many bound states below the threshold of the essential spectrum. Moreover, the authors bound from above 
the sequence of eigenvalues by a sequence which converges to the threshold of the essential spectrum at a known rate. 
Nevertheless, sharp spectral asymptotics on the number of eigenvalues remained unknown so far. 
We tackle this question in the present paper and give the precise rate of accumulation. It is reminiscent of~\cite{DOR15}, 
where the authors exhibit a similar result for a Dirichlet Laplacian in a conical layer. 
In this last paper, the authors also study the behaviour of the eigenvalues with respect to the aperture of the cone.
Here, we restrain ourselves to show that the eigenvalues depend monotonously
on the aperture of the cone.

In dimension greater than or equal to four very little is known so far. In the special case of an 
attractive $\delta$-interaction supported on a hyperconical surface
we prove that there is no discrete spectrum. Because a hyperconical surface splits the Euclidean space into 
a convex domain and a non-convex conical domain, it is worth mentioning that this result strengthens 
the difference with Robin Laplacians in convex circular conical domains. Indeed, according to~\cite{P15}, 
these Robin Laplacians have infinite discrete spectrum for any dimension greater than or equal to three.

Finally, we emphasise that, for attractive $\delta$-interactions supported on conical surfaces, the structure of the essential spectrum strongly depends on the smoothness of the cross section. 
For attractive $\delta$-interactions supported on general non-smooth conical surfaces 
this structure is expected to be more involved and 
we refer to~\cite{BDP14} for similar considerations about magnetic Laplacians on non-smooth conical domains.

\subsection{Hamiltonians with $\delta$-interactions on conical surfaces}

For $d\geq3$, let $(x_1,x_2,x_3,\dots,x_d)$ be the Cartesian coordinates on $\dR^d$. $(L^2(\dR^d); (\cdot,\cdot)_{\dR^d})$ and $(L^2(\dR^d;\dC^d), (\cdot,\cdot)_{\dR^d})$ denote the usual $L^2$-space and the $L^2$-space of vector-valued functions over $\R^d$, respectively. The first order $L^2$-based Sobolev space over $\R^d$ is denoted by $H^1(\dR^d)$.
We define the function $\dR^d\ni x \mapsto \rho(x) \in \dR_+$ as 
\begin{equation}\label{def:rho}
	\rho(x) := \sqrt{\medsum_{k=1}^{d-1} x_k^2}
\end{equation}
and introduce the conical hypersurface $\Cd\subset\dR^d$, as
\begin{equation}\label{def:cone}
	\Cd := 
		\big\{x = (x_1,x_2,\dots,x_d) \in\dR^d\colon x_d = (\cot\tt)\rho(x)\big\},
	\qquad 
	\tt\in(0,\pi/2).
\end{equation}
The parameter $\tt$ is the half-opening angle (aperture) of the cone (\cf Figure~\ref{fig:cone}). 
Since there is no possible confusion, we denote the conical hypersurface by $\cC$ instead of $\Cd$. 
We also denote by $(L^2(\cC); (\cdot,\cdot)_\cC))$ the $L^2$-space over $\cC$.
 
\begin{figure}[h!]
\figinit{0.75cm}
\figpt 0:(0,0)
\figpt 1:(0,2.776221382176024)
\figpt 2:(0,-2.776221382176024)
\figpt 3:(-4.442882938158366,0)
\figpt 4:(-5,0)
\figpt 5:(2,0)
\figpt 6:(-0.5,0.3)
\figpt 7:(-3.4,0)
\figpt 8:(-2.221441469079183,1.388110691088013)
\figpt 9:(-2.221441469079183,-1.388110691088013)
\figpt 10:(-2.221441469079183,0)
\figpt 11:(3.5,0)
\figpt 12:(4.75,0)
\figpt 13:(3.75,-0.25)
\figpt 14:(3.75,1)
\figpt 15:(4,0.20)
\figpt 16:(3.25,-0.45)
\figdrawbegin{}
\figdrawline [3,1]
\figdrawline [3,2]
\figdrawarrow [4,5]
\figdrawarrow [11,12]
\figdrawarrow [13,14]
\figdrawarrow [15,16]
\figdrawarcell 0 ; 0.5,2.776221382176024 (0,360, 0)
\figdrawarcell 10 ; 0.3,1.388110691088013 (90,260,0)
\figdrawarrowcircP 3;1.2[0,1]
\figset (dash=8)
\figdrawarcell 10 ; 0.3,1.388110691088013 (-90,90,0)
\figdrawend

\figvisu{\figBoxA}{}{
\figwritenw 6: {$\cC_{d,\tt}$}(0.3)
\figwritene 7:{$\theta$} (0.4)
\figwrites 12:{$x_{3}$} (0.15)
\figwritenw 14:{$x_{2}$} (0)
\figwritesw 16:{$x_{1}$}(0)
}

\centerline{\box\figBoxA}
\caption{The cone $\cC_{d,\tt}$ in dimension $d=3$.}
\label{fig:cone}
\end{figure}


For $\alpha > 0$, we introduce the symmetric, densely defined sesquilinear form
\begin{equation}\label{eqn:defqf}
	\Frm [ u, v ]  
	:= 	( \nabla u, \nabla v )_{\dR^d}  
							  - \alpha ( u|_\cC,v|_\cC)_\cC,
	\qquad	
	\dom\Frm := H^1(\dR^d).
\end{equation}
This form is closed and semibounded in $L^2(\dR^d)$ (\cf, \eg, \cite[Sec. 2]{BEKS94} and~\cite[Prop. 3.1]{BEL14_RMP}). 

\begin{dfn}\label{def:Op} 
	By the first representation theorem (\cite[Ch. VI, Thm. 2.1]{K}), 
	the form $\Frm$ is associated with a self-adjoint operator $\Op$ acting on $L^2(\R^d)$. 
	This operator is called Schr\"odinger operator with $\delta$-interaction of strength $\alpha > 0$ supported on $\cC$.
\end{dfn}

The operator $\Op$ can be understood as the Laplacian
on $\dR^d$ with a $\delta$-type coupling boundary condition on the conical surface $\cC$. Formally, one can write 
$\Op = -\Delta - \alpha\delta_{\cC}$ or
$\Op = -\Delta - \alpha\delta(x - \cC)$.
We refer to \cite{BEL14_RMP}, for a rigorous description of the action of $\Op$ and of its domain.

\subsection{Notations and main results} 

We introduce a few notation before stating the main results of this paper.
The set of positive integers is denoted by $\mathbb{N} := \{1,2,\dots\}$ and the set of natural integers is denoted by 
$\mathbb{N}_0 := \mathbb{N}\cup\{0\}$. Let $\sfT$ be a semi-bounded self-adjoint operator 
associated with the quadratic form $\frt$. 
We denote by $\sess(\sfT)$ and $\sd(\sfT)$
the essential and the discrete spectrum of $\sfT$, respectively.
By $\s(\sfT)$, we denote the spectrum of $\sfT$ 
(\ie $\s(\sfT) = \sess(\sfT)\cup\sd(\sfT)$).

We set $E_{\rm ess}(\sfT) := \inf\sess(\sfT)$ and, for $k\in\mathbb{N}$, $E_k(\sfT)$ denotes the $k$-th eigenvalue of $\sfT$
in the interval $(-\infty, E_{\rm ess}(\sfT))$. They are 
ordered non-decreasingly with multiplicities taken into account.
We define the counting function of $\sfT$ as
\[
	\N_E(\sfT) := \#\big\{k\in\dN\colon E_k(\sfT) < E\big\}, 
	\qquad E \le E_{\rm ess}(\sfT).
\]	
When working with the quadratic form $\frt$, 
we use the notations $\sess(\frt)$, $\sd(\frt)$, $\s(\frt)$,
$E_{\rm ess}(\frt)$, $E_k(\frt)$ and $\N_E(\frt)$  instead.

The first result is about the characterisation of the essential spectrum
and the qualitative description of the discrete spectrum for $\Op$.
\begin{thm}
	Let $\tt\in(0,\pi/2)$ and $\alpha>0$. The following statements hold:
	\begin{myenum}	
		\item for any dimension $d\geq3$, $\sess(\Op) = [-\alpha^2/4,+\infty)$;
		\item for $d=3$, $\#\sd(\Op) = \infty$; 
		\item for $d \ge 4$, $\#\sd(\Op) = 0$.
	\end{myenum}
\label{thm:struc_spec}
\end{thm}
Note that Theorem~\ref{thm:struc_spec} was already known in dimension $d=3$ (\cf \cite{BEL14}). 
As the proof of the structure of the essential spectrum in dimension $d \ge 4$ follows 
exactly the same lines as the one exposed in \cite[\S 2]{BEL14} (in dimension $d=3$), we omit it here for the sake of brevity.
The absence of discrete spectrum in item~(iii) differs from the results in \cite{EI01, EK03} where, in dimension $d=2$ or $d=3$,
it is shown that geometric deformations always induce bound states (at least
for $\aa > 0$ sufficiently large). To prove item~(iii),
we show that the operator $\Op$ is unitarily equivalent to an infinite orthogonal
sum of self-adjoint fiber operators and that the spectrum
of each fiber operator is included in $[-\aa^2/4,+\infty)$.

Then, relying on Theorem~\ref{thm:struc_spec}, 
we focus on properties of the discrete spectrum of $\Op$ in dimension $d=3$.
\begin{prop}
	Let $\aa > 0$. In dimension $d=3$, for all $k\in\dN$ the functions, $\tt \mapsto E_k( \Op )$
	are non-decreasing on $(0,\pi/2)$.
	\label{prop:non_decr_eigv}
\end{prop}
This proposition is reminiscent of a similar result in~\cite{EN03} 
for $\delta$-interactions supported on broken lines.
Nevertheless, our proof is somewhat simpler since we do not use the 
Birman-Schwinger principle. The idea of the proof is to show
that the discrete spectrum of $\Op$, below the point $-\aa^2/4$, coincides with the discrete spectrum of the lowest fiber operator below the same point.
Then, one can show that the lowest fiber operator is unitarily equivalent to another
operator, whose form domain is independent of $\tt$ and whose Rayleigh
quotient is a monotone function of $\tt$.

Finally, we state our main result on the spectral asymptotics of 
$\Op$ for $d = 3$.
%
%
%
\begin{thm}
	Let $\theta\in (0,\pi/2)$ and $\alpha > 0$. In dimension $d=3$, we have
	\[
		\N_{-\aa^2/4 - E}(\Op) 
		\sim 
		\frac{\cot\tt}{4\pi} |\ln(E)|,
		\quad E \arr 0+.
	\]
\label{th:acceig}
\end{thm}


The proof of Theorem~\ref{th:acceig} is
inspired by a similar strategy developed in \cite{DOR15} for Dirichlet
conical layers. Loosely speaking, we reduce the spectral asymptotics 
of $\Op$ to the known spectral asymptotics of
the one-dimensional Schr\"odinger operator
\[
	-\frac{\dd^2}{\dd x^2} - \frac{1}{4\sin^2\tt} \frac{1}{x^2}, 
	\qquad \text{on}~~(1,+\infty),
\]
with Dirichlet or Neumann boundary condition at $x=1$.
To this end, we use Dirichlet and Neumann bracketing combined with an IMS formula (\cf \cite{CFKS87}). We emphasise that 
the geometry of the bracketings depends on the spectral
parameter in a more sophisticated way than in~\cite{DOR15}.
Moreover, to estimate the operators involved in the Dirichlet and Neumann bracketings, 
we need spectral properties of two specific one-dimensional Schr\"odinger operators with a $\delta$-point interaction.

%


\subsection{Organisation of the paper} 

In Section~\ref{sec:Fib_dec}, we reduce the study of $\Op$ to a family of two-dimensional operators (fibers). 
This reduction allows to understand the structure of the discrete spectrum mentioned in Theorem~\ref{thm:struc_spec} 
and to prove Proposition~\ref{prop:non_decr_eigv}. After introducing one-dimensional model operators, Theorem~\ref{th:acceig} 
is proven in Section~\ref{sec:couneig}. Finally, we conclude the paper with Appendix~\ref{app:appB} about some properties of 
fiber decompositions.

\section{Fiber decomposition and its first applications}
\label{sec:Fib_dec}

In this section we prove Theorem~\ref{thm:struc_spec} and 
Proposition~\ref{prop:non_decr_eigv}. To take advantage of the symmetry of the problem, we work with a description of $\Op$ in cylindrical coordinates. Then, we reduce the study to a family of two-dimensional operators (fibers). For $d=3$, the fiber decomposition
is also used in the proof of the spectral asymptotics in Section~\ref{sec:couneig}.

\subsection{Hyper-cylindrical coordinates} 

Since the conical surface $\cC$ is axisymmetric, 
the problem is better described in (hyper-)cylindrical coordinates with $x_d$ as the reference axis. Let us denote these coordinates by 
$(r,z, \phi)\in\dR_+\times\dR\times\dS^{d-2}$, 
where $\dS^{d-2} \cong [0,\pi]^{d-3} \times[0,2\pi)$ is the unit sphere of dimension $d-2$ and $\mathfrak{m}_{d-2}$ its natural surface measure. 
We can write 
$\phi\in\dS^{d-2}$ as 
$\phi=(\phi_1,\dots,\phi_{d-2})\in[0,\pi]^{d-3}\times[0,2\pi)$ 
and, for all $k\in\{1,\dots,d-2\}$, we have             
\begin{equation}
	x_k = r \bigg(\prod_{p=1}^{k-1}\sin \phi_p\bigg) \cos \phi_k,\quad 
	x_{d-1} = r\prod_{p=1}^{d-2}\sin \phi_p, \quad x_d = z.
\label{eqn:var_cyl}
\end{equation}
For further use, we introduce the \textit{meridian domain} 
$\dR_+^2 = \dR_+\times\dR$, the \textit{meridian ray}
\begin{equation}\label{def:Gt}
	\Gt := 
	\big\{ (r,z) \in\dR^2_+\colon  z = r \cot\tt\big\},
\end{equation}
Now, we introduce some notations related to the cylindrical coordinates.

\begin{notn}$L_{\mathsf{cyl}}^2(\R^d)$ denotes the Hilbert space
\[
	L_{\mathsf{cyl}}^2(\R^d) = L^2(\R_+^2\times\mathbb{S}^{d-2},r^{d-2}\dd r \dd z \dd\m_{d-2}(\phi)).
\]
Similarly, we define the Sobolev cylindrical space $H_{\mathsf{cyl}}^1(\R^d)$ by
\[
	H_{\mathsf{cyl}}^1(\R^d) := \{u \in L_{\mathsf{cyl}}^2(\R^d) : \partial_r u, \partial_z u, r^{-1}|\nabla_{\mathbb{S}^{d-2}}u| 
	\in L_{\mathsf{cyl}}^2(\R^d)\},
\]
endowed with the norm
\[
	\|u\|_{H_\mathsf{cyl}^1(\R^d)}^2 := \|u\|_{L_{\mathsf{cyl}}^2(\R^d)}^2 + \|\partial_r u\|_{L_{\mathsf{cyl}}^2(\R^d)}^2 + 
	\|\partial_z u\|_{L_{\mathsf{cyl}}^2(\R^d)}^2 + \|r^{-1}|\nabla_{\mathbb{S}^{d-2}}u|\|_{L_{\mathsf{cyl}}^2(\R^d)}^2,
\]
where $\nabla_{\mathbb{S}^{d-2}}$ is the surface gradient on $\mathbb{S}^{d-2}$.

$\cC_0^\infty(\ov{\R_+^2})$ denotes the space of infinitely differentiable 
functions with compact support in $\ov{\R_+^2}$ and we introduce the space $\cC_{0,0}^\infty(\ov{\R_+^2})$ 
defined by
\[
	\cC_{0,0}^\infty(\ov{\R_+^2}) = \{u \in \cC_0^\infty(\ov{\R_+^2}) : \restric{u}{r=0} = 0\}.
\]
\label{notn:def_space}
\end{notn}

The change of variables \eqref{eqn:var_cyl} maps the whole space 
$\dR^d$ onto $\dR_+^2\times\dS^{d-2}$ and $\Op$ 
becomes the unbounded self-adjoint operator $\OpL$ acting on $L_{\mathsf{cyl}}^2(\R^d)$. By the first representation theorem, $\OpL$ can be seen as the operator associated with the quadratic form $\Frm^\mathsf{cyl}$, defined by the expression of $\Frm$ in cylindrical coordinates:
\[
	\begin{split}
		\Frm^\mathsf{cyl}[u] &:= \|\partial_r u\|_{L_{\mathsf{cyl}}^2(\R^d)}^2 + 
		\|\partial_z u\|_{L_{\mathsf{cyl}}^2(\R^d)}^2 + \|r^{-1}|\nabla_{\mathbb{S}^{d-2}}u|\|_{L_{\mathsf{cyl}}^2(\R^d)}^2\\
		& \qquad\qquad - \alpha \int_{\R_+\times\dS^{d-2}} |u(s\sin\tt,s\cos\tt,\phi)|^2\dd s \dd\mathfrak{m}_{d-2}(\phi).\\
		\dom\Frm^\mathsf{cyl} &:= H_\mathsf{cyl}^1(\R^d).
	\end{split}  
\]
%

\subsection{Spherical harmonics} 

First, we recall some known results on the spherical harmonics that can be found, \eg, in \cite[Chap. IV \S 2]{SW71}. 
$-\Delta_{\dS^{d-2}}$ denotes the Laplace-Beltrami on the sphere $\dS^{d-2}$ and its eigenvalues are given by $l(l+d-3)$, 
with $l\in\dN_0$ and the associated eigenspaces, 
denoted by $\cG_l^{d-2}$, are of dimension 
\[
	c(d,l) := \binom{d+l-2}{d-2} - \binom{d+l-4}{d-2},
	\qquad
	\text{where}\quad \binom{n}{k} :=\left\{ 	\begin{array}{cl}\frac{n!}{k!(n-k)!}& \text{if } n \ge k,\\
												0&\text{otherwise}.
									\end{array}
	\right.
\]	
Moreover, for all $k \in \{1,\dots,c(d,l)\}$, we denote by $Y_{l,k}^{d-2}$ the usual spherical harmonics. 
For all $l\in\dN_0$ and $k\in\{1,\dots,c(d,l)\}$, they satisfy
\begin{equation*}
	-\Delta_{\dS^{d-2}}Y_{l,k}^{d-2} = l(l+d-3)Y_{l,k}^{d-2},\qquad 
	\cG_l^{d-2} = \underset{k\in\{1,\dots,c(d,l)\}}{\spn}\{Y_{l,k}^{d-2}\}.
\end{equation*}
Finally, the family of all the spherical harmonics forms an orthonormal basis of $L^2(\dS^{d-2})$; \ie,
\[
	\ov{\underset{l\in\dN_0}{\spn}\cG_l^{d-2}}
	= 
	L^2(\dS^{d-2}).
\]
Now, decomposing into spherical harmonics and according to the terminology of 
\cite[\S XIII.16]{RS78}, we have the constant fiber sum
\begin{equation*}
\begin{array}{lcl}
	L^2(\dR_+^2\times\dS^{d-2},r^{d-2}\dd r \dd z \dd\mathfrak{m}_{d-2}(\phi)) 
	&=& 
	L^2(\dR_+^2,r^{d-2}\dd r\dd z) \otimes
	L^2(\dS^{d-2})\\
	& = & 
	\displaystyle\bigoplus_{l\in\dN_0}
	\bigoplus_{k=1}^{c(d,l)}L^2(\dR_+^2,r^{d-2}\dd r\dd z).
\end{array}
\end{equation*}
Secondly, for any function $u\in L_\mathsf{cyl}^2(\R^d)$, we define
	\[
		(\pi_{l,k}u)(r,z) := \ps{u}{Y_{l,k}^{d-2}}_{\dS^{d-2}} = \int_{\dS^{d-2}} u(r,z,\phi) \ov{Y_{l,k}^{d-2}}(\phi) \dd \mathfrak{m}_{d-2}(\phi).
	\]
We consider the family $(\Pi_{l,k})_{l\in\dN_0,k\in\{1,\dots,c(d,l)\}}$ of orthogonal projectors on $L_\mathsf{cyl}^2(\R^d)$, defined as
\begin{equation}\label{eq:Pikl}
	(\Pi_{l,k} u)(r,z,\phi) 
	:= 
	(\pi_{l,k} u)(r,z) Y_{l,k}^{d-2}(\phi). 
\end{equation}
By definition of the spherical harmonics and $H_\mathsf{cyl}^1(\R^d)$, we know that $\Pi_{l,k}(H_\mathsf{cyl}^1(\dR^d))\subset H_\mathsf{cyl}^1(\dR^d)$. 
We introduce the quadratic forms
\begin{equation}
	\frmLf ^{[l,k]}[u]  := \Frm^\mathsf{cyl}[u Y_{l,k}^{d-2}],\qquad \dom\frmLf^{[l,k]}  := \pi_{l,k}(H_\mathsf{cyl}^1(\R^d)).
	\label{eqn:def_fq1}
\end{equation}
By straightforward computations we first notice that the forms $\frmLf^{[l,k]}$ do not depend on $k$ and to simplify, we drop the index $k$. 
Thirdly, we get
\begin{equation}
\begin{split}
	\frmLf ^{[l]}[u] & = 
	\int_{\dR_+^2}\Big(|\p_r u|^2 + |\p_z u|^2 + 
	\frac{l(l+d-3)}{r^2}|u|^2\Big)r^{d-2}\dd r \dd z\\
	&\qquad\qquad\qquad\qquad\qquad -
	\alpha\int_{\dR_+}
	|u(s\sin\theta,s\cos\theta)|^2s^{d-2}\sin^{d-2}\theta\dd s,\\
	\dom\frmLf^{[l]} & = \left\{
	\begin{array}{lr}
		\{u : u, \p_r u, \p_z u \in L^2(\dR_+^2,r^{d-2}\dd r\dd z)\},
		& l = 0,\\
		\{u : u, \p_r u, \p_z u, r^{-1}u \in L^2(\dR_+^2,r^{d-2}\dd r\dd z)\},
		& l \ne 0.
	\end{array}
	\right.
\end{split}
\label{eqn:def_fq}
\end{equation}
We refer to \cite[\S II.3.a]{BDM99} for a full description of the domains of the above forms when $d=3$.

Using \eqref{eqn:def_fq1}, one can show that the quadratic forms $\frmLf^{[l]}$ are symmetric, closed, densely defined and semibounded on $\pi_{l,k}(L_\mathsf{cyl}^2(\R^d)) = L^2(\R_+^2,r^{d-2}\dd r\dd z)$. Hence, by the first representation theorem, each quadratic form $\frmLf^{[l]}$ is associated with a self-adjoint operator $\OpLf^{[l]}$ acting on $L^2(\R_+^2,r^{d-2}\dd r\dd z)$.

Using the precise description of $\dom\OpL$ given in \cite[Thm. 3.3 (a)]{BEL14_RMP} and the symmetry of $\cC$, one can show that
\[
	\Pi_{l,k}(\dom\OpL)\subset\dom\OpL,\quad \OpL(\Pi_{l,k}(\dom\OpL)) \subset \Pi_{l,k}(L_\mathsf{cyl}^2(\R^d)).
\]
The first representation theorem implies that the operator $\pi_{l,k}\restric{\OpL}{\Pi_{l,k}(\dom\OpL)}$ can be identified with $\OpLf^{[l]}$  and that they have the same spectrum. By \cite[\S 1.4]{S12}, the operator $\OpL$ decomposes as	
\begin{equation}
	\OpL = \bigoplus_{l\in\dN_0}\bigoplus_{k=1}^{c(d,l)} \OpLf^{[l]}.
\label{eqn:fib_dec}
\end{equation}
The self-adjoint operators $\OpLf^{[l]}$ are the fibers of $\OpL$ and this decomposition yields
\begin{equation}
	\s (\Op) = \s (\OpL) = 
	\ov{\cup_{l\in\dN_0}\cup_{k = 1}^{c(d,l)}\s(\OpLf^{[l]})} 
	= 
	\ov{\cup_{l\in\dN_0}\s(\OpLf^{[l]})}.
\label{eqn:dec_spec_un}
\end{equation}
For further use, for all $d\geq3$ and $l\geq0$, we define $\mathsf{Co}(Q_{\aa,\Gamma_\tt}^{[l]})$ as
	\[
		\mathsf{Co}(Q_{\aa,\Gamma_\tt}^{[l]})  := \left\{
	\begin{array}{lll}
		\cC_{0,0}^{\infty}(\ov{\R_+^2}),
		& \text{when }(d,l) = (3,l),
		&l>0,\\
		\cC_{0}^{\infty}(\ov{\R_+^2}),
		& \text{otherwise.}
	\end{array}
	\right.
	\]
With this definition, $\mathsf{Co}(Q_{\aa,\Gamma_\tt}^{[l]})$ is a form core of $Q_{\aa,\Gamma_\tt}^{[l]}$ (\cf Proposition \ref{prop:core_fdcyl}).


\subsection{Flat metric} 
In this subsection, after reformulating the problem in flat metric, we study the quadratic forms $\frmLf^{[l]}$ with the help of a unitarily equivalent form. 
First, we formulate the following proposition.
\begin{prop}\label{def:unitary}
	Let $\cH$ and $\cG$ be two Hilbert spaces and let $\sfU\colon \cH\arr \cG$ be a unitary operator. 
	Let $\frs$ be a closed, densely defined, symmetric and semibounded quadratic form on the Hilbert space $\cH$. 
	Define the quadratic form $\frt$ by
	\[
		\frt[u] := \frs[\sfU^{-1}v],\qquad \dom\frt := \sfU(\dom\frs).
	\]
	With this definition, the following statements hold.
	\begin{myenum}
		\item The form $\frt$ is closed, densely defined, symmetric, and semibounded on the
			Hilbert space $\cG$; and we say that $\frs$ and $\frt$ are unitarily equivalent.
		\item	The respective self-adjoint operators $\sfT$ and $\sfS$ (associated with these forms) are 
			unitarily equivalent and the relation $\sfS = \sfU^{-1}\sfT\sfU$ holds.
		\item	If $\mathsf{Co}(\frs)$ is a form core of $\frs$, then $\mathsf{U}(\mathsf{Co}(\frs))$ is a form core of $\frt$.
	\end{myenum}
\end{prop}
Proposition~\ref{def:unitary} is a direct consequence of the definition of $\frt$ and the definition of the unitary transform.

Secondly, we introduce the following unitary transform 
\begin{equation}\label{eqn:unitary_flat}
	\sfU \colon L^2(\dR_+^2,r^{d-2}\dd r\dd z) \arr L^2(\dR_+^2),
	\qquad 
	(\sfU u)(r,z) := r^{(d-2)/2}u(r,z)\equiv\wt{u}(r,z).
\end{equation}
\begin{prop} 
	Let $d \geq 3$ and $l\in\dN_0$ be such that 
	$(d,l)\neq(3,0)$. 
	Then $\frmLf^{[l]}$ is unitarily equivalent to the quadratic form
	\begin{equation*}
	\begin{split}
		\frmLft^{[l]}[\wt{u}] 
		& := 
		\int_{\dR_+^2}\Big(
		|\p_r \wt{u}|^2 + |\p_z\wt{u}|^2 + 
		\frac{\gamma(d,l)}{r^2}|\wt{u}|^2\Big) \dd r \dd z 
		- 
		\alpha\int_{\dR_+}|\wt{u}(s\sin\theta,s\cos\theta)|^2\dd s,\\
		\dom\frmLft^{[l]} & := \sfU ( \dom\frmLf^{[l]} ),
	\end{split}	
	\end{equation*}
	where $\gamma(d,l) = l (l+d-3) + \frac14(d-2)(d-4) \geq 0$. 
	Moreover, 
	$\mathsf{U}(\mathsf{Co}(\frmLf^{[l]}))$ is a form core of $\frmLft^{[l]}$ 
	and any function $\wt{u}\in\mathsf{U}(\mathsf{Co}(\frmLf^{[l]}))$ 
	satisfies $\restric{\wt{u}}{r=0} = 0$.
	\label{prop:fq_flatmetric}
\end{prop}

\begin{proof}
	Let $(d,l)\neq(3,0)$. First, 
	we define the quadratic form $\frmLft^{[l]}$ as
	\[
		\frmLft^{[l]}[\wt{u}] := \frmLf^{[l]}[\sfU^{-1}\wt{u}], \qquad \dom\frmLft^{[l]} := \sfU(\dom\frmLf^{[l]}),
	\]
	with $\sfU$ as in \eqref{eqn:unitary_flat}.
	By Proposition~\ref{def:unitary}\,(i) 
	we know that $\frmLft^{[l]}$ is a closed, densely defined, symmetric, and semibounded quadratic 
	form on $L^2(\R_+^2)$. 
	To get the expression of $\frmLft^{[l]}$ stated in Proposition~\ref{prop:fq_flatmetric}, 
	it is sufficient to check it on a form core of $\frmLft^{[l]}$. 
	Consequently, let us choose $u\in\mathsf{Co}(\frmLf^{[l]})$ and set $\wt u := \sfU u$. Then, we have
	\begin{equation*}
	\begin{split}
		\frmLft^{[l]}[\wt{u}] = \frmLf^{[l]}[r^{-(d-2)/2}\wt{u}] &=
		\int_{\dR_+^2}|\p_r(r^{-(d-2)/2}\wt{u})|^2 r^{d-2}\dd r\dd z 
		+ \int_{\dR_+^2} 
		\Big(|\p_z\wt{u}|^2 + \frac{l(l+d-3)}{r^2}|\wt{u}|^2\Big)
		\dd r\dd z\\
		&\qquad\qquad\qquad\qquad- \alpha\int_{\dR_+}
		|\wt{u}(s\sin\theta,s\cos\theta)|^2\dd s.
	\end{split}
	\end{equation*}
	A simple computation yields
	\begin{equation}
		|\p_r u|^2 =
		|\p_r(r^{-(d-2)/2}\wt{u})|^2 = 
		\frac{1}{r^{d-2}}|\p_r\wt{u}|^2 
		+ 
		\frac{(d-2)^2}{4r^d}|\wt{u}|^2 
		- 
		\frac{d-2}{2r^{d-1}}\p_r(|\wt{u}|^2).
	\label{eqn:der_comp}
	\end{equation}
	Integrating over $r\in\dR_+$ the last term in \eqref{eqn:der_comp}, 
	and integrating by parts, we end up with
	\begin{equation}
		\int_{\dR_+}
		\frac1{r^{d-1}}\p_r(|\wt{u}|^2)r^{d-2}\dd r 
		= 
		\int_{\dR_+}
		\frac1{r}\p_r(|\wt{u}|^2)\dd r 
		= \lim_{r\arr+\infty}(r^{-1}|\wt{u}|^2)-\lim_{r\arr 0}
		(r^{-1}|\wt{u}|^2) 
		+ 
		\int_{\dR_+} \frac1{r^2}|\wt{u}|^2\dd r.
	\label{eqn:ipp}
	\end{equation}
	The two limits in~\eqref{eqn:ipp} make sense and both are equal to zero.
	Indeed, we have $r^{-1}|\wt{u}|^2 = |u|^2 r^{d-3}$, 
	and $u$ is compactly supported,
	so 
	$|\wt{u}|^2 = |u|^2r^{d-2}\arr0$
	as $r\arr +\infty$ and the first limit in~\eqref{eqn:ipp} is zero. 
	Now, when $d=3$, $u\in\cC_{0,0}^\infty(\ov{\R_+^2})$ so,
	$r^{-1}|\wt{u}|^2 = |u|^2\arr 0$ as $r \arr 0$. When $d\geq4$, $u\in\cC_{0}^\infty(\ov{\R_+^2})$ and $r^{-1}|\wt{u}|^2 = |u|^2r^{d-3}\arr 0$ as $r \arr 0$.
	Hence, \eqref{eqn:ipp} rewrites as
	\begin{equation}
		\int_{\dR_+}
		\frac1{r^{d-1}}\p_r(|\wt{u}|^2)r^{d-2}\dd r 
		= 
		\int_{\dR_+} \frac1{r^2}|\wt{u}|^2\dd r.
	\label{eqn:ipp_fin}
	\end{equation}
	Using~\eqref{eqn:der_comp} and~\eqref{eqn:ipp_fin} we get the desired expression for $\frmLft^{[l]}$. 
\end{proof}

Now, we are ready to prove the following statement.

\begin{prop} 
	Let $d\geq3$ and $l\in\dN_0$ be such that $(d,l)\neq(3,0)$. Then we have
	\begin{equation*}
		\inf\s(\OpLf^{[l]}) \geq -\alpha^2/4.
	\end{equation*}
\label{prop:red_axy}
\end{prop}

\begin{proof}
	Instead of working with the operator $\OpLf^{[l]}$ and its associated form $\frmLf^{[l]}$  
	we work with the unitarily equivalent quadratic form 
	in the flat metric $\frmLft^{[l]}$. We want to apply the min-max principle to the quadratic form $\frmLft^{[l]}$. 
	To do so, it is sufficient to apply it with test functions in $\mathsf{U}(\mathsf{Co}	(\frmLf^{[l]}))$,
	where $\sfU$ is as in \eqref{eqn:unitary_flat}. 
	Let $\wt{u}\in\mathsf{U}(\mathsf{Co}(\frmLf^{[l]}))$, we have
	\begin{equation*}
		\frmLft^{[l]}[\wt{u}] 
		\geq 
		\int_{\dR_+^2}\big(|\p_r\wt{u}|^2 + |\p_z\wt{u}|^2\big) \dd r \dd z 
		- 
		\alpha\int_{\dR_+}|\wt{u}(s\sin\theta,s\cos\theta)|^{2}\dd s.
	\end{equation*}
	Thanks to Proposition \ref{prop:fq_flatmetric} we know that 
	$\wt{u}\in H_0^1(\dR_+^2)$. $\wt{u}$ can be extended by zero to the whole plane $\R^2$, defining a function $\wt{u}_0\in H^1(\R^2)$. We obtain
	\begin{equation}\label{eqn:min_maxbas}
		\frmLft^{[l]}[\wt{u}] 
	 	\ge
		\int_{\dR^2}\big(|\p_r\wt{u}_0|^2 + |\p_z\wt{u}_0|^2\big)
		\dd r \dd z - 
		\alpha\int_{\dR}|\wt{u}_0(s\sin\theta,s\cos\theta)|^{2}\dd s.
	\end{equation}
	The quadratic form on the right-hand side is the one of a 
	Schr\"odinger operator with an attractive $\delta$-interaction of strength 
	$\alpha > 0$ supported on a straight line in $\dR^2$. Its
	spectrum can be computed \textit{via} separation of variables and is $[-\aa^2/4,+\infty)$.
	The min-max principle applied to the form on the right hand side of \eqref{eqn:min_maxbas} yields
	\begin{equation*}
		\frmLft^{[l]}[\wt{u}] 
		\geq 
		-(\alpha^2/4)\|\wt{u}_0\|^2_{\dR^2}
		=
		-(\alpha^2/4)\|\wt{u}\|^2_{\dR^2_+}.
	\end{equation*}
	Finally, we get the inequality applying the min-max principle to $\frmLft^{[l]}$.
\end{proof}

Combining the structure of the essential spectrum, stated in 
Theorem~\ref{thm:struc_spec}, with Proposition \ref{prop:red_axy} we obtain the following corollary.

\begin{cor}\label{cor:spec}
	Let the self-adjoint operator $\Op$ be as in
	Definition~\ref{def:Op} and the self-adjoint operator $\OpLf^{[0]}$
	be as in \eqref{eqn:fib_dec}. 
	\begin{myenum}
		\item For $d = 3$, $\sd(\Op) = \sd(\OpLf^{[0]})$ holds and the multiplicities
		   		 of the corresponding eigenvalues coincide. 
		\item For $d \ge 4$, $\sd(\Op) = \varnothing$ holds.
	\end{myenum}
\end{cor}
In dimension $d\geq4$, it proves Theorem \ref{thm:struc_spec}\,(iii) about the emptiness of $\sd(\Op)$. In dimension $d=3$, it reduces the study of the eigenvalues of $\OpL$ to its axisymmetric fiber $\OpLf^{[0]}$. In the remainder of this paper, except if stated explicitly, $d=3$ and to simplify the notations, we drop the index $0$ and define
\begin{equation}
	\OpLf := \OpLf^{[0]},\qquad \frmLf := \frmLf^{[0]}.
	\label{eqn:not_axy}
\end{equation}
%

\subsection{Monotonicity of the eigenvalues} 

We prove Proposition~\ref{prop:non_decr_eigv} about the monotonicity of the eigenvalues of $\Op$ with respect to the half-opening angle
of the underlying cone $\cC$. Thanks to Corollary~\ref{cor:spec}, we know that we only have to focus on the axisymmetric 
fiber $\OpLf$. We describe the transition from the fiber 
form $\frmLf$ (in~\eqref{eqn:not_axy}) on the meridian domain to a unitarily equivalent
form on the inclined half-plane. 
This transition will be useful in the proof of 
Proposition~\ref{prop:non_decr_eigv} as well
as in further considerations. 

To this end, first, we define the rotation
\begin{equation}\label{def:m}
	s = z\cos\theta + r\sin\theta,\quad t=-z\sin\theta+r\cos\theta,
\end{equation} 
that transforms the meridian domain $\dR^2_+$ into the inclined half-plane (\cf Figure~\ref{fig:inchalp}) 
\begin{equation}\label{Omega_tt}
	\Omega_\tt := \{ (s,t)\in\dR^2 : s \sin\tt + t \cos\tt > 0\}. 
\end{equation}	
The meridian ray $\Gt$, defined in \eqref{def:Gt}, becomes the ray
\begin{equation}\label{def:G}
	\G :=  \{(s,0)\in\Omega_\tt \colon s > 0\}.
\end{equation}	

\begin{figure}[h!]
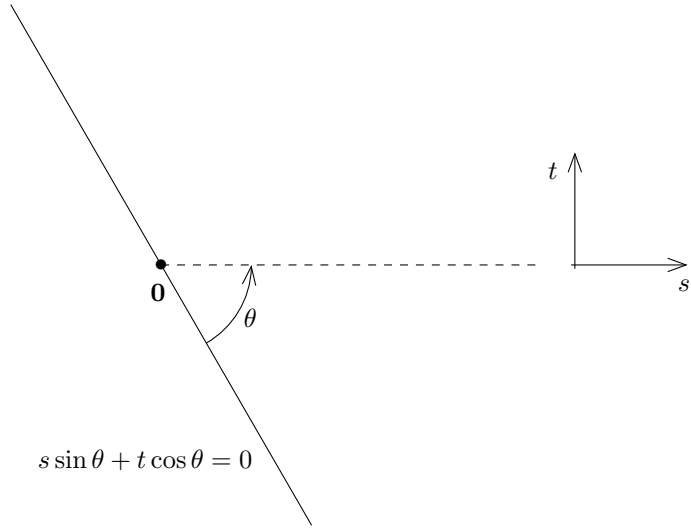

\figinit{1cm}
\figpt 0:(0,0)
\figpt 1:(5,0)

\figpt 2:(0,3.46410161)

\figpt 3:(-2,3.46410161)
\figpt 4:(2,-3.46410161) 

\figpt 5:(0.25,-0.4330127)
\figpt 6:(0.5,0)

\figpt 7:(0.6,-0.7)

\figpt 8:(1.5,-2.598076)

\figpt 9:(5.45,0)
\figpt 10:(7,0)

\figpt 11:(5.5,-0.05)
\figpt 12:(5.5,1.5)
\figdrawbegin{}
\figdrawline [3,4]
\figdrawarrow [9,10]
\figdrawarrow [11,12]
\figdrawarrowcircP 0;1.2[5,6]
\figset (dash=8)
\figdrawline [0,1]
\figdrawend

\figvisu{\figBoxA}{}{
\figwritee 7:{$\theta$} (0.5)
\figwrites 10:{$s$} (0.2)
\figwritesw 12:{$t$} (0.2)
\figwritew 8:{$s\sin\tt + t\cos\tt = 0$} (0.2)
\figset write(mark=$\bullet$)
\figwrites 0:{$\mathbf{0}$} (0.25)
}

\centerline{\box\figBoxA}
\caption{The inclined half-plane $\Omega_\tt$. The dashed line is the support of the $\delta$-interaction.}
\label{fig:inchalp}
\end{figure}

We associate the unitary operator
\begin{equation}\label{def:U}
	\sfU\colon L^2(\dR^2_+) \rightarrow L^2(\Omega_\tt),
	\qquad
	(\sfU u)(s,t)  := \wh{u}(s,t)=  u(s\sin\tt + t\cos\tt,s\cos\tt - t\sin\tt)
\end{equation}
with Rotation \eqref{def:m}. A straightforward computation yields that the quadratic form $\frmLf$ defined in \eqref{eqn:not_axy}, is unitarily equivalent to $\frmLfr$, defined by
\begin{equation}
\begin{split}
	\frmLfr[\wh{u}] & 
	:= 
	\int_{\Omega_\tt} 
	(|\p_s\wh{u}|^2+ |\p_t\wh{u}|^2)(s\sin\tt +  t\cos\tt) 
	\dd s \dd t 
	- 
	\alpha \int_{\Gamma}|\wh{u}(s,0)|^2s\sin\tt \dd s,\\
	\dom\frmLfr & := \sfU(\dom \frmLf).
\end{split}	
\label{eqn:fq_turn}
\end{equation}
To avoid the dependence on $\theta$ of the domain $\dom \frmLfr$, we perform the change of variables $(s,t)\mapsto (\check{s},\check{t}) = (s\tan\theta,t)$ that transforms the domain $\Omega_\tt$ into $\Omega := \Omega_{\pi/4}$. Setting $\check{u}(\check{s},\check{t})=\wh{u}(s,t)$, for $u\in\dom \frmLf$ we get, for the Rayleigh quotient
\begin{equation*}
\begin{array}{lcl}
	\displaystyle\frac{\frmLf[u]}{\|u\|^2_{\dR^2_+}} & 
	= 
	& \displaystyle 
	\frac{\int_\Omega 
	(\tan^2\theta|\p_{\check{s}}\check{u}|^2 + |\p_{\check{t}}\check{u}|^2)
	(\check{s} + \check{t})\cos\tt\cot\tt\dd\check{s}\dd\check{t}-\alpha
	\int_{\check{s}>0}
	|\check{u}(\check{s},0)|^2\check{s}\sin\tt\cot^2\tt\dd\check{s}}
	{\int_{\Omega}|\check{u}|^2(\check{s}+\check{t})\cos\theta\cot\theta\dd \check{s}\dd\check{t}}\\
	&&\\
	&=& 
	\displaystyle
	\frac{\int_\Omega (\tan^2\theta|\p_{\check{s}}\check{u}|^2 
	+ 
	|\p_{\check{t}}\check{u}|^2)
	(\check{s} + \check{t})\dd\check{s}\dd\check{t}-
	\alpha\int_{\check{s}>0}|\check{u}(\check{s},0)|^2\check{s}\dd\check{s}}{
	\int_{\Omega}|\check{u}|^2(\check{s}+\check{t})\dd \check{s}\dd\check{t}}.
	\end{array}
\end{equation*}
The right hand side of the last equation is the Rayleigh quotient of a 
quadratic form acting on $L^2(\Omega;(\check{s}+\check{t})\dd\check{s}\dd\check{t})$. 
Because this form is unitarily equivalent to $\frmLf$ and because its Rayleigh quotients are 
nondecreasing functions of $\tt$, the claim of Proposition~\ref{prop:non_decr_eigv} follows from the min-max formulae 
for the eigenvalues.

\section{Spectral asymptotics of $\Op$ in the three-dimensional case}
\label{sec:couneig}
In this section we prove Theorem~\ref{th:acceig}. The idea is to exhibit a lower and an upper bound
for the counting function of the operator $\OpLf$. We recall that all along this section $d=3$.


\subsection{Proof of Theorem~\ref{th:acceig}}

The proof of Theorem~\ref{th:acceig} relies on the following 
two propositions, whose proofs are postponed to Subsections~\ref{ssec:l_bnd}
and~\ref{ssec:u_bnd}, respectively.

\begin{prop}\label{prop:lower_bnd}
	Let $\aa > 0$ and $\tt\in(0,\pi/2)$. We have
	\[
		\liminf_{E\arr 0+}
			\frac{\N_{-\aa^2/4 - E}(\Op)}{|\ln E|} \ge \frac{\cot\tt}{4\pi}.
	\]
\end{prop}

\begin{prop}\label{prop:upper_bnd}
	Let $\aa > 0$ and $\tt\in(0,\pi/2)$. We have
	\[
		\limsup_{E\arr 0+}
			\frac{\N_{-\aa^2/4 - E}(\Op)}{|\ln E|} \le \frac{\cot\tt}{4\pi}.
	\]
\end{prop}
Propositions~\ref{prop:lower_bnd} and~\ref{prop:upper_bnd} imply
\[
	\lim_{E\arr 0+}
		\frac{\N_{-\aa^2/4 - E}(\Op)}{|\ln E|} = \frac{\cot\tt}{4\pi},
\]
which proves Theorem~\ref{th:acceig}.

\subsection{Auxiliary one-dimensional operators}\label{ssec:PI}

In this subsection we discuss some spectral properties of one-dimensional model
Schr\"odinger operators, which are used in the proofs of Propositions~\ref{prop:lower_bnd} 
and~\ref{prop:upper_bnd}.

Let us start by studying two Schr\"odinger operators with a point $\delta$-interaction.
For the spectral theory of one-dimensional Schr\"odinger operators with 
point $\delta$-interactions we refer to \cite[Chs. I.3, II.2, III.2]{AGHH}, 
the review paper \cite{KM13} and the references therein.

For $L > 0$, we define the interval $I := (- L,L)$ 
and introduce the Hilbert space $(L^2(I),(\cdot,\cdot)_I)$. 
Let us fix $\alpha > 0$ and 
consider the following two symmetric sesquilinear forms 
\begin{equation*}
\begin{split}
	\frmID [\varphi,\psi] & :=
			 (\varphi^\pp,\psi^\pp)_I - \aa\varphi(0)\ov{\psi(0)},
	 \qquad 
	\dom\frmID := H^1_0(I),\\
	\frmIN [\varphi,\psi] & := 
			(\varphi^\pp,\psi^\pp)_I - \aa\varphi(0)\ov{\psi(0)},
	\qquad 
	\dom\frmIN := H^1(I),
\end{split}
\end{equation*}
one can verify that both forms are
closed, densely defined, and semibounded in $L^2(I)$. 
For $\varphi \in H^2(I \setminus \{0\})$ set 
$\varphi_+ := \restric{\varphi}{(0,L)} $, $\varphi_- := \restric{\varphi}{(-L,0)} $ 
and 
$[\varphi^\pp](0) := \varphi^\pp(0+) - \varphi^\pp(0-)$.
Thanks to the first representation theorem (\cite[Ch. VI, Thm. 2.1]{K}), 
each of these quadratic forms is associated with a unique self-adjoint operator acting on $L^2(I)$.
The respective operators are given by
\begin{equation}
\begin{split}
	\opID \varphi 
		& = - (\varphi^\dpp_+\oplus \varphi^\dpp_-),\quad
	\dom\opID 
		= \big\{\varphi\in H^2(I\setminus\{0\}) \colon
			\varphi(\pm L) = 0,~[\varphi^\pp](0) = - \aa\varphi(0 \pm )
		 \big\},\\
	\opIN \varphi 
		& = - (\varphi^\dpp_+\oplus \varphi^\dpp_-),\quad	
	\dom\opIN 
		 = \big\{\varphi\in H^2(I\setminus\{0\}) \colon
		 	\varphi^\pp(\pm L) = 0,~[\varphi^\pp](0) = - \aa\varphi(0 \pm )
		\big\}.
\end{split}	
\end{equation}
As $H^1_0(I)$ and $H^1(I)$ are compactly embedded into $L^2(I)$ both operators 
$\opID$ and $\opIN$ have compact resolvent and their spectra 
consist of non-decreasing sequences of eigenvalues. 
The understanding of the first two eigenvalues of $\opID$
and $\opIN$, as functions of $L$, is important for our purposes. 
The next proposition is essentially proven 
in the paper~\cite{EY02}.

\begin{prop}\cite[Prop. 2.4, Prop. 2.5]{EY02}\label{prop:mod1D}
	The following statements hold:
	\begin{myenum}
		\item there exist $L_{\rm D} = L_{\rm D}(\aa) > 0$ and 
			$C_{\rm D} = C_{\rm D}(\aa),C_{\rm D}^\pp = C_{\rm D}^\pp(\aa) 
			> 0$ such that
			for all $L \ge L_{\rm D}$
			\[
				-\frac{\aa^2}{4} <  E_1(\opID) 
				\le 
				-\frac{\aa^2}{4} 
				+	C_{\rm D} e^{- C_{\rm D}^\pp L};
			\]	
		\item there exist $L_{\rm N} = L_{\rm N}(\aa) > 0$ and 
			$C_{\rm N} = C_{\rm N}(\aa), C_{\rm N}^\pp = C_{\rm N}^\pp(\aa)
			 > 0$ 
			such that for all $L \ge L_{\rm N}$
			\[
				-\frac{\aa^2}{4}  > E_1(\opIN)  \ge 
				-\frac{\aa^2}{4} - C_{\rm N} e^{- C_{\rm N}^\pp L};
			\]	
		\item for all $L > 0$, $E_2(\opID), E_2(\opIN) \ge 0$.
	\end{myenum}
\end{prop}

Further, we recall a result about another family of 
one-dimensional Schr\"{o}dinger operators. Let us introduce the interval $J := (1,+\infty)$ 
and the Hilbert space $(L^2(J),(\cdot,\cdot)_{J})$. 
Let $c>0$ be a positive constant and $V$ be the following potential $V(x) := x^{-2}$ on $J$. 
We consider the following closed, densely defined, symmetric, and semibounded sesquilinear forms
\begin{equation}\label{eq:op_KS}
\begin{split}
	\frmKSD [\varphi,\psi] 
		& = (\varphi^\pp,\psi^\pp)_{J} - c(V\varphi,\psi)_{J},
	\qquad 
	\dom\frmKSD := H^1_0(J),\\
	\frmKSN [\varphi,\psi] 
		& = (\varphi^\pp,\psi^\pp)_{J}
			 - c(V\varphi,\psi)_{J},
	\qquad 
	\dom\frmKSN := H^1(J),
\end{split}	
\end{equation}
in the Hilbert space $L^2(J)$.

By a compact perturbation argument one can show that
$\sess(\frmKSD) = \sess(\frmKSN) = [0,+\infty)$.
Below we provide a result on spectral asymptotics
of $\frmKSN$ and $\frmKSD$, essentially proven in~\cite{KS88} (\cf also~\cite{HM08} for further generalisations). 
\begin{thm}\cite[Thm. 1]{KS88}\label{th:KS88}
	Let $c > \frac14$. Then it holds that
	\begin{equation*}
		\N_{-E}(\frmKSD) \sim \N_{-E}(\frmKSN)
		\sim
		\frac{1}{2\pi}\, \sqrt{c - \frac14}\ |\ln E|,\qquad E \arr 0+.
	\end{equation*}
	In particular, we have $\#\sd(\frmKSD) = \#\sd(\frmKSN) =\infty$.
\end{thm}

\subsection{A lower bound on the counting function of $\Op$}
\label{ssec:l_bnd}

In this subsection we prove Proposition~\ref{prop:lower_bnd}.

\begin{proof}[Proof of Proposition~\ref{prop:lower_bnd}]
	Thanks to Corollary~\ref{cor:spec}\,(i) and unitary equivalence of the forms
	$\frmLf$ (in \eqref{eqn:not_axy}) and $\frmLfr$  (in \eqref{eqn:fq_turn})
	it is sufficient to prove
	\[
		\liminf_{E\arr 0+}
			\frac{\N_{-\aa^2/4 - E}(\frmLfr)}{|\ln E|} \ge \frac{\cot\tt}{4\pi}.
	\]		
	We split the proof of this inequality into three steps.

	\myemph{Step 1.}
	Let $R>0$, we define the intervals 
	$I_1 := ((\sin\tt)^{-1}+R,+\infty)$, $I_2 := (-R\tan\tt,R\tan\tt)$,
	and the half-strip 
	\[
		\Pi :=
			\big\{(s,t)\in\Omega_\tt \colon s > (\sin\tt)^{-1} + R, |t| < R\tan\tt
			\big\} = I_1\times I_2\subset\Omega_\tt;
	\]
	where the $(s,t)$-variables and $\Omega_\tt$ are related to the physical domain 		
	through the change of variables~\eqref{def:m} (\cf Figure~\ref{fig:defsigma}). 
	We 	introduce the quadratic form $\frmLfr^{\Pi}$, 
	defined as
	\begin{equation*}
	\begin{split}	
		\frmLfr^{\Pi}[u] & := 
		\int_{\Pi}(|\p_s u|^2 + |\p_t u|^2)(s\sin\theta+t\cos\theta)\dd s \dd t 
		- \alpha\int^{\infty}_{(\sin\tt)^{-1} + R}|u(s,0)|^2s\sin\theta\dd s,\\
		\dom\frmLfr^\Pi & := 
		\big\{u\in\dom\frmLfr \colon u = 0 \text{ on } 
		\Omega_\tt\setminus	\ov{\Pi}\big\}.
	\end{split}	
	\end{equation*}
	\begin{figure}[h!]
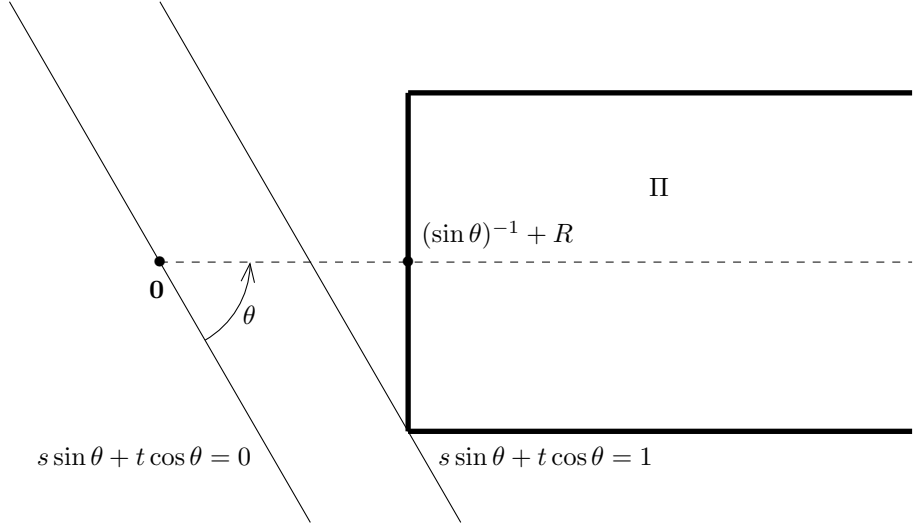

\figinit{1cm}
\figpt 0:(0,0)
\figpt 1:(10,0)

\figpt 2:(0,3.46410161)

\figpt 3:(-2,3.46410161)
\figpt 4:(2,-3.46410161) 

\figpt 5:(0.25,-0.4330127)
\figpt 6:(0.5,0)

\figpt 7:(0.6,-0.7)

\figpt 8:(1.5,-2.598076)

\figpt 9:(5.45,0)
\figpt 10:(7,0)

\figpt 11:(5.5,-0.05)
\figpt 12:(5.5,1.5)

\figpt 13:(0,3.46410161)
\figpt 14:(4,-3.46410161)
\figpt 15:(3.5,-2.598076)

\figpt 16:(3.3,-2.251666)
\figpt 17:(3.3,2.251666)
\figpt 18:(10,-2.251666)
\figpt 19:(10,2.251666)
\figpt 20:(3.3,0)
\figpt 21:(6,1)
\figdrawbegin{}
\figdrawline [3,4]
\figdrawline [13,14]
\figdrawarrowcircP 0;1.2[5,6]
\figset (width=2)
\figdrawline [16,17]
\figdrawline [16,18]
\figdrawline [17,19]
\figset(width=0)
\figset (dash=8)
\figdrawline [0,1]
\figdrawend

\figvisu{\figBoxA}{}{
\figwritee 7:{$\theta$} (0.5)
\figwritee 21:{$\Pi$} (0.5)
\figwritew 8:{$s\sin\tt + t\cos\tt =0$} (0.2)
\figwritee 15:{$s\sin\tt + t\cos\tt =1$} (0.2)
\figset write(mark=$\bullet$)
\figwrites 0:{$\mathbf{0}$} (0.25)
\figwritene 20:{$(\sin\tt)^{-1} + R$} (0.25)
}

\centerline{\box\figBoxA}
\caption{The inclined half-plane $\Omega_\tt$ and the half-strip $\Pi$. The dashed line is the support
of $\delta$-interaction.}
\label{fig:defsigma}
\end{figure}

	%

	Any $u\in\dom\frmLfr^{\Pi}$ can be extended by zero, 
	defining $u_0\in\dom\frmLfr$ such that $\frmLfr^{\Pi}[u] = \frmLfr[u_0]$. 
	Then, the min-max principle yields
	\begin{equation}
		\N_{-\alpha^2/4 - E}(\frmLfr^{\Pi})\leq\mathcal{N}_{-\alpha^2/4 - E}(\frmLfr).
	\label{eqn:low_bound1}
	\end{equation}

	\myemph{Step 2.}
	Let us define the unitary transform
	\[
		\sfU \colon L^2(\Pi; (s\sin\tt + t\cos\tt) \dd s \dd t) 
		\arr L^2(\Pi),\qquad
		(\sfU u)(s,t) := \sqrt{s\sin\tt + t\cos\tt}u(s,t).
	\]
	By straightforward computation, the form $\frmLfr^\Pi$ is unitarily
	equivalent, \textit{via} $\sfU$, to the form
	\begin{equation}
	\begin{split}
		\frmLfrt^{\Pi}[u] & 
		:= 
		\int_{\Pi}
		\Big(|\p_s u|^2 + |\p_t u|^2 - 
		\frac1{4(s\sin\theta+t\cos\theta)^2}|u|^2\Big)
		\dd s \dd t - 
		\alpha\int^\infty_{(\sin\tt)^{-1} + R}|u(s,0)|^2\dd s,\\
		\dom\frmLfrt^{\Pi} & := H^1_0(\Pi).
	\end{split}	
	\end{equation}
	Next, we bound $(s\sin\theta+t\cos\theta)^2$ from above by 
	$\sin^2\theta(s+R)^2$, obtaining
	\begin{equation}
		\frmLfrt^{\Pi}[u]\leq
		\int_{\Pi}\Big(
		|\p_s u|^2 + |\p_t u|^2 - 
		\frac1{4\sin^2\theta(s+R)^2}|u |^2\Big) \dd s \dd t - 
		\alpha\int ^\infty _{(\sin\tt)^{-1} + R}|u(s,0)|^2\dd s.
	\label{eqn:up_bound_tens}
	\end{equation}
	The right hand side of~\eqref{eqn:up_bound_tens} has two blocks with 
	separated variables. Since the Hilbert space $L^2(\Pi)$ decomposes as 
	$L^2(I_1) \myotimes L^2(I_2)$,
	the form on the right hand side of~\eqref{eqn:up_bound_tens} 
	admits the respective representation
	\begin{equation}\label{eq:tensor}
		\frq_{1,R} \otimes \fri_2 + \fri_1 \otimes \frq_{2,R},
	\end{equation}
	where $\fri_k$, $k=1,2$, is the form of the identity operator on 
	$L^2(I_k)$;	the forms $\frq_{k, R}$, $k=1,2$,
	are defined in the Hilbert spaces $L^2(I_k)$, $k=1,2$, 
	as
	\begin{subequations}\label{??}
	\begin{align*}
		\frq_{1,R}[\varphi, \psi]  & :=
		(\varphi^\pp, \psi^\pp)_{I_1} -(V_R\varphi,\psi)_{I_1},
		&
		\dom\frq_{1, R} := H^1_0(I_1),\\
		\frq_{2,R}[\varphi, \psi] & :=
		(\varphi^\pp, \psi^\pp)_{I_2} - \alpha\varphi(0)\ov{\psi(0)},
		&	\dom\frq_{2,R} := H^1_0(I_2);
	\end{align*}	
	\end{subequations}
	where the potential 
	$V_R$ is given by $V_R(s) := \frac{1}{4\sin^2\tt(s+R)^2}$.
	Using the unitary operator 
	\[	
		\sfV\colon L^2(1,+\infty)\arr L^2(I_1),
		\qquad
		(\sfV\varphi)(s) := \tfrac{1}{\sqrt{2 R + (\sin\tt)^{-1}}}
				\varphi\Big(\tfrac{s + R}{2 R + (\sin\tt)^{-1}}\Big),
	\]
	one finds that
	the forms $\frq_{1,R}$ and $(2 R + (\sin\tt)^{-1})^{-2}\frmKSD$	with $c = 1/(4\sin^2\tt)$ are
	unitarily equivalent (\cf Subsection~\ref{ssec:PI}).
	
	Thanks to \eqref{eqn:low_bound1}, \eqref{eqn:up_bound_tens}, \eqref{eq:tensor}, and the min-max principle, we obtain
	\begin{equation*}
	\begin{array}{lcl}
		\N_{-\alpha^2/4 - E}(\frmLfr) &
		\geq& 
		\#
		\big\{(k,j)\in\dN^2 \colon 
		E_k(\frq_{1,R}) + E_j(\frq_{2,R})\leq-\alpha^2/4-E
		\big\}\\  [0.6ex]
		&=&\displaystyle\sum_{j=1}^\infty
		\#\{k\in\dN \colon E_k({\frq_{1,R}}) \leq 
		-\alpha^2/4-E-E_j(\frq_{2,R})\}\\ [0.6ex]
		&\geq&\#
		\big\{k\in\dN \colon 
		E_k({\frq_{1,R}})\leq -\alpha^2/4-E-E_1(\frq_{2,R})\big\}.
	\end{array}
	\end{equation*}
	This inequality yields
	\begin{equation}
		\N_{-\alpha^2/4 - E}(\frmLfr) \geq 
		\N_{-\alpha^2/4 -E - E_1(\frq_{2,R})}(\frq_{1,R}).
	\label{eqn:cnt_ineq3}
	\end{equation}

    \myemph{Step 3.}
	Now, we choose $R$ depending on the spectral parameter $E>0$ as follows
	\begin{equation*}
		R = R(E) := M |\ln E |,\qquad M > 0,
	\end{equation*}
	in particular, we have $R(E) \arr +\infty$ as $E\arr 0+$.
	Let the constants $C_{\rm D}$, $C_{\rm D}'$ and $L_{\rm D}$ be as in Proposition~\ref{prop:mod1D}\,(i).
	Next, we choose	$M > 0$ sufficiently large such that
	$C_{\rm D}'M\tan\tt > 1$. Then for $E > 0$ 
	sufficiently small such that $\ln E < 0$ and $M |\ln E| \tan \tt > L_{\rm D}$, by Proposition~\ref{prop:mod1D}\,(i) we have
	\[
		|\aa^2/4 + E_1(\frq_{2,R(E)})| 
		\le C_{\rm D} \exp(C_{\rm D}' M\tan\tt \ln E)
		= 
		C_{\rm D} E^{C_{\rm D}'M\tan\tt} = o(E), \qquad E\arr 0+.
	\]
	Hence,
	\begin{equation}\label{eq:f(E)}
		f(E) 
		:= 
		\big(\alpha^2/4  + E + E_1(\frq_{2,R(E)})\big) \big(2 R(E) + (\sin\tt)^{-1}\big)^2
		= 
		4 M^2 E|\ln E|^2 + o(E|\ln E|^2), 	\qquad E\arr 0+.
	\end{equation}
	Using~\eqref{eqn:cnt_ineq3}, unitary equivalence of $\frq_{1,R}$ 
	and $(2 R + (\sin\tt)^{-1})^{-2}\frmKSD$ and Theorem~\ref{th:KS88} we get
	\[
	\begin{split}
		\liminf_{E \arr 0+}
		\frac{ \N_{-\aa^2/4 - E}(\frmLfr) }{ |\ln(E)|}
		& 
		\ge 
		\liminf_{E\arr 0+}
		\frac{\N_{- \aa^2/4 - E - E_1(\frq_{2,R(E)})}(
		\frq_{1,R(E)})}
		{ |\ln(E)|}\\
		& =
		\liminf_{E\arr 0+}
		\frac{ \N_{-f(E)}(\frmKSD)}
		{ |\ln(E)| }
		=
		\frac{\cot\tt}{4\pi}
		\liminf_{E \arr 0+}\frac{| \ln f(E) | }{|\ln E|}
		= \frac{\cot\tt}{4\pi},
	\end{split}	
	\]		
	where we used that $\frac{|\ln f(E)|}{|\ln E|} \arr 1$ as $E\arr 0+$ (\cf \eqref{eq:f(E)}).
	It ends the proof of Proposition \ref{prop:lower_bnd}.
\end{proof}


\subsection{An upper bound on the counting function of $\Op$}
\label{ssec:u_bnd}

The aim of this subsection is to prove Proposition~\ref{prop:upper_bnd}.
First, we provide an auxiliary lemma, whose proof is postponed to the end of the subsection.
To formulate this lemma, for $K > 0$, we define the domain
\[
	\Omega_\tt^{K} := 
	\big\{ (s,t) \in\dR^2\colon s\sin\tt + t\cos\tt > 2K \big\} \subset \Omega_\tt,
\]
and introduce the following symmetric quadratic form on the Hilbert space $L^2(\Omega_\tt^{K})$
\begin{equation}\label{eq:frm_aux}
\begin{split}
	\frmLWot [u] 
	& := 
	\int_{\Omega_\tt^{K}} |\p_s u|^2 + |\p_t u|^2 - \frac{|u|^2}{4(s\sin\tt + t\cos\tt)^2}\dd s \dd t 
	- 
	\alpha\int_{2K(\sin\tt)^{-1}}^\infty |u(s,0)|^2 \dd s,\\
	\dom\frmLWot & := H^1(\Omega_\tt^{K}).
\end{split}	 
\end{equation}
One can check that the form $\frmLWot$ is closed, densely defined and semibounded in $L^2(\Omega_\tt^{K})$. 

\begin{lem}\label{lem:bnd_aux}
	Let $\aa > 0$ and $\tt\in(0,\pi/2)$. For all $K > 0$ sufficiently large,
	the counting functions of $\Op$ and $\frmLWot$ satisfy
	\[
		\limsup_{E\arr 0+}
		\frac{\N_{-\aa^2/4 - E}(\Op)}{|\ln E|} 
		\le 
		\limsup_{E\arr 0+}
		\frac{\N_{-\aa^2/4 - E}(\frmLWot)}{|\ln E|}.
	\]
\end{lem}

Now we have all the tools to prove Proposition \ref{prop:upper_bnd}.




\begin{proof}[Proof of Proposition~\ref{prop:upper_bnd}]
	According to Lemma~\ref{lem:bnd_aux} it is sufficient to prove that for a fixed $K > 0$ sufficiently large, we have the bound
	\[
		\limsup_{E\arr 0+}
		\frac{\N_{-\aa^2/4 - E}(\frmLWot)}{|\ln E|} \le \frac{\cot\tt}{4\pi}
	\]	
	As in the proof of Proposition~\ref{prop:lower_bnd}, we split the proof of this inequality  
	into three steps.

	\myemph{Step 1.}
	Let us introduce the parameters: 
	$R > 0$, $m := \lfloor\sqrt{R}\rfloor$, $r := 2K (\sin\tt)^{-1}$ and the sequences $r_k := 3r + kR / m$, $d_k := (r_k\tan\tt)/2$  
	(for $k=0,1,2, \dots, m$). For the sake of convenience we set $r_{m+1} = +\infty$. We introduce the domains
	\[
		\Lambda_k := \big\{ (s,t) \in\dR^2 \colon s \in (r_k, r_{k+1}), 
		t \in (-d_k, d_k)\big\},
		\qquad	k = 0,1,\dots, m.
	\]	
	The inclusions
	$\Lambda_k \subset \Omega_\tt^{K}$ hold for all $k = 0,1,\dots,m$. 
	Indeed, for any $(s,t) \in \Lambda_k$ we have
	\[
		s\sin\tt + t\cos\tt > r_k \sin\tt - d_k\cos\tt = \frac{r_k\sin\tt}{2} 
		\ge \frac{3r\sin\tt}{2} = \frac{6K}{2} = 3K > 2K.
	\]	
	We also define the domains $\Lambda_{m+1}, \Lambda_{m+2} \subset \Omega_\tt^{K}$ 
	(\cf Figure~\ref{fig:subdomlambda}) as
	\begin{equation*}
		\Lambda_{m+1} := 
		\big\{(s,t)\in \Omega_\tt^{K} \colon s < 6K (\sin\tt)^{-1}, |t| < K (\cos\tt)^{-1} \big\},
		\qquad	
		\Lambda_{m+2}  := 
		\Omega_\tt^{K}\setminus \ov{\cup_{k=0}^{m+1}\Lambda_k}.
	\end{equation*}

	\begin{figure}[h!]
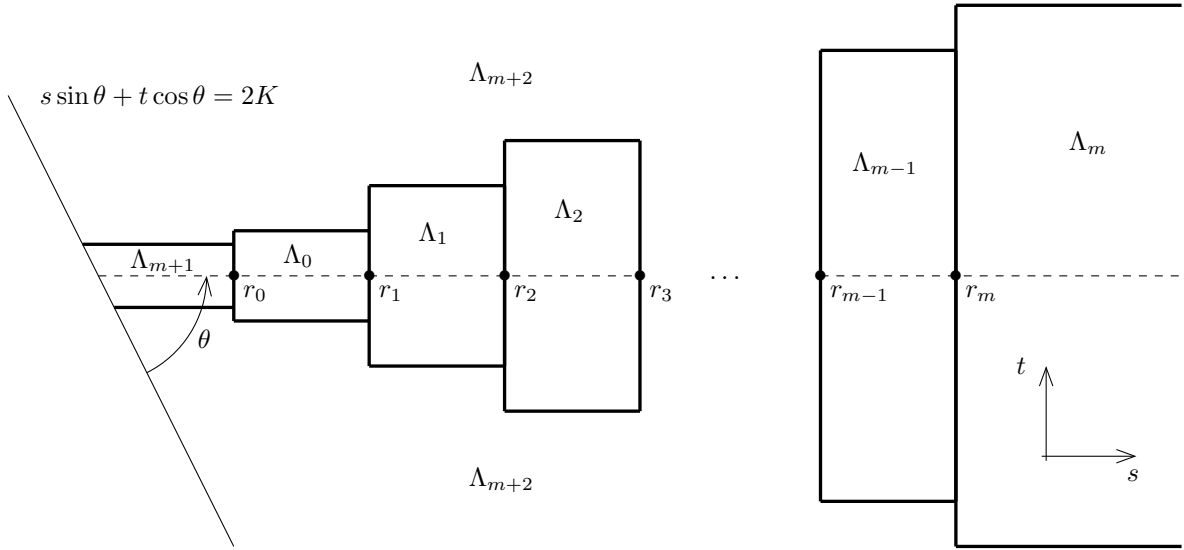

\figinit{1.2cm}

\figpt 0:(0,0)

\figpt 1:(-1,2)
\figpt 2:(1.5,-3)

\figpt 3:(6,0)

\figpt 4:(-0.175,0.35)
\figpt 5:(0.175,-0.35)
\figpt 401:(1.5,0.35)
\figpt 501:(1.5,-0.35)

\figpt 6:(1.5,0.5)
\figpt 7:(1.5,-0.5)

\figpt 8:(3,0.5)
\figpt 9:(3,-0.5)

\figpt 10:(3,1)
\figpt 11:(3,-1)
\figpt 12:(1.5,0.5)
\figpt 13:(1.5,-0.5)

\figpt 14:(4.5,1)
\figpt 15:(4.5,-1)

\figpt 16:(4.5,1.5)
\figpt 17:(4.5,-1.5)
\figpt 18:(6,1.5)
\figpt 19:(6,-1.5)

\figpt 20:(8,2.5)
\figpt 21:(8,-2.5)
\figpt 22:(9.5,2.5)
\figpt 23:(9.5,-2.5)

\figpt 24:(9.5,3)
\figpt 25:(9.5,-3)
\figpt 26:(12,3)
\figpt 27:(12,-3)

\figpt 28:(8,0)
\figpt 29:(12,0)

\figpt 30:(7,0)

\figpt 31:(1.5,0)
\figpt 32:(3,0)
\figpt 33:(4.5,0)
\figpt 34:(9.5,0)

\figpt 35:(2.25,0.25)
\figpt 36:(3.75,0.5)
\figpt 37:(5.25,0.75)
\figpt 38:(8.75,1.25)
\figpt 39:(11,1.5)

\figpt 40:(0.75,0.175)
\figpt 41:(4.5,2.25)
\figpt 42:(4.5,-2.25)

\figpt 43:(10.45,-2)
\figpt 44:(11.5,-2)

\figpt 45:(10.5,-2.05)
\figpt 46:(10.5,-1)

\figpt 47:(0.85,0)
\figpt 48:(0.35,-0.7)
\figpt 49:(1.1,-0.7)
\figdrawbegin{}
\figdrawline[1,2]
\figdrawarrow[43,44]
\figdrawarrow[45,46]
\figdrawarrowcircP 0;1.2[48,47]
\figset (dash=8)
\figdrawline[0,3]
\figdrawline[28,29]
\figset (dash=1)
\figset (width=1.25)
\figdrawline[4,401]
\figdrawline[5,501]
\figdrawline[6,7]
\figdrawline[10,11]
\figdrawline[12,8]
\figdrawline[9,13]
\figdrawline[10,14]
\figdrawline[11,15]
\figdrawline[14,15]
\figdrawline[16,18]
\figdrawline[17,19]
\figdrawline[16,17]
\figdrawline[18,19]
\figdrawline[20,22]
\figdrawline[21,23]
\figdrawline[20,21]
\figdrawline[22,23]
\figdrawline[24,26]
\figdrawline[25,27]
\figdrawline[24,25]
\figdrawend

\figvisu{\figBoxA}{}{
\figwriten 35:{$\Lambda_0$} (-0.15)
\figwriten 36:{$\Lambda_1$} (-0.15)
\figwriten 37:{$\Lambda_2$} (-0.15)
\figwriten 38:{$\Lambda_{m-1}$} (-0.15)
\figwriten 39:{$\Lambda_m$} (-0.15)
\figwriten 40:{$\Lambda_{m+1}$} (-0.15)
\figwriten 41:{$\Lambda_{m+2}$} (-0.15)
\figwriten 42:{$\Lambda_{m+2}$} (-0.15)
\figwrites 44:{$s$} (0.15)
\figwritew 46:{$t$} (0.15)
\figwritee 49:{$\tt$} (0)
\figwritee 1:{$s\sin\tt + t\cos\tt = 2 K$} (0.35)
\figwrites 30:{$\dots$} (0)
\figset write(mark=$\bullet$)
\figwritese 31:{$r_0$}(0.15)
\figwritese 32:{$r_1$}(0.15)
\figwritese 33:{$r_2$}(0.15)
\figwritese 3:{$r_3$}(0.15)
\figwritese 28:{$r_{m-1}$}(0.15)
\figwritese 34:{$r_m$}(0.15)
}

\centerline{\box\figBoxA}
\caption{Sketch of the different subdomains $\Lambda_k$ of $\Omega_\tt^{K}$ ($k\in\{0,\dots,m+2\}$). The dashed line is the support of the $\delta$-interaction.}
\label{fig:subdomlambda}
\end{figure}

	We introduce the notation $u_k = \restric{u}{\Lambda_k}$ ($k = 0,1,\dots,(m+2)$) and 
	we consider the following closed, densely defined, symmetric and semibounded quadratic 
	form $\frmLWone$ in $L^2(\Omega_\tt^{K})$, 
	defined as
	\begin{equation*}
	\begin{split}
		\frmLWone[u] 
		& := 
		\sum_{k=0}^{m+2} \|\nabla u_k\|^2_{\Lambda_k} -
		\int_{\Omega_\tt^{K}} \frac{|u|^2}{4(s\sin\tt + t\cos\tt)^2} \dd s \dd t
		-
		\alpha\int_{2K (\sin\tt)^{-1}}^\infty|u(s,0)|^2\dd s,\\
		\dom\frmLWone & :=\bigoplus_{k=0}^{m+2} H^1(\Lambda_k).
	\end{split}	 
	\end{equation*}
	This form admits a natural decomposition into parts corresponding to the sub-domains $\Lambda_k$
	\begin{equation*}
		\frmLWone[u] = \sum_{k=0}^{m+2} Q_{\alpha,\Lambda_k}[ u_k ],
	\end{equation*}
	where, for $k=0,1,\dots,(m+2)$, 
	the quadratic forms $Q_{\alpha,\Lambda_k}$ 
	have domains $\dom Q_{\alpha,\Lambda_k} := H^1(\Lambda_k)$ and are given by
	\begin{equation*}
	\begin{split}	
		Q_{\alpha,\Lambda_k}[u]  & :=
		\|\nabla u\|^2_{\Lambda_k} - \int_{\Lambda_k}\frac{|u|^2}{4(s\sin\tt + t\cos\tt)^2} \dd s \dd t
		- 
		\alpha\int_{r_k}^{r_{k+1}}|u(s,0)|^2\dd s, \quad k=0, \dots m,\\
		Q_{\alpha,\Lambda_{m+1}}[u]
		& := 
		\|\nabla u\|^2_{\Lambda_{m+1}}- \int_{\Lambda_{m+1}}\frac{|u|^2}{4(s\sin\tt + t\cos\tt)^2}\dd s \dd t
		- \alpha\int_{2K(\sin\tt)^{-1}}^{r_0}|u(s,0)|^2\dd s,\\
		Q_{\alpha,\Lambda_{m+2}}[u]&
		:= \|\nabla u\|^2_{\Lambda_{m+2}} - \int_{\Lambda_{m+2}}\frac{|u|^2}{4(s\sin\tt + t\cos\tt)^2}\dd s \dd t.
	\end{split}
	\end{equation*}
	As $\dom\frmLWot \subset \dom\frmLWone$, for any $u\in\dom\frmLWot$ 
	we have $\frmLWot[u] = \frmLWone[u]$ and we get the form ordering 
	$\frmLWone\prec \frmLWot$. 
	The min-max principle yields, for all $E>0$, the bound
	\begin{equation}\label{eq:cnt_upp_bnd}
		\N_{-\alpha^2/4 - E}(\frmLWot) 
		\leq 
		\sum_{k=0}^{m+2}
		\N_{-\alpha^2/4 - E}(Q_{\alpha,\Lambda_k}).
	\end{equation}

	\myemph{Step 2.}
	In this step we obtain bounds on the functions $E\mapsto \N_{-\aa^2/4 - E}(Q_{\alpha,\Lambda_k})$ 
	$(k = 0,1,2,\dots, (m+2))$. 
	First, we bound from above the functions 
	\[
		E\mapsto \N_{-\alpha^2/4 - E}(Q_{\alpha,\Lambda_{m+1}})
		\quad\text{and}\quad
		E\mapsto \N_{-\alpha^2/4 - E}(Q_{\alpha,\Lambda_{m+2}}).
	\]

	Because $H^1(\Lambda_{m+1})$ is compactly embedded into $L^2(\Lambda_{m+1})$
	the quadratic form $Q_{\alpha,\Lambda_{m+1}}$ is associated with an operator with compact resolvent.
	Therefore, since the domain $\Lambda_{m+1}$ does not depend on $R$, 
	there exists a constant $\frc_\tt = \frc_\tt(\aa, K) > 0$, 
	which depends on $\tt$, $\aa$ and $K$ (but \textit{not} on $R$),   	
	such that, for any $E>0$
	\begin{equation}\label{eq:bnd_m+1}
		\N_{-\alpha^2/4 - E }(Q_{\alpha,\Lambda_{m+1}}) 
		\leq 
		\N_{-\alpha^2/4}(Q_{\alpha,\Lambda_{m+1}})  = \frc_\tt.
	\end{equation}

	Further, for any $u\in\dom Q_{\alpha,\Lambda_{m+2}}$, we have
	\begin{equation*}
		Q_{\alpha,\Lambda_{m+2}}[u] \ge \|\nabla u\|_{\Lambda_{m+2}}^2  -\frac{1}{16K^2}\|u\|_{\Lambda_{m+2}}^2.
	\end{equation*}
	Consequently, for a fixed $K > 0$ such that $\frac{1}{16K^2} < \aa^2/4$, 
	the min-max principle yields $\N_{-\alpha^2/4}(Q_{\alpha,\Lambda_{m+2}}) = 0$. Hence, for any $E>0$, the following equation holds
	\begin{equation}\label{eq:bnd_m+2}
		\N_{-\alpha^2/4 - E }(Q_{\alpha,\Lambda_{m+2}})	\leq \N_{-\alpha^2/4}(Q_{\alpha,\Lambda_{m+2}}) = 0.
	\end{equation}

	Next, we obtain upper bounds for the functions
	\begin{equation*}
		E \mapsto \sum_{k=0}^{m-1}\N_{-\alpha^2/4 - E}(Q_{\alpha,\Lambda_k})
		\quad\text{and}\quad
		E \mapsto \N_{-\alpha^2/4 - E}(Q_{\alpha,\Lambda_m}).
	\end{equation*}
	To this end we define the non-increasing sequence
	\begin{equation}\label{eq:eps_k}      	
		\eps_k := \sup_{(s,t)\in \Lambda_k} \frac{1}{4(s\sin\tt + t\cos\tt)^2} 
			=  \frac{1}{r_k^2\sin^2\tt}, \qquad k = 0,1,\dots, m.
	\end{equation}
	We observe that for any $u\in\dom Q_{\alpha,\Lambda_k}$
	\begin{equation*} 
		Q_{\alpha,\Lambda_k}[u] \geq 
		\|\nabla u\|_{\Lambda_k}^2 
		- \eps_k\|u\|_{\Lambda_k}^2-
		\alpha \int_{r_k}^{r_{k+1}} |u(s,0)|^2\dd s  
		, \qquad k = 0,1,\dots, m.
	\end{equation*}
	Then, we define the intervals $I_1^k = (r_k,r_{k+1})$, $I_2^k = (-d_k,d_k)$,
	the potential $V_m(s) = \frac{1}{4\sin^2\tt (s - r_m/2)^2}$ and
	the following symmetric sesquilinear forms
	\begin{align*}
		\frq_{1,R}^k[\varphi,\psi] & :=  
		(\varphi',\psi')_{I_1^k} - \eps_k(\varphi,\psi)_{I_1^k}, &	
		\dom\frq_{1,R}^k & := H^1(I_1^k),\qquad k = 0,1,\dots, (m-1),\\
		\frq_{1,R}^m[\varphi,\psi] & 
		:=  
		(\varphi',\psi')_{I_1^m} - (V_m\varphi,\psi)_{I_1^m},
		& 	\dom\frq_{1,R}^m  & := H^1(I_1^m),\\
		\frq_{2,R}^k[\varphi,\psi] & :=
		(\varphi',\psi')_{I_2^k}  - \alpha \varphi(0)\ov{\psi(0)},&
		\dom\frq_{2,R}^k & := H^1(I_2^k),\qquad k = 0,1,\dots, m.
	\end{align*}
	One can check that all the forms are closed, densely defined and semibounded in $L^2$-spaces
	over their respective intervals. As $L^2(\Lambda_k) = L^2(I_1^k) \myotimes L^2(I_2^k)$, we introduce the
	quadratic forms 
	\begin{equation*}
		\wt Q_{\alpha, \Lambda_k} 
		:= 
		\frq_{1,R}^k \otimes \fri_2^k + \fri_1^k \otimes \frq_{2,R}^k, \qquad k=0,1,\dots, m,
	\end{equation*}
	and the form orderings $\wt{Q}_{\alpha,\Lambda_k} \prec Q_{\alpha,\Lambda_k}$
	hold for all $k = 0,1,\dots, m$. Here, for $j=1,2$ and $k = 0,2,\dots, m$, $\fri_j^k$ denote the forms of the identity
	operators in $L^2(I_j^k)$.  Hence, we arrive at
	the bound
	\[
	\begin{split}
		\N_{-\alpha^2/4-E}(Q_{\alpha,\Lambda_k})
		&\leq 
		\N_{-\alpha^2/4-E}(\wt Q_{\alpha,\Lambda_k}) 
		= 
		\# 
		\big\{(l,j)\in\dN^2 
		\colon E_l(\frq_{1,R}^k) + E_j(\frq_{2, R}^k) \leq -\alpha^2/4-E \big\}\\ 
		& =  
		\sum_{j=1}^\infty
		\#\big\{l\in \dN \colon 	E_l(\frq_{1,R}^k) \leq -\aa^2/4 - E - E_j( \frq_{2, R}^k )\big\}.
	\end{split}	
	\]
	Now, we choose $K > 0$ sufficiently large such that
	$\eps_1 < \aa^2/4$. Thanks to Proposition~\ref{prop:mod1D}\,(iii), 
	we know that all the summands, for $j>1$, in the above sum equal to zero. Thus, we get the same bound in a simplified form
	\begin{equation}\label{eq:intermediate_bnd}
		\N_{-\alpha^2/4-E}( Q_{\alpha,\Lambda_k} )
		\leq 
		\#\big\{l \in \dN \colon E_l(\frq_{1,R}^k) \leq - \aa^2/4 - E - E_1( \frq_{2, R}^k ) \big\}.
	\end{equation}
	For $k = 0,1, \dots, (m-1)$ we deduce from~\eqref{eq:intermediate_bnd} using~\eqref{eq:eps_k}
	and Proposition~\ref{prop:mod1D}\,(ii) that for $R > 0$ sufficiently large
	\begin{equation}\label{eq:Lambda_k_bnd}
	\begin{split}
		\N_{-\alpha^2/4-E}( Q_{\alpha,\Lambda_k} )
		&\leq 
		\#\big\{ l \in\dN_0 \colon m^2 \pi^2 l^2 R^{-2}                   	
		\leq -\alpha^2/4 + \eps_k - E_1(\frq_{2, R}^k)\big\}\\
		&\leq 
		1 + C_1 \sqrt{R} \sqrt{e^{-C_2 r_k} + r_k^{-2}}
		\leq 
		1 + \frac{C_3\sqrt{R}}{r_k},
	\end{split}
	\end{equation}
	where the positive constants $C_1, C_2$ and $C_3$ do not depend on $R$. 
	Summing the estimates~\eqref{eq:Lambda_k_bnd} over $k$,
	we end up with 
	\begin{equation}\label{eq:bnd_0_m-1}
	\begin{split}
		\sum_{k=0}^{m-1}\N_{-\alpha^2/4-E}(Q_{\alpha,\Lambda_k})
		& \le
		m + C_3\sqrt{R}\sum_{k=0}^{m-1}\frac{1}{r_k}\\
		&\le
		m + C_3\sqrt{R}\frac{1}{r_0} + C_3\int_0^R \frac{\dd x}{3r + x}\leq C_4 \sqrt{R}\\
	\end{split}	
	\end{equation}	
	for all $R > 0$ sufficiently large and a positive constant $C_4$ which does not depend on $R$.

	Further, for $k = m$ we obtain from~\eqref{eq:intermediate_bnd}
	\begin{equation}\label{eq:bnd_m}
	\begin{split}
		\N_{-\alpha^2/4-E}(Q_{\alpha,\Lambda_m}) &
		\leq 
		\#\big\{ l\in\dN \colon E_l(\frq_{1,R}^m) \leq -\alpha^2/4 - E-  E_1(\frq_{2,R}^m)\big\}\\
		&\leq
		\N_{-\alpha^2/4- E - E_1(\frq_{2,R}^m)}(\frq_{1,R}^m).
	\end{split}
	\end{equation}
	Using the unitary operator
	\begin{equation*}
		\sfU \colon L^2(1,+\infty) \arr L^2(I_1^m),
		\qquad 
		(\sfU \psi)(s) := \sqrt{\frac{2}{r_m}}\psi\bigg(\frac{2s}{r_m} - 1\bigg),
	\end{equation*}
	one finds by direct computations that the forms 
	$\frq_{1,R}^m$ and $\frac{4}{r_m^2} \frmKSN$ 
	with $c = 1/(4\sin^2\tt)$ are unitarily equivalent; 
	where the form $\frmKSN$ is defined in~\eqref{eq:op_KS}.
	
	Combining the bound~\eqref{eq:cnt_upp_bnd} and the estimates
	\eqref{eq:bnd_m+1}, \eqref{eq:bnd_m+2}, \eqref{eq:bnd_0_m-1}, \eqref{eq:bnd_m}
	we obtain		 
	\begin{equation}\label{eq:main_bnd}
		\N_{-\alpha^2/4-E}(Q_{\alpha,\Omega_\tt^{K}})
		\le
		\frc_\tt + C_4\sqrt{R} + \N_{-\aa^2/4 - E - E_1(\frq_{2,R}^m)}
		(\frq_{1,R}^m).
    	\end{equation}

	\myemph{Step 3.}
	Now, we choose $R$ depending on the spectral parameter $E>0$ as follows
	\begin{equation*}
		R = R(E) := M |\ln E |,\qquad M > 0,
	\end{equation*}
	in particular, we have $R(E) \arr +\infty$ as $E\arr 0+$.
	Let the constants $C_{\rm N}$, $C_{\rm N}'$ and $L_{\rm N}$ be as in Proposition~\ref{prop:mod1D}\,(ii).
   	Next, we choose $M > 0$ sufficiently large such that
	$C_{\rm N}'M\tan\tt > 2$. Then, for $E > 0$ 
	sufficiently small such that $\ln E < 0$ and $M|\ln E| \tan \tt/2 > L_{\rm N}$, by Proposition~\ref{prop:mod1D}\,(ii) we have 
	\[
		|\aa^2/4 + E_1(\frq_{2,R(E)}^m)| 
		\le C_{\rm N} \exp((C_{\rm N}' \tan\tt/2) (3r + M|\ln E|))
		= 
		\wt{C}_{\rm N} E^{(C_{\rm N}'M\tan\tt)/2} = o(E), \qquad E\arr 0+,
	\]
	where $\wt{C}_{\rm N}>0$. Hence,
	\begin{equation}\label{eq:f(E)_2}
		f(E) 
		:= 
		(r_m(E)^2/4)
		\big(\alpha^2/4  + E + E_1(\frq_{2,R(E)})\big)
		= 
		M^2 E|\ln E|^2/4 + o(E|\ln E|^2), 	\qquad E\arr 0+.
	\end{equation}
	Using~\eqref{eqn:cnt_ineq3}, unitary equivalence of $\frq_{1,R}^m$ 
	and $\frac{4}{r_m^2}\frmKSN$ and Theorem~\ref{th:KS88} we get
	\begin{equation}
	\begin{split}
		\limsup_{E\arr 0+}\frac{\N_{-\alpha^2/4 - E}(\frmLWot)}{|\ln E|}
		&\leq 
		\limsup_{E\arr 0+}\bigg(\frac{\frc_\tt}{|\ln E|} + 
		\frac{C_4\sqrt{M}}{\sqrt{|\ln E|}}
		+ 
		\frac{\N_{-\aa^2/4 - E - E_1(\frq_{2,R(E)}^m)}(\frq_{1,R(E)}^m )}{|\ln E|}
		\bigg)
		\\
		&\leq
		\limsup_{E\arr 0+}
		\frac{\N_{-f(E)}(\frmKSN)}{|\ln E|}
		 =
		\frac{\cot\tt}{4\pi}\limsup_{E\arr 0+}
		\frac{|\ln f(E)|}{|\ln E|}
		= 
		\frac{\cot\tt}{4\pi}.	 
	\label{eqn:maj_nbvpln}
	\end{split}
	\end{equation}

	where we used that $\frac{|\ln f(E)|}{|\ln E|}
	\arr 1$ as $E\arr 0+$ (\cf ~\eqref{eq:f(E)_2}).
	This ends the proof of Proposition \ref{prop:upper_bnd}.
\end{proof}




Now we provide the proof of Lemma~\ref{lem:bnd_aux}.

\begin{proof}[Proof of Lemma~\ref{lem:bnd_aux}]
	Thanks to Corollary~\ref{cor:spec}\,(i) it is sufficient to prove that
	\[
		\limsup_{E\arr 0+}
		\frac{\N_{-\aa^2/4 - E}(\frmLf)}{|\ln E|} 
		\le 
		\limsup_{E\arr 0+}
		\frac{\N_{-\aa^2/4 - E}(\frmLWot)}{|\ln E|},
	\]
	where the forms $\frmLf$ and $\frmLWot$
	are defined as in~\eqref{eqn:not_axy} and in~\eqref{eq:frm_aux},
	respectively.
	
	\myemph{Step 1.}
	Using an IMS formula we split the quadratic form of $\frmLf$
	into two forms, one acting on a strip-shaped 
	geometrical domain attached to
	the boundary $\p\dR_+^2$ of $\dR_+^2$, the other one acting away from it. 
	For this purpose, let us introduce a $\cC^\infty$-smooth cut-off 
	function $\chi_0 \colon \dR_+ \arr [0,1]$ such that
	\begin{equation*}
		\chi_0(r) := 
		\begin{cases}
			1, &r\leq1,\\
			0, &r\geq2.
		\end{cases}
	\end{equation*}
	We also introduce the function $\chi_1\colon \dR_+\arr [0,1]$ 
	such that $\chi^2_0 + \chi_1^2 \equiv 1$. Now, for $K > 0$, 
	we define $\chi_j^K(r) := \chi_j (K^{-1} r)$, $j=0,1$,
	and introduce the following bounded function
	\begin{equation*}
		W^K(r) := |(\chi^K_0)' (r)|^2 + |(\chi^K_1)'(r)|^2 	
		= K^{-2}(|\chi_0^\pp(K^{-1} r)|^2 + |\chi_1^\pp(K^{-1} r)|^2).
	\end{equation*}
	We set $W := W^1$ (for $K = 1$) and observe that $\|W^K\|_\infty =  K^{-2}\|W\|_\infty$. 
	Moreover, for any $u\in\dom\frmLf$, a simple computation (\cf \cite[Sec. 3.1]{CFKS87}) yields
	\begin{equation*}
		\frmLf[u] = 
		\frmLf[\chi_0^K u] + \frmLf[\chi_1^K u] - 
		\int_{\dR_+^2} W^K(r) |u(r,z)|^2 r \dd r \dd z.
	\end{equation*}
	Next we introduce the sub-domains 
	\[
		\Omega_0 := \{ (r,z) \in \dR_+^2 \colon  r \le 2K \},
		\qquad
		\Omega_1 := \{ (r,z) \in \dR_+^2 \colon r  \ge K \},
	\]
	of the meridian domain $\dR^2_+$ (note that $\supp W^K \subset \Omega_0$).
	For $j=0,1$, we set $\Gamma_j = \Gamma_\theta \cap \Omega_j$ and define
	$I_0 := (0, 2K(\sin\tt)^{-1})$, $I_1 := (K(\sin\tt)^{-1}, +\infty)$,
	$\Sigma_0 = \{2K\}\times \dR$, and $\Sigma_1 = \{K\} \times \dR$.
	Then we consider the quadratic forms $\frmLWj$, $j=0,1$, defined as
	\begin{equation}
	\begin{split}
		\frmLWj[u] & 
		:= 
		\int_{\Omega_j}
		\big(|\p_r u|^2 + |\p_z u|^2 - W^K(r)|u|^2\big) r \dd r \dd z 
		- 
		\alpha\int_{I_j}| u(s\sin\tt, s\cos\tt)|^2 s\sin\tt \dd s,\\
		\dom \frmLWj & := \big\{ \restric{u}{\Omega_j} \colon u \in \dom\frmLf, \restric{u}{\Sigma_j} = 0 \big\}.
	\end{split}	
	\label{eqn:dec_IMS1}
	\end{equation}
	Thanks to~\eqref{eqn:dec_IMS1}, we get that, for $j=0,1$ any $u \in \dom \frmLf$, 
	$\chi_j^K u \in \dom \frmLWj$ and we have the relation 
	\begin{equation*}
		\frmLf [u] = \frmLWz [\chi_0^K u] + \frmLWo [\chi_1^K u].
	\end{equation*}
	Using \cite[Lem. 5.2]{DLR12} we find
	\begin{equation}
		\N_{-\alpha^2/4 -E}(\frmLf) 
		\leq    	
		\N_{-\alpha^2/4 -E}(\frmLWz) + \N_{-\alpha^2/4 -E}(\frmLWo).
	\label{eqn:IMS_major}
	\end{equation}

	\myemph{Step 2.}  
	In this step we prove that for any $K > 0$ sufficiently large,
	there exists a constant $\wh\frc_\tt = \wh\frc_\tt (\aa, K) >0$ such that 
	\begin{equation}\label{eq:Q0_bnd}
		\N_{-\alpha^2/4 -E}(\frmLWz) \leq \wh\frc_\tt
	\end{equation}
	for all $E>0$. First, we introduce the quadratic form $\frmLz$, defined as
	\[
		\frmLz[u]  := 
		\int_{\Omega_0} \big( |\p_r u|^2 + |\p_z u|^2 \big) r \dd r\dd z - 
		\alpha\int_{I_0} | u(s\sin\tt,s\cos\tt)|^2 s\sin\tt \dd s,
		\qquad
		\dom \frmLz := \dom \frmLWz.
	\]
	One can check that the above form
	is closed, symmetric, densely defined and semibounded in $L^2(\Omega_0;r \dd r \dd z)$.
	Now, for any $u \in\dom \frmLWz $ we have
	\begin{equation}
		\frmLWz[u] \geq \frmLz[u] - K^{-2}\| W \|_\infty \|u\|^2_{L^2(\Omega_0;r \dd r \dd z)}.
	\label{eqn:bound_bel_epsilon}
	\end{equation}
	Consequently, we get that
	\begin{equation}
		\N_{-\alpha^2/4-E}(\frmLWz)
		\leq
		\N_{-\alpha^2/4 - E + K^{-2}\|W\|_\infty}(\frmLz).
	\label{eqn:choice_epsilon}
	\end{equation}
	Next, we choose $K > 0$ sufficiently large such that $-\alpha^2/4 + K^{-2} \|W\|_\infty \leq -\alpha^2/8$. 
	Combining \eqref{eqn:bound_bel_epsilon} and \eqref{eqn:choice_epsilon} 
	we obtain
	\begin{equation}\label{eq:Q0_Q0p_bnd}
		\N_{-\alpha^2/4-E}(\frmLWz) \leq \N_{-\alpha^2/8}(\frmLz).
	\end{equation}
	Secondly, let us split the domain $\Omega_0$ into two disjoint sub-domains 
	(\cf Figure \ref{fig:sigma_0})
	\begin{equation*}
		\Omega_{00} :=  \{(r,z)\in \Omega_0 \colon z \in  (0,2K\cot\theta)\},
		\qquad 
		\Omega_{01} :=  \{(r,z)\in \Omega_0 \colon z \notin  [0,2K\cot\theta]\}.
	\end{equation*}
	For $j= 0,1$, we define $\Sigma_{0j} := \Sigma_0 \cap \p\Omega_{0j}$ and we consider the quadratic forms $\frmLzz$ and $\frmLzo$ defined as
	\[
	\begin{split}
		\frmLzz[u] & := 
		\int_{\Omega_{00}}\big (|\p_r u|^2 + |\p_z u|^2\big) r \dd r \dd z - 
		\alpha\int_{I_0}|u(s\sin\tt,s\cos\tt)|^2 s \sin\tt \dd s,\\
		\frmLzo[u] & := 
		\int_{\Omega_{01}} (|\p_r u|^2 + |\p_z u|^2) r \dd r \dd z,\\
		\dom\frmLzj & := 
		\{u \colon u, \p_r u, \p_z u \in L^2(\Omega_{0j}, 
		r \dd r \dd z), u|_{\Sigma_{0j}} = 0\},\quad j=0,1.
	\end{split}
	\]
	One can check that the above forms
	are closed, symmetric, densely defined and semibounded in $L^2(\Omega_{00};r \dd r \dd z)$
	and in $L^2(\Omega_{01};r \dd r \dd z)$, respectively.
	\begin{figure}[h!]
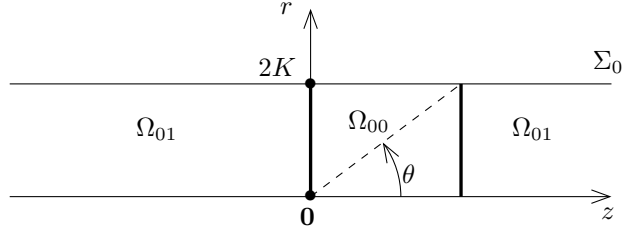

\figinit{1cm}
\figpt 0:(0,0)
\figpt 1:(-4,0)
\figpt 2:(4,0)
\figpt 3:(-4,1.5)
\figpt 4:(4,1.5)
\figpt 5:(0,1.5)
\figpt 6:(0,2.5)
\figpt 7:(2,1.5)
\figpt 8:(2,0)

\figpt 9:(1,0.65)
\figpt 10:(-2,0.75)
\figpt 11:(3,0.75)
\figpt 12:(0.5,1)

\figdrawbegin{}
\figdrawline [3,4]
\figdrawarrow [1,2]
\figdrawarrow [0,6]
\figdrawarrowcircP 0;1.2[2,7]
\figset (dash=8)
\figdrawline [0,7]
\figset (dash=1)
\figset (width=1.25)
\figdrawline [7,8]
\figdrawline [0,5]
\figdrawend

\figvisu{\figBoxA}{}{
\figwritese 9:{$\theta$} (0.31)
\figwriten 10:{$\Omega_{01}$} (0)
\figwriten 11:{$\Omega_{01}$} (0)
\figwritee 12:{$\Omega_{00}$} (0)
\figwrites 2:{$z$} (0.15)
\figwritew 6:{$r$} (0.15)
\figwriten 4:{$\Sigma_0$} (0.15)
\figset write(mark=$\bullet$)
\figwrites 0:{$\mathbf{0}$} (0.15)
\figwritenw 5:{$2K$} (0.15)
}

\centerline{\box\figBoxA}
\caption{The domain $\Omega_0$ and the sub-domains $\Omega_{00}$ and $\Omega_{01}$. The dashed line is the support of the $\delta$-interaction.}
\label{fig:sigma_0}
\end{figure}
	For $u\in \dom\frmLz$ and $j=0,1$, we define $u_j = \restric{u}{\Omega_{0j}}$ and get
	\begin{equation*}
		\frmLz[u] = \frmLzz[u_0] +\frmLzo[u_1].
	\end{equation*}
	The above equality and the min-max principle yield
	\begin{equation}\label{eq:Q0p_Q00p_Q01p_bnd}
		\N_{-\alpha^2/8}(\frmLz) \leq 
		\N_{-\alpha^2/8}(\frmLzz) + 
		\N_{-\alpha^2/8}(\frmLzo).
	\end{equation}

	Note that for all $u\in\dom \frmLzo$, we have
	$\frmLzo[u] \geq 0$.	Consequently, we get 
	\begin{equation}\label{eq:Q01p_bnd}
		\N_{-\alpha^2/8}(\frmLzo) = 0.
	\end{equation}
	Moreover, the quadratic form $\frmLzz$ is associated with the lowest fiber operator of a 
	three-dimensional Schr\"odinger operator
	with a surface $\delta$-interaction acting on a
	bounded domain with mixed boundary conditions 
	(Neumann and Dirichlet). This operator 
	has compact resolvent and its sequence of eigenvalues goes to infinity. 
	Hence, we obtain
	\begin{equation}\label{eq:Q00p_bnd}
		\N_{-\alpha^2/8}(\frmLzz) = \wh\frc_\tt < \infty,
	\end{equation}
	with some constant $\wh\frc_\tt = \wh\frc_\tt(\aa,K) > 0$.
	Combining \eqref{eq:Q0_Q0p_bnd}, \eqref{eq:Q0p_Q00p_Q01p_bnd}, \eqref{eq:Q01p_bnd} and \eqref{eq:Q00p_bnd}
	we obtain \eqref{eq:Q0_bnd}.

	\myemph{Step 3.}
	We remark that the domain 
	$\wt\Omega_1 = \{(s,t) \in \dR^2 \colon s\sin\tt + t\cos\tt > K\} 
	\subset\Omega_\tt$ is the image of 
	$\Omega_1$ under Rotation~\eqref{def:m}
	and we consider the unitary transform
	\begin{equation*}
		\sfU \colon L^2(\Omega_1; r \dd r \dd z) \arr  L^2(\wt\Omega_1),
		\qquad 
		(\sfU u)(s,t) :=  
		u(s\sin\tt + t\cos\tt,s\sin\tt - t \cos\tt)\sqrt{s\sin\tt +t\cos\tt}.
	\end{equation*}
	A straightforward computation yields that the quadratic form $\frmLWo$ 
	is unitarily equivalent, \textit{via} $\sfU$, to the form 
	\[
	\begin{split}
		\wt Q_1 [u] & := 
		\|\nabla u\|^2_{\wt\Omega_1} - 
		\int_{\wt\Omega_1} 
		\frac{|u|^2}{4(s\sin\tt + t\cos\tt)^2} + W^K(s\sin\tt + t\cos\tt)|u|^2
		\dd s \dd t
		-\aa\int_{K (\sin\tt)^{-1} }^\infty|u(s,0)|^2 \dd s\\
		\dom \wt Q_1 & := H^1_0(\wt\Omega_1).
	\end{split}	
	\]
	We introduce the sub-domains of $\wt\Omega_1$
	\[	
		\wt\Omega_{10} := \{(s,t)\in\dR^2 \colon 
		K < s\sin\tt + t \cos\tt < 2K\},
		\qquad 
		\wt\Omega_{11} := \{(s,t)\in \dR^2
		\colon  s\sin\tt + t \cos\tt > 2K\},
	\]
	and the forms
	\[
	\begin{split}
		&\wt Q_{10} [u]  
			:= 
			\|\nabla u\|^2_{\wt\Omega_{10}} - 
			\int_{\wt\Omega_{10}} 
			\frac{|u|^2}{4(s\sin\tt + t\cos\tt)^2} + 
			W^K(s\sin\tt + t\cos\tt)|u|^2\dd s \dd t
			-\aa\int_{K(\sin\tt)^{-1}}^{2K(\sin\tt)^{-1}}|u(s,0)|^2 \dd s,\\
		&\wt Q_{11} [u]  
		:= 
		\|\nabla u\|^2_{\wt\Omega_{11}} - 
		\int_{\wt\Omega_{11}} \frac{|u|^2}{4(s\sin\tt + t\cos\tt)^2}\dd s \dd t
			-\aa\int_{2K(\sin\tt)^{-1}}^{+\infty} |u(s,0)|^2 \dd s,\\	
		&\dom\wt Q_{10}   := 
		\big\{u\in H^1(\wt \Omega_{10}) \colon 
		u|_{\p \wt\Omega_1} = 0\big\},\qquad
		\dom\wt Q_{11} := H^1(\wt\Omega_{11}).
	\end{split}
	\]
	The above forms
	are closed, symmetric, densely defined and semibounded in $L^2(\wt\Omega_{10})$
	and in $L^2(\wt\Omega_{11})$, respectively.
	Once again, we get by the min-max principle
	\begin{equation}\label{eq:Q1_bnd}
		\N_{-\aa^2/4 - E}(\frmLWo) = 
		\N_{-\aa^2/4 - E}(\wt Q_1) \le 
		\N_{-\aa^2/4 - E}( \wt Q_{10} ) + \N_{-\aa^2/4 - E}( \wt Q_{11} ).
	\end{equation}

	\myemph{Step 4.} In this step we prove that for any $K>0$ sufficiently large, there exists a constant 
	$\wt \frc_\tt = \wt\frc_\tt(\aa, K) > 0$ such that	%
	\begin{equation}\label{eq:wtQ10_bnd}
		\N_{-\aa^2/4 - E}(\wt Q_{10}) \le \wt\frc_\tt,
	\end{equation}
	for all $E>0$. To do so, we introduce the quadratic form $\wt Q_{10}'$ defined as
	\begin{equation}\label{eqn:maj_Qprim}
	\wt Q_{10}' [u]  
			:= 
			\|\nabla u\|^2_{\wt\Omega_{10}}
			-\aa\int_{K(\sin\tt)^{-1}}^{2K(\sin\tt)^{-1}}|u(s,0)|^2 \dd s,\quad \dom \wt Q_{10}' := \dom\wt Q_{10}.
	\end{equation}
	One can check that the above form is closed, symmetric,
	densely defined and semibounded in $L^2(\wt\Omega_{10})$. Now, for any $u\in\dom\wt Q_{10}$ we have
	\begin{equation}\label{eqn:min_fqQprim}
	\wt Q_{10}[u] \ge
	\wt Q_{10}'[u] -
	K^{-2}(\|W\|_\infty + 1/4)\|u\|_{\wt\Omega_{10}}^2.
	\end{equation}
	Consequently, we get
	\[
	\N_{-\aa^2/4 - E}(\wt Q_{10}) \leq \N_{-\aa^2/4 - E + K^{-2}(\|W\|_\infty + 1/4)}(\wt Q_{10}').
	\]
	Next, we choose $K>0$ sufficiently large such that $-\aa^2/4 + K^{-2}(\|W\|_\infty + 1/4) < -\aa^2/8$. 
	Combining \eqref{eqn:maj_Qprim} and \eqref{eqn:min_fqQprim} we have
	\begin{equation}\label{eqn:bnd_1}
	\N_{-\aa^2/4 - E}(\wt Q_{10}) \leq
	\N_{-\aa^2/8}(\wt Q_{10}').
	\end{equation}
	Then, let us split the domain $\wt\Omega_{10}$ into two disjoint sub-domains
	\[
	\wt \Omega_{100} = 
	\{(s,t)Ê\in \wt \Omega_{10} : |t| < 1\},
	\quad
	\wt \Omega_{101} =
	\{(s,t)Ê\in \wt \Omega_{10} : |t| > 1\}.
	\]
	We denote by $\wt\Sigma_0$ the image of $\Sigma_0$ under Rotation \eqref{def:m} and, for $j=0,1$, let us define $\wt\Sigma_{0j} = \wt\Sigma_0\cap\p\Omega_{10j}$ ($j=0,1$). Then, we consider the quadratic forms $\wt Q_{100}'$ and $\wt Q_{101}'$ defined as
	\[
	\begin{split}
	\wt Q_{100}'[u] & :=
	\|\nabla u\|_{\wt\Omega_{100}}^2 -
	\alpha \displaystyle\int_{K(\sin\tt)^{-1}}^{2K(\sin\tt)^{-1}}|u(s,0)|^2\dd s,\\
	\wt Q_{101}'[u] & :=
	\|\nabla u\|_{\wt\Omega_{101}}^2,\\
	\dom\wt Q_{10j}'& :=
	\{uÊ\in H^1(\wt\Omega_{10j}) : \restric{u}{\wt\Sigma_{0j}} = 0\},\quad j=0,1.
	\end{split}
	\]
	The above forms
	are closed, symmetric, densely defined and semibounded in $L^2(\wt\Omega_{100})$
	and in $L^2(\wt\Omega_{101})$, respectively.
	For $u\in\dom\wt Q_{10}'$, we define $u_j = \restric{u}{\wt\Omega_{10j}}$ ($j=0,1$) and get
	\[
	\wt Q_{10}'[u] =
	\wt Q_{100}'[u_0] + \wt Q_{101}'[u_1].
	\]
	The above equality and the min-max principle yield
	\begin{equation}\label{eqn:bnd_2}
	\N_{-\aa^2/8}(\wt Q_{10}') \leq \N_{-\aa^2/8}(\wt Q_{100}') + \N_{-\aa^2/8}(\wt Q_{101}').
	\end{equation}
	For all $u\in\dom\wt Q_{101}'$, $\wt Q_{101}'[u]\ge0$ and we get
	\begin{equation}\label{eqn:bnd_3}
 	\N_{-\aa^2/8}(\wt Q_{101}') = 0.
 	\end{equation}
	The quadratic form $\wt Q_{100}'$ is the quadratic form of a Schr\"{o}dinger operator with a $\delta$-interaction supported on a line segment. 
	It acts on a bounded domain with mixed boundary conditions (Neumann and Dirichlet) thus, this operator has compact resolvent and its sequence of 
	eigenvalues goes to infinity. Hence, we have
	\begin{equation}\label{eqn:bnd_4}
	\N_{-\aa^2/8}(\wt Q_{100}') = \wt \frc_\tt < \infty,
	\end{equation}
	where $\wt \frc_\tt = \wt \frc_\tt(\aa,K) >0$. 
	Combining \eqref{eqn:bnd_1}, \eqref{eqn:bnd_2}, \eqref{eqn:bnd_3} and \eqref{eqn:bnd_4} we obtain \eqref{eq:wtQ10_bnd}.
	
	\myemph{Step 5.}
	To conclude, inserting~\eqref{eq:Q0_bnd},~\eqref{eq:Q1_bnd}~
	and~\eqref{eq:wtQ10_bnd} 
	into~\eqref{eqn:IMS_major} we get
	\[
	\begin{split}
		\limsup_{E\arr 0+}\frac{\N_{-\aa^2/4 - E}(\frmLf)}{|\ln E|}
		& 
		\le
        \limsup_{E\arr 0+}
		\bigg(
		\frac{\N_{-\aa^2/4 - E}(\frmLWz)}{|\ln E|} + 
		\frac{\N_{-\aa^2/4 - E}(\frmLWo)}{|\ln E|}
		\bigg)\\
		&
		\le
		\limsup_{E\arr 0+}\bigg(
		\frac{\N_{-\aa^2/4 - E}(\wt Q_{11})}{|\ln E|} 
		+ \frac{\wh\frc_\tt}{|\ln E|} + 
		\frac{\wt\frc_\tt}{|\ln E|}\bigg)
		=
		\limsup_{E\arr 0+}\frac{\N_{-\aa^2/4 - E}(\wt Q_{11})}{|\ln E|}.
 	\end{split}
	\]
	Finally, it remains to note that the form $\wt Q_{11}$ 
	is the form $Q_{\aa,\Omega^{K}_\tt}$ in~\eqref{eq:frm_aux}. 
	This ends the proof of Lemma \ref{lem:bnd_aux}.
\end{proof}

\appendix
\section{Quadratic forms $Q_{\aa, \Gamma_\tt}^{[l]}$}
\label{app:appB}
The aim of this appendix is to prove the following proposition about the
forms $Q_{\aa, \Gamma_\tt}^{[l]}$ in \eqref{eqn:def_fq}
and the spaces $\cC_0^\infty(\ov{\R_+^2})$, $\cC_{0,0}^\infty(\ov{\R_+^2})$ defined in Notation~\ref{notn:def_space}.
\begin{prop}
	Let $d\geq3$, $l\geq0$ and the quadratic forms $\frmLf^{[l]}$ be defined as in \eqref{eqn:def_fq}.
	Then the following statements hold:
	\begin{myenum}
		\item for $d = 3$ and any $l > 0$, $\cC_{0,0}^\infty(\ov{\R_+^2})$ is a form core for $Q_{\aa,\Gamma_\tt}^{[l]}$;
		\item for $d = 3$ and $l = 0$ or $d \geq 4$ and $l \geq 0$, $\cC_0^\infty(\ov{ \R_+^2 })$ is a form core for $Q_{\aa,\Gamma_\tt}^{[l]}$.
	\end{myenum}
	\label{prop:core_fdcyl}
\end{prop}
Before proving Proposition~\ref{prop:core_fdcyl}, we need to introduce a few notation. 
For any function $u\in L_\mathsf{cyl}^2(\R^d)$ we denote by $\wt{u}\in L^2(\R^d)$ 
the function $\wt{u}(x_1,\dots,x_d) = u(r,z,\phi)$ 
in the physical coordinates (\cf the change of variables \eqref{eqn:var_cyl}). 
Let us fix the dimension $d\geq3$ and $l\geq0$. We choose $M>0$ large enough such that 
for any $\wt{u}\in\dom Q_{\aa,\cC}$ and any $\wh{u}\in \dom Q_{\aa,\Gamma_\tt}^{[l]}$ we have
\[
	Q_{\aa,\cC}[\wt{u}] \geq - M \|\wt{u}\|_{\R^d}^{2},
	\qquad 
	Q_{\aa,\Gamma_\tt}^{[l]}[\wh{u}] \geq -M \|\wh{u}\|_{L^2(\R_+^2;r^{d-2}\dd r\dd z)}^{2}.
\]
We introduce the following norms associated with the quadratic forms 
$Q_{\aa,\cC}$ and $Q_{\aa,\Gamma_\tt}^{[l]}$ defined, for $\wt{u}\in\dom Q_{\aa,\cC}$ and 
$\wh{u}\in \dom Q_{\aa,\Gamma_\tt}^{[l]}$, by
\[
	\|\wt{u}\|_{Q_{\aa,\cC}}^2 := Q_{\aa,\cC}[\wt{u}] + (M+1)\|\wt{u}\|_{\R^d}^2,
	\qquad 
	\|\wh{u}\|_{Q_{\aa,\Gamma_\tt}^{[l]}}^2 
	:= 
	Q_{\aa,\Gamma_\tt}^{[l]}[u] + ( M + 1) \|\wh{u}\|_{L^2(\R_+^2;r^{d-2}\dd r\dd z)}^2.
\]
Now, we state three lemmas that are proven in the end of the appendix.
\begin{lem}
	Let $d\geq3$ and $l\geq0$. 
	Then the following set inclusions hold:
	\begin{myenum}	
		\item for $d=3$ and any $l>0$, $\cC_{0,0}^\infty(\ov{\R_+^2}) \subset \dom Q_{\aa,\Gamma_\tt}^{[l]}$;
		\item for $d=3$ and $l=0$ or $d\geq4$ and $l\geq0$, $\cC_0^\infty(\ov{ \R_+^2 }) \subset \dom Q_{\aa,\Gamma_\tt}^{[l]}$.
	\end{myenum}
	\label{lem:inclu_core}
\end{lem}

\begin{lem} 
	Let $d=3$ and $ l> 0$.   
	Then the following set inclusion
	$\cC_0^\infty(\ov{\R_+^2}) \cap \dom Q_{\aa,\Gamma_\tt}^{[l]}\subset\cC_{0,0}^\infty(\ov{\R_+^2})$
	holds.
\label{lem:reg_core}
\end{lem}

\begin{lem} 
	Let $d\geq3$, $l\geq0$ and $k\in\{1,\dots,c(d,l)\}$. 
	For any $u(r,z,\phi) = \wh{u}(r,z) Y_{l,k}^{d-2}(\phi) \in H_{\mathsf{cyl}}^1(\R^d)$ 
	with $\wh{u} \in L^2(\dR^2_+;r^{d-2}\dd r \dd z)$, 
	there exists $\wh{u}_n \in \cC_0^\infty( \ov{\R_+^2} )\cap \dom Q_{\aa,\Gamma_\tt}^{[l]}$
	such that 
	%
	%
	%
	\[
		\| \wh{u}_n - \wh{u}   \|_{Q_{\alpha, \Gamma_\tt }^{[l]}} \rightarrow 0,
		\qquad 
		n\rightarrow\infty.
	\]
\label{lem:cvg_normfq}
\end{lem}

Now, we have all the tools to prove Proposition~\ref{prop:core_fdcyl}.

\begin{proof}[Proof of Proposition \ref{prop:core_fdcyl}] 
	Let $d\geq3$, fix $l\in\dN_0$ and $k\in\{1,\dots,c(d,l)\}$. 
	Let $\wh{u} \in \dom Q_{\aa,\Gamma_\tt}^{[l]}$, 
	we define $u(r,z,\phi) = \wh{u}(r,z) Y_{l,k}^{d-2}(\phi)\in L_{\mathsf{cyl}}^2(\R^d)$. 
	One can show by direct computations that
	\[
	\begin{split}                                
		\|\p_r u\|_{L^2_{ \mathsf{cyl} }(\R^d)}^2 
		+  
		\|\p_z u\|_{L^2_{ \mathsf{cyl} }(\R^d)}^2 
		+ 
		\|r^{-1}\nabla_{\dS^{d-2}} u\|_{L^2_{\mathsf{cyl}}(\R^d)}^2 
		&= 
		\|\p_r \wh{u}\|_{L^2(\R_+^2;r^{d-2}\dd r\dd z)}^2 + \|\p_z \wh{u}\|_{L^2(\R_+^2;r^{d-2}\dd r\dd z)}^2 \\
		&\qquad\qquad\qquad\qquad 
		+  
		l(l + d - 3)\|r^{-1}\wh{u}\|_{L^2(\R_+^2;r^{d-2}\dd r\dd z)}^2.
	\end{split}
	\]
	As $\wh{u} \in \dom Q_{\aa,\Gamma_\tt}^{[l]}$ 
	the right hand side is finite and hence $u\in H_{\mathsf{cyl}}^1(\R^d)$. 
	Now, thanks to Lemma~\ref{lem:cvg_normfq} we know that there exists a sequence 
	$\wh{u}_n \in \cC_0^\infty(\ov{\R_+^2}) \cap \dom Q_{\alpha,\Gamma_\tt}^{[l]}$ 
	such that $\| \wh{u}_n - \wh{u} \|_{Q_{\aa,\Gamma_\tt}^{[l]}} \arr 0$ as $n\arr\infty$. 
	Because $\wh{u}_n \in \cC_0^\infty(\ov{\R^2_+})$ and thanks to Lemma~\ref{lem:inclu_core}\,(ii) 
	we obtain  item~(ii) in Proposition~\ref{prop:core_fdcyl}. 
	To get item~(i) in Proposition~\ref{prop:core_fdcyl}, 
	we use Lemma~\ref{lem:reg_core}, which yields $\wh{u}_n\in\cC_{0,0}^\infty(\ov{\R_+^2})$. 
	We conclude using Lemma~\ref{lem:inclu_core}\,(i).
\end{proof}

\begin{proof}[Proof of Lemma~\ref{lem:inclu_core}] 
	By definition of the form domains in~\eqref{eqn:def_fq}, 
	it is only necessary to check that for $d=3$, $l>0$ 
	and any $\wh{u}\in\cC_{0,0}^\infty(\ov{\R_+^2})$ 
	we have $r^{-1}\wh{u} \in L^2(\R_+^2;r\dd r \dd z)$. 
	Let us fix $\wh{u}\in\cC_{0,0}^\infty(\ov{\R_+^2})$, 
	using a Taylor expansion in the $r$-variable, 
	we can write $\wh{u}(r,z) = r g(r,z)$ with $g\in\cC_0^\infty(\ov{\R_+^2})$. 
	Consequently, we get
	\[
		\|r^{-1} \wh{u}\|_{L^2(\R_+^2;r\dd r \dd z)} = \|g\|_{L^2(\R_+^2;r\dd r \dd z)} < \infty,
	\]
	which concludes the proof.
\end{proof}

\begin{proof}[Proof of Lemma~\ref{lem:reg_core}] 
	Let $d=3$ and $l>0$. 
	We take $\wh{u}\in\cC_0^\infty(\ov{\R_+^2})\cap\dom Q_{\alpha,\Gamma_\tt}^{[l]}$ and, 
	using a Taylor expansion in the $r$-variable, 
	we can write $\wh{u}(r,z) = \wh{u}(0,z) +r g(r,z)$ with $g\in\cC_0^\infty(\ov{\R_+^2})$. 
	Because $r^{-1}\wh{u}\in L^2(\dR^2_+;r\dd r \dd z)$, we obtain
	that $r^{-1}\wh{u}(0,z) = r^{-1} \wh{u}(r,z) - g(r,z) \in L^2(\dR^2_+; r\dd r \dd z)$.
	Hence, we get
	\[
		\int_{\R_+^2}\frac{|\wh{u}(0,z)|^2}{r}\dd r \dd z  < \infty.
	\]
	Finiteness of the last integral necesserily implies $\wh{u}(0,z) = 0$ for any $z\in\R$. 
	Hence, $\wh{u}(r,z) = r g(r,z)$ and $\wh{u} \in \cC_{0,0}^\infty(\ov{\R_+^2})$.
\end{proof}

\begin{proof}[Proof of Lemma~\ref{lem:cvg_normfq}] 
	Let $d \geq 3$, $l \in \dN_0$ and $k \in \{1, \dots, c(d,l)\}$. 
	Let the orthogonal projector $\Pi_{l, k}$ in $L^2_{\mathsf{cyl}}(\dR^d)$
	be defined as in~\eqref{eq:Pikl}.
	Since $\wt{ u } \in H^1(\R^d)$, there exists a sequence $\wt{ v }_{ n } \in \cC_0^\infty(\R^d)$ 
	such that $\| \wt{ v }_n - \wt{ u } \|_{H^1(\R^d)} \arr 0$ as $n\rightarrow\infty$.

	Further, define the modified sequence 
	\[
		u_n := \Pi_{l,k}(v_n) \in L^2_{\mathsf{cyl}}(\R^d), \qquad n\in\dN.
	\] 
	We remark that 
	\[
		\wh{u}_n(r,z) := \ps{ u_n(r,z,\cdot) }{ Y_{l,k}^{d-2}(\cdot) }_{ \dS^{d-2} } \in \cC_0^\infty(\ov{\R^2_+}).
	\] 
	Thus, we can write $u_n(r,z,\phi) = \wh{u}_n(r, z) Y_{l,k}^{d-2}(\phi)\in H_\mathsf{cyl}^1(\R^d)$.
	

	Next, we prove that 
	\begin{equation}
		\|u_n - u\|_{H_{\mathsf{cyl}}^1(\R^d)} \rightarrow 0,\quad n\rightarrow\infty.
		\label{eqn:cvg_projH}
	\end{equation}
	By orthogonality, we have
	\[
		\|v_n - u\|_{L_{\mathsf{cyl}}^2(\R^d)}^2 
		= 
		\|u_n - u\|_{L_{\mathsf{cyl}}^2(\R^d)}^2 + \| (I - \Pi_{k,l})v_n \|_{L_{\mathsf{cyl}}^2(\R^d)}^2,
	\]
	where the left hand side tends to zero as $n\arr\infty$. Hence, we get
	\begin{equation}
		\|u_n - u\|_{L_{\mathsf{cyl}}^2(\R^d)} \rightarrow 0,\qquad n\rightarrow\infty.
	\label{eqn:cvgH1}
	\end{equation}
	Because $v_n$ is smooth and compactly supported, we have the following useful commutation relations
	\[
		\p_r u_n = \Pi_{l,k}(\p_r v_n)
		\quad\text{and}\quad 
		\p_z u_n = \Pi_{l,k}(\p_z v_n),
	\]
	which yield
	\begin{equation}
	\begin{split}
		\|\p_r (u_n - u)\|_{L_{\mathsf{cyl}}^2(\R^d)} 
		&= 
		\|\Pi_{l,k}( \p_r (v_n - u) )\|_{L_{\mathsf{cyl}}^2(\R^d)}  
		\leq 
		\| \p_r (v_n - u) \|_{L_{\mathsf{cyl}}^2(\R^d)} \arr 0, \qquad n\arr\infty,\\
		\| \p_z (u_n - u) \|_{L_{\mathsf{cyl}}^2(\R^d)} 
		&= 
		\| \Pi_{l,k} ( \p_z (v_n - u) )\|_{L_{\mathsf{cyl}}^2(\R^d)}    
		\leq 
		\| \p_z ( v_n - u )\|_{L_{\mathsf{cyl}}^2(\R^d)} \arr 0,\qquad n\arr\infty. 
	\end{split}
	\label{eqn:cvgH2}
	\end{equation}
	Using that spherical harmonics are eigenfunctions of the
	self-adjoint Laplace-Beltrami operator $-\Delta_{\dS^{d-2}}$ associated with the quadratic form
	$H^1(\dS^{d-2}) \ni \psi \mapsto \| \nabla_{\dS^{d-2}} \psi\|^2_{\dS^{d-2}}$ on $L^2(\dS^{d-2})$,
	we have for any fixed $(r, z)\in\R_+^2$
	\begin{equation*}
	\begin{split}
		\ps{ \nabla_{\dS^{d-2}} (u_n - u) }{  \nabla_{\dS^{d-2}}(v_n - u_n)}_{\dS^{d-2}}
		& = \ps{(\wh{u}_n - \wh{u}) (-\Delta_{\dS^{d-2}} Y_{l,k}^{d-2}) }{  (I - \Pi_{l,k})v_n }_{\dS^{d-2}}\\
		& = l(l+ d-3) \ps{u_n - u}{  (I - \Pi_{l,k})v_n }_{\dS^{d-2}}\\
		& = l(l+ d-3) \ps{\Pi_{l,k}(v_n - u_n)}{  (I - \Pi_{l,k})v_n }_{\dS^{d-2}} = 0.
	\end{split}
	\label{eqn:orth_rel2}
	\end{equation*}
	The above relation gives
	\[
		\|\nabla_{\dS^{d-2}} (v_n - u)\|_{\dS^{d-2}}^2 
		= 
		\|\nabla_{\dS^{d-2}} (u_n - u)\|_{\dS^{d-2}}^2 
		+ 
		\|\nabla_{\dS^{d-2}} (v_n - u_n)\|_{\dS^{d-2}}^2.
	\]
	Multiplying the latter equality by $r^{-2}$ and integrating in $(r,z)$ we get
	\begin{equation}
		\|r^{-1}\nabla_{\dS^{d-2}}(u_n - u) \|_{L_{\mathsf{cyl}}^2(\R^d)}^2 
		\leq 
		\|r^{-1}\nabla_{\dS^{d-2}}(v_n - u)\|_{L_{\mathsf{cyl}}^2(\R^d)}^2 
		\leq 
		\|v_n - u\|_{H_{\mathsf{cyl}}^1(\R^d)}^2\rightarrow 0,\qquad n\rightarrow\infty.
		\label{eqn:cvgH3}
	\end{equation}
	Finally, combining~\eqref{eqn:cvgH1},~\eqref{eqn:cvgH2},~\eqref{eqn:cvgH3} 
	and the fact that $u\in H_\mathsf{cyl}^1(\R^d)$ we get~\eqref{eqn:cvg_projH}.

	According to \cite[Prop. 3.1 and its proof]{BEL14_RMP} and Notation~\ref{notn:def_space}, 
	the $H^1_{\mathsf{cyl}}(\dR^d)$-norm is equivalent to the norm $\|\cdot\|_{Q_{\aa,\cC}}$
	after suitable identification of functions. Hence,~\eqref{eqn:cvg_projH} yields
	\[
		\| \wt{u}_n - \wt{u} \|_{Q_{\alpha,\cC}} \rightarrow 0,\quad n\rightarrow\infty.
	\]
	By direct computation, we get
	\[
		\|\wt{u}_n - \wt{u}\|_{Q_{\alpha,\cC}} = \|\wh{u}_n - \wh{u}\|_{Q_{\alpha,\Gamma_\tt}^{[l]}},
	\]
	which concludes the proof of Lemma~\ref{lem:cvg_normfq}.
\end{proof}


\subsection*{Acknowledgements}
V.~Lotoreichik is supported by the Czech Science Foundation (GA\v{C}R) under the project 14-06818S.
He is grateful for the stimulating research stay and the hospitality at the
Basque Center for Applied Mathematics in May 2015 where a part of this paper was written.
T.~Ourmi\`eres-Bonafos is supported by the Basque Government through the BERC 2014-2017 program and by Spanish Ministry of Economy and Competitiveness MINECO: BCAM Severo Ochoa excellence accreditation SEV-2013-0323.

\bibliographystyle{abbrv}

\end{document}